\documentclass[reqno, 12pt]{amsart}
\pdfoutput=1  
\makeatletter
\let\origsection=\section \def\section{\@ifstar{\origsection*}{\mysection}}
\def\mysection{\@startsection{section}{1}\z@{.7\linespacing\@plus\linespacing}{.5\linespacing}{\normalfont\scshape\centering\S}}
\makeatother  

\usepackage{amsmath,amssymb,amsthm}
\usepackage{mathrsfs}
\usepackage{mathabx}\changenotsign
\usepackage{dsfont}
\usepackage[dvipsnames]{xcolor}
\usepackage{wrapfig}
\usepackage{caption}
\usepackage{floatflt}
\usepackage{xcolor}
\usepackage[backref]{hyperref}
\hypersetup{
    colorlinks,
    linkcolor={red!60!black},
    citecolor={green!60!black},
    urlcolor={blue!60!black}
}
\usepackage[dvipsnames]{xcolor}
\usepackage[open,openlevel=2,atend]{bookmark}
\usepackage{multirow}
\usepackage[abbrev,msc-links,backrefs]{amsrefs}
\usepackage{doi}

\renewcommand{\PrintDOI}[1]{\doi{#1}}

\usepackage[T1]{fontenc}
\usepackage{lmodern}
\usepackage[babel]{microtype}
\usepackage[english]{babel}
\usepackage{geometry}
\numberwithin{equation}{section}
\numberwithin{figure}{section}
\usepackage{enumitem}
\def\rmlabel{\upshape({\itshape \roman*\,})}

\def\alabel{\upshape({\itshape \alph*\,})}

\def\nlabel{\upshape({\itshape \arabic*\,})}
\def\Ext{\mathrm{Ext}}
\usepackage{booktabs}

\def\ag#1{	\tikz{\def\nn{#1};
	\pgfmathsetmacro\wie{3*\nn-1};
	\pgfmathsetmacro\mm{5*\nn};
	\pgfmathsetmacro\kk{\nn-1};
	\foreach \i in {0,...,\mm}{
		\coordinate (x\i) at (\i*360/\wie:1.3cm);}
		
	\foreach \i in {0,...,\wie}{	
		\foreach \j [evaluate=\j as \k using \i+\j+\nn] in {0,...,\kk}{
			\draw[green!75!black] (x\i)--(x\k);}
	}		
	
	\foreach \i in {0,...,\wie}{	
			\draw[black, fill=black] (x\i) circle (1pt);
	}		}}

\newcommand{\mref}[1]{\ifmmode\textrm{\ref{#1}}\else\ref{#1}\fi}
\usepackage{float}
\let\polishlcross=\l
\def\l{\ifmmode\ell\else\polishlcross\fi}

\def\qand{\quad\text{and}\quad}

\let\setminus=\smallsetminus

\makeatletter
\def\moverlay{\mathpalette\mov@rlay}
\def\mov@rlay#1#2{\leavevmode\vtop{   \baselineskip\z@skip \lineskiplimit-\maxdimen
   \ialign{\hfil$\m@th#1##$\hfil\cr#2\crcr}}}
\newcommand{\charfusion}[3][\mathord]{
    #1{\ifx#1\mathop\vphantom{#2}\fi
        \mathpalette\mov@rlay{#2\cr#3}
      }
    \ifx#1\mathop\expandafter\displaylimits\fi}
\makeatother

\newcommand{\dcup}{\charfusion[\mathbin]{\cup}{\cdot}}
\newcommand{\bigdcup}{\charfusion[\mathop]{\bigcup}{\cdot}}

\DeclareFontFamily{U}  {MnSymbolC}{}
\DeclareSymbolFont{MnSyC}         {U}  {MnSymbolC}{m}{n}
\DeclareFontShape{U}{MnSymbolC}{m}{n}{
    <-6>  MnSymbolC5
   <6-7>  MnSymbolC6
   <7-8>  MnSymbolC7
   <8-9>  MnSymbolC8
   <9-10> MnSymbolC9
  <10-12> MnSymbolC10
  <12->   MnSymbolC12}{}
\DeclareMathSymbol{\powerset}{\mathord}{MnSyC}{180}

\usepackage{tikz}
\usetikzlibrary{calc,decorations.pathmorphing}
\usetikzlibrary{math}
\pgfdeclarelayer{background}
\pgfdeclarelayer{foreground}
\pgfdeclarelayer{front}
\pgfsetlayers{background,main,foreground,front}

\usepackage{subcaption}
\let\epsilon=\varepsilon
\let\eps=\epsilon
\let\rho=\varrho
\let\theta=\vartheta

\def\NN{{\mathds N}}

\def\ZZ{{\mathds Z}}

\def\CC{{\mathds C}}
\def\DD{{\mathds D}}

\newcommand{\cC}{\mathcal{C}}
\newcommand{\cD}{\mathcal{D}}

\newcommand{\cQ}{\mathcal{Q}}

\newcommand{\fD}{\mathfrak{D}}

\def\fA{\mathfrak{A}}

\newcommand{\ex}{\mathrm{ex}}

\newcommand{\Nn}{\mathrm{N}}

\newcommand{\ccF}{\mathscr{F}}

\newtheoremstyle{note}  {4pt}  {4pt}  {\sl}  {}  {\bfseries}  {.}  {.5em}          {}
\newtheoremstyle{introthms}  {3pt}  {3pt}  {\itshape}  {}  {\bfseries}  {.}  {.5em}          {\thmnote{#3}}
\newtheoremstyle{remark}  {2pt}  {2pt}  {\rm}  {}  {\bfseries}  {.}  {.3em}          {}

\theoremstyle{plain}
\newtheorem{thm}{Theorem}[section]
\newtheorem{lemma}[thm]{Lemma}
\newtheorem{prop}[thm]{Proposition}

\newtheorem{conj}[thm]{Conjecture}
\newtheorem{quest}[thm]{Question}
\newtheorem{cor}[thm]{Corollary}

\newtheorem{claim}[thm]{Claim}

\theoremstyle{note}
\newtheorem{dfn}[thm]{Definition}
\newtheorem*{dfn*}{Definition}

\theoremstyle{remark}

\usepackage{accents}

\newcommand{\re}[1]{\textcolor{red!75!black}{#1}}
\newcommand{\bl}[1]{\textcolor{blue!75!black}{#1}}
\newcommand{\gr}[1]{\textcolor{green!65!black}{#1}}

\newcommand{\Gr}{\re{\Gamma_{\mathrm{red}}}}
\newcommand{\Gg}{\gr{\Gamma_{\mathrm{green}}}}
\newcommand{\Gb}{\bl{\Gamma_{\mathrm{blue}}}}
\def\grot{\Upsilon}
\def\lra{\longleftrightarrow}
\let\phi=\varphi
\usepackage{mathtools}
\usepackage{lineno}
\newcommand*\patchAmsMathEnvironmentForLineno[1]{\expandafter\let\csname old#1\expandafter\endcsname\csname #1\endcsname
\expandafter\let\csname oldend#1\expandafter\endcsname\csname end#1\endcsname
\renewenvironment{#1}{\linenomath\csname old#1\endcsname}{\csname oldend#1\endcsname\endlinenomath}}\newcommand*\patchBothAmsMathEnvironmentsForLineno[1]{\patchAmsMathEnvironmentForLineno{#1}\patchAmsMathEnvironmentForLineno{#1*}}\AtBeginDocument{\patchBothAmsMathEnvironmentsForLineno{equation}\patchBothAmsMathEnvironmentsForLineno{align}\patchBothAmsMathEnvironmentsForLineno{flalign}\patchBothAmsMathEnvironmentsForLineno{alignat}\patchBothAmsMathEnvironmentsForLineno{gather}\patchBothAmsMathEnvironmentsForLineno{multline}}
\usepackage{multicol}
\usepackage{todonotes}

\newcommand{\expl}[1]\relax

\usepackage{marginnote}
 
  \newcommand{\ch}[1]{\marginnote{{\bf TL}: #1}}

\newcommand{\G}{\Gamma}
\def\Aut{\mathrm{Aut}}
\def\dd{\mathrm{d}}
\let\vn=\varnothing

\newcommand{\Dp}[1]{$\cD_{#1}$} 
\newcommand{\Qp}[1]{$\cQ_{#1}$}  

\begin{document}
\title[Strong Brandt-Thomass\'e Theorems]{Strong Brandt-Thomass\'e Theorems}
\author[T.~\L uczak]{Tomasz \L uczak}
\address{Adam Mickiewicz University, Faculty of Mathematics and Computer Science, Pozna\'n, Poland}
\email{tomasz@amu.edu.pl}
\author[J.~Polcyn]{Joanna Polcyn}
\address{Adam Mickiewicz University, Faculty of Mathematics and Computer Science, Pozna\'n, Poland}
\email{joaska@amu.edu.pl}
\thanks{The first author is partially supported by National Science Centre, Poland, grant 2022/47/B/ST1/01517.}
\author[Chr.~Reiher]{Christian Reiher}
\address{Fachbereich Mathematik, Universit\"at Hamburg, Hamburg, Germany}
\email{Christian.Reiher@uni-hamburg.de}

\begin{abstract}
	Solving a long standing conjecture of Erd\H{o}s and Simonovits, Brandt 
	and Thomass\'e proved that the chromatic number of each triangle-free 
	graph $G$ such that $\delta(G)>|V(G)|/3$ is at most four. 
	In fact, they showed the much stronger result that every maximal triangle-free 
	graph $G$ satisfying this minimum degree condition is a blow-up of either an 
	Andr\'asfai or a Vega graph.
	
	Here we establish the same structural conclusion on $G$ under the weaker 
	assumption that for $m\in\{2, 3, 4\}$ every sequence of $3m$ vertices has 
	a subsequence of length $m+1$ with a common neighbour. In forthcoming work 
	this will be used to solve an old problem of Andr\'asfai in Ramsey-Tur\'an 
	theory.   	
\end{abstract}

\maketitle

\section{Introduction}\label{sec:intro}
\subsection{Minimum degree conditions}
One of the earliest results of modern graph theory is Mantel's theorem~\cite{M} 
from~1907 on the maximum number of edges in a triangle-free graph. 
The interest in the structure of dense triangle-free graphs has 
been revived in the early seventies. Andr\'asfai, Erd\H{o}s, and S\'os~\cite{AES} 
observed that all triangle-free graphs~$G$ with~$n$ vertices and minimum 
degree $\delta(G)>2n/5$ are bipartite. 
Moreover, Hajnal constructed a family of triangle-free graphs $G$ with minimum 
degree $(1/3-o(1))|V(G)|$ and arbitrarily large chromatic number. 
Inspired by this example Erd\H{o}s and Simonovits~\cite{ES73} conjectured in~1973  
that every triangle-free graph $G$ with $n$ vertices and 
minimum degree $\delta(G)>n/3$ is three-colourable.
In fact, at this time the only known maximal triangle-free graphs~$G$ 
on~$n$ vertices with minimum degree $\delta(G)>n/3$ were blow-ups 
of Andr\'asfai graphs, 
which were introduced by Andr\'asfai~\cite{A} a few years earlier. 
For every positive integer~$k$ there is an Andr\'asfai graph $\G_k$ with vertex 
set $\ZZ/(3k-1)\ZZ$ and all edges $ij$ such that $i-j\in \{k, k+1, \dots, 2k-1\}$. 
Hence, we have $\G_1=K_2$, $\G_2 = C_5$, and Figure~\ref{fig:ag} shows some further 
Andr\'asfai graphs. By a {\it blow-up} of a given graph $G$ we mean another graph 
obtained by replacing each vertex of $G$ by a non-empty independent set of vertices 
and each edge of $G$ by the complete bipartite graph between 
the vertex classes corresponding to its end vertices.

\begin{figure}[h]
	\centering
	\begin{multicols}{4}
		\ag{3} \\ \ag{4} \\ \ag{5} \\ \ag{6}
	\end{multicols}
	\caption{The Andr\'asfai graphs $\G_3$, $\G_4$,  $\G_5$, and $\G_6$.}
	\label{fig:ag}
\end{figure}
  
H\"aggkvist~\cite{H} refuted the Erd\H{o}s-Simonovits conjecture in the 
early eighties. His counterexample is a blow-up of a certain triangle-free 
graph $\grot$ on eleven vertices with chromatic number four, often called 
the {\it Mycielski graph}~\cite{Myc} by Polish authors (see Figure~\ref{fig:graphsB}) 
or the {\it Gr\"otzsch graph} by German authors (see Figure~\ref{fig:graphsC}).  
An appropriate choice of `weights' indicated in Figure~\ref{fig:graphsD}
leads to blow-ups $\hat\grot$ on $n$ vertices 
with $\delta(\hat\grot)\ge 10n/29>n/3$.  

Later work of Chen, Jin, and Koh~\cite{CJK} showed that containing $\grot$ 
is the only possible obstruction to satisfying the Erd\H{o}s-Simonovits 
conjecture. 
More precisely, all $\{K_3, \grot\}$-free graphs~$G$ on~$n$ vertices 
with $\delta(G)>n/3$ are contained in blow-ups of Andr\'asfai graphs and, 
therefore, three-colourable.

\begin{figure}[h!]
\begin{subfigure}[b]{.18\textwidth}
\centering
\begin{tikzpicture}[scale=.7]
\coordinate (c) at (0,0);
\foreach \i in {1,...,5} {
	\coordinate (a\i) at (18+\i*72:1);
	\coordinate (b\i) at (18+\i*72:2);	
}
\fill (c) circle (2pt);
\foreach \i in {1,...,5} {
	\draw (c)--(a\i);
	\fill (a\i) circle (2pt);		
	\fill (b\i) circle (2pt);
}
\draw (b1)--(b2)--(b3)--(b4)--(b5)--(b1);
\draw (b1)--(a2)--(b3)--(a4)--(b5)--(a1)--(b2)--(a3)--(b4)--(a5)--(b1);
\phantom{\draw (0,-1.8) circle (1pt);}
\end{tikzpicture}
\caption{Mycielski}\label{fig:graphsB}
\end{subfigure}
\hfill
\begin{subfigure}[b]{.2\textwidth}
\centering
\begin{tikzpicture}[scale=.7]
\coordinate (c) at (0,0);
\foreach \i in {1,...,5} {
	\coordinate (a\i) at (-18+\i*72:1);
	\coordinate (b\i) at (18+\i*72:1.5);	
	\coordinate (d\i) at (18+\i*72:2);
}
\fill (c) circle (2pt);
\foreach \i in {1,...,5} {
	\draw (c)--(a\i);
	\fill (a\i) circle (2pt);		
	\fill (b\i) circle (2pt);
}
\draw (a1)--(b1)--(a2)--(b2)--(a3)--(b3)--(a4)--(b4)--(a5)--(b5)--(a1);
\draw (b1) [out=162, in=72] to (d2) [out=-108, in=162] to (b3);
\draw (b2) [out=234, in=144] to (d3) [out=-36, in=234] to (b4);
\draw (b3) [out=-54, in=216] to (d4) [out=36, in=-54] to (b5);
\draw (b4) [out=18, in=-72] to (d5) [out=108, in=18] to (b1);
\draw (b5) [out=90, in=0] to (d1) [out=180, in=90] to (b2);
\end{tikzpicture}
\caption{Gr\"otzsch}
\label{fig:graphsC}
\end{subfigure}
\hfill
\begin{subfigure}[b]{.2\textwidth}
\centering
\begin{tikzpicture}[scale=.5]			
\coordinate (c) at (0,0);
\foreach \i in {1,...,5} {
	\coordinate (a\i) at (-18+\i*72:1.6);
	\coordinate (b\i) at (18+\i*72:2.3);	
	\coordinate (d\i) at (18+\i*72:3);
}
\draw [line width=6pt, black!20] (a1)--(b1)--(a2)--(b2)--(a3)--(b3)--(a4)--(b4)--(a5)--(b5)--(a1);
\draw (b1) [line width=6pt, black!20,out=162, in=72] to (d2) [out=-108, in=162] to (b3);
\draw (b2) [line width=6pt, black!20,out=234, in=144] to (d3) [out=-36, in=234] to (b4);
\draw (b3) [line width=6pt, black!20,out=-54, in=216] to (d4) [out=36, in=-54] to (b5);
\draw (b4) [line width=6pt, black!20,out=18, in=-72] to (d5) [out=108, in=18] to (b1);
\draw (b5) [line width=6pt, black!20,out=90, in=0] to (d1) [out=180, in=90] to (b2);
\foreach \i in {1,...,5} {
	\draw[line width=6pt, black!20] (c)--(a\i);
	\draw [fill=white](a\i) circle (7pt);		
	\draw [fill=white](b\i) circle (7pt);
	\fill [blue]($(b\i) +(.1,-.06)$) circle (2pt);
	\fill [blue]($(b\i) +(-.1,-.06)$) circle (2pt);
	\fill [blue]($(b\i) +(0,.1)$) circle (2pt);
	\fill [blue]($(a\i) +(.1,0)$) circle (2pt);
	\fill[blue] ($(a\i) +(-.1,0)$) circle (2pt);
}
\draw [fill=white](c) circle (7pt);
\fill [blue](-.1,.1) circle (2pt);
\fill [blue](-.1,-.1) circle (2pt);
\fill [blue](.1,.1) circle (2pt);
\fill [blue](.1,-.1) circle (2pt);
\end{tikzpicture}
\caption{H\"aggkvist}\label{fig:graphsD}
\end{subfigure}\hfill
\begin{subfigure}[b]{.24\textwidth}
\centering
\begin{tikzpicture}[scale=1]			
\def\pro{1.3cm};
\def\pu{1pt};
\coordinate (v) at (30:\pro);
\coordinate (a) at (90:\pro);
\coordinate (w) at (150:\pro);
\coordinate (b) at (210:\pro);
\coordinate (u) at (270:\pro);
\coordinate (c) at (330:\pro);				
\coordinate (x) at (\pro*1.4,\pro*.9);
\coordinate (y) at (\pro*1.4,-\pro*.9);			
\foreach \i in {1,...,8} \coordinate (v\i) at (135-\i*45:\pro/2);
\draw [red!80!black, ,domain=5:89, dashed, thick] plot ({\pro*.6*cos(\x)},{\pro*.6*sin(\x)}); 
\draw [blue, dashed, domain=142:178, thick] plot ({\pro*.6*cos(\x)},{\pro*.6*sin(\x)}); 
\draw [green!80!black, dashed,domain=230:313, thick] plot ({\pro*.6*cos(\x)},{\pro*.6*sin(\x)}); 
\draw [thick] (x)--(y);
\draw (c)--(x)--(a);
\draw (v)--(y)--(u);
\draw [rounded corners=15] (x) -- (-.2*\pro,\pro*1.2)--(-\pro*1.2,\pro*.5)--(b);
\draw [rounded corners=15] (y) -- (-.2*\pro,-\pro*1.2)--(-\pro*1.2,-\pro*.5)--(w);		
\foreach \i in {4,...,6} \draw (v\i)--(v1);
\foreach \i in {5,...,7} \draw (v\i)--(v2);
\foreach \i in {6,...,8} \draw (v\i)--(v3);
\foreach \i in {3,...,5} \draw  (v\i)--(v8);
\draw (v4)--(v7);
\draw  (a)--(v)--(c)--(u)--(b)--(w)--(a);
\foreach \i in {a,b,c, v, u, w, x, y} \draw [thick] (\i) circle (\pu);
\foreach \i in {1,...,8}\draw [thick] (v\i) circle (\pu);
\fill (x) circle (\pu);
\fill (y) circle (\pu);
\foreach \i in {c, w, v7, v8}\fill[blue!75!white]  (\i) circle (\pu);
\foreach \i in {a,u, v1, v2, v3}\fill[red!75!white]  (\i) circle (\pu);
\foreach \i in {v,b, v4, v5, v6}\fill[green!75!white]  (\i) circle (\pu);
\node at ($(a)+(0,-.15)$) {\tiny{$\re{a}$}};
\node at ($(b)+(.15,.1)$) {\tiny{$\gr{b}$}};
\node at ($(c)+(-.15,.1)$) {\tiny{$\bl{c}$}};
\node at ($(v)+(-.15,-.1)$) {\tiny{$\gr{v}$}};
\node at ($(w)+(.15,-.1)$) {\tiny{$\bl{w}$}};
\node at ($(u)+(0,.15)$) {\tiny{$\re{u}$}};
\node at ($(x)+(0,.15)$) {\tiny{$x$}};
\node at ($(y)+(0,-.15)$) {\tiny{$y$}};
\node  [scale=.7] at ($(v1)+(-.1,.1)$) {\tiny{$\re{0}$}};
\node [scale=.7] at ($(v3)+(.1,-.13)$) {\tiny{$\re{i-1}$}};
\node [scale=.7] at ($(v4)+(.1,.05)$) {\tiny{$\gr{i}$}};
\node [scale=.7] at ($(v6)+(-.05,-.1)$) {\tiny{$\gr{2i-1}$}};
\node [scale=.7] at ($(v7)+(-.1,-.15)$) {\tiny{$\bl{2i}$}};
\node [scale=.7] at ($(v8)+(0,.1)$) {\tiny{$\bl{3i-2}$}};
\end{tikzpicture}
\caption{Vega}
\label{fig:graphsE}
\end{subfigure}
\caption{Some relevant graphs.}
\label{fig:graphs}
\end{figure}

Brandt and Pisanski~\cite{BP}, on the other hand, discovered that the 
Mycielski-Gr\"otzsch graph starts a new sequence of four-chromatic triangle-free 
graphs, which they called Vega graphs; they have chromatic number four and 
admit (regular) blow-ups violating the Erd\H{o}s-Simonovits conjecture. 
Following some further work on triangle-free graphs of large minimum 
degree (see, e.g.,~\cites{L, T}) Brandt and Thomass\'e then proved in 
an unpublished manuscript~\cite{BT} that every maximal triangle-free graph~$G$ 
on~$n$ vertices with $\delta(G)>n/3$ is a blow-up of either an Andr\'asfai graph 
or a Vega graph. It follows that all such graphs are four-colourable, which 
establishes a relaxed version of the Erd\H{o}s-Simonovits conjecture. 

Given their importance, we would briefly like to describe {\it Vega graphs} here. 
For every integer $i\ge 2$ the graph $\grot_i^{00}$, shown in 
Figure~\ref{fig:graphsE}, consists of an inner Andr\'asfai graph $\G_i$, 
an external hexagon $\cC_6=\re{a}\gr{v}\bl{c}\re{u}\gr{b}\bl{w}$, and two outer 
vertices~$x$,~$y$ joined to each other and to~$\cC_6$ as in the picture. 
Moreover, the vertices of~$\cC_6$ are connected to the vertices of~$\G_i$ of 
the same colour (\re{red}, \gr{green}, or \bl{blue}). There are further Vega graphs 
$\grot_i^{10}$, $\grot_i^{01}$, and $\grot_i^{11}$ obtained from $\grot_i^{00}$
by deleting one or both of $y$ and $\gr{2i-1}$. For instance, 
$\grot_2^{11}=\grot_2^{00}-\{i, \gr{3}\}$ is isomorphic to the Mycielski-Gr\"otzsch
graph $\grot$ (see Figure~\ref{fig:43}). A more detailed definition of Vega graphs 
will be given at the beginning of Section~\ref{sec:vega}. 

\subsection{Existence of common neighbours}
Our main result is similar to the Brandt-Thomass\'e theorem, but instead of 
a minimum degree hypothesis we shall use an assumption on the existence of 
common neighbours. The motivation for studying such problems is another 
conjecture on dense triangle-free graphs due to Andr\'asfai. In his already 
referenced article~\cite{A} he proposes to investigate the largest number 
$\ex(n, s)$ of edges that a triangle-free graph on~$n$ vertices can have if its 
independence number is at most $s$. So for $s\ge n/2$ we have $\ex(n, s)=\lfloor n^2/4\rfloor$
by Mantel's theorem. After proving that for~$s\in (2n/5, n/2]$ certain 
blow-ups of~$\G_2$ are optimal, Andr\'asfai conjectured that 
for every~$s>n/3$ the maximum~$\ex(n, s)$ is achieved by an appropriate blow-up 
of some Andr\'asfai graph. His work is the first contribution to a branch of extremal 
graph theory nowadays called {\it Ramsey-Tur\'an theory}.  
For some recent partial results on Andr\'asfai's conjecture we refer 
to~\cites{Vega, LPR2, LPR3}. An excellent survey by S\'os and 
Simonovits~\cite{SS} provides further background on Ramsey-Tur\'an theory.

In a forthcoming article we plan to resolve Andr\'asfai's conjecture in the 
sense of establishing 
\begin{equation}\label{eq:LPR}
	\ex(n, s)
	=
	\frac12k(k-1)n^2-k(3k-4)ns+\frac12(3k-4)(3k-1)s^2
\end{equation}
whenever $s\in (n/3, n/2]$ and $k=\lceil s/(3s-n)\rceil$. As explained 
in~\cite{Vega} there is always a blow-up of~$\G_k$ achieving equality 
and for some values of~$k$ there are (perhaps unexpected) blow-ups of 
Vega graphs for which equality holds as well. 
There will be one step in the proof of~\eqref{eq:LPR}, where we want to infer that 
some auxiliary graph~$\ccF$ admits a homomorphism into some `well-behaved' graph, 
such as an Andr\'asfai or Vega graph. This graph~$\ccF$ is always triangle-free,
but it can have vertices of small degree. Thus we need to prove a version 
of the Brandt-Thomass\'e theorem under an assumption which will turn out to hold
in our intended application, and this is what we shall do here. The alternative hypothesis is of the following form.  

\begin{dfn}\label{df:hDk}
	A graph $G$ has property \Dp{k} for some $k\in \NN$ if for every $m\in [k]$ 
	and every sequence $x_1, \dots, x_{3m}\in V(G)$ of (not necessarily distinct) 
	vertices of $G$ there  is a vertex $y\in V(G)$ such that 
		\[
		|\{i \in [3m]\colon x_iy\in E(G)\}|\ge m+1\,.
	\]
	\end{dfn}

A simple counting argument discloses that every graph $G$ on $n$ vertices 
with $\delta(G)> n/3$ has the property \Dp{k} for every $k\ge 1$. One advantage 
of these properties, however, is that they are preserved under taking   
blow-ups. That is, a graph $G$ satisfies \Dp{k} if and only if all its  
blow-ups do. Another feature of \Dp{k} is that---in contrast to the minimum 
degree condition---it is a sensible property of infinite graphs. 
We offer some further remarks on this topic in the last section, 
but throughout the main body of this article we shall tacitly assume that our 
graphs are finite.  

\begin{thm}\label{thm:D}
	A maximal triangle-free graph satisfies \Dp{4} if and only if 
	it is a blow-up of either an Andr\'asfai or a Vega graph. 
\end{thm}

Therefore the class of blow-ups of Andr\'asfai and Vega graphs is definable 
by a single first order property of graphs, which seems somewhat surprising 
to us. Andr\'asfai and Vega graphs themselves are then definable as twin-free
graphs in this class, which is another first-order property. Next, the 
difference between Andr\'asfai and Vega graphs is that the former 
are $\grot$-free, while the latter contain $\grot$. We thus arrive at the 
astonishing conclusion that both the class of Andr\'asfai graphs and the class 
of Vega graphs are definable by a first-order sentence in the language of graph
theory. 

The proof of Theorem~\ref{thm:D} begins with a case distinction 
whether the given graph $G$ is $\grot$-free or not. If it is, we look at 
a maximal Andr\'asfai subgraph~$\G_k$ of~$G$ and show that~$G$ is a blow-up 
of~$\G_k$. Similarly, if $\grot\subseteq G$ we take a maximal Vega subgraph 
of~$G$ and argue that~$G$ is a blow-up thereof. It turns out that in the former 
case the property~\Dp{3} rather than~\Dp{4} suffices. As it reflects the  
historical progress made by Chen, Jin, Koh~\cite{CJK} and 
Brandt, Thomass\'e~\cite{BT}, we would like to state this fact separately.

\begin{thm}\label{thm:1.3}
	Let $G$ be a maximal triangle-free graph.	\begin{enumerate}[label=\alabel]
		\item\label{it:1.3a} If $\grot\nsubseteq G$ and $G$ satisfies \Dp{3}, 
			then $G$ is a blow-up of some Andr\'asfai graph. 
		\item\label{it:1.3b} If $\grot\subseteq G$ and $G$ satisfies \Dp{4}, 
			then $G$ is a blow-up of some Vega graph. 
	\end{enumerate}
\end{thm}

Finally it should be pointed out that many, but presumably not all, steps 
in the earlier works~\cites{CJK, BT} use only~\Dp{4} rather than the 
full force of $\delta(G) > n/3$. So there is some overlap between our 
proof and the arguments employed by Chen, Jin, and Koh~\cite{CJK}, and 
by Brandt and Thomass\'e~\cite{BT}.   

\subsection*{Organisation}
In the next section we introduce the central concepts of our approach and 
provide a brief description of important intermediate steps. 
The proofs of the parts~\ref{it:1.3a} and~\ref{it:1.3b} of 
Theorem~\ref{thm:1.3} will then be completed in Section~\ref{sec:A} and
Section~\ref{sec:vega}, respectively. We conclude by mentioning an even 
stronger version of Theorem~\ref{thm:1.3} and discussing some problems 
for further research in Section~\ref{sec:final}. 

\section{Preliminaries}\label{sec:prelim}
We follow standard graph theoretic notation. Given a graph $G$ we denote its sets 
of vertices and edges by $V(G)$ and $E(G)$, respectively. 
For brevity we often write $xy\in E(G)$ instead of $\{x,y\}\in E(G)$. 
By $\Nn(v)$ we mean the neighbourhood of $v\in V(G)$. 
The graph obtained from $G$ by removing a vertex $v$ together with all incident 
edges is denoted by~${G-v}$. If $H$ is a subgraph of $G$ 
and $v\in V(G)\setminus V(H)$, then $H+v$ refers to the graph
\[
	\bigl(V(H)\cup \{v\}, E(H)\cup \{vx\in E(G)\colon x\in V(H)\}\bigr)\,.
\]

Two vertices $v$, $w$ are called {\it twins} if they  have the same 
neighbourhood, i.e., $\Nn(v)=\Nn(w)$. Since a maximal triangle-free graph~$G$ 
has property~\Dp{k} if and only if all its blow-ups have this 
property, it would suffice to prove our main result for twin-free graphs~$G$. 
But in order to detect twins in the graphs $G$ under consideration we shall work 
with the following slightly more general concept.   
	
\begin{dfn}\label{df:Otwin}
	Let $H$ be a subgraph of $G$. If $q\in V(H)$ and $q'\in V(G)$ satisfy 
		\[
		\Nn(q)\cap V(H) = \Nn(q')\cap V(H)\,,
	\]
		then $q'$ is called an {\it $H$-twin} of $q$ (see Figure~\ref{fig:new_23A}).
\end{dfn}

Notice that every $q\in V(H)$ is an $H$-twin of itself. Moreover, 
if $q'\not\in V(H)$, then the graph $H-q+q'$, which will be denoted by $H(q')$ 
in the sequel, is isomorphic to $H$. Now we are ready to define a central concept 
of our approach.

\begin{dfn}\label{df:twin}
	Let $F$ and $G$ be two graphs. 
	\begin{enumerate}[label=\alabel]
		\item\label{it:twina} For an edge $e\in E(F)$ we say that $G$ has 
			the {\it $(F,e)$-twin property} if the following holds: 
			\begin{quotation}
				If $H$ is a subgraph of $G$ isomorphic to $F$, the edge $qz\in E(H)$ 
				corresponds to~$e$, and $q',z'\in V(G)$ are $H$-twins of $q$, $z$, 
				then $q'z'$ is an edge of $G$ (see Figure \ref{fig:new_23B}).
			\end{quotation}
		\item\label{it:twinb} If $G$ has the $(F,e)$-twin property for 
			every $e\in E(F)$, we say that $G$ has the {\it $F$-twin property}. 
	\end{enumerate}
\end{dfn}

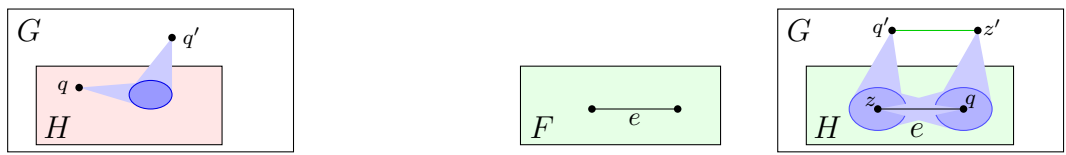
\begin{figure}[ht]
\centering
\begin{subfigure}[b]{0.47\textwidth}
\centering
\begin{tikzpicture}[scale=.95]
	\draw (-1,1) rectangle (3,-1);
	\node at (-.7,.7) {$G$};
	\draw [fill=red!10] (-.6,.2) rectangle (2,-.9);
	\node at (-.3,-.65) {$H$};
	\coordinate (q1) at (0,-.1);
	\coordinate (q2) at (1.3,.6);
	\fill [blue!20] (1,-.02)--(q1)--(1,-.4);
	\fill [blue!20] (.66,-.2) -- (q2)--(1.3, -.2);
	\draw [blue!90!black,fill = blue!40](1,-.2) ellipse (.3cm and .2cm);
	\foreach \i in {1,2} \fill (q\i) circle (1.5pt);
	\node [left] at (q1) {\tiny $q$};
	\node [right] at (q2) {\tiny $q'$};
\end{tikzpicture}
\caption{The vertex $q'$ is an $H$-twin of $q$}
\label{fig:new_23A} 
\end{subfigure}
\hfill
\begin{subfigure}[b]{0.47\textwidth}
	\centering
	\begin{tikzpicture}[scale=.95]
			\coordinate (q2) at (.6,.7);
		\coordinate (z2) at (1.8,.7);
	\draw [fill=green!10] (-4.6,.2) rectangle (-1.8,-.9);
	\node at (-4.3,-.65) {$F$};
	\fill (-3.6,-.4) circle (1.5pt);
	\fill (-2.4,-.4) circle (1.5pt);
	\draw (-3.6,-.4)--(-2.4,-.4);
	\node at (-3, -.55) {\footnotesize $e$};
	\draw (-1,1) rectangle (3,-1);
	\node at (-.7,.7) {$G$};
	\draw [fill=green!10] (-.6,.2) rectangle (2.3,-.9);
	\node at (-.3,-.65) {$H$};
	\coordinate (q1) at (.4,-.4);
	\coordinate (z1) at (1.6,-.4);
		\fill [ blue!20] (1.5,-.1)--(q1)--(1.5,-.7);
	\fill [blue!20] (.4,-.1)--(z1)--(.4,-.7);
	\fill [blue!20] (.01,-.31) -- (q2)--(.75, -.31);
	\fill [blue!20] (1.2,-.35) -- (z2)--(2, -.35);
	\draw [blue!70,fill = blue!30](.4,-.4) ellipse (.4cm and .3cm);
	\draw [blue!70,fill = blue!30](1.6,-.4) ellipse (.4cm and .3cm);
		\fill [blue!20] (1.2,-.25)--(q1)--(1.2,-.6);
	\fill [blue!20] (.9,-.3)--(z1)--(.9,-.55);
		\node at (.95, -.7) {$e$};
	\node at ($(q1)+(-.1,.1)$) {\tiny $z$};
	\node at ($(z1)+(.1,.1)$) {\tiny $q$};
	\node at ($(q2)+(-.15,.05)$) {\tiny $q'$};
	\node at ($(z2)+(.2,.05)$) {\tiny $z'$};
	\draw [green!80!black] (q2)--(z2);
		\draw (q1)--(z1);
	\foreach \i in {q1,z1,q2, z2} \fill (\i) circle (1.5pt);
\end{tikzpicture}
\caption{The graph $G$ has the $(F,e)$-twin property}
\label{fig:new_23B} 
\end{subfigure}
\caption{The concepts in Definitions~\ref{df:Otwin} and~\ref{df:twin}.}
\label{fig:new_23}
\end{figure}

Under some mild assumptions that will often be satisfied in what follows, 
the next lemma asserts that if we want to prove 
a vertex $q'$ to be an $H$-twin of another vertex $q\in V(H)$, then we 
can temporarily replace some vertex $z\in V(H-q)$ by any of its $H$-twins $z'$. 

\begin{lemma}\label{f:24}
	Suppose that $F$ is a maximal triangle-free graph and $G$ is a triangle-free 
	graph possessing the $F$-twin property. Let $H$ be a copy of $F$ in $G$, let 
	$z'\in V(G)\setminus V(H)$ be an $H$-twin of $z\in V(H)$, and set $H'=H(z')$.
	
	If a vertex $q'$ is an $H'$-twin of $q\in V(H)\setminus \{z\}$, then it is 
	an $H$-twin of $q$ as well. 
\end{lemma}

\begin{proof}
	Knowing $\Nn(q')\cap V(H') = \Nn(q)\cap V(H')$ we want to show 
	\[
		\Nn(q')\cap V(H) = \Nn(q)\cap V(H)\,.
	\]
		So we only need to establish that $z$ is adjacent to either both or none 
	of $q$, $q'$. 
	
	Assume first that $qz\notin E(G)$. Since $H$ is maximal triangle-free, there 
	is a vertex $y\in V(H)$ with $yq, yz \in E(G)$ (see Figure~\ref{fig:new_24A}). 
	Because $q'$ is an $H'$-twin of $q$ we obtain $yq'\in E(G)$ (see 
	Figure~\ref{fig:new_24B}) and the hypothesis $K_3\not\subseteq G$ 
	yields indeed $q'z\not\in E(G)$.
	
	Suppose next that $qz \in E(G)$, which entails $qz'\in E(G)$, since $z'$ 
	is an $H$-twin of $z$ (see Figure~\ref{fig:new_24C}). 
	Due to $H'\cong H\cong F$  the $F$-twin property applies to the 
	$H'$-twins~$q'$,~$z$ of~$q$,~$z'$. This shows that $q'z$ is an edge of $G$,
	thereby completing the proof (see Figure~\ref{fig:new_24D}).
 \end{proof}
	
\begin{figure}[ht]
\centering
\begin{subfigure}[b]{0.22\textwidth}
\centering
\begin{tikzpicture}[scale=.85]
	\draw (-1,1) rectangle (3,-1.1);
	\node at (-.7,.7) {$G$};
	\draw [fill=green!10] (-.6,.2) rectangle (2.5,-1);
	\node at (-.3,-.65) {$H'$};
	\coordinate (q1) at (.5,-.8);
	\coordinate (z1) at (1.5,-.8);
	\coordinate (y) at (.8,-.2);
	\node [right] at (q1) {\tiny $z'$};
	\node [right] at (z1) {\tiny $q$};
	\coordinate (q2) at (.5,.7);
	\coordinate (z2) at (1.5,.7);
	\node [left] at (q2) {\tiny $z$};
	\node [right] at (z2) {\tiny $q'$};
	\fill [blue!20] (0,-.2)--(q1)--(1.1,-.2);
	\fill [blue!20] (0,-.2) -- (q2)--(1.2, -.2);
	\fill [blue!20] (.7,-.2)--(z1)--(1.8,-.2);
	\fill [blue!20] (.7,-.2) -- (z2)--(1.8, -.2);
	\fill [white](.6,-.2) ellipse (.6cm and .2cm);
	\fill [white](1.2,-.2) ellipse (.6cm and .2cm);
	\draw [opacity = .5, blue!90!black,fill = blue!40](.6,-.2) ellipse (.6cm and .2cm);
	\draw [opacity = .5,blue!90!black,fill = blue!40](1.2,-.2) ellipse (.6cm and .2cm);
	\foreach \i in {q1,z1,q2, z2, y} \fill (\i) circle (1.5pt);
	\node [right] at (y) {\tiny $y$};
	\draw (q2)--(y)--(z1);
\end{tikzpicture}
\vskip -.1cm
\caption{}
\label{fig:new_24A} 
\end{subfigure}
\hfill
\begin{subfigure}[b]{0.22\textwidth}
\centering
\begin{tikzpicture}[scale=.85]
	\draw (-1,1) rectangle (3,-1.1);
	\node at (-.7,.7) {$G$};
	\draw [fill=green!10] (-.6,.2) rectangle (2.5,-1);
	\node at (-.3,-.65) {$H'$};
	\coordinate (q1) at (.5,-.8);
	\coordinate (z1) at (1.5,-.8);
	\coordinate (y) at (.8,-.2);
	\node [right] at (q1) {\tiny $z'$};
	\node [right] at (z1) {\tiny $q$};
	\coordinate (q2) at (.5,.7);
	\coordinate (z2) at (1.5,.7);
	\node [left] at (q2) {\tiny $z$};
	\node [right] at (z2) {\tiny $q'$};
	\fill [blue!20] (0,-.2)--(q1)--(1.1,-.2);
	\fill [blue!20] (0,-.2) -- (q2)--(1.2, -.2);
	\fill [blue!20] (.7,-.2)--(z1)--(1.8,-.2);
	\fill [blue!20] (.7,-.2) -- (z2)--(1.8, -.2);
	\fill [white](.6,-.2) ellipse (.6cm and .2cm);
	\fill [white](1.2,-.2) ellipse (.6cm and .2cm);
	\draw [opacity = .5, blue!90!black,fill = blue!40](.6,-.2) ellipse (.6cm and .2cm);
	\draw [opacity = .5,blue!90!black,fill = blue!40](1.2,-.2) ellipse (.6cm and .2cm);
	\foreach \i in {q1,z1,q2, z2, y} \fill (\i) circle (1.5pt);
	\node [right] at (y) {\tiny $y$};
	\draw (q2)--(y)--(z1);
	\draw (y)--(z2);
\end{tikzpicture}
\vskip -.1cm
\caption{}
\label{fig:new_24B} 
\end{subfigure}
\hfill
\begin{subfigure}[b]{0.22\textwidth}
\centering
\begin{tikzpicture}[scale=.85]
	\draw (-1,1) rectangle (3,-1.1);
	\node at (-.7,.7) {$G$};
	\draw [fill=green!10] (-.6,.2) rectangle (2.5,-1);
	\node at (-.3,-.65) {$H'$};
	\coordinate (q1) at (1.5,-.3);
	\coordinate (z1) at (.4,-.3);
	\coordinate (q2) at (.5,.7);
	\coordinate (z2) at (1.5,.7);
	\node [left] at (q2) {\tiny $z$};
	\node [right] at (z2) {\tiny $q'$};
	\fill [blue!20] (1.5,0)--(z1)--(1.5,-.6);
	\fill [blue!20] (.4,0)--(q1)--(.4,-.6);
	\fill [blue!20] (0,-.2) -- (q2)--(.75, -.2);
	\fill [blue!20] (1.2,-.2) -- (z2)--(2, -.2);
	\draw [blue!70,fill = blue!30](.4,-.3) ellipse (.4cm and .3cm);
	\draw [blue!70,fill = blue!30](1.6,-.3) ellipse (.4cm and .3cm);
	\fill [blue!20] (1.2,-.1)--(z1)--(1.2,-.5);
	\fill [blue!20] (.9,-.15)--(q1)--(.9,-.45);
	\draw (q2)--(z1)--(q1);
	\foreach \i in {q1,z1,q2, z2} \fill (\i) circle (1.5pt);
	\node at ($(z1)+(-.15,0)$) {\tiny $q$};
	\node [right] at (q1) {\tiny $z'$};
\end{tikzpicture}
\vskip -.1cm
\caption{}
\label{fig:new_24C} 
\end{subfigure}
\hfill
\begin{subfigure}[b]{0.22\textwidth}
\centering
\begin{tikzpicture}[scale=.85]
	\draw (-1,1) rectangle (3,-1.1);
	\node at (-.7,.7) {$G$};
	\draw [fill=green!10] (-.6,.2) rectangle (2.5,-1);
	\node at (-.3,-.65) {$H'$};
	\coordinate (q2) at (.5,.7);
	\coordinate (z2) at (1.5,.7);
	\node [left] at (q2) {\tiny $z$};
	\node [right] at (z2) {\tiny $q'$};
	\fill [ blue!20] (1.5,0)--(z1)--(1.5,-.6);
	\fill [blue!20] (.4,0)--(q1)--(.4,-.6);
	\fill [blue!20] (0,-.2) -- (q2)--(.75, -.2);
	\fill [blue!20] (1.2,-.2) -- (z2)--(2, -.2);
	\draw [blue!70,fill = blue!30](.4,-.3) ellipse (.4cm and .3cm);
	\draw [blue!70,fill = blue!30](1.6,-.3) ellipse (.4cm and .3cm);
	\fill [blue!20] (1.2,-.1)--(z1)--(1.2,-.5);
	\fill [blue!20] (.9,-.15)--(q1)--(.9,-.45);
	\draw (z2)--(q2)--(z1)--(q1);
	\foreach \i in {q1,z1,q2, z2} \fill (\i) circle (1.5pt);
	\node at ($(z1)+(-.15,0)$) {\tiny $q$};
	\node [right] at (q1) {\tiny $z'$};
\end{tikzpicture}
\vskip -.1cm
\caption{}
\label{fig:new_24D} 
\end{subfigure}
\vskip -.3cm
\caption{The proof of Lemma~\ref{f:24}.}
\label{fig:new_24}
\end{figure}
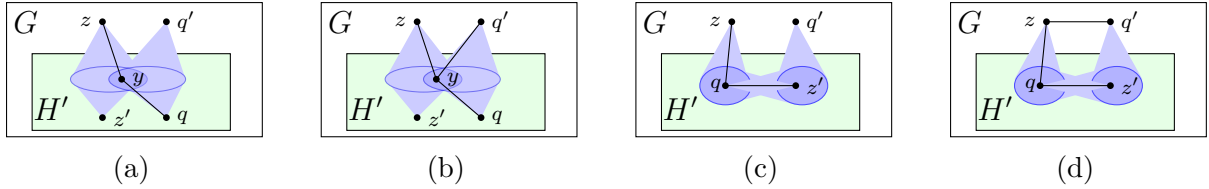
 
Let us now explain our strategy for proving that a given graph is a 
blow-up of one of its subgraphs. 

\begin{lemma}\label{lem:2407}
	Let $G$ be a triangle-free graph and let $\Omega$ be a twin-free subgraph 
	of $G$ which is maximal triangle-free. If 
		\begin{enumerate}[label=\rmlabel]
		\item\label{it:twin} $G$ has the $\Omega$-twin property
		\item\label{it:attach} and every vertex of $G$ is an $\Omega$-twin 
			of some vertex of $\Omega$, 
	\end{enumerate}
		then $G$ is a blow-up of $\Omega$.
\end{lemma}

In practice, assertions verifying assumption~\ref{it:twin} will be called 
{\it twin lemmata} and statements confirming~\ref{it:attach} will be referred to
as {\it attachment lemmata}. Each of the two subsequent sections has its own 
twin lemma (cf.\ Lemma~\ref{l:twin} and Lemma~\ref{lem:vtwin}) and its own 
attachment lemma (cf.\ Lemma~\ref{l:attachment} and Lemma~\ref{l:329}).

\begin{proof}[Proof of Lemma~\ref{lem:2407}]
	Set $A_q=\{q'\in V(G)\colon q' \text{ is an $\Omega$-twin of $q$}\}$
	for every $q\in V(\Omega)$. These sets are mutually disjoint, because 
	we assumed $\Omega$ to be twin-free. So due to~\ref{it:attach} we have 
	a partition	
		\[
		V(G) = \bigdcup_{q\in V(\Omega)} A_q
	\]
		and~\ref{it:twin} informs us that whenever $qz\in E(\Omega)$ 
	all $A_q$-$A_z$-edges are in $E(G)$. Since $\Omega$  is maximal 
	triangle-free,~$G$ can have no further edges. 
\end{proof}

\section{Andr\'asfai graphs}\label{sec:A}
This entire section is devoted to the proof of part \ref{it:1.3a} of 
Theorem~\ref{thm:1.3}. Our first step simplifies the assumption 
$\grot\not\subseteq G$. Since the Mycielski-Gr\"otzsch graph $\grot$ is maximal 
triangle-free, all its occurrences in triangle-free graphs must be induced.  
Together with the fact that $\grot$ contains an induced hexagon 
this shows that triangle-free graphs containing $\grot$ contain an induced 
hexagon as well. It turns out that this implication can be reversed for 
maximal triangle-free graphs with property~\Dp{3}.

\begin{lemma}\label{l:21}
	Let $G$ be a maximal triangle-free graph satisfying \Dp{3}.
	If $G$ contains an induced hexagon, then it contains the 
	Mycielski-Gr\"otzsch graph as well. 
\end{lemma}

\begin{proof}
	Let $a_1-b_3-a_2-b_1-a_3-b_2-a_1$ be an induced hexagon. Its vertex set 
	has only two independent subsets of size three, namely $\{a_1, a_2, a_3\}$ 
	and $\{b_1, b_2, b_3\}$. So by \Dp{2} we may assume that there exists a 
	common neighbour $x$ of $b_1, b_2$, and $b_3$. 	Because $G$ is a maximal 
	triangle-free graph 
	and $a_1b_1, a_2b_2, a_3b_3\not\in E(G)$, there are 
	common neighbours $c_i$ of~$a_i$,~$b_i$ for $i=1, 2, 3$ 
	(see Figure~\ref{fig:21a}). Since $G$ is triangle-free, the 
	vertices $x$, $c_1$, $c_2$, $c_3$ are distinct.
	
	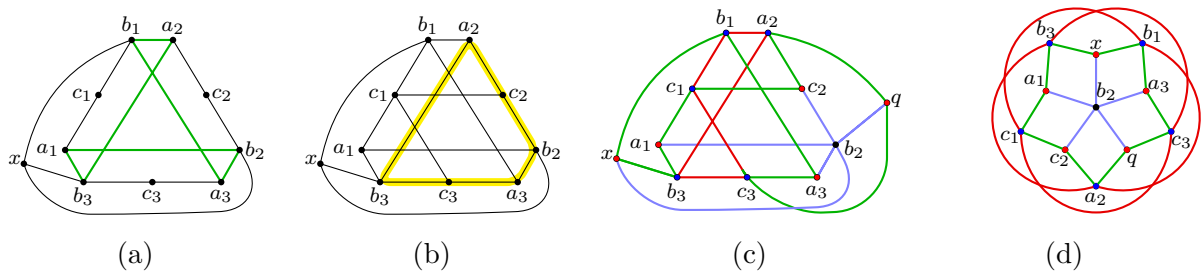
\begin{figure}[ht]
	\centering
	\begin{subfigure}[t]{0.235\textwidth}
	\centering
	\begin{tikzpicture}[scale=.61]
		\coordinate (c3) at (0,-.1);
		\coordinate (b3) at (-1.5,-.1);
		\coordinate (a3) at (1.5,-.1);
		\coordinate (b1) at (-.45,3);
		\coordinate (a2) at (.45,3);
		\coordinate (a1) at (-1.9,.6);
		\coordinate (b2) at (1.9,.6);
		\coordinate (c1) at ($(b1)!.5!(a1)$);
		\coordinate (c2) at ($(b2)!.5!(a2)$);
		\coordinate (x) at (-2.8,.3);
		\draw (b3)--(x) edge [bend left = 30] (b1);
		\draw [rounded corners=20] (x)--(-2.5,-.8)--(0,-.8)--(2.6,-.7)--(b2); 
		\draw [thick, green!70!black] (a1)--(b2)--(a3)--(b1)--(a2)--(b3)--cycle;
		\fill (x) circle (2pt);
		\foreach \i in {1,...,3} {
			\draw (b\i)--(a\i);
			\fill (a\i) circle (2pt);		
			\fill (b\i) circle (2pt);
			\fill (c\i) circle (2pt);
		}
		\node [above=-2pt] at (b1) {\tiny $b_1$};
		\node [left=-2pt] at (a1) {\tiny $a_1$};
		\node [below=-2pt] at (b3) {\tiny $b_3$};
		\node [below=-1pt] at (a3) {\tiny $a_3$};
		\node [right=-2pt] at (b2) {\tiny $b_2$};
		\node [above=-2pt] at (a2) {\tiny $a_2$};
		\node [left=-2pt] at (c1) {\tiny $c_1$};
		\node [right=-2pt] at (c2) {\tiny $c_2$};
		\node [below=-1pt] at (c3) {\tiny $c_3$};
		\node at ($(x)+(-.2,.1)$) {\tiny $x$};
	\end{tikzpicture}
	\caption{}
	\label{fig:21a}
	\end{subfigure}
	\hfill
	\begin{subfigure}[t]{0.235\textwidth}
	\centering
	\begin{tikzpicture}[scale=.61]
		\draw [line width = 3pt, yellow, rounded corners] (a2)--(b2)--(a3)--(b3)--cycle;
		\fill (x) circle (2pt);
		\foreach \i in {1,...,3} {
			\fill (a\i) circle (2pt);		
			\fill (b\i) circle (2pt);
			\fill (c\i) circle (2pt);
			\draw (b\i)--(a\i);
		}
		\draw (b3)--(x) edge [bend left = 30] (b1);
		\draw [rounded corners=20] (x)--(-2.5,-.8)--(0,-.8)--(2.6,-.7)--(b2); 
		\draw (a1)--(b2)--(a3);
		\draw (b1) -- (a3);
		\draw (b1)--(a2)--(b3) -- (a1);
		\draw (c2)--(c1)--(c3);
		\node [above=-2pt] at (b1) {\tiny $b_1$};
		\node [left=-2pt] at (a1) {\tiny $a_1$};
		\node [below=-2pt] at (b3) {\tiny $b_3$};
		\node [below=-1pt] at (a3) {\tiny $a_3$};
		\node [right=-2pt] at (b2) {\tiny $b_2$};
		\node [above=-2pt] at (a2) {\tiny $a_2$};
		\node [left=-2pt] at (c1) {\tiny $c_1$};
		\node [right=-2pt] at (c2) {\tiny $c_2$};
		\node [below=-1pt] at (c3) {\tiny $c_3$};
		\node at ($(x)+(-.2,.1)$) {\tiny $x$};
	\end{tikzpicture}
	\caption{}
	\label{fig:21b}
	\end{subfigure}
	\hfill
	\begin{subfigure}[t]{.27\textwidth}
	\centering
	\begin{tikzpicture}[scale=.62]
		\coordinate (c3) at (0,-.1);
		\coordinate (b3) at (-1.5,-.1);
		\coordinate (a3) at (1.5,-.1);
		\coordinate (b1) at (-.45,3);
		\coordinate(a2) at (.45,3);
		\coordinate (a1) at (-1.9,.6);
		\coordinate (b2) at (1.9,.6);
		\coordinate (c1) at ($(b1)!.5!(a1)$);
		\coordinate (c2) at ($(b2)!.5!(a2)$);
		\coordinate (x) at (-2.8,.3);
		\coordinate (q) at (3, 1.5);
		\draw [thick, green!70!black](a2)edge [bend left = 20](q);
		\draw [thick, blue!50] (c2)--(b2)--(a3)--(b2)--(q)--(b2)--(a1);
		\draw [thick, red!90!black] (c1)--(c3)--(b3)--(a2)--(b1)--(c1);
		\draw [thick, green!70!black] (x)--(b3)--(a1)--(c1)--(c2)--(a2);
		\draw [thick, green!70!black] (b1)--(a3)--(c3);
		\draw [thick, green!70!black] (b3)--(x) edge [bend left = 30] (b1);
		\draw [thick, green!70!black, rounded corners=20] (c3)--(.5,-.9)--(3,-.8)--(q);
		\draw [thick, blue!50,rounded corners=20] (x)--(-2.5,-.8)--(0,-.8)--(2.6,-.7)--(b2); 	
		\foreach \i in {1,...,3} {
			\fill (a\i) circle (2pt);		
			\fill (b\i) circle (2pt);
			\fill (c\i) circle (2pt);
		}
		\fill (x) circle (2pt);
		\fill (q) circle (2pt);
		\foreach \i in {x, a3, q, c2, a1} 	\fill [red](\i) circle (1.6pt);		
		\foreach \i in {b3, b1, c3, a2, c1}	\fill [blue](\i) circle (1.6pt);
		\node [above=-2pt] at (b1) {\tiny $b_1$};
		\node [left=-2pt] at (a1) {\tiny $a_1$};
		\node [below=-2pt] at (b3) {\tiny $b_3$};
		\node [below=-1pt] at (a3) {\tiny $a_3$};
		\node at ($(b2)+(.4,-.1)$) {\tiny $b_2$};
		\node [above=-2pt] at (a2) {\tiny $a_2$};
		\node [left=-2pt] at (c1) {\tiny $c_1$};
		\node [right=-2pt] at (c2) {\tiny $c_2$};
		\node [below] at (c3) {\tiny $c_3$};
		\node at ($(x)+(-.2,.1)$) {\tiny $x$};
		\node at ($(q)+(.2,.05)$) {\tiny $q$};
	\end{tikzpicture}
	\caption{}
	\label{fig:21c}
	\end{subfigure}
	\hfill
	\begin{subfigure}[t]{.23\textwidth}
	\centering
	\begin{tikzpicture}[scale=.7]
		\hskip .4cm
		\coordinate (c) at (0,0);
		\foreach \i in {1,...,5} {
			\coordinate (a\i) at (18+\i*72:1);
			\coordinate (b\i) at (54+\i*72:1.5);	
			\coordinate (d\i) at (54+\i*72:2);
		}
		\foreach \i in {1,...,5} {
			\draw [blue!50, thick] (c)--(a\i);
			\fill (a\i) circle (1.8pt);		
			\fill (b\i) circle (1.8pt);
		}
		\fill (c) circle (1.8pt);
		\draw [green!70!black, thick](a1)--(b1)--(a2)--(b2)--(a3)--(b3)--(a4)--(b4)--(a5)--(b5)--(a1);
		\draw (b1) [out=198, in=108, red!90!black, thick] to (d2) [out=-72, in=198] to (b3);
		\draw (b2) [out=270, in=180, red!90!black, thick] to (d3) [out=0, in=270] to (b4);
		\draw (b3) [out=-18, in=252, red!90!black, thick] to (d4) [out=72, in=-18] to (b5);
		\draw (b4) [out=54, in=-36, red!90!black, thick] to (d5) [out=144, in=54] to (b1);
		\draw (b5) [out=126, in=36, red!90!black, thick] to (d1) [out=216, in=126] to (b2);
		\node at ($(c)+(.2,.3)$) {\tiny $b_2$};
		\node at ($(a1)+(0,.2)$) {\tiny $x$};
		\node at ($(a2)+(-.2,.2)$) {\tiny $a_1$};
		\node at ($(a3)+(-.1,-.2)$) {\tiny $c_2$};
		\node at ($(a4)+(.1,-.2)$) {\tiny $q$};
		\node at ($(a5)+(.25,.1)$) {\tiny $a_3$};
		\node at ($(b1)+(-.06,.27)$) {\tiny $b_3$};
		\node at ($(b2)+(-.2,-.1)$) {\tiny $c_1$};
		\node at ($(b3)+(0,-.25)$) {\tiny $a_2$};
		\node at ($(b4)+(.2,-.2)$) {\tiny $c_3$};
		\node at ($(b5)+(.2,.2)$) {\tiny $b_1$};
		\foreach \i in {1,...,5} {
			\fill [red](a\i) circle (1.6pt);		
			\fill [blue](b\i) circle (1.6pt);
		}
	\end{tikzpicture}	
	\caption{}
	\label{fig:21d}
	\end{subfigure}
	\caption{The proof of Lemma~\ref{l:21}.}
	\label{fig:21}
	\end{figure}
 
	By~\Dp{3} there is a four-element subset $T$  
	of $\{a_1, a_2, a_3, b_1, b_2, b_3, c_1, c_2, c_3\}$ possessing a common
	neighbour $t$. As $T$ contains at most one vertex 
	from each of the edges $b_1c_1$,~$b_2c_2$,~$b_3c_3$, we may 
	assume $a_1\in T$. Similarly at least one of $b_1$, $b_2$, $b_3$ is in~$T$. 
	Together with the independence of $T$ this yields 
	$b_1\in T$, whence $T=\{a_1, b_1, c_2, c_3\}$. 
	Thus we can replace $c_1$ by $t$ and this argument allows us to assume 
	$c_1c_2, c_1c_3\in E(G)$ (see Figure~\ref{fig:21b}). 
	
	Next we apply~\Dp{2} to the hexagon $a_2-c_2-b_2-a_3-c_3-b_3-a_2$. 
	Due to the symmetry between the indices $2$ and $3$ we can suppose,
	without loss of generality, that there exists a common neighbour $q$ 
	of $a_2$, $b_2$, $c_3$ (see Figure \ref{fig:21c}). 
	Altogether we have now found a copy of the Mycielski-Gr\"otzsch 
	graph~$\grot$ in~$G$ (see Figure~\ref{fig:21d}). 
\end{proof}

In the remainder of this section we do not need to appeal to~\Dp{3} directly 
anymore. In other words, we shall obtain an explicit description of the class~$\fA$
of maximal triangle-free graphs on at least two vertices not containing an 
induced hexagon. As it will turn out, $\fA$ is simply the class of blow-ups of 
Andr\'asfai graphs. Let us recall at this moment that for every positive 
integer~$k$ the Andr\'asfai graph $\G_k$ has vertex set $\ZZ/(3k-1)\ZZ$ and 
all edges~$ij$ such that $i-j\in \{k, k+1, \dots, 2k-1\}$.
As promised in Section~\ref{sec:prelim} we shall establish a twin lemma and 
an attachment lemma.

An edge $ij$ of the Andr\'asfai graph $\Gamma_k$ is called {\it short} 
if $i-j=\pm k$ and {\it long} otherwise. So all edges of~$\Gamma_1$ 
and~$\Gamma_2$ are short and, up to symmetry,~$04$ is the only long edge 
of~$\Gamma_3$. For long edges the twin property requires no further assumptions.

\begin{lemma}
	If $e$ denotes a long edge of an Andr\'asfai graph $\G_k$, then every $G\in\fA$
	has the $(\G_k, e)$-twin property.
\end{lemma}

\begin{proof}
	We start with the special case $k=3$, i.e., we show that every $G\in\fA$ has 
	the $(\G_3, 04)$-twin property.   
	Assume contrariwise that $\Gamma_3\subseteq G$ and that $0',4'\in V(G)$ 
	are non-adjacent $\Gamma_3$-twins of $0,4\in V(\G_3)$. Let $r$ be a 
	common neighbour of $0'$, $4'$ (see Figure~\ref{fig:new_gamma3A}).
		
	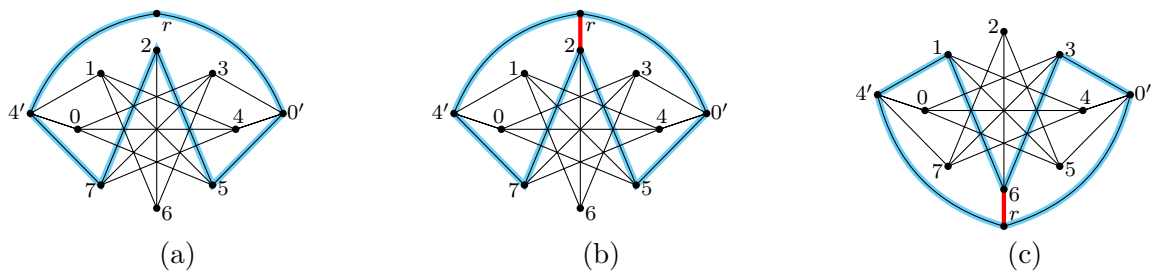
\begin{figure}[ht]
	\centering
	\begin{subfigure}[b]{.3\textwidth}
	\begin{tikzpicture}[scale=.7]
		\coordinate (b) at (-2.4, .3);
		\coordinate (c) at (2.4, .3);
		\coordinate (r) at (0,2.2);
		\foreach \i in {1,...,8} \coordinate (a\i) at (\i*45:1.5);
		\draw (c) edge [cyan!50, line width=2.5, bend right = 30](r);
		\draw (b) edge [cyan!50, line width = 2.5, bend left= 30](r);
		\draw [cyan!50, line width=2.5] (b)--(a5)--(a2)--(a7)--(c);
		\draw (c) edge [bend right = 30](r);
		\draw (b) edge [bend left= 30](r);		
		\draw (a1) -- (a4) --(a7) --(a2) -- (a5) -- (a8) --(a3) --(a6) --(a1);
		\draw (a1) -- (a5);
		\draw (a2) -- (a6);
		\draw (a3) -- (a7);
		\draw (a4) -- (a8);	
		\draw (a3)--(b)--(a4)--(b)--(a5);
		\draw (a7)--(c)--(a8)--(c)--(a1);
		\fill (c) circle (2pt);
		\fill (b) circle (2pt);
		\fill (r) circle (2pt);
		\foreach \i in {1,...,8} \fill (a\i) circle (2pt);		
		\node at ($(r)+(.2,-.25)$) {\tiny $r$};
		\node at ($(a1)+(.2,.1)$) {\tiny $3$};
		\node at ($(a2)+(-.2,.1)$) {\tiny $2$};
		\node at ($(a3)+(-.2,.1)$) {\tiny $1$};
		\node at ($(a4)+(-.05,.25)$) {\tiny $0$};
		\node at ($(a5)+(-.2,-.1)$) {\tiny $7$};
		\node at ($(a6)+(.2,-.1)$) {\tiny $6$};
		\node at ($(a7)+(.2,-.05)$) {\tiny $5$};
		\node at ($(a8)+(.05,.25)$) {\tiny $4$};
		\node at ($(b)+(-.25,.05)$) {\tiny $4'$};
		\node at ($(c)+(.25,.05)$) {\tiny $0'$};
	\end{tikzpicture}
	\vskip -.2cm
	\caption{}
	\label{fig:new_gamma3A}
	\end{subfigure}
	\hfill
	\begin{subfigure}[b]{.3\textwidth}
	\begin{tikzpicture}[scale=.7]
		\draw (c) edge [cyan!50, line width=2.5, bend right = 30](r);
		\draw (b) edge [cyan!50, line width = 2.5, bend left= 30](r);
		\draw [cyan!50, line width=2.5] (b)--(a5)--(a2)--(a7)--(c);
		\draw [ultra thick, red] (r)--(a2);
		\draw (c) edge [bend right = 30](r);
		\draw (b) edge [bend left= 30](r);		
		\draw (a1) -- (a4) --(a7) --(a2) -- (a5) -- (a8) --(a3) --(a6) --(a1);
		\draw (a1) -- (a5);
		\draw (a2) -- (a6);
		\draw (a3) -- (a7);
		\draw (a4) -- (a8);	
		\draw (a3)--(b)--(a4)--(b)--(a5);
		\draw (a7)--(c)--(a8)--(c)--(a1);
		\fill (c) circle (2pt);
		\fill (b) circle (2pt);
		\fill (r) circle (2pt);
		\foreach \i in {1,...,8} \fill (a\i) circle (2pt);		
		\node at ($(r)+(.2,-.25)$) {\tiny $r$};
		\node at ($(a1)+(.2,.1)$) {\tiny $3$};
		\node at ($(a2)+(-.2,.1)$) {\tiny $2$};
		\node at ($(a3)+(-.2,.1)$) {\tiny $1$};
		\node at ($(a4)+(-.05,.25)$) {\tiny $0$};
		\node at ($(a5)+(-.2,-.1)$) {\tiny $7$};
		\node at ($(a6)+(.2,-.1)$) {\tiny $6$};
		\node at ($(a7)+(.2,-.05)$) {\tiny $5$};
		\node at ($(a8)+(.05,.25)$) {\tiny $4$};
		\node at ($(b)+(-.25,.05)$) {\tiny $4'$};
		\node at ($(c)+(.25,.05)$) {\tiny $0'$};
	\end{tikzpicture}
	\vskip -.2cm
	\caption{}
	\label{fig:new_gamma3B}
	\end{subfigure}
	\hfill
	\begin{subfigure}[b]{.3\textwidth}
	\begin{tikzpicture}[scale=.7]
		\coordinate (r) at (0,-2.2);
		\draw (c) edge [cyan!50, line width=2.5, bend left = 30](r);
		\draw (b) edge [cyan!50, line width = 2.5, bend right= 30](r);
		\draw [cyan!50, line width=2.5] (b)--(a3)--(a6)--(a1)--(c);
		\draw [ultra thick, red] (r)--(a6);
		\draw (c) edge [bend left = 30](r);
		\draw (b) edge [bend right= 30](r);		
		\draw (a1) -- (a4) --(a7) --(a2) -- (a5) -- (a8) --(a3) --(a6) --(a1);
		\draw (a1) -- (a5);
		\draw (a2) -- (a6);
		\draw (a3) -- (a7);
		\draw (a4) -- (a8);	
		\draw (a3)--(b)--(a4)--(b)--(a5);
		\draw (a7)--(c)--(a8)--(c)--(a1);
		\fill (c) circle (2pt);
		\fill (b) circle (2pt);
		\fill (r) circle (2pt);
		\foreach \i in {1,...,8} \fill (a\i) circle (2pt);		
		\node at ($(r)+(.2,.2)$) {\tiny $r$};
		\node at ($(a1)+(.2,.1)$) {\tiny $3$};
		\node at ($(a2)+(-.2,.1)$) {\tiny $2$};
		\node at ($(a3)+(-.2,.15)$) {\tiny $1$};
		\node at ($(a4)+(-.05,.25)$) {\tiny $0$};
		\node at ($(a5)+(-.2,-.1)$) {\tiny $7$};
		\node at ($(a6)+(.2,-.1)$) {\tiny $6$};
		\node at ($(a7)+(.2,-.05)$) {\tiny $5$};
		\node at ($(a8)+(.05,.25)$) {\tiny $4$};
		\node at ($(b)+(-.25,.05)$) {\tiny $4'$};
		\node at ($(c)+(.25,.05)$) {\tiny $0'$};
	\end{tikzpicture}
	\vskip -.2cm
	\caption{}
	\label{fig:new_gamma3C}
	\end{subfigure}
	\vskip -.3cm
	\caption{The $(\G_3, e)$-twin property.}
	\label{fig:new_gamma3}
	\end{figure}
 		
	As the hexagon $0'-r-4'-7-2-5-0'$ cannot be induced, we have $2r\in E(G)$ 
	(see Figure \ref{fig:new_gamma3B}). 
	Similarly, the hexagon $0'-r-4'-1-6-3-0'$ discloses $6r\in E(G)$
	(see Figure~\ref{fig:new_gamma3C}).
	But now $26r$ is a triangle, which is absurd. 
	
	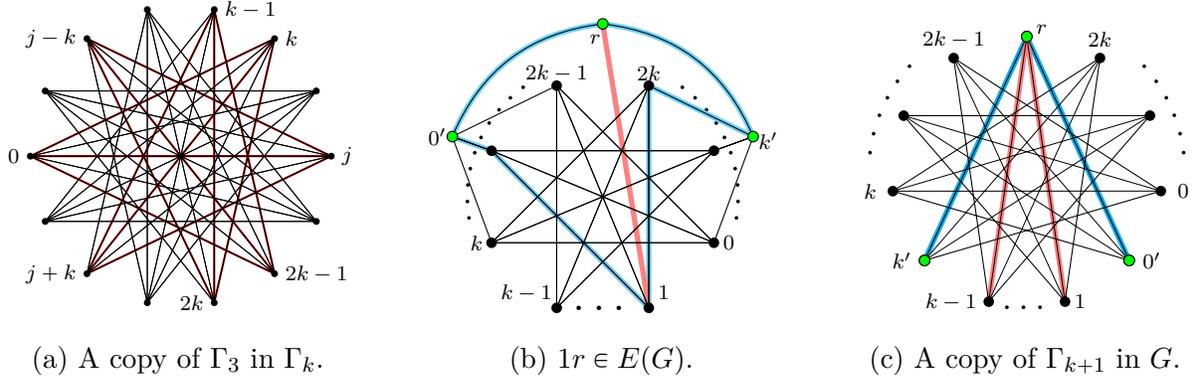
\begin{figure}[h!]
	\centering
	\begin{subfigure}[b]{.3\textwidth}
	\begin{tikzpicture}[scale=1]
		\coordinate (c) at (0,0);
		\foreach \i in {0,...,30} {
			\coordinate (a\i) at (180+\i*25.714:2);
		}
		\draw [red,  thick] (a0)--(a7);
		\draw [red,  thick] (a4)--(a10);
		\draw [red,  thick]  (a2)--(a9);
		\draw [red,  thick] (a5)--(a12);
		\draw [red,  thick] (a0)--(a9)--(a4)--(a12)--(a7)--(a2)--(a10)--(a5)--(a0);
		\foreach \i in {0,...,13} {		
			\fill (a\i) circle (1.3pt);
			\foreach \j [evaluate=\j as \k using \i+\j] in {5,...,9}{\draw [line width=.05](a\i)--(a\k);}		
		}
		\node [left] at (a0) {\tiny $0$};
		\node [right] at (a7) {\tiny $j$};
		\node [right] at (a9) {\tiny $k$};
		\node [right] at (a10) {\tiny $k-1$};
		\node [left] at (a12) {\tiny $j-k$};
		\node [left] at (a2) {\tiny $j+k$};
		\node [left] at (a4) {\tiny $2k$};
		\node [right] at (a5) { \tiny $2k-1$};
	\end{tikzpicture}
	\caption{A copy of $\G_3$ in $\G_k$.}\label{fig:gamma3_2}
	\end{subfigure}
	\hfill
	\begin{subfigure}[b]{0.36\textwidth}
	\centering
	\begin{tikzpicture}[scale=1]
		\foreach \j in {1,...,10}{
			\fill (119+\j*7:1.8) circle (.7pt);
			\fill (-16+\j*7:1.8) circle (.7pt);
		}
		\foreach \i in {0,...,12} \coordinate (a\i) at (-22.5+\i*45:1.6);
		\coordinate (b) at (-2,.8);
		\coordinate (c) at (2,.8);
		\coordinate (r) at (0,2.3);
		\draw [cyan!50, line width=2pt] (c)edge [bend left=-30](r);
		\draw [cyan!50, line width=2pt] (r) edge [bend left=-30] (b);
		\draw [cyan!50, line width=2pt] (b)--(a4)--(a7)--(a2)--(c);
		\draw [red!50, line width=2pt] (a7)--(r);
		\foreach \i in {0,...,7} {
			\fill (a\i) circle (2pt);
			\foreach \j [evaluate=\j as \k using \i+\j] in {3,4,5}{
			\draw (a\i)--(a\k);}
		}
		\draw (a0)--(c)--(a1)--(c)--(a2);
		\draw (a3)--(b)--(a4)--(b)--(a5);
		\draw (c) edge [bend left = -30](r);
		\draw (b) edge [bend right= -30](r);
		\draw[fill=green] (c) circle (2pt);
		\draw[fill=green] (b) circle (2pt);
		\draw[fill=green] (r) circle (2pt);
		\node at ($(r)+(-.1,-.2)$) {\tiny $r$};
		\node at ($(b)+(-.2,0)$) {\tiny $0'$};
		\node at ($(c)+(.2,-.05)$) {\tiny $k'$};
		\node at ($(a2)+(0,.17)$) {\tiny $2k$};
		\node at ($(a3)+(0,.17)$) {\tiny $2k-1$};
		\node at ($(a5)+(-.2,0)$) {\tiny $k$};
		\node at ($(a6)+(-.4,.2)$) {\tiny $k-1$};
		\node at ($(a7)+(.2,.2)$) {\tiny $1$};
		\node at ($(a0)+(.2,0)$) {\tiny $0$};
		\node at($(a6)!.5!(a7)$) {\Large $\dots$};
	\end{tikzpicture}
	\caption{$1r\in E(G)$.}
	\label{fig:twina} 
	\end{subfigure}
	\hfill    
	\begin{subfigure}[b]{0.3\textwidth}
	\centering
	\begin{tikzpicture}[scale=1]
		\foreach \i in {1,...,11} {
			\coordinate (a\i) at (90+\i*32.73:1.8);
		}
		\foreach \j in {1,...,6}{
			\fill (131+\j*7:2.1) circle (.7pt);
			\fill (0+\j*7:2.1) circle (.7pt);
		}
		\draw [line width = 2pt, cyan] (a7) -- (a11);
		\draw [line width = 2pt, cyan] (a11) -- (a4);
		\draw [line width = 2pt, red!50] (a6) -- (a11) ;
		\draw [line width = 2pt, red!50] (a11) -- (a5);
		\foreach \i in {1,...,11} {
			\fill (a\i) circle (2pt);		
		}
		\foreach \i in {4,7,11}{
			\draw[fill=green] (a\i) circle (2pt);	
		}
		\draw (a11)--(a4)--(a8)--(a1)--(a5)--(a9)--(a2)--(a6)--(a10)--(a3)--(a7)--(a11);
		\draw [ultra thin] (a1)--(a6)--(a11)--(a5)--(a10)--(a4)--(a9)--(a3)--(a8)--(a2)--(a7)--(a1);
		\foreach \i in {4,7,11}{
			\draw[fill=green] (a\i) circle (2pt);	
		}
		\node at ($(a11)+(.2,.1)$) { \tiny $r$};
		\node at ($(a1)+(0,.25)$) { \tiny $2k-1$};
		\node at ($(a3)+(-.3,0)$) { \tiny $k$};
		\node at ($(a4)+(-.3,0)$) { \tiny $k'$};
		\node at ($(a5)+(-.5,0)$) { \tiny $k-1$};
		\node at ($(a6)+(.2,0)$) { \tiny $1$};
		\node at ($(a7)+(.3,0)$) { \tiny $0'$};
		\node at ($(a8)+(.3,0)$) { \tiny $0$};
		\node at ($(a10)+(0,.25)$) { \tiny $2k$};
		\node at (0,-1.8) {\large $\dots$};
	\end{tikzpicture}
	\caption{A copy of $\G_{k+1}$ in $G$.}
	\label{fig:twinb} 		
	\end{subfigure}   
	\caption{The $\G_k$-twin property for long and short edges.}
	\label{fig:twin}
	\end{figure}
 
	Next we generalise this to all long edges.
	By symmetry we may assume that the given long edge of $\G_k$ is of the 
	form $e=0j$, where $k<j<2k-1$. Figure~\ref{fig:gamma3_2} 
	shows a copy of~$\Gamma_3$ in~$\Gamma_k$ one of whose long edges 
	corresponds to~$0j$. Thus the assertion follows from the special case 
	treated earlier. 
\end{proof}

\begin{lemma}[Twin lemma]\label{l:twin}
	Every $\G_{k+1}$-free graph $G\in \fA$ has the $\Gamma_k$-twin property.
\end{lemma}

\begin{proof}	It remains to consider short edges. In fact, for reasons of symmetry,
	it suffices to show that $G$ has the $(\G_k, 0k)$-twin 
	property. Assume for the sake of contradiction that $\Gamma_k\subseteq G$ 
	and that $0'$, $k'$ are non-adjacent $\Gamma_k$-twins of $0$, $k$, respectively. 
	Let $r$ be a common neighbour of $0'$, $k'$.
	Whenever $0<j<k$ the hexagon 
	\[
	0'-r-k'-(j+2k-1) - j - (j+k)-0'
	\]
	shows $rj\in E(G)$ (for $j=1$ this is illustrated in Figure~\ref{fig:twina}). 
	So $\{0',1,\dots,k-1,k'\}\subseteq \Nn(r)$ and $V(\Gamma_k)\cup \{0',k',r\}$ 
	induces a copy of $\Gamma_{k+1}$ in $G$, which is absurd 
	(see Figure~\ref{fig:twinb}).
\end{proof}

\begin{lemma}[Attachment lemma]\label{l:attachment}
	If $\Gamma_k\subseteq G\in \fA$ and $G$ is $\G_{k+1}$-free, 
	then every $q\in V(G)$ is a $\Gamma_k$-twin of some vertex of $\Gamma_k$. 
\end{lemma}

\begin{proof}	
	We begin with the following very special case. 
	
	\begin{claim}\label{l:28}
		If $j\in V(\Gamma_k)$, and $j+k, j-k\in \Nn(q)$, 
		then $q$ is a $\Gamma_k$-twin of $j$. 
	\end{claim}
	
	\begin{proof}
		By symmetry we can assume $j=k$, so that $0, 2k \in \Nn(q)$. 
		For every vertex $m\in [2k+1,3k-2]$ the hexagon 
		\[
			q-0-(m-k)-m-(m+k)-2k-q
		\]
		shows $qm\in E(G)$ (see Figure \ref{fig:AtLemA}). 
		So $\Nn(k)\cap V(\G_k)=\{2k,\dots, 3k-2,0\}\subseteq \Nn(q)\cap V(\G_k)$ 
		and, since $G$ is triangle-free, this holds with equality.
	\end{proof}
	
	\begin{figure}[ht]
	\centering
	\vskip -.3cm
	\begin{subfigure}[b]{.49\textwidth}
	\centering
\hskip 1cm	\begin{tikzpicture}[scale=1]
		\coordinate (c) at (0,0);
		\foreach \i in {0,...,30} {
			\coordinate (a\i) at (90+\i*25.714:1.7);
		}
		\coordinate (q) at (-.7, 1.9);
		\foreach \i in {0,...,13} {		
			\foreach \j [evaluate=\j as \k using \i+\j] in {5,...,9}{\draw [line width=.05] (a\i)--(a\k);}		
		}
		\def\cOne{red!70!black}
		\draw (a12) edge [\cOne, thick, bend right = 40](q);
		\draw (q) edge [\cOne, thick, bend right =20] (a2);
		\draw [\cOne, thick] (a12) -- (a6) -- (a1)--(a10);
		\draw [\cOne, thick] (a10)--(a2) ;
		\draw [green!70!black, thick] (q)--(a1);
		\foreach \i in {0,...,13} {		
			\fill (a\i) circle (1.3pt);
		}
		\fill (q) circle (1.3pt);
		\node [left] at (a7) {\tiny $k$};
		\node [left] at (a6) {\tiny $m-k$};
		\node [left] at (a2) {\tiny $2k$};
		\node at ($(a1)+(.2,.05)$) {\tiny $m$};
		\node [right] at (a10) {\tiny $m+k$};
		\node [left] at (a3) {\tiny $2k-1$};
		\node [right] at (a11) {\tiny $1$};
		\node  at ($(a12)+(.75,.1)$) { \tiny $0=3k-1$};
		\node [above] at (q) {\tiny $q$};
	\end{tikzpicture}
	\vskip -.2cm
	\caption{}\label{fig:AtLemA}
	\end{subfigure}
	\hfill
	\begin{subfigure}[b]{0.49\textwidth}
	\centering
\hskip -1cm	\begin{tikzpicture}[scale=1]
		\coordinate (c) at (0,0);
		\foreach \i in {0,...,30} {
			\coordinate (a\i) at (90+\i*25.714:1.7);
		}
		\coordinate (q) at (.3, -1.8);
		\coordinate (z) at (1,-1.7);
		\foreach \i in {0,...,13} {		
			\foreach \j [evaluate=\j as \k using \i+\j] in {5,...,9}{\draw [line width=.05] (a\i)--(a\k);}		
		}
		\def\cOne{green!70!black}
		\def\cTwo{green!70!black}
		\draw [\cOne, thick] (q)--(a7);
		\draw [\cTwo, thick] (a8)--(z)--(q);
		\draw [\cOne, thick] (a7) -- (a12) -- (a3)--(a8);
		\foreach \i in {0,...,13} {		
			\fill (a\i) circle (1.3pt);
		}
		\fill (q) circle (1.3pt);
		\fill (z) circle (1.3pt);
		\node [left] at (a7) {\tiny $k$};
		\node [left] at (a6) {\tiny $k+1$};
		\node [left] at (a2) {\tiny $2k$};
		\node at ($(a8)+(.4,.05)$) {\tiny $k-1$};
		\node [left] at (a3) {\tiny $2k-1$};
		\node [right] at (a11) {\tiny $1$};
		\node at ($(a12)+(.8,0)$) { \tiny $0=3k-1$};
		\node at ($(a13)+(.45,-.02)$) {\tiny $3k-2$};
		\node at ($(q)+(.07,.15)$) {\tiny $q$};
		\node at ($(z)+(.1,-.1)$) {\tiny $z$};
	\end{tikzpicture}
	\vskip -.2cm
	\caption{}
	\label{fig:AtLemB} 
	\end{subfigure}
	\caption{The proof of the attachment lemma.}
	\label{fig:AtLem}
	\end{figure}
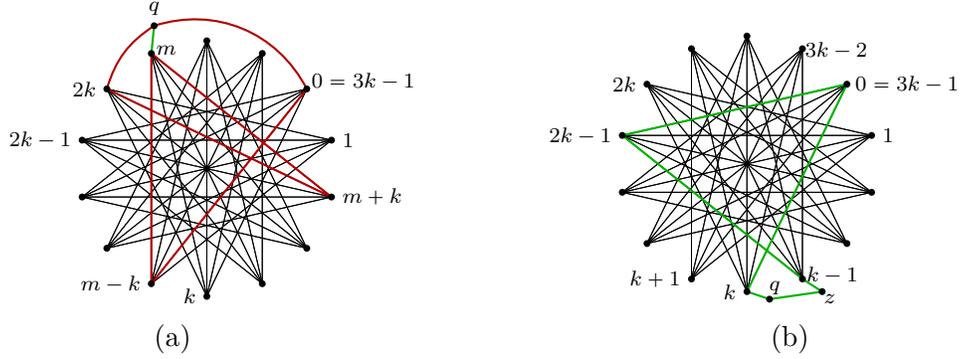
 
	Let us proceed with a less special case. 
	
	\begin{claim}\label{c:special}
		If $q$ has a neighbour in $V(\Gamma_k)$, then it is a $\G_k$-twin of 
		some vertex of $\Gamma_k$.
	\end{claim}
		
	\begin{proof}
		By symmetry we can suppose that $kq\in E(G)$ and $(k-1)q\notin E(G)$. 
		Let $z$ be a common neighbour of $k-1$ and $q$ (see Figure~\ref{fig:AtLemB}). 
		The hexagon 
		\[
			z-q-k-0-(2k-1)-(k-1)-z
		\]
		shows that either $q(2k-1)$ or $z0$ is an edge. In the first case 
		Claim~\ref{l:28} entails that~$q$ is a~$\Gamma_k$-twin of $0$. 
		In the second case Claim~\ref{l:28} implies that $z$ is a $\Gamma_k$-twin 
		of $2k-1$ and we can form $\Gamma_k'=\Gamma_k(z)$. 
		Another application of Claim~\ref{l:28} reveals that~$q$ is 
		a $\Gamma_k'$-twin of $0$. 
		Since $G$ has the $\G_k$-twin property, Lemma~\ref{f:24} tells us that~$q$ 
		is also a $\Gamma_k$-twin of $0$. 
	\end{proof}

	Proceeding with the general case we consider an arbitrary vertex $q\in V(G)$. 
	Assuming $0q\notin E(G)$ we take a common neighbour $z$ of $q$, $0$. We already know 
	that $z$ is a $\Gamma_k$-twin of some $j\in V(\Gamma_k)$. Now $q$ has a neighbour 
	belonging to $\Gamma_k'=\Gamma_k(z)$ and, therefore $q$ is a $\Gamma_k'$-twin of 
	some $i\in V(\Gamma_k)\setminus \{j\}$. By Lemma~\ref{f:24} $q$ is also 
	a $\Gamma_k$-twin of $i$. 
\end{proof}

The main result of this section reads as follows.

\begin{prop}\label{prop:210}
	A maximal triangle-free graph on at least two vertices contains no 
	induced hexagon if and only if it is a blow-up of some Andr\'asfai graph.
\end{prop}

\begin{proof}
	We will only require and prove the forward implication in the sequel, 
	leaving the (almost obvious) reverse direction to the reader. Since the 
	given graph~$G$ is maximal triangle-free and has at least two vertices, 
	it needs to contain a copy of $\Gamma_1$. Let $k\in\NN$ be maximal such 
	that $G$ has a subgraph isomorphic to $\G_k$. By Lemma~\ref{l:twin}
	and Lemma~\ref{l:attachment} the assumptions of Lemma~\ref{lem:2407}
	are satisfied for some $\Omega\cong\G_k$. Thus $G$ is a blow-up of $\G_k$.
\end{proof}

It should be clear that Lemma~\ref{l:21} combined with 
Proposition~\ref{prop:210} yields Theorem~\ref{thm:1.3}\ref{it:1.3a}.
			
\section{Vega graphs}\label{sec:vega}
The goal of this section is to establish Theorem~\ref{thm:1.3}\ref{it:1.3b}.
We start by defining Vega graphs. For every $i\ge 2$ there 
is a {\it Vega graph} $\grot_i^{00}$ shown in Figure~\ref{fig:Vega1}. 
In the middle we see an Andr\'asfai graph $\G_i$ together with a three-colouring 
$V(\G_i)=\Gr\dcup\Gg\dcup\Gb$
of its vertex set, where 
$\Gr=\{\re{0}, \dots, \re{i-1}\}$ is a set of \re{red} vertices, 
the set $\Gg=\{\gr{i}, \dots, \gr{2i-1}\}$ is \gr{green}, and 
$\Gb=\{\bl{2i}, \dots, \bl{3i-2}\}$ is \bl{blue}.
The vertices of the {\it external hexagon} 
$\cC_6=\re{a}\gr{v}\bl{c}\re{u}\gr{b}\bl{w}$ are connected to the vertices of 
the same colour of the inner Andr\'asfai graph, so that 
\[
	\Gr=\Nn(\re{a})\cap \Nn( \re{u})\,,
	\qquad 
	\Gg=\Nn(\gr{b})\cap \Nn(\gr{v})\,,
	\quad \text{ and } \quad
	\Gb=\Nn(\bl{c})\cap \Nn(\bl{w})\,.
\]
Finally, there is an edge $xy$ joined to the hexagon so that 
\[
	\{\re{a},\gr{b},\bl{c}\}\subseteq \Nn(x) 
	\quad \text{ and } \quad
	\{\re{u},\gr{v},\bl{w}\}\subseteq \Nn(y)\,.
\]
This completes the description of $\grot_i^{00}$.

\begin{figure}
\centering	
\begin{tikzpicture}[scale=.7]	
	\def\pro{2.5cm};
	\def\pu{1.5pt};
	\coordinate (v) at (30:\pro);
	\coordinate (a) at (90:\pro);
	\coordinate (w) at (150:\pro);
	\coordinate (b) at (210:\pro);
	\coordinate (u) at (270:\pro);
	\coordinate (c) at (330:\pro);				
	\coordinate (x) at (\pro*1.4,\pro*.9);
	\coordinate (y) at (\pro*1.4,-\pro*.9);			
	\foreach \i in {1,...,8} \coordinate (v\i) at (135-\i*45:\pro/2);
	\draw [red!80!black, ,domain=5:89, dashed, thick] plot ({\pro*.6*cos(\x)},{\pro*.6*sin(\x)}); 
	\draw [blue, dashed, domain=142:178, thick] plot ({\pro*.6*cos(\x)},{\pro*.6*sin(\x)}); 
	\draw [green!80!black, dashed,domain=230:313, thick] plot ({\pro*.6*cos(\x)},{\pro*.6*sin(\x)}); 
	\draw [thick] (x)--(y);
	\draw (c)--(x)--(a);
	\draw (v)--(y)--(u);
	\draw [rounded corners=25] (x) -- (-.2*\pro,\pro*1.2)--(-\pro*1.2,\pro*.5)--(b);
	\draw [rounded corners=25] (y) -- (-.2*\pro,-\pro*1.2)--(-\pro*1.2,-\pro*.5)--(w);		
	\foreach \i in {4,...,6} \draw (v\i)--(v1);
	\foreach \i in {5,...,7} \draw (v\i)--(v2);
	\foreach \i in {6,...,8} \draw (v\i)--(v3);
	\foreach \i in {3,...,5} \draw  (v\i)--(v8);
	\draw (v4)--(v7);
	\draw (a)--(v)--(c)--(u)--(b)--(w)--(a);
	\foreach \i in {a,b,c, v, u, w, x, y} \draw [thick] (\i) circle (\pu);
	\foreach \i in {1,...,8}\draw [thick] (v\i) circle (\pu);
	\fill (x) circle (\pu);
	\fill (y) circle (\pu);
	\foreach \i in {c, w, v7, v8}\fill[blue!75!white]  (\i) circle (\pu);
	\foreach \i in {a,u, v1, v2, v3}\fill[red!75!white]  (\i) circle (\pu);
	\foreach \i in {v,b, v4, v5, v6}\fill[green!75!white]  (\i) circle (\pu);
	\node at ($(a)+(0,-.25)$) {\scriptsize{$\re{a}$}};
	\node at ($(b)+(.25,.15)$) {\scriptsize{$\gr{b}$}};
	\node at ($(c)+(-.25,.15)$) {\scriptsize{$\bl{c}$}};
	\node at ($(v)+(-.25,-.15)$) {\scriptsize{$\gr{v}$}};
	\node at ($(w)+(.25,-.15)$) {\scriptsize{$\bl{w}$}};
	\node at ($(u)+(0,.25)$) {\scriptsize{$\re{u}$}};
	\node at ($(x)+(0,.2)$) {\scriptsize{$x$}};
	\node at ($(y)+(0,-.25)$) {\scriptsize{$y$}};
	\node  [scale=.8] at ($(v1)+(-.15,.15)$) {\scriptsize{$\re{0}$}};
	\node  [scale=.8] at ($(v3)+(.2,-.23)$) {\scriptsize{$\re{i-1}$}};
	\node [scale=.8] at ($(v4)+(.2,.1)$) {\scriptsize{$\gr{i}$}};
	\node  [scale=.8] at ($(v6)+(-.05,-.15)$) {\scriptsize{$\gr{2i-1}$}};
	\node  [scale=.8] at ($(v7)+(-.1,-.25)$) {\scriptsize{$\bl{2i}$}};
	\node  [scale=.8] at ($(v8)+(0,.2)$) {\scriptsize{$\bl{3i-2}$}};
\end{tikzpicture}
\caption{The Vega graph $\grot_i^{00}$.}
\label{fig:Vega1}
\end{figure}
 
For each $i\ge 2$ there are three further {\it Vega graphs} 
$\grot_i^{10}=\grot_i^{00}-y$, $\grot_i^{01}=\grot_i^{00}-\gr{(2i-1)}$,
and $\grot_i^{11}=\grot_i^{00}-\{y, \gr{2i-1}\}$.
Thus for $\mu, \nu \in \{0,1\}$ the vertex $y$ belongs to $\grot_i^{\mu\nu}$
if and only if $\mu=0$, while $\nu=0$ indicates the presence of $\gr{2i-1}$.
	
\subsection{Automorphisms}\label{subsec:auto}
We will write $\Aut(\Omega)$ for the automorphism group of a given 
graph~$\Omega$, i.e., for the group of adjacency preserving bijections 
$V(\Omega)\longrightarrow V(\Omega)$. 
The main reason why knowing automorphisms of 
Vega graphs will be helpful for us is that they often allow us to reduce 
the number of cases we need to consider. Moreover, they sometimes suggest 
non-obvious embeddings of smaller Vega graphs into larger ones, that are 
in turn useful when considering a maximal Vega subgraph of a given graph $G$ 
we wish to analyse. In all cases, the lists of automorphisms we provide 
could be shown to be exhaustive, but there is no need for verifying this.    
	
We start with three automorphisms of order two that exist for all $i\ge 2$ 
and appropriate values of $\mu$, $\nu$. First, for both indices $\nu\in\{0, 1\}$
the composition~$\sigma$ of the four transpositions
\[
	x\lra  y\,, 
	\qquad 
	\re{a}\lra \re{u}\,, 
	\qquad 
	\gr{b} \lra \gr{v}\,, 
	\qquad 
	\bl{c}\lra \bl{w}
\]
is an automorphism of $\grot_i^{0\nu}$. 

Second, $\grot_i^{\mu 0}$ has an 
automorphism~$\tau_0$ exchanging the colours \re{red} and \gr{green}. 
More precisely,~$\tau_0$ is the composition of the transpositions
\[
	\re{a} \lra \gr{b}\,, 
	\qquad 
	\re{u} \lra \gr{v}
\]
with the reflection $j\longmapsto 2i-1-j$ of the inner Andr\'asfai graph. 
	
Similarly, $\grot_i^{\mu 1}$ has an automorphism~$\tau_1$ exchanging 
\bl{blue} and \gr{green}, namely the composition of 
\[
	\gr{b }\lra \bl{c}\,, 
	\qquad 
	\gr{v}\lra \bl{w}
\]
with the reflection $j\longmapsto i-1-j$ of $\G_i-(\gr{2i-1})$.
 
It could be shown that for $i\ge 3$ these automorphisms generate the 
entire automorphism group, i.e., that
\[
	\Aut(\grot^{\mu\nu}_i)
	=
	\begin{cases}
		\{1,\sigma, \tau_0, \sigma\tau_0\}  & \text {if } (\mu, \nu) = (0, 0)\cr
		\{1,\sigma, \tau_1, \sigma\tau_1\}  & \text {if } (\mu, \nu) = (0, 1)\cr
		\{1,\tau_0\} & \text{if } (\mu, \nu) = (1, 0)\cr
		\{1,\tau_1\} & \text{if } (\mu, \nu) = (1, 1),
	\end{cases}
\]
but we do not need this knowledge in the sequel.
	
What will be important, however, is that for $i=2$ and $(\mu, \nu)\ne (0, 1)$
there are further `sporadic' automorphisms. We begin their discussion with an 
alternative way of drawing~$\grot_2^{00}$: Start with $\Gamma_3$, add simultaneously 
four twins as indicated in Figure~\ref{fig:Vega_2_rhoC}, and join them to a new vertex. 

\begin{figure}[ht]
\vskip -.3cm
\centering
\begin{subfigure}[b]{0.32\textwidth}
\begin{tikzpicture}[scale=.9]
	\coordinate (r) at (3,0);
	\foreach \i in {1,...,8} {
		\coordinate (a\i) at (\i*45:1.4);
		\coordinate (b\i) at (\i*45:3);
		\coordinate (c\i) at (180+\i*45:2.3);
	}
	\draw (c3) -- (r) -- (c5);
	\draw [rounded corners=35] (r) -- (2.5,2.4)--(c7);
	\draw [rounded corners=35] (r) -- (2.5,-2.4)--(c1);
	\draw (a3) -- (a7);
	\draw (r)--(c3);
	\draw (a6)--(c3)--(a8)--(a4)--(c7)--(a2)--(a6);
	\draw (a2)--(c5)--(a8)--(a5)--(a2);
	\draw (a1)--(a4)--(c1)--(a6)--(a1);	
	\draw (a6)--(a3)--(a8)--(a3)--(c7);
	\draw (a2)--(a7)--(a4)--(a7)--(c3);
	\draw (c1)--(a5)--(a1)--(c5)--(r);
	\foreach \i in {1,3,5,7}{
		\fill (c\i) circle (2pt);
	}			
	\foreach \i in {1,...,8} {
		\fill (a\i) circle (2pt);		
	}
	\fill (r) circle (2pt);
	\fill (a3) circle (1.7pt);
	\fill (a7) circle (1.7pt);
	\foreach \i in {r,c3,c7} \draw [fill=blue] (\i) circle (2pt);
	\foreach \i in {c5,a5,a2,a8} \draw [fill=green!90!black] (\i) circle (2pt);
	\foreach \i in {a1,c1,a4,a6} \draw [fill=red] (\i) circle (2pt);
	\node at ($(r)+(.3,0)$) {\small $\bl{4}$};
	\node at ($(a1)+(.2,-.1)$) {\small $\re{0}$};
	\node at ($(a2)+(0,.25)$) {$\gr{v}$};
	\node at ($(a3)+(-.2,-.05)$) {$x$};
	\node at ($(a4)+(-.3,0)$) {$\re{u}$};
	\node at ($(a5)+(-.2,.1)$) {\small $\gr{3}$};
	\node at ($(a6)+(0,-.25)$) {$\re{a}$};
	\node at ($(a7)+(.2,.1)$) {$y$};
	\node at ($(a8)+(.3,0)$) {$\gr{b}$};
	\node at ($(c1)+(-.2,-.1)$) {\small $\re{1}$};
	\node at ($(c3)+(.2,-.2)$) {$\bl{w}$};
	\node at ($(c5)+(.2,.1)$) {\small $\gr{2}$};
	\node at ($(c7)+(-.2,.1)$) {$\bl{c}$};
\end{tikzpicture}
\vskip -.3cm
\caption{}
\label{fig:Vega_2_rhoC}
\end{subfigure}
\hfill
\begin{subfigure}[b]{0.32\textwidth}
\begin{tikzpicture}[scale=.9]
	\coordinate (r) at (3,0);
	\foreach \i in {1,...,8} {
		\coordinate (a\i) at (\i*45:1.4);
		\coordinate (b\i) at (\i*45:3);
		\coordinate (c\i) at (180+\i*45:2.3);
	}
	\draw (c3) -- (r) -- (c5);
	\draw [blue, rounded corners=35] (r) -- (2.5,2.4)--(c7);
	\draw [yellow, thick, rounded corners=35] (r) -- (2.5,-2.4)--(c1);
	\draw [dashed, rounded corners=35] (r) -- (2.5,-2.4)--(c1);
	\draw [thick] (a3) -- (a7);
	\draw [blue] (r)--(c3);
	\draw [thick, cyan] (a6)--(c3)--(a8)--(a4)--(c7)--(a2)--(a6);
	\draw [green!80!black] (a2)--(c5)--(a8)--(a5)--(a2);
	\draw [red] (a1)--(a4)--(c1)--(a6)--(a1);	
	\draw (a6)--(a3)--(a8)--(a3)--(c7);
	\draw  (a2)--(a7)--(a4)--(a7)--(c3);
	\draw [yellow, thick] (c1)--(a5)--(a1)--(c5)--(r);
	\draw [dashed] (c1)--(a5)--(a1)--(c5)--(r);
	\foreach \i in {1,3,5,7}{
		\draw [thick] (c\i) circle (1.7pt);
	}			
	\foreach \i in {1,...,8} {
		\draw [thick] (a\i) circle (1.7pt);		
	}
	\fill (r) circle (1.7pt);
	\fill (a3) circle (1.7pt);
	\fill (a7) circle (1.7pt);
	\foreach \i in {r,c3,c7} \draw [fill=blue] (\i) circle (2pt);
	\foreach \i in {c5,a5,a2,a8} \draw [fill=green!90!black] (\i) circle (2pt);
	\foreach \i in {a1,c1,a4,a6} \draw [fill=red] (\i) circle (2pt);
	\node at ($(r)+(.3,0)$) {\small $\bl{4}$};
	\node at ($(a1)+(.2,-.1)$) {\small $\re{0}$};
	\node at ($(a2)+(0,.25)$) {$\gr{v}$};
	\node at ($(a3)+(-.2,-.05)$) {$x$};
	\node at ($(a4)+(-.3,0)$) {$\re{u}$};
	\node at ($(a5)+(-.2,.1)$) {\small $\gr{3}$};
	\node at ($(a6)+(0,-.25)$) {$\re{a}$};
	\node at ($(a7)+(.2,.1)$) {$y$};
	\node at ($(a8)+(.3,0)$) {$\gr{b}$};
	\node at ($(c1)+(-.2,-.1)$) {\small $\re{1}$};
	\node at ($(c3)+(.2,-.2)$) {$\bl{w}$};
	\node at ($(c5)+(.2,.1)$) {\small $\gr{2}$};
	\node at ($(c7)+(-.2,.1)$) {$\bl{c}$};
\end{tikzpicture}
\vskip -.3cm
\caption{}
\label{fig:Vega_2_rhoB}
\end{subfigure}
\hfill
\begin{subfigure}[b]{0.32\textwidth}
\begin{tikzpicture}[scale=.77]
	\coordinate (v) at (30:2.2cm);
	\coordinate (a) at (90:2.2cm);
	\coordinate (w) at (150:2.2cm);
	\coordinate (b) at (210:2.2cm);
	\coordinate (u) at (270:2.2cm);
	\coordinate (c) at (330:2.2cm);
	\coordinate (x) at (3,1.5);
	\coordinate (y) at (3,-1.5);
	\foreach \i in {0,...,8}{
		\coordinate (v\i) at (90 + \i*72:1cm);
	}
	\draw [thick] (x)--(y);
	\draw (c)--(x)--(a);
	\draw (v)--(y)--(u);
	\draw [rounded corners=35] (x) -- (0,3)--(-3,1)--(b);
	\draw [rounded corners=35] (y) -- (0,-3)--(-3,-1)--(w);
	\draw [thick, yellow](v0)--(v2)--(v4)--(v1)--(v3)--(v0);
	\draw [dashed] (v0)--(v2)--(v4)--(v1)--(v3)--(v0);
	\draw [thick, cyan] (a)--(v)--(c)--(u)--(b)--(w)--(a);
	\foreach \i in {a,b,c, v, u, w, x, y, v1, v2, v3, v4, v5, v6, v7, v8}{
		\draw[ thick]  (\i) circle (2pt);
	}	
	\draw [red!75!black] (v0) --(a)--(v4);
	\draw [green!74!black] (b) edge [bend right=20](v3);
	\draw [green!74!black](v3) edge [bend right=20] (v);
	\draw [green!74!black] (b)--(v2);
	\draw [green!74!black] (v2) edge [bend left=20] (v);
	\draw [blue!80!black] (w)--(v1);
	\draw [blue!80!black, rounded corners=15] (v1)--(-.7,-1.3)--(c);
	\draw [red!75!black, rounded corners=25] (v0)--(-2,.4)--(u);
	\draw [red!75!black, rounded corners=15] (v4)--(1.5,-.4)--(u);		
	\foreach \i in {c, w,v1}{
		\fill[blue!75!white]  (\i) circle (2pt);
	}	
	\foreach \i in {a,u,v0,v4}{
		\fill[red!75!white]  (\i) circle (2pt);
	}	
	\foreach \i in {v,b,v2,v3}{
		\fill[green!75!white]  (\i) circle (2pt);
	}		
	\fill (x) circle (2pt);
	\fill (y) circle (2pt);	
	\node at ($(a)+(-.25,.07)$) {$\re{a}$};
	\node at ($(b)+(.2,.3)$) {$\gr{b}$};
	\node at ($(c)+(.2,-.1)$) {$\bl{c}$};
	\node at ($(v)+(.2,.1)$) {$\gr{v}$};
	\node at ($(w)+(.4,-.05)$) {$\bl{w}$};
	\node at ($(u)+(-.25,-.1)$) {$\re{u}$};
	\node at ($(x)+(.3,.05)$) {$x$};
	\node at ($(y)+(.3,-.05)$) {$y$};
	\node at ($(v0)+(.23,.05)$) {\small $\re{0}$};
	\node at ($(v4)+(.2,.05)$) {\small $\re{1}$};
	\node at ($(v3)+(.2,-.15)$) {\small $\gr{2}$};
	\node at ($(v2)+(.2,-.15)$) {\small $\gr{3}$};
	\node at ($(v1)+(-.1,-.3)$) {\small $\bl{4}$};
\end{tikzpicture}
\vskip -.3cm
\caption{}
\label{fig:Vega_2_rhoA}
\end{subfigure}
\vskip -.3cm
\caption{Three pictures of $\grot_2^{00}$.}
\label{fig:Vega_2_rho}
\end{figure}
 	
The dihedral group $\DD_4$ acts in the usual way by rotations and reflections 
on the ``imaginary square'' $\bl{c}\re{1}\bl{w}\gr{2}$. This yields a faithful 
action of $\DD_4$ on $\grot_2^{00}$ with fixed point~\bl{$4$} and, in fact, 
it could be proved that 
\[
	\Aut(\grot_2^{00})\cong \DD_4\,. 
\]
We shall occasionally use the reflection~$\varrho$ about the line $\re{a}\gr{v}$. So explicitly~$\varrho$ is the composition of the five transpositions
\[
	\bl{c}\lra\gr{2}\,,  
	\qquad
	\re{u}\lra \gr{b}\,,
	\qquad 
	\re{1}\lra\bl{w}\,, 
	\qquad
	x\lra \re{0}\,,
	\qquad 
	\gr{3}\lra y
\]
(see Figure~\ref{fig:1419}).
As~$\varrho$ exchanges $\gr{3}$ and $y$, it establishes an exceptional 
isomorphism between~$\grot_2^{01}$ and~$\grot_2^{10}$. 
More generally we shall always regard $\varrho$ as an isomorphism 
from~$\grot_2^{\mu\nu}$ to~$\grot_2^{\nu\mu}$. 
The graph~$\grot_2^{01}$ has only the four standard automorphisms 
mentioned earlier. 

Finally, $\grot_2^{11}$ is isomorphic to the Mycielski-Gr\"otzsch graph 
(see Figure~\ref{fig:43}) and its automorphism group can be shown to 
be the symmetry group $\DD_5$ of the ``imaginary pentagon''~$\re{01}\bl{w}x\gr{v}$.

\begin{figure}[h]		
\centering
\begin{tikzpicture}[scale=.8]			
	\coordinate (v) at (30:1.8cm);
	\coordinate (a) at (90:1.8cm);
	\coordinate (w) at (150:1.8cm);
	\coordinate (b) at (210:1.8cm);
	\coordinate (u) at (270:1.8cm);
	\coordinate (c) at (330:1.8cm);
	\coordinate (x) at (2.5,1.5);
	\draw (a)--(x)--(c);
	\draw [rounded corners=20] (x) -- (-.5,2.3)--(-2.3,.8)--(b);
	\draw [cyan, thick]  (a)--(v)--(c)--(u)--(b)--(w)--(a);
	\foreach \i in {0,...,5} \coordinate (v\i) at (90-72*\i:.9);
	\coordinate (v3) at (234:.7);
	\draw [black!10] (v0)--(v3)--(v1);
	\draw [thick] (v0)--(v2)--(v4)--(v1);
	\draw [red!75!black] (v0) --(a)--(v1);
	\draw [green!74!black] (b)--(v2)--(v);
	\draw [blue!80!black] (w)--(v4);
	\draw [blue!80!black, rounded corners=15] (v4)--(-.6,-1.3)--(c);
	\draw [red!75!black, rounded corners=25] (v0)--(-2,.4)--(u);
	\draw [red!75!black, rounded corners=15] (v1)--(1.5,-.4)--(u);
	\node at ($(a)+(-.2,.1)$) {$\re{a}$};
	\node at ($(b)+(-.2,-.2)$) {$\gr{b}$};
	\node at ($(c)+(.2,-.2)$) {$\bl{c}$};
	\node at ($(v)+(.05,.2)$) {$\gr{v}$};
	\node at ($(w)+(0,.2)$) {$\bl{w}$};
	\node at ($(u)+(-.3,-.1)$) {$\re{u}$};
	\node at ($(x)+(-.1,.2)$) {$x$};
	\node at ($(v0)+(-.2,.2)$) {\footnotesize \re{0}};
	\node at ($(v1)+(.1,.2)$) {\footnotesize \re{1}};
	\node at ($(v2)+(.2,-.05)$) {\footnotesize  \gr{2}};
	\node at ($(v4)+(-.15,-.15)$) {\footnotesize  \bl{4}};
	\node [black!10] at ($(v3)+(-.1,.2)$) {\footnotesize 3};
	\foreach \i in {a,b,c,u,v,w,v0,v1,v2,v4,v5, x} \fill (\i) circle (2pt);
	\foreach \i in {a,u, v0, v1}{
		\fill [red]  (\i) circle (1.7pt);
	}	
	\foreach \i in {c, w,v4}{
		\fill[blue!75!white]  (\i) circle (1.7pt);
	}	
	\foreach \i in {v,b,v2}{
		\fill[green]  (\i) circle (1.7pt);
	}				
	\fill [black!10] (v3) circle (2pt);
	\coordinate (c) at (5,0);
	\foreach \i in {1,...,5} {
		\coordinate (a\i) at ($(-18+\i*72:1)+(5,0)$);
		\coordinate (b\i) at ($(18+\i*72:1.5)+(5,0)$);	
		\coordinate (d\i) at ($(18+\i*72:2)+(5,0)$);
	}
	\draw (b3)--(a4)--(b4)--(a4)--(c);
	\draw (b1) [thick, cyan, out=162, in=72] to (d2) [out=-108, in=162] to (b3);
	\draw (b2) [blue!75!black,  out=234, in=144] to (d3) [out=-36, in=234] to (b4);
	\draw (b3) [green!75!black, out=-54, in=216] to (d4) [out=36, in=-54] to (b5);
	\draw (b4) [thick, cyan, out=18, in=-72] to (d5) [out=108, in=18] to (b1);
	\draw (b5) [thick, out=90, in=0] to (d1) [out=180, in=90] to (b2);
	\draw [thick] (a2)--(b2);
	\draw [thick] (a1)--(b5);
	\draw [thick, cyan] (b3)--(a3)--(c)--(a5)--(b4);
	\draw [red!75!black] (b1)--(a1)--(c)--(a2)--(b1);
	\draw [blue!75!black] (b2)--(a3);
	\draw [green!75!black] (b5)--(a5);
	\fill (c) circle (2pt);
	\foreach \i in {1,...,5} {
		\fill (a\i) circle (2pt);		
		\fill (b\i) circle (2pt);
	}
	\foreach \i in {a2, a1, c, b1} \fill [red] (\i) circle (1.7pt);
	\foreach \i in {b2, a3, b4} \fill [blue] (\i) circle (1.7pt);
	\foreach \i in {b5, b3, a5} \fill [green] (\i) circle (1.7pt); 
	\node at ($(c)+(-.3,.15)$) {$\re{a}$};
	\node at ($(a3)+(-.25,-.15)$) {$\bl{w}$};
	\node at ($(a4)+(0,-.23)$) {$x$};
	\node at ($(a5)+(.25,-.1)$) {$\gr{v}$};
	\node at ($(b3)+(-.1,-.2)$) {$\gr{b}$};
	\node at ($(b4)+(.1,-.2)$) {$\bl{c}$};
	\node at ($(b1)+(0,.2)$) {$\re{u}$};
	\node at ($(a1)+(.17,.15)$) {\footnotesize  $\re{0}$};
	\node at ($(a2)+(-.17,.15)$) {\footnotesize  $\re{1}$};
	\node at ($(b2)+(-.2,.05)$) {\footnotesize $\bl{4}$};
	\node at ($(b5)+(.2,.05)$) {\footnotesize  $\gr{2}$};
\end{tikzpicture}
\caption{The isomorphism $\grot^{11}_2\cong\grot$.}
\label{fig:43}
\end{figure}		
 
Let us conclude this subsection by showing that we cannot obtain~$\grot_i^{11}$ 
from~$\grot_i^{00}$ by deleting an edge.  

\begin{lemma}\label{c:310}
	If $i\ge 2$ and $q, z\in V(\grot_i^{00})$ satisfy 
	$\grot_i^{00}-\{q,z\}\cong \grot_i^{11}$, then $qz\not\in E(\grot_i^{00})$.
\end{lemma}

\begin{proof}
	Assume contrariwise that for some edge $qz$ of $\grot_i^{00}$ the 
	graphs $\grot_i^{00}-\{q,z\}$ and $\grot_i^{11}$ are isomorphic. 
	We label the vertices of $\grot_i^{00}$ as in Figure~\ref{fig:Vega1}. 
	Recall that in $\grot_i^{11}$ any two non-adjacent vertices have a common 
	neighbour. Since $\re{u}$ is the only common neighbour of~$\bl{c}$ and~$\re{0}$, 
	this proves that if $q=\re{u}$, then $z\in \{\bl{c},\re{0}\}$.
	But $\re{u}$ is also the only common neighbour of~$y$ and $\re{1}$ and, 
	therefore, $q=\re{u}$ would imply $z\in \{y,\re{1}\}$ as well.  
	Altogether the case $q=\re{u}$ is impossible. By $\sigma$- and $\tau_0$-symmetry
	this argument actually shows
	\begin{equation}\label{eq:38}
		q, z\not\in \{\re{a}, \gr{b}, \re{u}, \gr{v}\}\,.
	\end{equation}
	We indicate degrees of vertices in $\grot_i^{00}$ by $\dd(\cdot)$.
	Because of 
	\[
		|E(\grot_i^{00})|-|E(\grot_i^{11})|=\dd(y)+\dd(\gr{2i-1})=4+(i+2)=i+6
	\]
	and $qz\in E(\grot_i^{00})$ we have
	\begin{equation}\label{eq:39}
		\dd(q)+\dd(z)=i+7\,.
	\end{equation}
	As $\grot^{00}_i$ has the degree table
	\begin{center}
	\begin{tabular}{c|c|c|c}
			$t\in$&  $\{\re{a}, \gr{b}, \re{u}, \gr{v}\}$&  
			$\G_i\cup\{\bl{c}, \bl{w}\}$ & $\{x, y\}$
			\\ \hline
			$\dd(t)=$&  $i+3$&  $i+2$ & $4$
	\end{tabular}\,\,,
	\end{center}
	it follows from~\eqref{eq:38} and~\eqref{eq:39} that $i=3$ 
	and $q, z\in \G_i\cup\{\bl{c}, \bl{w}\}$. 
	It is not difficult to see, however, that if two adjacent vertices 
	belonging to this set are deleted from $\grot_i^{00}$, 
	then an even number of the vertices $\re{a},\gr{b},\re{u},\gr{v}$ 
	keeps the degree $i+3$. The graph~$\grot^{11}_i$, on the other hand, 
	has exactly one such vertex (namely, the vertex which would be 
	called~$\re{a}$ in the standard labelling of~$\grot^{11}_i$). 
	This contradiction concludes the proof.
\end{proof}

\subsection{Properties of maximal triangle-free graphs satisfying 
\texorpdfstring{\Dp{4}}{D4}}
	
Let $\fD_4$ be the class of maximal triangle-free graphs satisfying \Dp{4}. 
In this subsection we present two lemmata on subgraphs of such graphs. 
The first of them concerns the {\it cube}, that is the graph remaining 
from~$K_{4,4}$ after the deletion of a perfect matching. Brandt~\cite{B}
proved that maximal triangle-free graphs $G$ with $\delta(G)>|V(G)|/3$ 
contain no induced cubes. His argument goes through under the weaker 
assumption $G\in \fD_4$ and for the sake of completeness we would like 
to provide full details. 
	
\begin{lemma}[Cube lemma]\label{l:31}
	No graph in $\fD_4$ contains an induced cube.
\end{lemma}
\begin{proof}
	Assume contrariwise that some $G\in \fD_4$ has eight vertices 
	\[
		a_1, a_2, a_3, a_4, b_1, b_2, b_3, b_4
	\]
	such that $a_ib_j\in E(G)$ whenever $i, j\in [4]$ are distinct, 
	whilst $a_ib_i\not\in E(G)$ for all $i\in [4]$.

	\begin{figure}[h!]
	\centering
	\begin{subfigure}[b]{0.23\textwidth}
	\begin{tikzpicture}[scale=.7]
		\coordinate (a1) at (-1,1.5);
		\coordinate (b4) at (-1.5,1);
		\coordinate (a3) at (1,-1.5);
		\coordinate (b2) at (1.5,-1);
		\coordinate (b3) at (-1,-1.5);
		\coordinate (a4) at (-1.5,-1);
		\coordinate (b1) at (1,1.5);
		\coordinate (a2) at (1.5,1);
		\coordinate (c1) at (0,1.5);
		\coordinate (c2) at (1.5,0);
		\coordinate (c3) at (0,-1.5);
		\coordinate (c4) at (-1.5,0);
		\foreach \i in {1,...,4} {
			\fill (a\i) circle (2pt);
			\fill (b\i) circle (2pt);	
			\fill [blue] (c\i) circle (2pt);
			\draw [blue, thick] (a\i)--(b\i);
		}
		\draw (a1)--(b2)--(a3)--(b4)--(a1)--(b3)--(a4)--(b1)--(a2)--(b3);
		\draw (b2)--(a4);
		\draw (a3) -- (b1);
		\draw (b4) --(a2);
		\node at ($(a1)+(0,.2)$) {\tiny $a_2$};
		\node at ($(a2)+(.28,0)$) {\tiny $a_3$};
		\node at ($(a3)+(0,-.2)$) {\tiny $a_4$};
		\node at ($(a4)+(-.23,0)$) {\tiny $a_1$};
		\node at ($(b1)+(.08,.2)$) {\tiny $b_2$};
		\node at ($(b2)+(.28,0)$) {\tiny $b_3$};
		\node at ($(b3)+(0,-.15)$) {\tiny $b_4$};
		\node at ($(b4)+(-.23,0)$) {\tiny $b_1$};
		\node [blue] at ($(c1)+(0,.2)$) {\tiny $c_2$};
		\node [blue] at ($(c2)+(.28,0)$) {\tiny $c_3$};
		\node [blue] at ($(c3)+(0,-.2)$) {\tiny $c_4$};
		\node [blue] at ($(c4)+(-.23,0)$) {\tiny $c_1$};
	\end{tikzpicture}
	\caption{}
	\label{fig:Wurfela}
	\end{subfigure}
	\hfill
	\begin{subfigure}[b]{0.23\textwidth}
	\begin{tikzpicture}[scale=.7]
		\coordinate (a1) at (-1,1.5);
		\coordinate (b4) at (-1.5,1);
		\coordinate (a3) at (1,-1.5);
		\coordinate (b2) at (1.5,-1);
		\coordinate (b3) at (-1,-1.5);
		\coordinate (a4) at (-1.5,-1);
		\coordinate (b1) at (1,1.5);
		\coordinate (a2) at (1.5,1);
		\coordinate (c1) at (0,1.5);
		\coordinate (c2) at (1.5,0);
		\coordinate (c3) at (0,-1.5);
		\coordinate (c4) at (-1.5,0);
		\coordinate (t) at (0,-1.3);
		\draw [red, thick ] (b3)--(t)--(a3);
		\draw [red, thick ] (c4) --(t) -- (c2) --(t) --(c1);
		\fill (t) circle (2pt);
		\foreach \i in {1,...,4} {
			\fill (a\i) circle (2pt);
			\fill (b\i) circle (2pt);	
			\fill  (c\i) circle (2pt);
			\draw (a\i)--(b\i);
		}
		\draw (a1)--(b2)--(a3)--(b4)--(a1)--(b3)--(a4)--(b1)--(a2)--(b3);
		\draw (b2)--(a4);
		\draw (a3) -- (b1);
		\draw (b4) --(a2);
		\node at ($(a1)+(0,.2)$) {\tiny $a_2$};
		\node at ($(a2)+(.28,0)$) {\tiny $a_3$};
		\node at ($(a3)+(0,-.2)$) {\tiny $a_4$};
		\node at ($(a4)+(-.23,0)$) {\tiny $a_1$};
		\node at ($(b1)+(.08,.2)$) {\tiny $b_2$};
		\node at ($(b2)+(.28,0)$) {\tiny $b_3$};
		\node at ($(b3)+(0,-.15)$) {\tiny $b_4$};
		\node at ($(b4)+(-.23,0)$) {\tiny $b_1$};
		\node at ($(c1)+(0,.2)$) {\tiny $c_2$};
		\node at ($(c2)+(.28,0)$) {\tiny $c_3$};
		\node at ($(c3)+(0,-.2)$) {\tiny $c_4$};
		\node at ($(c4)+(-.23,0)$) {\tiny $c_1$};
		\node [red] at ($(t)+(.35,.1)$) {\tiny $t$};
	\end{tikzpicture}
	\caption{}
	\label{fig:Wurfelb}
	\end{subfigure}
	\hfill
	\begin{subfigure}[b]{0.23\textwidth}
	\begin{tikzpicture}[scale=.7]
		\coordinate (a1) at (-1,1.5);
		\coordinate (b4) at (-1.5,1);
		\coordinate (a3) at (1,-1.5);
		\coordinate (b2) at (1.5,-1);
		\coordinate (b3) at (-1,-1.5);
		\coordinate (a4) at (-1.5,-1);
		\coordinate (b1) at (1,1.5);
		\coordinate (a2) at (1.5,1);
		\coordinate (c1) at (0,1.5);
		\coordinate (c2) at (1.5,0);
		\coordinate (c3) at (0,-1.5);
		\coordinate (c4) at (-1.5,0);
		\coordinate (s) at (0,1.3);
		\draw [black!20] (b2)--(a3)--(b4)--(a1)--(b3)--(a4)--(b1)--(a2)--(b3);
		\draw [black!20] (b3)--(a3) -- (b1);
		\draw [black!20] (c4) --(c3)--(c1)--(c3)--(c2);
		\draw [thick] (b4)--(a4)--(b1)--(a2);
		\draw [thick] (a4)--(b2)--(a2)--(b4);
		\draw [thick] (b4)--(a1)--(b2);
		\draw [thick] (a1) -- (b1);
		\draw [thick, red] (b1)--(s)--(a1)--(s)--(c2)--(s)--(c4);
		\fill (s) circle (2pt);
		\foreach \i in {1,2,4} {
			\fill (a\i) circle (2pt);
			\fill (b\i) circle (2pt);	
			\fill  (c\i) circle (2pt);
		}
		\foreach \i in {a3,b3,c3} \fill [black!20] (\i) circle (2pt);
		\node at ($(a1)+(0,.2)$) {\tiny $a_2$};
		\node at ($(a2)+(.28,0)$) {\tiny $a_3$};
		\node [black!20] at ($(a3)+(0,-.2)$) {\tiny $a_4$};
		\node at ($(a4)+(-.23,0)$) {\tiny $a_1$};
		\node at ($(b1)+(.08,.2)$) {\tiny $b_2$};
		\node at ($(b2)+(.28,0)$) {\tiny $b_3$};
		\node [black!20] at ($(b3)+(0,-.15)$) {\tiny $b_4$};
		\node at ($(b4)+(-.23,0)$) {\tiny $b_1$};
		\node at ($(c1)+(0,.2)$) {\tiny $c_2$};
		\node at ($(c2)+(.28,0)$) {\tiny $c_3$};
		\node [black!20] at ($(c3)+(0,-.2)$) {\tiny $c_4$};
		\node at ($(c4)+(-.23,0)$) {\tiny $c_1$};
		\node [red] at ($(s)+(.3, -.08)$) {\tiny $s$};
	\end{tikzpicture}
	\caption{}
	\label{fig:Wurfelc}
	\end{subfigure}
	\hfill
	\begin{subfigure}[b]{0.23\textwidth}
	\begin{tikzpicture}[scale=.7]
		\coordinate (a1) at (-1,1.5);
		\coordinate (b4) at (-1.5,1);
		\coordinate (a3) at (1,-1.5);
		\coordinate (b2) at (1.5,-1);
		\coordinate (b3) at (-1,-1.5);
		\coordinate (a4) at (-1.5,-1);
		\coordinate (b1) at (1,1.5);
		\coordinate (a2) at (1.5,1);
		\coordinate (c1) at (0,1.5);
		\coordinate (c2) at (1.5,0);
		\coordinate (c3) at (0,-1.5);
		\coordinate (c4) at (-1.5,0);
		\draw [red, thick] (c3)--(c2)--(c1)--(c4)--(c3);
		\draw (a1)--(b2)--(a3)--(b4)--(a1)--(b3)--(a4)--(b1)--(a2)--(b3);
		\draw (a2)--(b2)--(a4)--(b4);
		\draw (a3) -- (b1);
		\draw [thick, red] (b1)--(a1);
		\draw (b4) --(a2);
		\draw [thick, red] (a3)--(b3);
		\foreach \i in {1,...,4} {
			\fill (a\i) circle (2pt);
			\fill (b\i) circle (2pt);	
			\fill  (c\i) circle (2pt);
		}
		\node at ($(a1)+(0,.2)$) {\tiny $a_2$};
		\node at ($(a2)+(.28,0)$) {\tiny $a_3$};
		\node at ($(a3)+(0,-.2)$) {\tiny $a_4$};
		\node at ($(a4)+(-.23,0)$) {\tiny $a_1$};
		\node at ($(b1)+(.08,.2)$) {\tiny $b_2$};
		\node at ($(b2)+(.28,0)$) {\tiny $b_3$};
		\node at ($(b3)+(0,-.15)$) {\tiny $b_4$};
		\node at ($(b4)+(-.23,0)$) {\tiny $b_1$};
		\node at ($(c1)+(0,.2)$) {\tiny $c_2$};
		\node  at ($(c2)+(.28,0)$) {\tiny $c_3$};
		\node at ($(c3)+(0,-.2)$) {\tiny $c_4$};
		\node at ($(c4)+(-.23,0)$) {\tiny $c_1$};
	\end{tikzpicture}
	\caption{}
	\label{fig:Wurfeld}
	\end{subfigure}
	\caption{The proof of the cube lemma.}
	\label{fig:lem31}
	\end{figure}
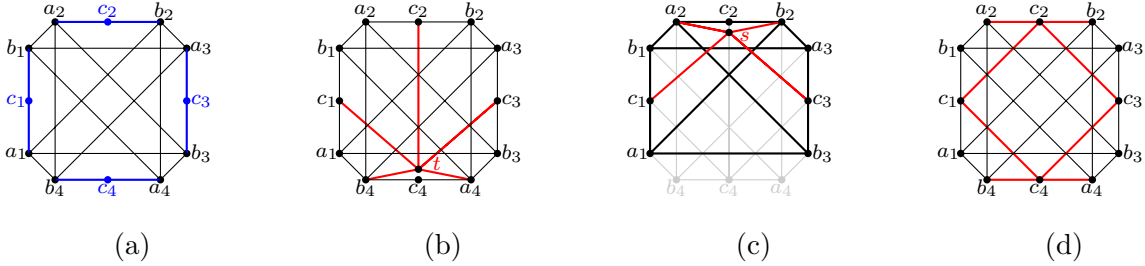
 	
	Because $G$ has diameter two, there exist vertices 
	$c_i\in \Nn(a_i)\cap \Nn(b_i)$ for all $i\in [4]$ 
	(see Figure~\ref{fig:Wurfela}). 
	By~\Dp{4} there is a five-element set 
	$T\subseteq\{a_1, a_2, a_3, a_4, b_1, b_2, b_3, b_4, c_1, c_2, c_3, c_4\}$
	possessing a common neighbour $t$.
	Since $T$ contains at most one vertex from each of the edges $b_ic_i$, we may 
	assume $a_4\in T$. Similarly we obtain $b_4\in T$ and thus 
	$T=\{a_4, b_4, c_1, c_2, c_3\}$. So $c_4$ can be replaced by $t$ 
	(see Figure~\ref{fig:Wurfelb})
	and, without loss of generality, we may assume 
	\begin{equation}\label{eq:31}
		c_1c_4, c_2c_4,c_3c_4\in E(G)\,.
	\end{equation}
	Next we apply \Dp{3} to the nine vertices 
	$a_1$, $a_2$, $a_3$, $b_1$, $b_2$, $b_3$, $c_1$, $c_2$, $c_3$,
	thereby finding a vertex $s$ which is, without loss of generality,  
	adjacent to $a_2$, $b_2$, $c_1$, $c_3$ (see Figure~\ref{fig:Wurfelc}).		
	Replacing $c_2$ by $s$ we change~\eqref{eq:31} to the more `symmetric'
	configuration
	\[
		c_1c_2, c_2c_3, c_3c_4, c_4c_1\in E(G)\,.
	\]
	But now the largest independent set among the twelve vertices 
	$a_i$, $b_i$, $c_i$ with $i\in [4]$
	has size four, contrary to~\Dp{4}. 
\end{proof}

In the sequel, whenever we apply the above lemma, we write

\begin{center}
	\begin{tabular}{cc}
		$(Q)$&
		\begin{tabular}{c|c|c|c}
			$a_1$&  $a_2$&  $a_3$ & $a_4$
			\\ \hline
			$b_1$&  $b_2$&  $b_3$ & $b_4$
		\end{tabular}
	\end{tabular}
\end{center}		
to denote the cube $Q$ with vertices $a_i$, $b_i$ and 
$E(Q) =\{a_i b_j\colon i\neq j \text{ and }i,j\in[4]\}$. 
Notice that if $Q$ appears as a non-induced subgraph of some triangle-free 
graph $G$, then one of the edges $a_ib_i$ needs to be present in $G$.

\medskip

We shall now take a closer look at the graph $N$ displayed in Figure~\ref{fig:NN},
which we have already encountered in the proofs of Lemma~\ref{l:21} 
and Lemma~\ref{l:31}. It will be convenient to read the indices $0$, $1$, $2$ 
in~$N$ modulo $3$. As we have seen in earlier proofs, if $N$ is a subgraph of 
some graph $G$ satisfying~\Dp{3}, then one of the three sets 
$\{a_i, b_i, c_{i+1}, c_{i+2}\}$ needs to have a 
common neighbour. It turns out that for $G\in\fD_4$ much more is true. 	

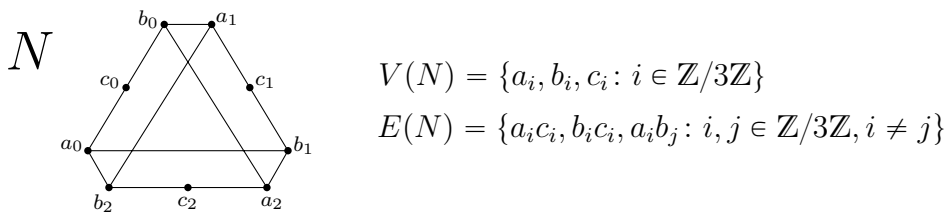
\begin{figure}[h!]
\vskip -.2cm
\centering		
\begin{tikzpicture}[scale=.7]
	\coordinate (c2) at (0,-.1);
	\coordinate (b2) at (-1.5,-.1);
	\coordinate (a2) at (1.5,-.1);
	\coordinate (b0) at (-.45,3);
	\coordinate(a1) at (.45,3);
	\coordinate (a0) at (-1.9,.6);
	\coordinate (b1) at (1.9,.6);
	\coordinate (c0) at ($(b0)!.5!(a0)$);
	\coordinate (c1) at ($(b1)!.5!(a1)$);
	\draw (a0)--(b0);
	\draw (a1)--(b1);
	\draw (a2)--(b2);
	\draw (b0)--(a1)--(b2)--(a0)--(b1)--(a2)--cycle;
	\foreach \i in {0,1,2}{
		\fill (a\i) circle (2pt);
		\fill (b\i) circle (2pt);
		\fill (c\i) circle (2pt);
	}
	\node at ($(b0)+(-.3,.1)$) {\tiny $b_0$};
	\node at ($(c0)+(-.3,.1)$) {\tiny $c_0$};
	\node at ($(a0)+(-.3,.1)$) {\tiny $a_0$};
	\node at ($(b1)+(.3,.1)$) {\tiny $b_1$};
	\node at ($(c1)+(.3,.1)$) {\tiny $c_1$};
	\node at ($(a1)+(.3,.1)$) {\tiny $a_1$};
	\node at ($(b2)+(-.1,-.3)$) {\tiny $b_2$};
	\node at ($(c2)+(0,-.3)$) {\tiny $c_2$};
	\node at ($(a2)+(.1,-.3)$) {\tiny $a_2$};
	\node at (-3, 2.5) {\LARGE $N$};
	\node at (7.3, 2) {$V(N) = \{a_i, b_i, c_i\colon i \in \ZZ/3\ZZ\}$};
	\node at (9, 1) {$E(N) =\{a_ic_i, b_ic_i, a_ib_j\colon i,j\in \ZZ/3\ZZ, i\neq j\}$};
\end{tikzpicture}
\vskip -.3cm
\caption{The graph $N$.}
\label{fig:NN}
\vskip -.3cm
\end{figure}
 
\begin{lemma}\label{l:35}
	If $N$ is a subgraph of some $G\in \fD_4$, then 
	for every $i\in \ZZ/3\ZZ$ either $c_{i-1}c_{i+1}$ 
	is an edge of $G$ or there is a common neighbour 
	of $a_i$, $b_i$, $c_{i-1}$, $c_{i+1}$.
\end{lemma}

\begin{proof}
	The only independent sets of size four in $N$ are
	\[
		\{a_0,b_0,c_1,c_2\}\,,
		\qquad 
		\{a_1,b_1,c_0,c_2\}\,,
		\quad\text{ and } \quad 
		\{a_2,b_2,c_0,c_1\}\,.
	\]
	For every $i\in \ZZ/3\ZZ$ let $X_i$ be the set of common neighbours 
	of $a_i$, $b_i$, $c_{i-1}$, $c_{i+1}$. As $G$ satisfies~\Dp{3}, the 
	sets $X_0$, $X_1$, and $X_2$ cannot be empty simultaneously and thus 
	we can assume~${X_0\ne\vn}$. For reasons of symmetry we only need to 
	prove that if $c_0c_2\not\in E(G)$, then~$X_1\ne\vn$. To this end we 
	consider a common neighbour $y$ of $c_0$, $c_2$ (see Figure~\ref{fig:lem31A}).
	
	\begin{figure}[ht]
	\vskip -.3cm
	\centering
	\begin{subfigure}[b]{.173\textwidth}
	\centering
	\begin{tikzpicture}[scale=.51]
		\coordinate (c2) at (0,-.1);
		\coordinate (b2) at (-1.5,-.1);
		\coordinate (a2) at (1.5,-.1);
		\coordinate (b0) at (-.45,3);
		\coordinate (a1) at (.45,3);
		\coordinate (a0) at (-1.9,.6);
		\coordinate (b1) at (1.9,.6);
		\coordinate (c0) at ($(b0)!.5!(a0)$);
		\coordinate (c1) at ($(b1)!.5!(a1)$);
		\coordinate (y) at (0,1.25);
		\draw [thin] (a0)--(b0);
		\draw [thin] (a1)--(b1);
		\draw [thin] (a2)--(b2);
		\draw [thick, blue] (b0)--(a1)--(b2)--(a0)--(b1)--(a2)--cycle;
		\draw [red, thick] (c2)--(y)--(c0);
		\foreach \i in {0,1,2}{
			\fill [blue] (a\i) circle (2pt);
			\fill [blue] (b\i) circle (2pt);
		}
		\fill (c1) circle (2pt);
		\fill [red](c0) circle (2pt);
		\fill [red](c2) circle (2pt);
		\fill [red](y) circle (2pt);
		\node at ($(b0)+(-.35,.1)$) {\tiny $b_0$};
		\node at ($(c0)+(-.35,.1)$) {\tiny $c_0$};
		\node at ($(a0)+(-.35,.1)$) {\tiny $a_0$};
		\node at ($(b1)+(.35,.1)$) {\tiny $b_1$};
		\node at ($(c1)+(.35,.1)$) {\tiny $c_1$};
		\node at ($(a1)+(.35,.1)$) {\tiny $a_1$};
		\node at ($(b2)+(-.1,-.3)$) {\tiny $b_2$};
		\node at ($(c2)+(0,-.3)$) {\tiny $c_2$};
		\node at ($(a2)+(.1,-.3)$) {\tiny $a_2$};
		\node at ($(y)+(.2,0)$) {\tiny $y$};
	\end{tikzpicture}
	\vskip -.2cm
	\caption{}\label{fig:lem31A}
	\end{subfigure}
	\hfill
	\begin{subfigure}[b]{.268\textwidth}
	\centering
	\begin{tikzpicture}[scale=.8]
		\foreach \i in {1,...,8} {
			\coordinate (a\i) at (\i*45:1);
		}
		\coordinate (b1) at (-2.2,0.5);
		\coordinate (b2) at (2.2,.5);
		\coordinate (b3) at (1.5,-1.5);
		\coordinate(b4) at (-1.5,-1.5);
		\draw [black!20] (a7)--(b2)--(b3);
		\draw [red, thick] (a5)--(b1)--(b4);
		\draw [blue, thick, rounded corners=1] (a4)--(a7)--(a2)--(a6)--(b3)--(a8)--cycle;
		\draw [ultra thin] (a4)--(b4)--(a6);
		\draw  [ultra thin] (a2)--(a5)--(a8);
		\foreach \i in {2,4,5,6,7,8} {
			\fill (a\i) circle (1.5pt);		
		}
		\foreach \i in {1,3,4} {
			\fill (b\i) circle (1.5pt);		
		}
		\fill [black!20] (b2) circle (1.5pt);
		\foreach \i in {b4,b1,a5} \fill [red] (\i) circle (1.5pt);
		\foreach \i in {a4,a7,a2,a6,b3,a8} \fill [blue] (\i) circle (1.5pt);
		\node at ($(a2)+(0,.25)$) {\tiny $b_2$};
		\node at ($(a4)+(-.2,.2)$) {\tiny $b_0$};
		\node at ($(a5)+(-.2,-.2)$) {\tiny $c_2$};
		\node at ($(a6)+(0,-.3)$) {\tiny $a_0$};
		\node at ($(a7)+(.2,-.2)$) {\tiny $a_1$};
		\node at ($(a8)+(.2,.2)$) {\tiny $a_2$};
		\node at ($(b1)+(-.2,.1)$) {\tiny $y$};
		\node at ($(b3)+(.3,-.1)$) {\tiny $b_1$};
		\node [black!20] at ($(b2)+(.2,.1)$) {\tiny $c_1$};
		\node at ($(b4)+(-.3,-.1)$) {\tiny $c_0$};
	\end{tikzpicture}
	\vskip -.2cm
	\caption{}
	\label{fig:lem31B}
	\end{subfigure} 
	\hfill
	\begin{subfigure}[b]{.268\textwidth}
	\centering
	\begin{tikzpicture}[scale=.8]
		\foreach \i in {1,...,8} {
			\coordinate (a\i) at (\i*45:1);
		}
		\coordinate (b1) at (-2.2,0.5);
		\coordinate (b2) at (2.2,.5);
		\coordinate (b3) at (1.5,-1.5);
		\coordinate(b4) at (-1.5,-1.5);
		\draw [ultra thin] (a3)--(b1)--(b4)--(a6);
		\draw [ultra thin] (a1)--(a4)--(a7)--(a2)--(a5)--(a8)--(a3)--(a6)--(a1);
		\draw [ultra thin] (a3) -- (a7) -- (b2);
		\draw [ultra thin] (b4) -- (a4);
		\draw [ultra thin] (a4) -- (a8) -- (b3);
		\draw [ultra thin] (b1)--(a5);
		\draw [ultra thin] (a1)--(b2)--(b3)--(a6);
		\draw [ultra thin] (a1) -- (a5);
		\draw [ultra thin] (a2) -- (a6);
		\draw [blue, thick, rounded corners=1] (a2)--(a5)--(a8)--(a4)--(a7)--cycle;
		\draw [red, thick] (b1)--(b4);
		\draw [red, thick] (b2)--(b3);
		\foreach \i in {1,...,8} {
			\fill (a\i) circle (1.5pt);		
		}
		\foreach \i in {2,4,5,7,8} {
			\fill [blue](a\i) circle (1.5pt);		
		}
		\foreach \i in {1,...,4} {
			\fill [red] (b\i) circle (1.5pt);		
		}
		\node [blue] at (.25,.8) {\tiny $P$};
		\node at ($(a1)+(.2,.2)$) {\tiny $x_0$};
		\node at ($(a2)+(0,.25)$) {\tiny $b_2$};
		\node at ($(a3)+(-.2,.2)$) {\tiny $z$};
		\node at ($(a4)+(-.2,.2)$) {\tiny $b_0$};
		\node at ($(a5)+(-.2,-.2)$) {\tiny $c_2$};
		\node at ($(a6)+(0,-.3)$) {\tiny $a_0$};
		\node at ($(a7)+(.2,-.2)$) {\tiny $a_1$};
		\node at ($(a8)+(.2,.2)$) {\tiny $a_2$};
		\node at ($(b1)+(-.2,.1)$) {\tiny $y$};
		\node at ($(b3)+(.3,-.1)$) {\tiny $b_1$};
		\node at ($(b2)+(.2,.1)$) {\tiny $c_1$};
		\node at ($(b4)+(-.3,-.1)$) {\tiny $c_0$};
		\node [red] at (2.05,-.5) {\tiny $e_2$};
		\node [red] at (-2,-.5) {\tiny $e_1$};
	\end{tikzpicture}
	\vskip -.2cm
	\caption{}
	\label{fig:lem31C}
	\end{subfigure}
	\hfill
	\begin{subfigure}[b]{.268\textwidth}
	\centering
	\begin{tikzpicture}[scale=.8]
		\foreach \i in {1,...,8} {
			\coordinate (a\i) at (\i*45:1);
		}
		\coordinate (b1) at (-2.2,.15);
		\coordinate (b2) at (2.1,-.1);
		\coordinate (b3) at (1.5,-1.5);
		\coordinate(b4) at (-1.5,-1.6);
		\coordinate (q) at (1.4,1.6);
		\draw [red, thick] (a8)--(q)--(b2) -- (q) --(a2);
		\draw (a3)--(b1)--(b4)--(a6);
		\draw (a1) -- (a4) -- (a7) -- (a2) -- (a5) -- (a8) -- (a3) -- (a6) -- (a1);
		\draw (a3) -- (a7) -- (b2);
		\draw (b4) -- (a4);
		\draw (a4) -- (a8);
		\draw (b1)--(a5);
		\draw (a1)--(b2);
		\draw (a1) -- (a5);
		\draw (a2) -- (a6);
		\draw [black!20]  (a8) -- (b3) --(b2) --(b3)--(a6);
		\foreach \i in {1,...,8} {
			\fill (a\i) circle (1.5pt);		
		}
		\foreach \i in {1,2,4} {
			\fill (b\i) circle (1.5pt);		
		}
		\fill (q) circle (1.5pt);
		\fill [black!20] (b3) circle (1.5pt);
		\node at ($(a1)+(.2,.1)$) {\tiny $x_0$};
		\node at ($(a2)+(0,.25)$) {\tiny $b_2$};
		\node at ($(a3)+(-.2,.2)$) {\tiny $z$};
		\node at ($(a4)+(-.2,.2)$) {\tiny $b_0$};
		\node at ($(a5)+(-.2,-.2)$) {\tiny $c_2$};
		\node at ($(a6)+(0,-.3)$) {\tiny $a_0$};
		\node at ($(a7)+(.2,-.2)$) {\tiny $a_1$};
		\node at ($(a8)+(.25,0)$) {\tiny $a_2$};
		\node at ($(b1)+(-.2,.1)$) {\tiny $y$};
		\node [black!20] at ($(b3)+(.3,-.1)$) {\tiny $b_1$};
		\node at ($(b2)+(.2,.1)$) {\tiny $c_1$};
		\node at ($(b4)+(-.3,-.1)$) {\tiny $c_0$};
		\node at ($(q)+(.2,.1)$) {\tiny $q$};
	\end{tikzpicture}
	\vskip -.2cm
	\caption{}
	\label{fig:lem31D}
	\end{subfigure}
	\vskip -.2cm
	\caption{The proof of Lemma~\ref{l:35}.}
	\label{fig:lem34a}
	\vskip -.3cm
	\end{figure}
 	
	Working with the hexagon \bl{$b_0a_1b_2a_0b_1a_2$} and the path \re{$c_0yc_2$} 
	(see Figure~\ref{fig:lem31B}) we see that the graph $N+y-c_1$ has only three 
	independent sets of size four, namely 
	\[
		\{a_1, b_1, c_0, c_2\}\,,
		\qquad 
		\{a_0, a_1, a_2, y\}\,,
		\qquad 
		\{b_0, b_1, b_2, y\}\,.
	\]
	If the first of them has a neighbour we are done, so due to~\Dp{3}
	and $a$-$b$-symmetry we can assume that there is a common neighbour~$z$ 
	of $\{a_0,a_1,a_2,y\}$. Together with an arbitrary vertex $x_0\in X_0$ 
	we can now build the configuration shown in Figure~\ref{fig:lem31C}.
	
	Let us now look at the set $A$ consisting of the nine vertices belonging
	to the pentagon $\bl{P=b_2a_1b_0a_2c_2}$ and the two edges $\re{e_1=yc_0}$, 
	$\re{e_2=c_1b_1}$ (see Figure~\ref{fig:lem31C}). By~\Dp{3} there is an 
	independent set $U\subseteq A$ of size four possessing a common neighbour~$q$. 
	Clearly,~$U$ contains two vertices from~\bl{$P$} and one vertex from each 
	of the edges~$\re{e_1}$, $\re{e_2}$.

	If $\{b_0,a_2\}\cap U = \vn$, then only the possibility 
	$U=\{c_2,a_1,c_0,b_1\}$ remains, and we reach 
	$q\in X_1$, as required. Now assume for the sake of contradiction 
	that either $b_0$ or $a_2$ is in $U$. Both cases can be treated analogously 
	and we only display the argument for $a_2\in U$. 
	Now we have 
		$\{a_2, c_1, b_2\} \subseteq U$
	and the largest independent set among the twelve vertices 
	\[
		y, c_0, q, c_1, b_2, z, b_0, c_2, a_0, a_1, a_2, x_0
	\]
	has size four (see Figure~\ref{fig:lem31D}), contrary to~\Dp{4}.
\end{proof}

When using Lemma~\ref{l:35} in the sequel, we will sometimes draw 
the configuration at hand as in Figure~\ref{fig:1224}. The dashed orange 
non-edge forces the existence of a green vertex together with four green edges. 

\begin{figure}[h!]
\centering
\begin{tikzpicture}[scale=.75]
	\coordinate (c2) at (0,-.1);
	\coordinate (b2) at (-1.5,-.1);
	\coordinate (a2) at (1.5,-.1);
	\coordinate (b0) at (-.45,3);
	\coordinate(a1) at (.45,3);
	\coordinate (a0) at (-1.9,.6);
	\coordinate (b1) at (1.9,.6);
	\coordinate (c0) at ($(b0)!.5!(a0)$);
	\coordinate (c1) at ($(b1)!.5!(a1)$);
	\coordinate (n) at (0,1.1);
	\draw (b0)--(a1)--(b2)--(a2)--cycle;
	\draw (c0)--(a0)--(b1)--(a1);
	\draw (c0)--(b0);
	\draw (a0)--(b2);
	\draw (a2)--(b1);
	\draw [orange, thick, dashed] (c0)--(c1);
	\draw [green!90!black, thick] (c0)--(n)--(c1)--(n)--(a2)--(n)--(b2);
	\draw [fill=green] (n) circle (2pt);
	\foreach \i in {0,1,2}{
		\fill (a\i) circle (2pt);
		\fill (b\i) circle (2pt);
		\fill  (c\i) circle (2pt);
	}
\end{tikzpicture}
\caption{Applications of Lemma~\ref{l:35}.}
\label{fig:1224}
\end{figure}
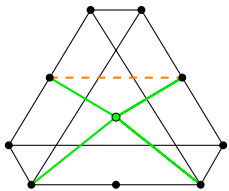
\vspace{-1em} 	
\subsection{Gr\"otzsch subgraphs}\label{subsec:43}
In this subsection we label the vertices of the Mycielski-Gr\"otzsch 
graph $\grot$ as shown in Figure~\ref{fig:beautifulLA}. So the eleven vertices
are called $a_i$, $b_i$, $c$ with indices $i\in \ZZ/5\ZZ$, and the twenty 
edges of $\grot$ are all pairs of the form $a_ic$, $a_ib_{i\pm 2}$, or $b_ib_{i+2}$.

The results that follow deal with graphs $G$ such that 
$\grot \subseteq G\in \fD_4$. Here are some questions motivating them. 
\begin{enumerate}
	\item[$\bullet$] Which subsets of $V(\grot)$ have common neighbours? 
	\item[$\bullet$] How far can we go in the direction of proving 
		the $\grot$-twin property? 
\end{enumerate}
At first sight some of the ensuing statements may seem very weak.
This is because we do not `know' at the present level of generality 
how the given copy of~$\grot$ `sits' in the Vega graph of which~$G$ 
is a blow-up, so that there is still a large number of possibilities. 
For the very same reason, however, the results obtained here turn out
to be very flexible later, when we study the 
scenario $\grot_i^{\mu\nu}\subseteq G\in\fD_4$. The fact 
that $\grot_i^{\mu\nu}$ can contain `many' copies of $\grot$ then means that 
results on $\grot$ tend to be applicable in several distinct ways. 

\begin{lemma}[Beautiful lemma]\label{l:beauti}
	If $i\in \ZZ/5\ZZ$, $\grot\subseteq G\in \fD_4$, and $q\in V(G)$ is 
	adjacent to~$a_{i-1}$,~$a_{i+1}$, then $a_iq\in E(G)$.
\end{lemma}	

\begin{proof}
	Due to symmetry we can assume $i=1$, so that $a_0q, a_2q\in E(G)$ 
	(see Figure~\ref{fig:beautifulLB}). 
	Suppose $a_1q\not\in E(G)$ and let $z$ 
	be a common neighbour of $a_1$, $q$. 
	The graph $\grot-a_3+q+z$ can be drawn as in Figure~\ref{fig:beautifulLC}. 
	Its only independent set of size $5$ 
	is $\{\gr{a_0},\gr{a_4},\gr{b_0},\gr{b_4},\gr{z}\}$. 
	If \re{$b_2'$} denotes a common neighbour of this set guaranteed by \Dp{4}, 
	then the graph $\grot-b_2-a_4 +b'_2+q+z$ drawn in Figure~\ref{fig:beautifulLD}  
	has no independent set of size five, which contradicts $G\in \fD_4$. 
\end{proof}
	
\begin{figure}[h!]
\centering
\vskip -.4cm
\begin{subfigure}[b]{.195\textwidth}
\centering
\begin{tikzpicture}[scale=.7]
	\coordinate (c) at (0,0);
	\foreach \i in {1,...,5} {
		\coordinate (a\i) at (-18+\i*72:1);
		\coordinate (b\i) at (18+\i*72:1.5);	
		\coordinate (d\i) at (18+\i*72:2);
	}
	\draw (b1) [orange, thick,out=162, in=72] to (d2) [out=-108, in=162] to (b3);
	\draw (b2) [out=234, in=144] to (d3) [out=-36, in=234] to (b4);
	\draw (b3) [cyan, thick, out=-54, in=216] to (d4) [out=36, in=-54] to (b5);
	\draw (b4) [cyan, thick, out=18, in=-72] to (d5) [out=108, in=18] to (b1);
	\draw (b5) [out=90, in=0] to (d1) [out=180, in=90] to (b2);
	\draw [cyan, thick] (b1)--(a1)--(c)--(a4)--(b3);
	\draw [cyan, thick]  (b5)--(a5)--(b4);
	\draw [orange, thick] (c)--(a5);
	\draw [orange, thick] (a4)--(b4);
	\draw [orange, thick] (a1)--(b5);
	\draw [green!75!black, thick] (c)--(a3)--(b2)--(a2)--(c);
	\draw (b3)--(a3);
	\draw (a2)--(b1);
	\node at ($(c)+(-.3,.1)$) {\tiny $c$};
	\node at ($(a1)+(.15,.15)$) {\tiny $a_2$};
	\node at ($(a2)+(-.2,.2)$) {\tiny $a_3$};
	\node at ($(a3)+(-.2,-.15)$) {\tiny $a_4$};
	\node at ($(a4)+(.1,-.25)$) {\tiny $a_0$};
	\node at ($(a5)+(.3,-.1)$) {\tiny $a_1$};
	\node at ($(b1)+(-.06,.27)$) {\tiny $b_0$};
	\node at ($(b2)+(-.22,.1)$) {\tiny $b_1$};
	\node at ($(b3)+(-.1,-.2)$) {\tiny $b_2$};
	\node at ($(b4)+(.2,-.2)$) {\tiny $b_3$};
	\node at ($(b5)+(.2,.2)$) {\tiny $b_4$};
	\fill (c) circle (2pt);
	\foreach \i in {1,...,5} {
		\fill (a\i) circle (2pt);		
		\fill (b\i) circle (2pt);
	}
\end{tikzpicture}
\vskip -.2cm
\caption{}	
\label{fig:beautifulLA}
\end{subfigure}
\hfill
\begin{subfigure}[b]{.255\textwidth}
\centering
\begin{tikzpicture}[scale=.9]
	\foreach \i in {1,...,8} {
		\coordinate (a\i) at (\i*45:.9);
		\coordinate (c\i) at (\i*45:1.8);
	}
	\draw [thick, cyan](a1) -- (a4) -- (a7) -- (a2) -- (a5) -- (a8) -- (a3) -- (a6) -- (a1);
	\draw [thick, orange] (a3) -- (a7);
	\draw [orange, thick] (a4) -- (a8);
	\draw [orange, thick](a1) -- (a5);
	\draw [orange, thick](a2) -- (a6);
	\draw [green!75!black, thick] (a4)--(c3)--(c4)--(c5)--(a4);
	\draw (a3)--(c4)--(a5);
	\draw (a2) --(c3);
	\draw (c5)--(a6);
	\draw [red, thick] (a1)--(c8)--(a7);
	\foreach \i in {1,...,8} {
		\fill (a\i) circle (2pt);		
	}
	\foreach \i in {3,4,5, 8} {
		\fill (c\i) circle (2pt);		
	}
	\node at ($(a1)+(.2,.1)$) {\tiny $a_0$};
	\node at ($(a2)+(0,.3)$) {\tiny $b_0$};
	\node at ($(a3)+(-.2,.2)$) {\tiny $b_4$};
	\node at ($(a4)+(-.2,.1)$) {\tiny $c$};
	\node at ($(a5)+(-.2,-.2)$) {\tiny $b_3$};
	\node at ($(a6)+(0,-.3)$) {\tiny $b_2$};
	\node at ($(a7)+(.2,-.2)$) {\tiny $a_2$};
	\node at ($(a8)+(.2,.1)$) {\tiny $a_1$};
	\node at ($(c3)+(-.3,.1)$) {\tiny $a_3$};
	\node at ($(c4)+(-.2,-.1)$) {\tiny $b_1$};
	\node at ($(c5)+(-.2,-.2)$) {\tiny $a_4$};
	\node at ($(c8)+(.2,.1)$) {\tiny $q$};
\end{tikzpicture}
\vskip -.2cm
\caption{}	
\label{fig:beautifulLB}
\end{subfigure}
\hfill
\begin{subfigure}[b]{.255\textwidth}
\centering
\begin{tikzpicture}[scale=.9]
	\foreach \i in {1,...,8} {
		\coordinate (a\i) at (\i*45:.9);
		\coordinate (c\i) at (\i*45:1.8);
	}
	\draw (a1) -- (a4) -- (a7) -- (a2) -- (a5) -- (a8) -- (a3) -- (a6) -- (a1);
	\draw (a3) -- (a7);
	\draw (a4) -- (a8);
	\draw (a1) -- (a5);
	\draw (a2) -- (a6);
	\draw [ultra thin, black!50] (a4)--(c3)--(c4);
	\draw (c4)--(c5)--(a4);
	\draw (a3)--(c4)--(a5);
	\draw [ultra thin, black!50] (a2) --(c3);
	\draw (c5)--(a6);
	\draw (a1)--(c8)--(a7);
	\draw [red, thick] (a8)--(c7)--(c8);
	\foreach \i in {1,...,8} {
		\fill (a\i) circle (2pt);		
	}
	\foreach \i in {4,5, 7,8} {
		\fill (c\i) circle (2pt);		
	}
	\fill [black!20] (c3) circle (1.8pt);
	\foreach \i in {a1, c5, a2, a3, c7} \fill [green!80!black] (\i) circle (2pt); 
	\node [green!75!black] at ($(a1)+(.2,.1)$) {\tiny $a_0$};
	\node [green!75!black] at ($(a2)+(0,.3)$) {\tiny $b_0$};
	\node [green!75!black] at ($(a3)+(-.2,.2)$) {\tiny $b_4$};
	\node at ($(a4)+(-.2,.1)$) {\tiny $c$};
	\node at ($(a5)+(-.2,-.2)$) {\tiny $b_3$};
	\node at ($(a6)+(0,-.3)$) {\tiny $b_2$};
	\node at ($(a7)+(.2,-.2)$) {\tiny $a_2$};
	\node at ($(a8)+(.2,.1)$) {\tiny $a_1$};
	\node [black!40!white] at ($(c3)+(-.3,.1)$) {\tiny $a_3$};
	\node at ($(c4)+(-.2,-.1)$) {\tiny $b_1$};
	\node [green!75!black] at ($(c5)+(-.2,-.2)$) {\tiny $a_4$};
	\node at ($(c8)+(.2,.1)$) {\tiny $q$};
	\node [green!75!black] at ($(c7)+(.2,-.1)$) {\tiny $z$};
\end{tikzpicture}
\vskip -.2cm
\caption{}	
\label{fig:beautifulLC}
\end{subfigure}
\hfill
\begin{subfigure}[b]{.255\textwidth}
\centering
\begin{tikzpicture}[scale=.9]			
	\foreach \i in {1,...,8} {
		\coordinate (a\i) at (\i*45:.9);
		\coordinate (c\i) at (\i*45:1.8);
	}
	\draw (a1) -- (a4) -- (a7) -- (a2) -- (a5) -- (a8);
	\draw [red, thick] (a3) -- (a6) -- (a1);
	\draw (a3) -- (a7);
	\draw  (a4) -- (a8);
	\draw (a1) -- (a5);
	\draw [red, thick](a2) -- (a6);
	\draw (a4)--(c3)--(c4);
	\draw [ultra thin, black!50](c4)--(c5)--(a4);
	\draw (a3)--(c4);
	\draw (c4)--(a5);
	\draw (a2) --(c3);
	\draw [red, ultra thin] (c5)--(a6);
	\draw [red, thick](a6)--(c7);
	\draw (a1)--(c8)--(a7);			
	\draw (a8)--(c7)--(c8);
	\foreach \i in {1,...,8} {
		\fill (a\i) circle (2pt);		
	}
	\foreach \i in {3,4,7, 8} {
		\fill (c\i) circle (2pt);		
	}
	\foreach \i in {a1, a2, a3, c7} \fill [green!80!black] (\i) circle (2pt); 
	\fill [green!30] (c5) circle (1.8pt);
	\node at ($(a1)+(.2,.1)$) {\tiny $a_0$};
	\node at ($(a2)+(0,.3)$) {\tiny $b_0$};
	\node at ($(a3)+(-.2,.2)$) {\tiny $b_4$};
	\node at ($(a4)+(-.2,.1)$) {\tiny $c$};
	\node at ($(a5)+(-.2,-.2)$) {\tiny $b_3$};
	\node at ($(a6)+(0,-.3)$) {\tiny \re{$b_2'$}};
	\node at ($(a7)+(.2,-.2)$) {\tiny $a_2$};
	\node at ($(a8)+(.2,.1)$) {\tiny $a_1$};
	\node at ($(c3)+(-.3,.1)$) {\tiny $a_3$};
	\node at ($(c4)+(-.2,-.1)$) {\tiny $b_1$};
	\node [black!40 ] at ($(c5)+(-.2,-.2)$) {\tiny $a_4$};
	\node at ($(c8)+(.2,.1)$) {\tiny $q$};
	\node at ($(c7)+(.2,-.1)$) {\tiny $z$};			
\end{tikzpicture}
\vskip -.2cm
\caption{}	
\label{fig:beautifulLD}
\end{subfigure}
\vskip -.2cm
\caption{The proof of the beautiful lemma.}
\label{fig:beautifulL}
\vskip -.2cm
\end{figure}
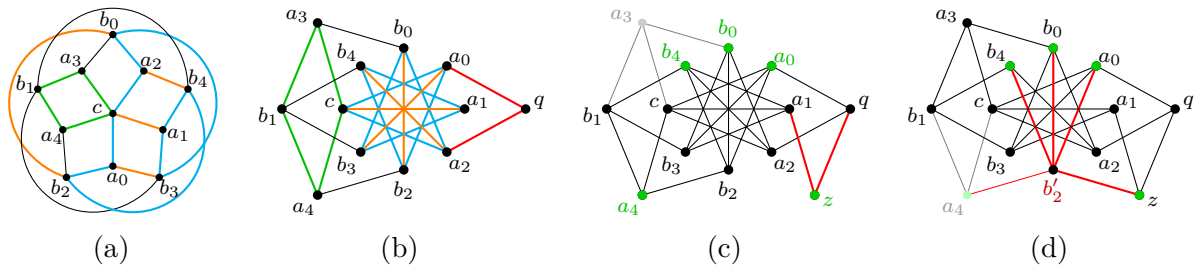 		

The Mycielski-Gr\"otzsch graph has three kinds of edges. 
For those containing the central vertex $c$ the twin property
demands no additional assumptions. 

\begin{lemma}\label{l:37}
	For every $i\in \ZZ/5\ZZ$ all graphs $G\in \fD_4$ have 
	the $(\grot,ca_i)$-twin property.
\end{lemma}
\begin{proof}
	Due to symmetry we can assume $i=1$. Suppose $\grot\subseteq G$ and 
	that $a_1', c'\in V(G)$ are $\grot$-twins of $a_1$, $c$. 
	Since $a_0, a_2\in \Nn(c')$, the beautiful lemma applied to~$\grot(a_1')$ 
	and~$c'$ instead of~$\grot$ and~$q$ yields $a_1'c'\in E(G)$.
\end{proof}

For the other edges of $\grot$ the twin property cannot be proved unconditionally. But the situation can be analysed satisfactorily as follows. 

\begin{lemma}\label{l:38}
	Let $i\in \ZZ/5\ZZ$, $\eps\in \{1, -1\}$, and $\grot\subseteq G\in \fD_4$. 
	If $a_i'$, $b'_{i+2\eps}$ are non-adjacent $\grot$-twins 
	of $a_i$, $b_{i+2\eps}$, then there exists a common neighbour 
	of $\{a_i', a_{i+2\eps}, b_{i+\eps}, b'_{i+2\eps}\}$.
\end{lemma}

\begin{proof}
	By symmetry we can assume $i=0$ and $\eps=1$. Let $r$ be a common 
	neighbour of~$a_0'$,~$b_2'$ (see Figure~\ref{fig:lem37A}). 
	Due to $a_2b_1\not\in E(G)$ our claim follows from Lemma~\ref{l:35} 
	(see Figure~\ref{fig:lem37B}).
\end{proof}

\begin{figure}[ht]
\centering
\def\co{blue}
\def\ct{red}
\vskip -.3cm
\begin{subfigure}[b]{.23\textwidth}
\begin{tikzpicture}[scale=.8]
	\coordinate (c) at (0,0);
	\foreach \i in {1,...,5} {
		\coordinate (a\i) at (-18+\i*72:1);
		\coordinate (b\i) at (18+\i*72:1.5);	
		\coordinate (d\i) at (18+\i*72:2.1);
	}
	\coordinate (r) at (-.5,-.8);
	\draw (b3) [thin, black!30,out=-54, in=216] to (d4) [out=36, in=-54] to (b5);
	\draw (b5) [thin, black!30, out=90, in=0] to (d1) [out=180, in=90] to (b2);
	\draw (b1) [thick, \co, out=162, in=72] to (d2) [out=-108, in=162] to (b3);
	\draw (b2) [thick, \ct, out=234, in=144] to (d3) [out=-36, in=234] to (b4);		
	\draw (b4) [thick, \co, out=18, in=-72] to (d5) [out=108, in=18] to (b1);
	\draw [thin, black!30] (b1)--(a1)--(c)--(a4);
	\draw [thin, black!30] (b5)--(a5)--(b4);
	\draw [thin, black!30]  (c)--(a5);
	\draw [thin, black!30] (a4)--(b4);
	\draw [thin, black!30] (a1)--(b5);
	\draw [thin, black!30] (c)--(a3)--(b2)--(a2)--(c);
	\draw [thin, black!30] (b3)--(a3);
	\draw [thin, black!30]  (a2)--(b1);
	\draw [thick,\ct] (a4)--(r)--(b3);
	\draw [thick,\co] (b4)--(a4)--(c)--(a3)--(b3);
	\draw [thick,\ct] (c)--(a1)--(b1);
	\draw [thick,\ct] (b2)--(a3);
	\node at ($(c)+(-.3,.1)$) {\tiny $c$};
	\node at ($(a1)+(.05,.25)$) {\tiny $a_2$};
	\node [black!20] at ($(a2)+(-.2,.2)$) {\tiny $a_3$};
	\node at ($(a3)+(-.2,-.15)$) {\tiny $a_4$};
	\node at ($(a4)+(.1,-.3)$) {\tiny $a'_0$};
	\node [black!20] at ($(a5)+(.25,-.1)$) {\tiny $a_1$};
	\node at ($(b1)+(-.06,.27)$) {\tiny $b_0$};
	\node  at ($(b2)+(-.3,.1)$) {\tiny $b_1$};
	\node at ($(b3)+(-.15,-.2)$) {\tiny $b'_2$};
	\node at ($(b4)+(.2,-.2)$) {\tiny $b_3$};
	\node[black!20] at ($(b5)+(.2,.2)$) {\tiny $b_4$};
	\node at ($(r) + (.2,.1)$) {\tiny $r$};
	\fill (c) circle (2pt);
	\foreach \i in {1,...,5} {
		\fill (a\i) circle (2pt);		
		\fill (b\i) circle (2pt);
		}
	\fill (r) circle (2pt);
	\foreach \i in {a4, b3, r} \fill [cyan] (\i) circle (1.7pt);
	\foreach \i in { b5, a2, a5} \fill [black!20] (\i) circle (2pt);
\end{tikzpicture}
\vskip -.2cm
\caption{}	
\label{fig:lem37A}
\end{subfigure}
\hfill
\begin{subfigure}[b]{.25\textwidth}
\begin{tikzpicture}[scale=.8]			
	\coordinate (c2) at (0,-.1);
	\coordinate (b2) at (-1.5,-.1);
	\coordinate (a2) at (1.5,-.1);
	\coordinate (b0) at (-.45,3);
	\coordinate (a1) at (.45,3);
	\coordinate (a0) at (-1.9,.6);
	\coordinate (b1) at (1.9,.6);
	\coordinate (c0) at ($(b0)!.5!(a0)$);
	\coordinate (c1) at ($(b1)!.5!(a1)$);
	\coordinate (q) at (0, 1);
	\draw [thick, \co] (a0)--(b2)--(a1)--(b0)--(a2)--(b1)--(a0);
	\draw [thick, \ct] (b0) --(a0);
	\draw [thick, \ct] (b1) --(a1);
	\draw [thick, \ct] (b2) --(a2);
	\draw [green!90!black, thick] (c0)--(q)--(c1)--(q)--(b2)--(q)--(a2);
	\draw [thick, dashed, orange] (c0)--(c1);
	\draw [fill=green] (q) circle (1.8pt);
	\foreach \i in {0,1,2}{
		\fill (a\i) circle (2pt);
		\fill (b\i) circle (2pt);
		\fill (c\i) circle (2pt);
	}
	\foreach \i in {a2, b2, c2} \fill [cyan] (\i) circle (1.7pt);
	\def\s{\tiny}
	\node at ($(b0)+(-.3,.1)$) {\s $b_0$};
	\node at ($(c0)+(-.3,.1)$) {\s $a_2$};
	\node at ($(a0)+(-.3,.1)$) {\s $c$};
	\node at ($(b1)+(.3,.1)$) {\s $a_4$};
	\node at ($(c1)+(.3,.1)$) {\s $b_1$};
	\node at ($(a1)+(.3,.1)$) {\s $b_3$};
	\node at ($(b2)+(-.1,-.3)$) {\s $a'_0$};
	\node at ($(c2)+(0,-.3)$) {\s $r$};
	\node at ($(a2)+(.1,-.3)$) {\s $b'_2$};
\end{tikzpicture}
\vskip -.2cm
\caption{}	
\label{fig:lem37B}
\end{subfigure}
\hfill
\begin{subfigure}[b]{.24\textwidth}
\begin{tikzpicture}[scale=.8]		
	\coordinate (c) at (0,0);
	\foreach \i in {1,...,5} {
		\coordinate (a\i) at (-18+\i*72:1);
		\coordinate (b\i) at (18+\i*72:1.5);	
		\coordinate (d\i) at (18+\i*72:2.1);
	}
	\coordinate (r) at (162:2.2);
	\draw (b5) [thick, \ct, out=90, in=0] to (d1) [out=180, in=90] to (b2);
	\draw (b1) [thick, \ct, out=162, in=72] to (r) [out=-108, in=162] to (b3);
	\draw (b2) [thick, \ct, out=234, in=144] to (d3) [out=-36, in=234] to (b4);		
	\draw (b4) [thick, \co, out=18, in=-72] to (d5) [out=108, in=18] to (b1);
	\draw (b3) [thick, \co, out=-54, in=216] to (d4) [out=36, in=-54] to (b5);
	\draw [thin, black!30] (b1)--(a1)--(c)--(a4);
	\draw [thin, black!30] (b5)--(a5)--(b4);
	\draw [thin, black!30] (c)--(a5);
	\draw [thin, black!30] (a4)--(b4);
	\draw [thin, black!30] (a1)--(b5);
	\draw [thin, black!30] (c)--(a3)--(b2)--(a2)--(c);
	\draw [thin, black!30] (b3)--(a3);
	\draw [thin, black!30] (a2)--(b1);
	\draw [thick,\co] (b1)--(a1)--(b5);
	\draw [thick,\co] (b3)--(a4)--(b4);
	\draw [thick,\ct] (a1)--(c)--(a4);
	\node at ($(c)+(-.3,.1)$) {\tiny $c$};
	\node at ($(a1)+(.2,.15)$) {\tiny $a_2$};
	\node [black!20] at ($(a2)+(-.2,.2)$) {\tiny $a_3$};
	\node [black!20] at ($(a3)+(-.2,-.15)$) {\tiny $a_4$};
	\node at ($(a4)+(.1,-.3)$) {\tiny $a_0$};
	\node  [black!20] at ($(a5)+(.25,-.1)$) {\tiny $a_1$};
	\node at ($(b1)+(.03,.27)$) {\tiny $b'_0$};
	\node  at ($(b2)+(-.2,.1)$) {\tiny $b_1$};
	\node at ($(b3)+(-.15,-.2)$) {\tiny $b'_2$};
	\node at ($(b4)+(.2,-.2)$) {\tiny $b_3$};
	\node at ($(b5)+(.25,.1)$) {\tiny $b_4$};
	\node at ($(r) + (-.2,.1)$) {\tiny $r$};
	\fill (c) circle (2pt);
	\foreach \i in {1,...,5} {
		\fill (a\i) circle (2pt);		
		\fill (b\i) circle (2pt);
	}
	\fill (r) circle (2pt);
	\foreach \i in {b1, b3} \fill [lime] (\i) circle (1.7pt);
	\fill [cyan] (r) circle (1.7pt);
	\foreach \i in {a2, a3, a5} \fill [black!20] (\i) circle (2pt);
\end{tikzpicture}
\vskip -.2cm
\caption{}	
\label{fig:lem37C}
\end{subfigure}
\hfill
\begin{subfigure}[b]{.255\textwidth}
\begin{tikzpicture}[scale=.8]			
	\coordinate (c2) at (0,-.1);
	\coordinate (b2) at (-1.5,-.1);
	\coordinate (a2) at (1.5,-.1);
	\coordinate (b0) at (-.45,3);
	\coordinate (a1) at (.45,3);
	\coordinate (a0) at (-1.9,.6);
	\coordinate (b1) at (1.9,.6);
	\coordinate (c0) at ($(b0)!.5!(a0)$);
	\coordinate (c1) at ($(b1)!.5!(a1)$);
	\coordinate (q) at (0, 1);
	\draw [thick, \co] (a0)--(b2)--(a1)--(b0)--(a2)--(b1)--(a0);
	\draw [thick, \ct] (b0) --(a0);
	\draw [thick, \ct] (b1) --(a1);
	\draw [thick, \ct] (b2) --(a2);
	\draw [thick, green!90!black] (c0)--(q)--(c1)--(q)--(b2)--(q)--(a2);
	\draw [thick, dashed, orange] (c0)--(c1);
	\draw [fill=green] (q) circle (1.8pt);
	\foreach \i in {0,1,2}{
		\fill (a\i) circle (2pt);
		\fill (b\i) circle (2pt);
		\fill (c\i) circle (2pt);
	}
	\foreach \i in {a2, b2} \fill [lime] (\i) circle (1.7pt);
	\fill [cyan](c2) circle (1.7pt); 
	\def\s{\tiny}
	\node at ($(b0)+(-.3,.1)$) {\s $b_4$};
	\node at ($(c0)+(-.3,.1)$) {\s $b_1$};
	\node at ($(a0)+(-.3,.1)$) {\s $b_3$};
	\node at ($(b1)+(.3,.1)$) {\s $a_0$};
	\node at ($(c1)+(.3,.1)$) {\s $c$};
	\node at ($(a1)+(.3,.1)$) {\s $a_2$};
	\node at ($(b2)+(-.1,-.3)$) {\s $b'_0$};
	\node at ($(c2)+(0,-.3)$) {\s $r$};
	\node at ($(a2)+(.1,-.3)$) {\s $b'_2$};
\end{tikzpicture}
\vskip -.2cm
\caption{}	
\label{fig:lem37D}
\end{subfigure}
\vskip -.2cm
\caption{The proofs of Lemma~\ref{l:38} and Lemma~\ref{l:39}.}
\label{fig:lem37}
\end{figure}
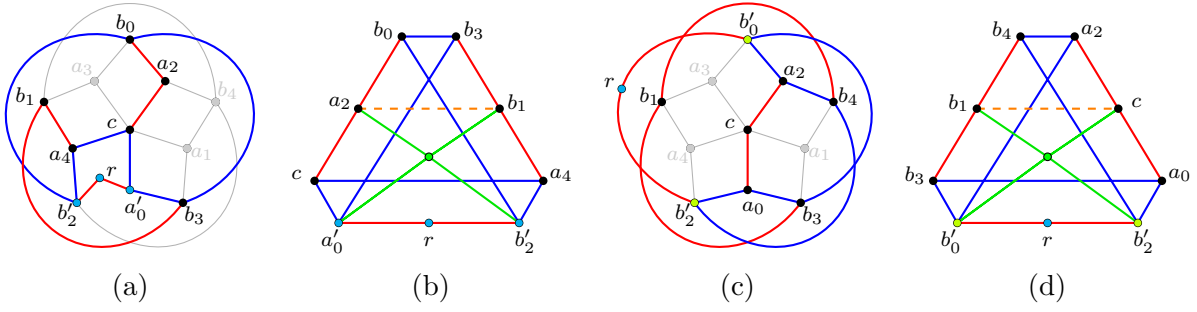
 		
\begin{lemma}\label{l:39}
	If $i\in \ZZ/5\ZZ$, $\grot\subseteq G\in \fD_4$, and 
	$b'_i$, $b'_{i+2}$ are non-adjacent $\grot$-twins of $b_i$, $b_{i+2}$, 
	then there is a common neighbour 
	of $\{b_i', b_{i+1}, b'_{i+2}, c\}$.
\end{lemma}

\begin{proof}
	It suffices to consider the case $i=0$. If~$r$ denotes a common neighbour 
	of $b_0'$, $b_2'$ (see Figure~\ref{fig:lem37C}), then due to $b_1c\not\in E(G)$ 
	Lemma~\ref{l:35} leads to the desired vertex 
	(see Figure~\ref{fig:lem37D}).
\end{proof}

The three foregoing lemmata will assist us later when proving the twin lemma for 
Vega graphs (cf.\ Lemma~\ref{lem:vtwin}). In an attempt to facilitate later references
we visualise the beautiful lemma and the two previous lemmata in 
Figure~\ref{fig:1244}. 
The idea is that in Figure~\ref{fig:1244a} the existence of the two blue edges 
leads to the green one. Moreover, in the Figures~\ref{fig:1244b} and~\ref{fig:1244c} 
the dashed orange non-edge forces the existence of the green vertex together with 
four green edges.

\begin{figure}[h]
\centering
\begin{subfigure}[b]{.31\textwidth}
\begin{tikzpicture}[scale=1]			
	\coordinate (c) at (0,0);
	\foreach \i in {1,...,5} {
		\coordinate (a\i) at ($(-18+\i*72:1)$);
		\coordinate (b\i) at ($(18+\i*72:1.5)$);	
		\coordinate (d\i) at ($(18+\i*72:2)$);
	}
	\coordinate (q) at (.3,-.5);			
	\draw [green!90!black, thick] (q)--(a4);						
	\draw [cyan, thick] (a3)--(q)--(a5);
	\draw (b4)--(a4)--(c);
	\draw (a4)--(b3);
	\draw (b1) [out=162, in=72] to (d2) [out=-108, in=162] to (b3);
	\draw (b2) [out=234, in=144] to (d3) [out=-36, in=234] to (b4);
	\draw (b3) [out=-54, in=216] to (d4) [out=36, in=-54] to (b5);
	\draw (b4) [out=18, in=-72] to (d5) [out=108, in=18] to (b1);
	\draw (b5) [out=90, in=0] to (d1) [out=180, in=90] to (b2);
	\draw (a2)--(b2);
	\draw (a1)--(b5);
	\draw (b3)--(a3)--(c)--(a5)--(b4);
	\draw (b1)--(a1)--(c)--(a2)--(b1);
	\draw (b2)--(a3);
	\draw (b5)--(a5);
	\fill (c) circle (2pt);
	\foreach \i in {1,...,5} {
		\fill (a\i) circle (2pt);		
		\fill (b\i) circle (2pt);
	}
	\draw [fill=cyan](q) circle (1.8pt);
\end{tikzpicture}
\caption{Beautiful Lemma}	
\label{fig:1244a}
\end{subfigure}
\hfill
\begin{subfigure}[b]{.31\textwidth}
\begin{tikzpicture}[scale=1]			
	\coordinate (c) at (0,0);
	\foreach \i in {1,...,5} {
		\coordinate (a\i) at ($(-18+\i*72:1)$);
		\coordinate (b\i) at ($(18+\i*72:1.5)$);	
		\coordinate (d\i) at ($(18+\i*72:2)$);
	}
	\coordinate (q) at (-.5,.1);			
	\draw (b4)--(a4)--(c);
	\draw [orange, thick, dashed] (a4)--(b3);
	\draw (b1) [out=162, in=72] to (d2) [out=-108, in=162] to (b3);
	\draw (b2) [out=234, in=144] to (d3) [out=-36, in=234] to (b4);
	\draw (b3) [out=-54, in=216] to (d4) [out=36, in=-54] to (b5);
	\draw (b4) [out=18, in=-72] to (d5) [out=108, in=18] to (b1);
	\draw (b5) [out=90, in=0] to (d1) [out=180, in=90] to (b2);
	\draw (a2)--(b2);
	\draw (a1)--(b5);
	\draw (b3)--(a3)--(c)--(a5)--(b4);
	\draw (b1)--(a1)--(c)--(a2)--(b1);
	\draw (b2)--(a3);
	\draw (b5)--(a5);
	\foreach \i in {a4, b3, b2, a1}{
		\draw [green!90!black, thick] (q)--(\i);
	}
	\draw [fill=green](q) circle (2pt);
	\fill (c) circle (2pt);
	\foreach \i in {1,...,5} {
		\fill (a\i) circle (2pt);		
		\fill (b\i) circle (2pt);
	}
	\foreach \i in {a4, b3, b2, a1}{
			\draw [fill=cyan!80!black] (\i) circle (1.7pt);
	}
\end{tikzpicture}
\caption{Lemma~\ref{l:38}} 
\label{fig:1244b}
\end{subfigure}
\hfill
\begin{subfigure}[b]{.31\textwidth}
\begin{tikzpicture}[scale=1]			
	\coordinate (c) at (0,0);
	\foreach \i in {1,...,5} {
		\coordinate (a\i) at ($(-18+\i*72:1)$);
		\coordinate (b\i) at ($(18+\i*72:1.5)$);	
		\coordinate (d\i) at ($(18+\i*72:2)$);
	}
	\coordinate (q) at (0,.6);			
	\draw (b4)--(a4)--(c);
	\draw (a4) --(b3);
	\draw (b1) [out=162, in=72] to (d2) [out=-108, in=162] to (b3);
	\draw (b2) [out=234, in=144] to (d3) [out=-36, in=234] to (b4);
	\draw (b3) [out=-54, in=216] to (d4) [out=36, in=-54] to (b5);
	\draw (b4) [out=18, in=-72] to (d5) [out=108, in=18] to (b1);
	\draw (b5) [orange, thick, dashed, out=90, in=0] to (d1) [out=180, in=90] to (b2);
	\draw (a2)--(b2);
	\draw (a1)--(b5);
	\draw (b3)--(a3)--(c)--(a5)--(b4);
	\draw (b1)--(a1)--(c)--(a2)--(b1);
	\draw (b2)--(a3);
	\draw (b5)--(a5);
	\foreach \i in {b2, b5, b1, c}{
		\draw [green!90!black, thick] (q)--(\i);
	}
	\draw [fill=green](q) circle (2pt);
	\fill (c) circle (2pt);
	\foreach \i in {1,...,5} {
		\fill (a\i) circle (2pt);		
		\fill (b\i) circle (2pt);
	}
	\foreach \i in {b2, b5, b1, c}{
		\draw [fill=cyan!80!black] (\i) circle (1.7pt);
	}
\end{tikzpicture}
\caption{Lemma~\ref{l:39}}	
\label{fig:1244c}
\end{subfigure}
\caption{}
\label{fig:1244}
\end{figure}
 
\begin{lemma}\label{l:310}
	Let $\grot\subseteq G\in \fD_4$. If $u\in V(G)$ is adjacent to $a_1$, $a_4$ 
	and $v\in V(G)$ is adjacent to $u$, $b_0$, then either $b_1v\in E(G)$, 
	or $b_4v\in E(G)$, or $v$ is adjacent to $\grot$-twins of $b_1$ and $b_4$. 
\end{lemma}		

\begin{figure}[h!]
\centering
\def\co{green!80!black}
\def\ct{cyan!80!black}
\def\ch{blue}
\def\cf{red}
\begin{subfigure}[b]{.24\textwidth}
\begin{tikzpicture}[scale=.8]
	\coordinate (c) at (0,0);
	\foreach \i in {1,...,5} {
		\coordinate (a\i) at (-18+\i*72:1);
		\coordinate (b\i) at (18+\i*72:1.5);	
		\coordinate (d\i) at (18+\i*72:2);
		\draw [black!30](a\i)--(c);
	}
	\coordinate (r) at (162:2.2);
	\coordinate (u) at (.4, -.6);
	\coordinate (v) at (.5,.2);
	\draw[black!30] (a3)--(u)--(a5)--(u);
	\draw [thick, \cf] (u)--(v)--(b1);
	\draw [black!30] (a1)--(b1)--(a2)--(b2)--(a3)--(b3)--(a4)--(b4)--(a5)--(b5)--(a1);
	\draw [black!30] (a4)--(u);
	\draw [thick, \ch] (b4)--(a5)--(u)--(a3)--(b3);
	\draw [thick, \cf] (a3)--(b2);
	\draw [thick, \cf] (a5)--(b5);
	\draw (b5) [black!30, out=90, in=0] to (d1) [out=180, in=90] to (b2);		
	\draw (b2) [thick, \cf, out=234, in=144] to (d3) [out=-36, in=234] to (b4);		
	\draw (b3) [thick, \cf, out=-54, in=216] to (d4) [out=36, in=-54] to (b5);
	\draw (b1) [thick, \ch, out=162, in=72] to (d2) [out=-108, in=162] to (b3);
	\draw (b4) [thick, \ch,  out=18, in=-72] to (d5) [out=108, in=18] to (b1);
	\node [black!20] at ($(c)+(-.3,.1)$) {\tiny $c$};
	\node [black!20] at ($(a1)+(.15,.2)$) {\tiny $a_2$};
	\node [black!20] at ($(a2)+(-.25,.2)$) {\tiny $a_3$};
	\node at ($(a3)+(-.2,-.1)$) {\tiny $a_4$};
	\node [black!20] at ($(a4)+(.1,-.3)$) {\tiny $a_0$};
	\node at ($(a5)+(.25,-.1)$) {\tiny $a_1$};
	\node at ($(b1)+(-.06,.27)$) {\tiny $b_0$};
	\node at ($(b2)+(-.2,.1)$) {\tiny $b_1$};
	\node at ($(b3)+(-.15,-.2)$) {\tiny $b_2$};
	\node at ($(b4)+(.2,-.2)$) {\tiny $b_3$};
	\node at ($(b5)+(.2,.2)$) {\tiny $b_4$};
	\node at ($(u) + (.1,-.15)$) {\tiny $u$};
	\node at ($(v) + (.15,0)$) {\tiny $v$};
	\fill (c) circle (2pt);
	\foreach \i in {1,...,5} {
		\fill (a\i) circle (2pt);		
		\fill (b\i) circle (2pt);
	}
	\fill (u) circle (2pt);
	\fill (v) circle (2pt);
	\fill [cyan] (u) circle (1.7pt);
	\fill [cyan] (v) circle (1.7pt);
	\foreach \i in {a2, c, a1, a4} \fill [black!20] (\i) circle (2pt);
\end{tikzpicture}
\vskip -.2cm
\caption{}	
\label{fig:lem39A}
\end{subfigure}
\hfill
\begin{subfigure}[b]{.255\textwidth}
\begin{tikzpicture}[scale=.8]			
	\coordinate (c2) at (0,-.1);
	\coordinate (b2) at (-1.5,-.1);
	\coordinate (a2) at (1.5,-.1);
	\coordinate (b0) at (-.45,3);
	\coordinate(a1) at (.45,3);
	\coordinate (a0) at (-1.9,.6);
	\coordinate (b1) at (1.9,.6);
	\coordinate (c0) at ($(b0)!.5!(a0)$);
	\coordinate (c1) at ($(b1)!.5!(a1)$);
	\coordinate (q) at (-.2, 1.3);
	\coordinate (q2) at (.1, .95);
	\draw [thick, \ch] (a0)--(b2)--(a1)--(b0)--(a2)--(b1)--(a0);
	\draw [thick, \cf] (b0) --(a0);
	\draw [thick, \cf] (b1) --(a1);
	\draw [thick, \cf] (b2) --(a2);
	\draw [thick, dashed, orange] (c2)--(c1);
	\draw [thick, dashed, orange] (c0)--(c2);
	\draw [thick, \co] (c0)--(q)--(c2)--(q)--(b1)--(q)--(a1);
	\draw [\ct, thick] (c1)--(q2)--(c2)--(q2)--(b0)--(q2)--(a0);
	\foreach \i in {0,1,2}{
		\fill (a\i) circle (2pt);
		\fill (b\i) circle (2pt);
		\fill (c\i) circle (2pt);
	}
	\fill [cyan] (b2) circle (1.7pt);
	\fill [cyan](c2) circle (1.7pt); 
	\draw [fill=\co] (q) circle (2pt);
	\draw [fill=\ct] (q2) circle (2pt);
	\node at ($(q2)+(.25,-.1)$) {\tiny $b_1'$};
	\node at ($(q)+(-.2,-.1)$) {\tiny $b_4'$};
	\node at ($(b0)+(-.3,.1)$) {\tiny $b_3$};
	\node at ($(c0)+(-.3,.1)$) {\tiny $b_1$};
	\node at ($(a0)+(-.3,.1)$) {\tiny $a_4$};
	\node at ($(b1)+(.3,.1)$) {\tiny $b_2$};
	\node at ($(c1)+(.3,.1)$) {\tiny $b_4$};
	\node at ($(a1)+(.3,.1)$) {\tiny $a_1$};
	\node at ($(b2)+(-.1,-.3)$) {\tiny $u$};
	\node at ($(c2)+(0,-.3)$) {\tiny $v$};
	\node at ($(a2)+(.1,-.3)$) {\tiny $b_0$};
\end{tikzpicture}
\vskip -.2cm
\caption{}	
\label{fig:lem39B}
\end{subfigure}
\hfill
\begin{subfigure}[b]{.24\textwidth}
\begin{tikzpicture}[scale=.8]			
	\coordinate (c) at (0,0);
	\coordinate (t) at (.3,-.5);
	\foreach \i in {1,...,5} {
		\coordinate (a\i) at (-18+\i*72:1);
		\coordinate (b\i) at (18+\i*72:1.5);	
		\coordinate (d\i) at (18+\i*72:2);
		\draw (c)--(a\i);
	}
	\draw (a1)--(b1)--(a2)--(b2)--(a3)--(b3)--(a4)--(b4)--(a5)--(b5)--(a1);
	\draw (b1) [out=162, in=72] to (d2) [out=-108, in=162] to (b3);
	\draw (b2) [out=234, in=144] to (d3) [out=-36, in=234] to (b4);
	\draw (b4) [out=18, in=-72] to (d5) [out=108, in=18] to (b1);
	\draw (b5) [thick, \co, out=90, in=0] to (d1) [out=180, in=90] to (b2);
	\draw (b3) [thick, \co, out=-54, in=216] to (d4) [out=36, in=-54] to (b5);
	\draw [thick, \co] (a5)--(b5)--(a1);
	\draw [thick, \ct] (b3)--(a4)--(b4)--(a4)--(c);
	\draw [green, thick] (t)--(a4);						
	\draw [cyan, thick] (a3)--(t)--(a5);
	\draw [fill=cyan](t) circle (2pt);
	\fill (c) circle (2pt);
	\foreach \i in {1,...,5} {
		\fill (a\i) circle (2pt);		
		\fill (b\i) circle (2pt);
	}
	\foreach \i in {a4, b1, b3} \fill [cyan] (\i) circle (1.7pt);
	\draw [fill=\co] (b5) circle (2pt);
	\draw [fill=\ct] (a4) circle (2pt);
	\node [cyan!80!black] at ($(t)+(.12,-.23)$) {\tiny $a_3$};
	\node at ($(c)+(-.3,.15)$) {\tiny $b_3$};
	\node at ($(a1)+(.1,.2)$) {\tiny $a_1$};
	\node at ($(a2)+(-.2,.2)$) {\tiny $a_0$};
	\node at ($(a3)+(-.15,-.15)$) {\tiny $b_0$};
	\node at ($(a4)+(.1,-.3)$) {\tiny $b_1'$};
	\node at ($(a5)+(.25,-.1)$) {\tiny $b_1$};
	\node at ($(b1)+(-.06,.2)$) {\tiny $u$};
	\node at ($(b2)+(-.25,.1)$) {\tiny $b_2$};
	\node at ($(b3)+(-.1,-.2)$) {\tiny $v$};
	\node at ($(b4)+(.2,-.2)$) {\tiny $a_4$};
	\node at ($(b5)+(.2,.15)$) {\tiny $b_4'$};
\end{tikzpicture}
\vskip -.2cm
\caption{}	
\label{fig:lem39C}
\end{subfigure}
\hfill
\begin{subfigure}[b]{.24\textwidth}
\begin{tikzpicture}[scale=.8]			
	\coordinate (c) at (0,0);
	\coordinate (t) at (.3,-.5);	
	\foreach \i in {1,...,5} {
		\coordinate (a\i) at (-18+\i*72:1);
		\coordinate (b\i) at (18+\i*72:1.5);	
		\coordinate (d\i) at (18+\i*72:2);
		\draw (c)--(a\i);
	}
	\draw (a1)--(b1)--(a2)--(b2)--(a3)--(b3)--(a4)--(b4)--(a5)--(b5)--(a1);
	\draw (b1) [out=162, in=72] to (d2) [out=-108, in=162] to (b3);
	\draw (b2) [out=234, in=144] to (d3) [out=-36, in=234] to (b4);		
	\draw (b4) [out=18, in=-72] to (d5) [out=108, in=18] to (b1);
	\draw (b5) [thick, \ct, out=90, in=0] to (d1) [out=180, in=90] to (b2);
	\draw (b3) [thick, \ct, out=-54, in=216] to (d4) [out=36, in=-54] to (b5);
	\draw [thick, \co] (b3)--(a4)--(b4)--(a4)--(c);
	\draw [thick, \ct] (a5)--(b5)--(a1);
	\draw [green, thick](t)--(a4);						
	\draw [cyan, thick] (a3)--(t)--(a5);
	\draw [fill=cyan](t) circle (2pt);
	\fill (c) circle (2pt);
	\foreach \i in {1,...,5} {
		\fill (a\i) circle (2pt);		
		\fill (b\i) circle (2pt);
	}
	\foreach \i in { b5, b1, b3} \draw [fill=cyan] (\i) circle (2pt);
	\draw [fill=\co] (a4) circle (2pt);
	\draw [fill=\ct] (b5) circle (2pt);
	\node [cyan!80!black] at ($(t)+(.12,-.23)$) {\tiny $a_2$};
	\node at ($(c)+(-.3,.15)$) {\tiny $b_2$};
	\node at ($(a1)+(.1,.2)$) {\tiny $a_4$};
	\node at ($(a2)+(-.2,.2)$) {\tiny $a_0$};
	\node at ($(a3)+(-.15,-.15)$) {\tiny $b_0$};
	\node at ($(a4)+(.1,-.3)$) {\tiny $b_4'$};
	\node at ($(a5)+(.25,-.1)$) {\tiny $b_4$};
	\node at ($(b1)+(-.06,.2)$) {\tiny $u$};
	\node at ($(b2)+(-.25,.1)$) {\tiny $b_3$};
	\node at ($(b3)+(-.1,-.2)$) {\tiny $v$};
	\node at ($(b4)+(.2,-.2)$) {\tiny $a_1$};
	\node at ($(b5)+(.2,.15)$) {\tiny $b_1'$};
\end{tikzpicture}		
\vskip -.2cm
\caption{}	
\label{fig:lem39D}
\end{subfigure}
\vskip -.3cm
\caption{The proof of Lemma~\ref{l:310}.}
\label{fig:lem39}
\end{figure}
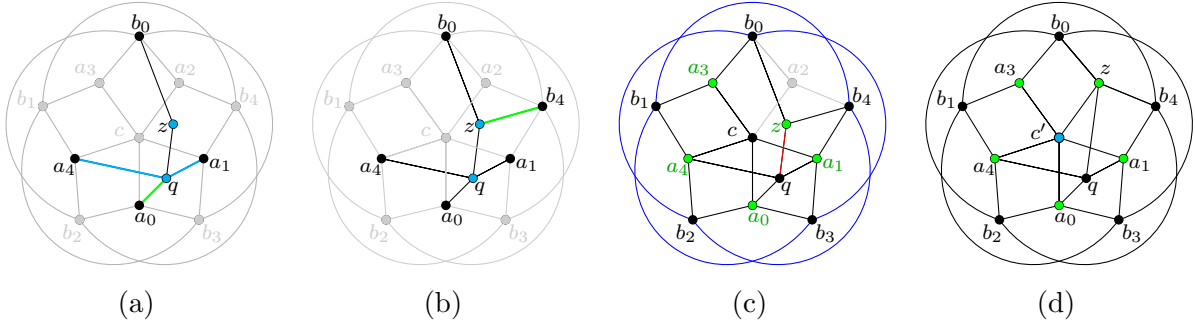
 \vskip-.3cm

\begin{proof}
	The beautiful lemma implies $a_0u\in E(G)$  (see Figure~\ref{fig:lem39A}). 
	Suppose that neither~$b_1v$ nor~$b_4v$ is an edge of~$G$. 
	Due to Lemma~\ref{l:35} there are common neighbours~$b_1'$ and~$b'_4$ 
	of $\{a_4,b_3,b_4,v\}$ and $\{a_1,b_1,b_2,v\}$, respectively   
	(see Figure~\ref{fig:lem39B}). 
	
	It remains to show $a_3b_1', a_2b_4'\in E(G)$, for then $b_1'$, $b_4'$
	are the desired $\grot$-twins in $\Nn(v)$. Both statements follow from 
	the beautiful lemma applied to appropriate copies of $\grot$, as indicated 
	in the Figures~\ref{fig:lem39C} and~\ref{fig:lem39D}.
\end{proof}

We proceed with a series of results that are drawn schematically in 
Figure~\ref{fig:1138}. Again the blue vertices and edges and the orange dashed 
non-edges force the existence of green vertices and edges. 
The colour light-green indicates twins. 

\begin{figure}[h!]
\vskip -.2cm
\begin{subfigure}[b]{.24\textwidth}
\begin{tikzpicture}[scale=.7]			
	\coordinate (c) at (0,0);
	\foreach \i in {1,...,5} {
		\coordinate (a\i) at ($(-18+\i*72:1)$);
		\coordinate (b\i) at ($(18+\i*72:1.5)$);	
		\coordinate (d\i) at ($(18+\i*72:2)$);
	}
	\coordinate (s) at (-1.43, .3);
	\coordinate (t) at (1.43, .3);				
	\coordinate (q) at (-.1,.6);			
	\draw [cyan, thick] (c)--(q)--(b1);
	\draw [orange, thick, dashed] (b5)--(q)--(b2);
	\draw (b4)--(a4)--(c);
	\draw (a4)--(b3);
	\draw (b1) [out=162, in=72] to (d2) [out=-108, in=162] to (b3);
	\draw (b2) [out=234, in=144] to (d3) [out=-36, in=234] to (b4);
	\draw (b3) [out=-54, in=216] to (d4) [out=36, in=-54] to (b5);
	\draw (b4) [out=18, in=-72] to (d5) [out=108, in=18] to (b1);
	\draw (b5) [out=90, in=0] to (d1) [out=180, in=90] to (b2);
	\draw (a2)--(b2);
	\draw (a1)--(b5);
	\draw (b3)--(a3)--(c)--(a5)--(b4);
	\draw (b1)--(a1)--(c)--(a2)--(b1);
	\draw (b2)--(a3);
	\draw (b5)--(a5);
	\fill (c) circle (2pt);
	\foreach \i in {1,...,5} {
		\fill (a\i) circle (2pt);		
		\fill (b\i) circle (2pt);
	}
	\draw [green!90!black, thick] (s)--(q)--(t);
	\draw [fill = lime] (s) circle (1.8pt);
	\draw [fill = lime] (t) circle (1.8pt);
	\draw [fill=cyan](q) circle (1.8pt);
\end{tikzpicture}
\caption{Corollary~\ref{cor:311}}
\label{fig:1138a}
\end{subfigure}
\hfill
\begin{subfigure}[b]{.24\textwidth}
\begin{tikzpicture}[scale=.7]			
	\coordinate (c) at (0,0);
	\foreach \i in {1,...,5} {
		\coordinate (a\i) at ($(-18+\i*72:1)$);
		\coordinate (b\i) at ($(18+\i*72:1.5)$);	
		\coordinate (d\i) at ($(18+\i*72:2)$);
	}
	\coordinate (q) at (.6,-.4);			
	\draw [orange, thick, dashed] (q)--(b1);
	\draw [green!90!black, thick] (a1)--(q)--(a4)--(q)--(a2);						
	\draw [cyan, thick] (a3)--(q)--(a5);
	\draw (b4)--(a4)--(c);
	\draw (a4)--(b3);
	\draw (b1) [out=162, in=72] to (d2) [out=-108, in=162] to (b3);
	\draw (b2) [out=234, in=144] to (d3) [out=-36, in=234] to (b4);
	\draw (b3) [out=-54, in=216] to (d4) [out=36, in=-54] to (b5);
	\draw (b4) [out=18, in=-72] to (d5) [out=108, in=18] to (b1);
	\draw (b5) [out=90, in=0] to (d1) [out=180, in=90] to (b2);
	\draw (a2)--(b2);
	\draw (a1)--(b5);
	\draw (b3)--(a3)--(c)--(a5)--(b4);
	\draw (b1)--(a1)--(c)--(a2)--(b1);
	\draw (b2)--(a3);
	\draw (b5)--(a5);
	\fill (c) circle (2pt);
	\foreach \i in {1,...,5}{
		\fill (a\i) circle (2pt);		
		\fill (b\i) circle (2pt);
	}
	\draw [fill=cyan](q) circle (1.8pt);
		\end{tikzpicture}
\caption{Lemma~\ref{l:312}}
\label{fig:1138b}
\end{subfigure}
\hfill
\begin{subfigure}[b]{.24\textwidth}
\begin{tikzpicture}[scale=.7]			
	\coordinate (c) at (0,0);
	\foreach \i in {1,...,5} {
		\coordinate (a\i) at ($(-18+\i*72:1)$);
		\coordinate (b\i) at ($(18+\i*72:1.5)$);	
		\coordinate (d\i) at ($(18+\i*72:2)$);
	}
	\coordinate (q) at (.6,-.8);			
	\coordinate (z) at (-.6, -.8);
	\draw [green!90!black, thick] (a2)--(z)--(b3)--(z)--(a5);
	\draw [green!90!black, thick] (q)--(z);
	\draw [fill=green!80!black] (z) circle (1.8pt);
	\draw [cyan, thick] (b2)--(q)--(a4);
	\draw [orange, thick, dashed] (q)--(a5);
	\draw [green!90!black, thick] (q)--(b1);
	\draw (b4)--(a4)--(c);
	\draw (a4)--(b3);
	\draw (b1) [out=162, in=72] to (d2) [out=-108, in=162] to (b3);
	\draw (b2) [out=234, in=144] to (d3) [out=-36, in=234] to (b4);
	\draw (b3) [out=-54, in=216] to (d4) [out=36, in=-54] to (b5);
	\draw (b4) [out=18, in=-72] to (d5) [out=108, in=18] to (b1);
	\draw (b5) [out=90, in=0] to (d1) [out=180, in=90] to (b2);
	\draw (a2)--(b2);
	\draw (a1)--(b5);
	\draw (b3)--(a3)--(c)--(a5)--(b4);
	\draw (b1)--(a1)--(c)--(a2)--(b1);
	\draw (b2)--(a3);
	\draw (b5)--(a5);
	\fill (c) circle (2pt);
	\foreach \i in {1,...,5}{
		\fill (a\i) circle (2pt);		
		\fill (b\i) circle (2pt);
	}
	\draw [fill=cyan](q) circle (1.8pt);
\end{tikzpicture}	
\caption{Lemma~\ref{l:313}}
\label{fig:1138c}	
\end{subfigure}
\hfill
\begin{subfigure}[b]{.24\textwidth}
\begin{tikzpicture}[scale=.7]			
	\coordinate (c) at (0,0);
	\foreach \i in {1,...,5} {
		\coordinate (a\i) at ($(-18+\i*72:1)$);
		\coordinate (t\i) at ($(-18+\i*75:1.15)$);
		\coordinate (b\i) at ($(18+\i*72:1.5)$);	
		\coordinate (s\i) at ($(18+\i*75:1.65)$);
		\coordinate (d\i) at ($(18+\i*72:2)$);
	}
	\coordinate (q) at (-.5,.2);			
	\draw [cyan, thick] (a4)--(q)--(b2);
	\draw [orange, thick, dashed] (a1)--(q)--(b1);
	\draw (b4)--(a4)--(c);
	\draw (a4)--(b3);
	\draw (b1) [out=162, in=72] to (d2) [out=-108, in=162] to (b3);
	\draw (b2) [out=234, in=144] to (d3) [out=-36, in=234] to (b4);
	\draw (b3) [out=-54, in=216] to (d4) [out=36, in=-54] to (b5);
	\draw (b4) [out=18, in=-72] to (d5) [out=108, in=18] to (b1);
	\draw (b5) [out=90, in=0] to (d1) [out=180, in=90] to (b2);
	\draw (a2)--(b2);
	\draw (a1)--(b5);
	\draw (b3)--(a3)--(c)--(a5)--(b4);
	\draw (b1)--(a1)--(c)--(a2)--(b1);
	\draw (b2)--(a3);
	\draw (b5)--(a5);
	\fill (c) circle (2pt);
	\foreach \i in {1,...,5} {
		\fill (a\i) circle (2pt);		
		\fill (b\i) circle (2pt);
	}
	\draw [green!90!black, thick] (t1)--(q)--(s1);
	\draw [fill = lime] (t1) circle (1.8pt);
	\draw [fill = lime] (s1) circle (1.8pt);
	\draw [fill=cyan](q) circle (1.8pt);
\end{tikzpicture}
\caption{Lemma~\ref{l:314}}
\label{fig:1138d}
\end{subfigure}
\vskip -.3cm
\caption{}
\label{fig:1138}
\end{figure}

\vspace{-1em}
 
\begin{cor}\label{cor:311}
	Let $\grot\subseteq G\in \fD_4$. If $q\in V(G)$ is adjacent to $c$, $b_0$, 
	then either $q$ is an $\grot$-twin of $a_2$, or it is an $\grot$-twin of $a_3$, 
	or it is adjacent to $\grot$-twins of $b_1$ and $b_4$.
\end{cor}

\begin{proof}
	Apply Lemma~\ref{l:310} to $u=c$ and $v=q$. 
\end{proof}

The maximal independent sets of $\grot$ are the neighbourhoods of vertices 
and the five sets of the form $\{a_{i-1},a_i,a_{i+1},b_i\}$. It can happen 
that some of these sets have common neighbours in an ambient graph $G$ 
belonging to $\fD_4$. Given $i\in \ZZ/5\ZZ$ and $\grot\subseteq G$ we shall 
write $\Ext(\grot, a_i)$ for the set of common neighbours 
of $\{a_{i-1}, a_{i+1}, b_i\}$. More generally, if~$\Omega \subseteq G$ is 
isomorphic to $\grot$ and $a\in V(\Omega)$ has degree three in $\Omega$, 
then $\Ext(\Omega, a)$ is defined analogously. The beautiful lemma implies
$\Ext(\Omega, a)\subseteq \Nn(a)$. If $\Ext(\Omega, a) = \vn$, 
we say that~$a$ is {\it reliable (with respect to $\Omega$)}.

\begin{lemma}\label{l:312}
	If $\grot\subseteq G\in \fD_4$ and $q\in V(G)$ is adjacent to $a_1$, $a_4$, 
	then either~$q$ is an $\grot$-twin of $c$ or $q\in \Ext(\grot,a_0)$.
\end{lemma}

\begin{proof}
	The beautiful lemma tells us $qa_0\in E(G)$. If $q\not\in \Ext(\grot, a_0)$, 
	then $qb_0\notin E(G)$ and we can pick a common neighbour $z$ of $q$, $b_0$ 
	(see Figure \ref{fig:lem311A}). 
	Plugging $u=q$, $v=z$ into Lemma~\ref{l:310} we learn that either $zb_1\in E(G)$, 
	or $zb_4\in E(G)$, or $z$ is adjacent to $\grot$-twins of both $b_1$, $b_4$. 
	By symmetry we may assume that $\Nn(z)$ contains some $\grot$-twin $b_4'$ of $b_4$. 
	As we only need to show $a_2, a_3\in \Nn(q)$, it is permissible to replace $\grot$
	by $\grot(b_4')$ and, hence, we can even assume $b_4z\in E(G)$ 
	(see Figure~\ref{fig:lem311B}). 
			
	\begin{figure}[h!]
	\centering
	\begin{subfigure}[b]{.24\textwidth}
	\centering
	\begin{tikzpicture}[scale=.9]			
		\coordinate (c) at (0,0);		
		\foreach \i in {1,...,5}{
			\coordinate (a\i) at (-18+\i*72:1);
			\coordinate (b\i) at (18+\i*72:1.5);	
			\coordinate (d\i) at (18+\i*72:2);
			\draw [black!30](a\i)--(c);
		}
		\coordinate (r) at (162:2.2);
		\coordinate (u) at (.4, -.6);
		\coordinate (v) at (.5,.2);
		\draw[black] (a4)--(u)--(a3)--(u)--(a5)--(u)--(v)--(b1);
		\draw [black!30] (a1)--(b1)--(a2)--(b2)--(a3)--(b3)--(a4)--(b4)--(a5)--(b5)--(a1);
		\draw (b5) [black!30, out=90, in=0] to (d1) [out=180, in=90] to (b2);		
		\draw (b2) [black!30, out=234, in=144] to (d3) [out=-36, in=234] to (b4);		
		\draw (b3) [black!30, out=-54, in=216] to (d4) [out=36, in=-54] to (b5);
		\draw (b1) [black!30, out=162, in=72] to (d2) [out=-108, in=162] to (b3);
		\draw (b4) [black!30, out=18, in=-72] to (d5) [out=108, in=18] to (b1);
		\draw [thick, cyan] (a3)--(u)--(a5);
		\draw [thick, green] (u)--(a4);
		\node [black!20] at ($(c)+(-.3,.1)$) {\tiny $c$};
		\node [black!20] at ($(a1)+(.1,.2)$) {\tiny $a_2$};
		\node [black!20] at ($(a2)+(-.2,.2)$) {\tiny $a_3$};
		\node at ($(a3)+(-.15,-.15)$) {\tiny $a_4$};
		\node at ($(a4)+(.1,-.2)$) {\tiny $a_0$};
		\node at ($(a5)+(.25,-.1)$) {\tiny $a_1$};
		\node at ($(b1)+(.03,.2)$) {\tiny $b_0$};
		\node [black!20] at ($(b2)+(-.23,.1)$) {\tiny $b_1$};
		\node [black!20] at ($(b3)+(-.1,-.2)$) {\tiny $b_2$};
		\node [black!20] at ($(b4)+(.2,-.2)$) {\tiny $b_3$};
		\node [black!20] at ($(b5)+(.2,.15)$) {\tiny $b_4$};
		\node at ($(u) + (.1,-.15)$) {\tiny $q$};
		\node at ($(v) + (-.15,-.05)$) {\tiny $z$};
		\fill (c) circle (2pt);
		\foreach \i in {1,...,5} {
			\fill (a\i) circle (2pt);		
			\fill (b\i) circle (2pt);
		}
		\fill (u) circle (2pt);
		\fill (v) circle (2pt);
		\fill [cyan] (u) circle (1.7pt);
		\fill [cyan] (v) circle (1.7pt);
		\foreach \i in {a2,c, a1, b5, b2, b3, b4} \fill [black!20] (\i) circle (2pt);		
	\end{tikzpicture}
	\vskip -.2cm
	\caption{}	
	\label{fig:lem311A}
	\end{subfigure}
	\hfill
	\begin{subfigure}[b]{.24\textwidth}
	\centering
	\begin{tikzpicture}[scale=.9]						
		\draw[black] (a4)--(u)--(a3)--(u)--(a5)--(u)--(v)--(b1)--(v)--(b5);
		\draw [black!20] (a1)--(b1)--(a2)--(b2)--(a3)--(b3)--(a4)--(b4)--(a5)--(b5)--(a1);
		\draw [black!20] (a1)--(c)--(a2)--(c)--(a3)--(c)--(a4)--(c)--(a5);
		\draw (b5) [black!20, out=90, in=0] to (d1) [out=180, in=90] to (b2);		
		\draw (b2) [black!20, out=234, in=144] to (d3) [out=-36, in=234] to (b4);		
		\draw (b3) [black!20, out=-54, in=216] to (d4) [out=36, in=-54] to (b5);
		\draw (b1) [black!20, out=162, in=72] to (d2) [out=-108, in=162] to (b3);
		\draw (b4) [black!20,  out=18, in=-72] to (d5) [out=108, in=18] to (b1);
		\draw [green, thick] (v)--(b5);
		\node [black!20] at ($(c)+(-.3,.1)$) {\tiny $c$};
		\node [black!20] at ($(a1)+(.1,.2)$) {\tiny $a_2$};
		\node [black!20] at ($(a2)+(-.2,.2)$) {\tiny $a_3$};
		\node at ($(a3)+(-.15,-.15)$) {\tiny $a_4$};
		\node at ($(a4)+(.1,-.2)$) {\tiny $a_0$};
		\node at ($(a5)+(.25,-.1)$) {\tiny $a_1$};
		\node at ($(b1)+(.03,.2)$) {\tiny $b_0$};
		\node [black!20] at ($(b2)+(-.23,.1)$) {\tiny $b_1$};
		\node [black!20] at ($(b3)+(-.1,-.2)$) {\tiny $b_2$};
		\node [black!20] at ($(b4)+(.2,-.2)$) {\tiny $b_3$};
		\node at ($(b5)+(.2,.15)$) {\tiny $b_4$};
		\node at ($(u) + (.1,-.15)$) {\tiny $q$};
		\node at ($(v) + (-.15,-.05)$) {\tiny $z$};
		\fill (c) circle (2pt);
		\foreach \i in {1,...,5} {
			\fill (a\i) circle (2pt);		
			\fill (b\i) circle (2pt);
		}
		\fill (u) circle (2pt);
		\fill (v) circle (2pt);
		\fill [cyan] (u) circle (1.7pt);
		\fill [cyan] (v) circle (1.7pt);
		\foreach \i in {a1,a2,c, b2, b3, b4} \fill [black!20] (\i) circle (2pt);		
	\end{tikzpicture}
	\vskip -.2cm
	\caption{}	
	\label{fig:lem311B}
	\end{subfigure}
	\hfill
	\begin{subfigure}[b]{.24\textwidth}
	\centering
	\begin{tikzpicture}[scale=.9]						
		\draw (a4)--(u)--(a3)--(u)--(a5)--(u)--(v)--(b1)--(v)--(b5);
		\draw (b1)--(a2)--(b2)--(a3)--(b3)--(a4)--(b4)--(a5)--(b5);
		\draw (c)--(a2)--(c)--(a3)--(c)--(a4)--(c)--(a5);
		\draw [black!20] (c)--(a1)--(b1)--(a1)--(b5);
		\draw (b5) [ blue, out=90, in=0] to (d1) [out=180, in=90] to (b2);		
		\draw (b2) [blue, out=234, in=144] to (d3) [out=-36, in=234] to (b4);		
		\draw (b3) [blue, out=-54, in=216] to (d4) [out=36, in=-54] to (b5);
		\draw (b1) [blue, out=162, in=72] to (d2) [out=-108, in=162] to (b3);
		\draw (b4) [blue, out=18, in=-72] to (d5) [out=108, in=18] to (b1);
		\draw [red](u)--(v);
		\node at ($(c)+(-.3,.1)$) {\tiny $c$};
		\node [black!20] at ($(a1)+(.1,.2)$) {\tiny $a_2$};
		\node at ($(a2)+(-.2,.2)$) {\tiny \gr{$a_3$}};
		\node at ($(a3)+(-.15,-.15)$) {\tiny \gr{$a_4$}};
		\node at ($(a4)+(.1,-.2)$) {\tiny \gr{$a_0$}};
		\node at ($(a5)+(.25,-.1)$) {\tiny \gr{$a_1$}};
		\node at ($(b1)+(.03,.2)$) {\tiny $b_0$};
		\node at ($(b2)+(-.23,.1)$) {\tiny $b_1$};
		\node at ($(b3)+(-.1,-.2)$) {\tiny $b_2$};
		\node at ($(b4)+(.2,-.2)$) {\tiny $b_3$};
		\node at ($(b5)+(.2,.15)$) {\tiny $b_4$};
		\node at ($(u) + (.1,-.15)$) {\tiny $q$};
		\node at ($(v) + (-.15,-.05)$) {\tiny \gr{$z$}};
		\fill (c) circle (2pt);
		\foreach \i in {1,...,5} {
			\fill (a\i) circle (2pt);		
			\fill (b\i) circle (2pt);
		}
		\fill (u) circle (2pt);
		\fill (v) circle (2pt);
		\fill [black!20] (a1) circle (2pt);		
		\foreach \i in {v, a4, a5, a3, a2} \fill [green] (\i) circle (1.7pt);
	\end{tikzpicture}
	\vskip -.2cm
	\caption{}	
	\label{fig:lem311C}
	\end{subfigure}
	\hfill		
	\begin{subfigure}[b]{.24\textwidth}
	\centering
	\begin{tikzpicture}[scale=.9]			
		\coordinate (v) at (a1);
		\draw (a1)--(b1)--(a2)--(b2)--(a3)--(b3)--(a4)--(b4)--(a5)--(b5)--(a1);
		\draw (a1)--(c)--(a2)--(c)--(a3)--(c)--(a4)--(c)--(a5);
		\draw (a4)--(u)--(a3)--(u)--(a5)--(u)--(v)--(b1)--(v)--(b5);
		\draw (b5) [out=90, in=0] to (d1) [out=180, in=90] to (b2);		
		\draw (b2) [out=234, in=144] to (d3) [out=-36, in=234] to (b4);		
		\draw (b3) [out=-54, in=216] to (d4) [out=36, in=-54] to (b5);
		\draw (b1) [out=162, in=72] to (d2) [out=-108, in=162] to (b3);
		\draw (b4) [out=18, in=-72] to (d5) [out=108, in=18] to (b1);
		\node at ($(c)+(-.3,.1)$) {\tiny $c'$};
		\node at ($(a1)+(.1,.2)$) {\tiny $z$};
		\node at ($(a2)+(-.2,.2)$) {\tiny $a_3$};
		\node at ($(a3)+(-.15,-.15)$) {\tiny $a_4$};
		\node at ($(a4)+(.1,-.2)$) {\tiny $a_0$};
		\node at ($(a5)+(.25,-.1)$) {\tiny $a_1$};
		\node at ($(b1)+(.03,.2)$) {\tiny $b_0$};
		\node at ($(b2)+(-.23,.1)$) {\tiny $b_1$};
		\node at ($(b3)+(-.1,-.2)$) {\tiny $b_2$};
		\node at ($(b4)+(.2,-.2)$) {\tiny $b_3$};
		\node at ($(b5)+(.2,.15)$) {\tiny $b_4$};
		\node at ($(u) + (.1,-.15)$) {\tiny $q$};
		\fill (c) circle (2pt);
		\foreach \i in {1,...,5}{
			\fill (a\i) circle (2pt);		
			\fill (b\i) circle (2pt);
		}
		\fill (u) circle (2pt);
		\fill (v) circle (2pt);
		\draw [fill=cyan] (c) circle (2pt);
		\foreach \i in {v, a4, a5, a3, a2} \fill [green] (\i) circle (1.7pt);
	\end{tikzpicture}
	\vskip -.2cm
	\caption{}	
	\label{fig:lem311D}
	\end{subfigure}
	\vskip -.2cm
	\caption{The proof of Lemma \ref{l:312}.}
	\label{fig:lem311}
	\end{figure}
 		
	Due to \Dp{4} the graph $\grot-a_2+q+z$ depicted in Figure~\ref{fig:lem311C}
	has an independent set~$T$ of size five possessing a common neighbour $c'$. 
	Since $T$ contains at most two vertices of the pentagon $B=\bl{b_0b_2b_4b_1b_3}$
	and at most one vertex from each of the edges $qz$, $a_3c$, at least one 
	of $a_0$, $a_1$, $a_4$ needs to be in $T$. Thus $c$ and $q$ are not 
	in~$T$ or, in other words,~$T$ is a subset of the 
	ten-cycle $b_0a_3b_1a_4b_2a_0b_3a_1b_4z$,
	whence $T=\{\gr{z},\gr{a_0}, \gr{a_1}, \gr{a_3}, \gr{a_4}\}$. 
		
	The beautiful lemma applied to $a_4, z\in \Nn(q)$ and the graph $\grot-c-a_2+c'+z$ 
	drawn in Figure \ref{fig:lem311D} shows $a_3\in \Nn(q)$. A final application of 
	the beautiful lemma to $a_1, a_3\in \Nn(q)$ and~$\grot$ gives $a_2\in \Nn(q)$, 
	wherefore $q$ is indeed an $\grot$-twin of~$c$. 
\end{proof}

\begin{lemma}\label{l:313}
	If $\grot\subseteq G\in \fD_4$ and $q\in V(G)$ is adjacent to $a_0$, $b_1$, 
	but not to $a_1$, then it is adjacent to $b_0$ and to some vertex 
	in $\Ext(\grot, a_2)$. 
\end{lemma}
					
\begin{figure}[ht]
\centering
\begin{subfigure}[b]{.32\textwidth}
\centering
\begin{tikzpicture}[scale=.8]
	\coordinate (c) at (0,0);
	\foreach \i in {1,...,5} {
		\coordinate (a\i) at (-18+\i*72:1);
		\coordinate (b\i) at (18+\i*72:1.5);	
		\coordinate (d\i) at (18+\i*72:2);
		\draw [black!30](a\i)--(c);
	}
	\coordinate (q) at (-.3, -.4);
	\coordinate (z) at (.5,-.7);
	\draw [black!30] (a1)--(b1)--(a2)--(b2)--(a3)--(b3)--(a4)--(b4)--(a5)--(b5)--(a1);
	\draw (a4)--(q)--(b2);
	\draw (q)--(z)--(a5);
	\draw [thick, blue] (b2)--(q)--(a4)--(c)--(a5)--(b5);
	\draw [thick, red] (a4)--(b3);
	\draw [thick, red] (q)--(z)--(a5);
	\draw [thick, red] (c)--(a2)--(b2);
	\draw (b2) [black!30, out=234, in=144] to (d3) [out=-36, in=234] to (b4);		
	\draw (b1) [black!30, out=162, in=72] to (d2) [out=-108, in=162] to (b3);
	\draw (b4) [black!30,  out=18, in=-72] to (d5) [out=108, in=18] to (b1);
	\draw (b5) [thick, blue, out=90, in=0] to (d1) [out=180, in=90] to (b2);		
	\draw (b3) [thick, red, out=-54, in=216] to (d4) [out=36, in=-54] to (b5);
	\node at ($(c)+(-.3,.1)$) {\tiny $c$};
	\node [black!20] at ($(a1)+(.1,.2)$) {\tiny $a_2$};
	\node at ($(a2)+(-.2,.2)$) {\tiny $a_3$};
	\node [black!20]  at ($(a3)+(-.2,-.15)$) {\tiny $a_4$};
	\node at ($(a4)+(.1,-.3)$) {\tiny $a_0$};
	\node at ($(a5)+(.25,-.1)$) {\tiny $a_1$};
	\node [black!20]at ($(b1)+(-.06,.27)$) {\tiny $b_0$};
	\node at ($(b2)+(-.23,.1)$) {\tiny $b_1$};
	\node at ($(b3)+(-.15,-.2)$) {\tiny $b_2$};
	\node [black!20] at ($(b4)+(.2,-.2)$) {\tiny $b_3$};
	\node at ($(b5)+(.2,.2)$) {\tiny $b_4$};
	\node at ($(q) + (-.15,-.15)$) {\tiny $q$};
	\node at ($(z) + (0,-.2)$) {\tiny $z$};
	\fill (c) circle (2pt);
	\foreach \i in {1,...,5} {
		\fill (a\i) circle (2pt);		
		\fill (b\i) circle (2pt);
	}
	\fill (q) circle (2pt);
	\fill (z) circle (2pt);
	\fill [cyan] (q) circle (1.7pt);
	\fill [cyan] (z) circle (1.7pt);
	\foreach \i in {b1,b4,a1,a3} \fill [black!20] (\i) circle (2pt);
\end{tikzpicture}
\caption{}
\label{fig:lem312A}
\end{subfigure}
\hfill
\begin{subfigure}[b]{.32\textwidth}
\centering
\begin{tikzpicture}[scale=.8]			
	\coordinate (c2) at (0,-.1);
	\coordinate (b2) at (-1.5,-.1);
	\coordinate (a2) at (1.5,-.1);
	\coordinate (b0) at (-.45,3);
	\coordinate(a1) at (.45,3);
	\coordinate (a0) at (-1.9,.6);
	\coordinate (b1) at (1.9,.6);
	\coordinate (c0) at ($(b0)!.5!(a0)$);
	\coordinate (c1) at ($(b1)!.5!(a1)$);
	\coordinate (q) at (0, 1.1);
	\draw [thick, blue] (a0)--(b2)--(a1)--(b0)--(a2)--(b1)--(a0);
	\draw [thick, red] (b0) --(a0);
	\draw [thick, red] (b1) --(a1);
	\draw [thick, red] (b2) --(a2);
	\draw [thick, dashed, orange] (c0)--(c1);
	\draw [green!80!black] (c0)--(q)--(c1)--(q)--(b2)--(q)--(a2);
	\foreach \i in {0,1,2}{
		\fill (a\i) circle (2pt);
		\fill (b\i) circle (2pt);
		\fill (c\i) circle (2pt);
	}
	\foreach \i in {b2, c2, q} \draw [fill=cyan] (\i) circle (2pt);
	\def\z{\tiny}
	\node at ($(b0)+(-.3,.1)$) {\z $b_4$};
	\node at ($(c0)+(-.3,.1)$) {\z $b_2$};
	\node at ($(a0)+(-.3,.1)$) {\z $a_0$};
	\node at ($(b1)+(.2,.1)$) {\z $c$};
	\node at ($(c1)+(.23,.1)$) {\z $a_3$};
	\node at ($(a1)+(.3,.1)$) {\z $b_1$};
	\node at ($(b2)+(-.1,-.25)$) {\z $q$};
	\node at ($(c2)+(0,-.25)$) {\z $z$};
	\node at ($(a2)+(.1,-.2)$) {\z $a_1$};
	\node at ($(q)+(0,-.23)$) {\tiny $t$};
\end{tikzpicture}
\caption{}
\label{fig:lem312B}
\end{subfigure}
\hfill
\begin{subfigure}[b]{.32\textwidth}
\centering
\begin{tikzpicture}[scale=.8]			
	\coordinate (c) at (0,0);
	\foreach \i in {1,...,5} {
		\coordinate (a\i) at (-18+\i*72:1);
		\coordinate (b\i) at (18+\i*72:1.5);	
		\coordinate (d\i) at (18+\i*72:2);
		\draw (a\i)--(c);
	}
	\coordinate (q) at (.3, -.5);
	\draw (a1)--(b1)--(a2)--(b2)--(a3)--(b3)--(a4)--(b4)--(a5)--(b5)--(a1);
	\draw (b2) [out=234, in=144] to (d3) [out=-36, in=234] to (b4);		
	\draw (b1) [red, thick,out=162, in=72] to (d2) [out=-108, in=162] to (b3);
	\draw (b4) [out=18, in=-72] to (d5) [out=108, in=18] to (b1);
	\draw (b5) [blue, thick,  out=90, in=0] to (d1) [out=180, in=90] to (b2);		
	\draw (b3) [red, thick, out=-54, in=216] to (d4) [out=36, in=-54] to (b5);
	\draw [blue, thick] (b1)--(a2)--(b2);
	\draw [blue, thick] (b5)--(a5);
 	\draw [red, thick] (a5)--(c)--(a2);
 	\draw [green!80!black] (b2)--(a3)--(b3)--(a3)--(c);
	\draw [thick, cyan] (a3)--(q)--(a5);
	\draw [thick, green] (q)--(a4);
	\fill (c) circle (2pt);
	\foreach \i in {1,...,5} {
		\fill (a\i) circle (2pt);		
		\fill (b\i) circle (2pt);
	}
	\foreach \i in {a3,q} \draw [fill=cyan] (\i) circle (2pt);
	\node at ($(c)+(-.3,.1)$) {\tiny $b_2$};
	\node at ($(a1)+(.1,.2)$) {\tiny $a_4$};
	\node at ($(a2)+(-.2,.2)$) {\tiny $b_4$};
	\node at ($(a3)+(-.15,-.15)$) {\tiny $t$};
	\node at ($(a4)+(.1,-.3)$) {\tiny $b_0$};
	\node at ($(a5)+(.25,-.1)$) {\tiny $a_0$};
	\node at ($(b1)+(-.06,.27)$) {\tiny $b_1$};
	\node  at ($(b2)+(-.23,.1)$) {\tiny $a_1$};
	\node at ($(b3)+(-.15,-.2)$) {\tiny $a_3$};
	\node  at ($(b4)+(.2,-.2)$) {\tiny $b_3$};
	\node at ($(b5)+(.2,.1)$) {\tiny $c$};
	\node at ($(q) + (.15,-.15)$) {\tiny $q$};
\end{tikzpicture}
\vskip -.2cm
\caption{}
\label{fig:lem312C}
\end{subfigure}
\caption{The proof of Lemma~\ref{l:313}.}
\label{fig:lem312}
\end{figure}

\begin{proof}
	Let $z$ be a common neighbour of $q$, $a_1$ (see Figure~\ref{fig:lem312A}). 
	In view of $b_2a_3\notin E(G)$,  Lemma~\ref{l:35} tells us that 
	there is a common neighbour~$t$ of $q$, $a_1$, $b_2$, and $a_3$ 
	(see Figure~\ref{fig:lem312B}). 
	Clearly, $t$ is in $\Ext(\grot, a_2)$. Due to $a_0, t\in \Nn(q)$ the beautiful 
	lemma applied to $\grot -a_2 + t$  yields $b_0\in \Nn(q)$ 
	(see Figure~\ref{fig:lem312C}).
\end{proof}
 
\begin{lemma}\label{l:314}
	If $\grot\subseteq G\in \fD_4$ and $a_0, b_1\in \Nn(q)$, then either 
	$q\in \Ext(\grot, a_1)$, or $b_0q\in E(G)$, or $q$ is adjacent to $\grot$-twins 
	of $a_2$, $b_0$. 
\end{lemma}		
			
\begin{figure}[h!]
\vskip -.3cm
\centering
\begin{subfigure}[b]{.24\textwidth}
\centering
\begin{tikzpicture}[scale=.8]
	\coordinate (c) at (0,0);
	\foreach \i in {1,...,5}{
		\coordinate (a\i) at (-18+\i*72:1);
		\coordinate (b\i) at (18+\i*72:1.5);	
		\coordinate (d\i) at (18+\i*72:2);
		\draw [black!30](a\i)--(c);
	}
	\coordinate (q) at (-.3, -.4);
	\coordinate (z) at (.5,-.7);
	\draw [black!30] (a1)--(b1)--(a2)--(b2)--(a3)--(b3)--(a4)--(b4)--(a5)--(b5)--(a1);
	\draw [red, thick] (a4)--(q)--(b2);
	\draw [thick, red] (c)--(a1)--(b5);
	\draw [thick, red] (a2)--(b1);
	\draw [thick, blue] (b2)--(a2)-- (c)--(a4)--(b3);
	\draw (b2) [black!30, out=234, in=144] to (d3) [out=-36, in=234] to (b4);		
	\draw (b4) [black!30,  out=18, in=-72] to (d5) [out=108, in=18] to (b1);
	\draw (b5) [blue, thick, out=90, in=0] to (d1) [out=180, in=90] to (b2);		
	\draw (b3) [thick, blue, out=-54, in=216] to (d4) [out=36, in=-54] to (b5);	
	\draw (b1) [red, thick, out=162, in=72] to (d2) [out=-108, in=162] to (b3);
	\node at ($(c)+(-.3,.1)$) {\tiny $c$};
	\node at ($(a1)+(.05,.25)$) {\tiny $a_2$};
	\node at ($(a2)+(-.2,.2)$) {\tiny $a_3$};
	\node [black!20]  at ($(a3)+(-.2,-.15)$) {\tiny $a_4$};
	\node at ($(a4)+(.1,-.25)$) {\tiny $a_0$};
	\node [black!20]  at ($(a5)+(.25,-.1)$) {\tiny $a_1$};
	\node at ($(b1)+(.03,.2)$) {\tiny $b_0$};
	\node at ($(b2)+(-.23,.1)$) {\tiny $b_1$};
	\node at ($(b3)+(-.15,-.2)$) {\tiny $b_2$};
	\node [black!20] at ($(b4)+(.2,-.2)$) {\tiny $b_3$};
	\node at ($(b5)+(.2,.2)$) {\tiny $b_4$};
	\node at ($(q) + (-.15,-.15)$) {\tiny $q$};
	\fill (c) circle (2pt);
	\foreach \i in {1,...,5}{
		\fill (a\i) circle (2pt);		
		\fill (b\i) circle (2pt);
	}
	\fill (q) circle (2pt);
	\fill [cyan] (q) circle (1.7pt);
	\foreach \i in {b4,a5,a3} \fill [black!20] (\i) circle (2pt);
\end{tikzpicture}
\vskip -.2cm
\caption{}	
\label{fig:lem313A}
\end{subfigure}
\hfill
\begin{subfigure}[b]{.24\textwidth}
\centering
\begin{tikzpicture}[scale=.75]			
	\coordinate (c2) at (0,-.1);
	\coordinate (b2) at (-1.5,-.1);
	\coordinate (a2) at (1.5,-.1);
	\coordinate (b0) at (-.45,3);
	\coordinate(a1) at (.45,3);
	\coordinate (a0) at (-1.9,.6);
	\coordinate (b1) at (1.9,.6);
	\coordinate (c0) at ($(b0)!.5!(a0)$);
	\coordinate (c1) at ($(b1)!.5!(a1)$);
	\coordinate (q) at (0, 1.1);
	\draw [thick, blue] (a0)--(b2)--(a1)--(b0)--(a2)--(b1)--(a0);
	\draw [thick, red] (b0) --(a0);
	\draw [thick, red] (b1) --(a1);
	\draw [thick, red] (b2) --(a2);
	\draw [thick, dashed, orange] (c0)--(c2);
	\draw [thick,green!80!black,] (c0)--(q)--(c2)--(q)--(b1)--(q)--(a1);
	\fill (q) circle (2pt);
	\fill [cyan] (q) circle (1.7pt);
	\foreach \i in {0,1,2}{
		\fill (a\i) circle (2pt);
		\fill (b\i) circle (2pt);
		\fill (c\i) circle (2pt);
	}
	\draw [fill=cyan](c2) circle (2pt); 
	\node at ($(b0)+(-.3,.1)$) {\tiny $b_4$};
	\node at ($(c0)+(-.3,.1)$) {\tiny $a_2$};
	\node at ($(a0)+(-.2,.1)$) {\tiny $c$};
	\node at ($(b1)+(.3,.1)$) {\tiny $a_3$};
	\node at ($(c1)+(.3,.1)$) {\tiny $b_0$};
	\node at ($(a1)+(.3,.1)$) {\tiny $b_2$};
	\node at ($(b2)+(-.1,-.3)$) {\tiny $a_0$};
	\node at ($(c2)+(0,-.3)$) {\tiny $q$};
	\node at ($(a2)+(.1,-.3)$) {\tiny $b_1$};
	\node at ($(q)+(-.25,-.15)$) {\tiny $b_0'$};
\end{tikzpicture}
\vskip -.2cm
\caption{}	
\label{fig:lem313B}
\end{subfigure}
\hfill
\begin{subfigure}[b]{.24\textwidth}
\centering
\begin{tikzpicture}[scale=.75]			
	\coordinate (c2) at (0,-.1);
	\coordinate (b2) at (-1.5,-.1);
	\coordinate (a2) at (1.5,-.1);
	\coordinate (b0) at (-.45,3);
	\coordinate (a1) at (.45,3);
	\coordinate (a0) at (-1.9,.6);
	\coordinate (b1) at (1.9,.6);
	\coordinate (c0) at ($(b0)!.5!(a0)$);
	\coordinate (c1) at ($(b1)!.5!(a1)$);
	\coordinate (q) at (0, 1.1);
	\draw (a0)--(b2)--(a1)--(b0)--(a2)--(b1)--(a0);
	\draw (b0) --(a0);
	\draw (b1) --(a1);
	\draw (b2) --(a2);
	\draw [green!80!black, thick] (c1)--(q)--(c2)--(q)--(b0)--(q)--(a0);
	\draw [thick, dashed, orange] (c2)--(c1);
	\foreach \i in {0,1,2}{
		\fill (a\i) circle (2pt);
		\fill (b\i) circle (2pt);
		\fill (c\i) circle (2pt);
	}
	\fill (q) circle (2pt);
	\draw [fill=cyan] (q) circle (2pt);
	\draw [fill=cyan](c2) circle (2pt); 
	\node at ($(b0)+(-.3,.1)$) {\tiny $b_4$};
	\node at ($(c0)+(-.3,.1)$) {\tiny $a_2$};
	\node at ($(a0)+(-.2,.1)$) {\tiny $c$};
	\node at ($(b1)+(.3,.1)$) {\tiny $a_3$};
	\node at ($(c1)+(.3,.1)$) {\tiny $b_0$};
	\node at ($(a1)+(.3,.1)$) {\tiny $b_2$};
	\node at ($(b2)+(-.1,-.3)$) {\tiny $a_0$};
	\node at ($(c2)+(0,-.3)$) {\tiny $q$};
	\node at ($(a2)+(.1,-.3)$) {\tiny $b_1$};
	\node at ($(q)+(.25,-.1)$) {\tiny $a_2'$};
\end{tikzpicture}
\vskip -.2cm
\caption{}	
\label{fig:lem313C}
\end{subfigure}
\hfill
\begin{subfigure}[b]{.24\textwidth}
\centering
\begin{tikzpicture}[scale=.75]			
	\coordinate (c2) at (-.45,-.1);
	\coordinate (b2) at (-1.5,-.1);
	\coordinate (a2) at (1.5,-.1);
	\coordinate (b0) at (-.45,3);
	\coordinate(a1) at (.45,3);
	\coordinate (a0) at (-1.9,.6);
	\coordinate (b1) at (1.9,.6);
	\coordinate (c0) at ($(b0)!.5!(a0)$);
	\coordinate (c1) at ($(b1)!.5!(a1)$);
	\coordinate (q) at (.2, 1.1);
	\coordinate (r) at (.5, .4);
	\draw [black!20] (c0)--(q);
	\draw [black!20] (b2)--(a0)--(b0)--(a1);
	\draw [black!20] (b0)--(a2);
	\draw [black!20] (a0)--(b1);
	\draw (b2)--(a1);
	\draw (a1)--(b1)--(a2)--(b2);
	\draw (q)--(c2)--(q)--(b1)--(q)--(a1);
	\draw (b2)--(r)--(c1)--(r)--(a2);
	\foreach \i in {1,2}{
		\fill (a\i) circle (2pt);
		\fill (b\i) circle (2pt);
		\fill (c\i) circle (2pt);
	}
	\fill (q) circle (2pt);
	\fill (r) circle (2pt);
	\fill [cyan] (q) circle (1.7pt);
	\fill [cyan](c2) circle (1.7pt); 
	\foreach \i in {a0,c0,b0} \fill [black!20] (\i) circle (2pt);
	\node [black!20] at ($(b0)+(-.3,.1)$) {\tiny $b_4$};
	\node [black!20]  at ($(c0)+(-.3,.1)$) {\tiny $a_2$};
	\node [black!20] at ($(a0)+(-.2,.1)$) {\tiny $c$};
	\node at ($(b1)+(.3,.1)$) {\tiny $a_3$};
	\node at ($(c1)+(.3,.1)$) {\tiny $b_0$};
	\node at ($(a1)+(.3,.1)$) {\tiny $b_2$};
	\node at ($(b2)+(-.1,-.3)$) {\tiny $a_0$};
	\node at ($(c2)+(0,-.3)$) {\tiny $q$};
	\node at ($(a2)+(.1,-.3)$) {\tiny $b_1$};
	\node at ($(q)+(-.33,-.1)$) {\tiny $b_0'$};
	\node at ($(r)+(0,-.23)$) {\tiny $b_3$};
\end{tikzpicture}
\vskip -.2cm
\caption{}	
\label{fig:lem313D}
\end{subfigure}
\vskip -.2cm
\caption{The proof of Lemma \ref{l:314}.}
\label{fig:lem313}
\vskip -.2cm
\end{figure}
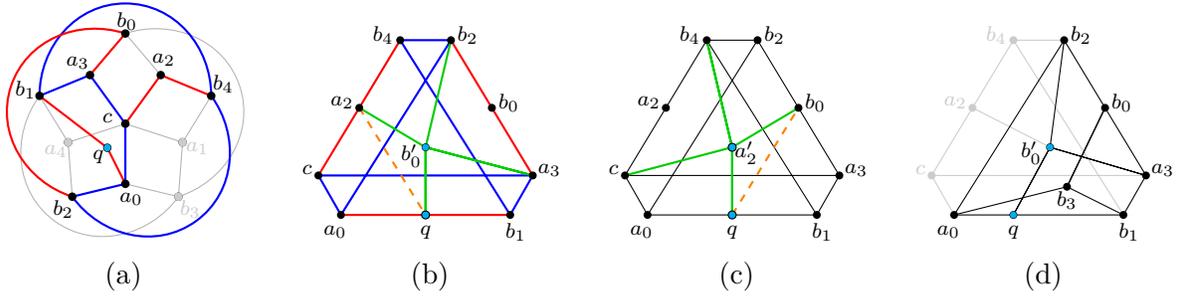

\begin{proof}
	If $a_2q\in E(G)$, then $q\in \Ext(\grot, a_1)$ and we are done. So we may 
	suppose $a_2, b_0\notin \Nn(q)$ and need to prove that $\Nn(q)$ 
	contains $\grot$-twins of those two vertices. Due to Lemma~\ref{l:35} 
	there are vertices $b'_0$ and $a'_2$ adjacent to $\{a_2,a_3,b_2,q\}$ 
	and  $\{b_0,b_4,c,q\}$, respectively (see Figures~\ref{fig:lem313B} 
	and~\ref{fig:lem313C}).
	Obviously $a_2'$ is an $\grot$-twin of $a_2$. 
	Moreover, the cube
		\begin{center}
			\begin{tabular}{cc}
				$(Q)$&
		\begin{tabular}{c|c|c|c}
			$q$&  $b_3$&  $b_2$ & $a_3$
			\\ \hline
			$b_0$&  $b_0'$&  $b_1$ & $a_0$
		\end{tabular}
		\end{tabular}
	\end{center}	
	drawn in Figure~\ref{fig:lem313D} yields $b_0'b_3\in E(G)$ and thus $b_0'$ is 
	an $\grot$-twin of $b_0$. 
\end{proof}
 
When applying one of the three previous lemmata it is often useful to know that 
certain vertices in the copy of $\grot$ under consideration are reliable, as 
this could eliminate one of several possible outcomes. So far, however, 
we have no way of inferring reliability. The last two lemmata of this subsection 
change this situation.
	
\begin{lemma}\label{l:315}
	If $\grot\subseteq G\in \fD_4$, then at least one of $a_1$, $a_2$, $a_3$ is reliable. 
\end{lemma}

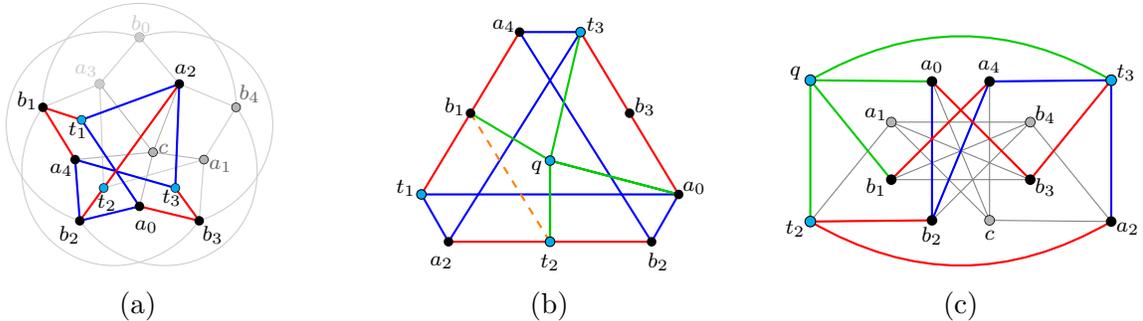
\begin{figure}[ht]
\centering
\vskip -.3cm
\begin{subfigure}[b]{.32\textwidth}
\centering
\begin{tikzpicture}[scale=.9]
	\coordinate (c) at (.2,-.2);
	\foreach \i in {1,...,5} {
		\coordinate (a\i) at (-18+\i*72:1);
		\coordinate (b\i) at (18+\i*72:1.5);	
		\coordinate (d\i) at (18+\i*72:2);
		\coordinate (t\i) at (18 +\i*72:.9);
		\draw [black!30](a\i)--(c);
	}
	\draw [black!20] (a1)--(b1)--(a2)--(b2)--(a3)--(b3)--(a4)--(b4)--(a5)--(b5)--(a1);
	\draw [black!20](a1)--(t2)--(a4);
	\draw [black!20](a2)--(t3)--(a5);
	\draw [black!20](a3)--(t4)--(a1);
	\draw [thick, blue] (t4)--(a3)--(b3)--(a4)--(t2)--(a1)--(t4);
	\draw [thick, red] (a3)--(b2)--(t2);
	\draw [thick, red] (b3)--(t3)--(a1);
	\draw [thick, red] (a4)--(b4)--(t4);
	\draw (b2) [black!20, out=234, in=144] to (d3) [out=-36, in=234] to (b4);		
	\draw (b1) [black!20, out=162, in=72] to (d2) [out=-108, in=162] to (b3);
	\draw (b4) [black!20,  out=18, in=-72] to (d5) [out=108, in=18] to (b1);
	\draw (b5) [black!20, out=90, in=0] to (d1) [out=180, in=90] to (b2);		
	\draw (b3) [black!20, out=-54, in=216] to (d4) [out=36, in=-54] to (b5);
	\node [black!60] at ($(c)+(.15,.05)$) {\tiny $c$};
	\node [black] at ($(a1)+(.1,.2)$) {\tiny $a_2$};
	\node [black!20] at ($(a2)+(-.2,.2)$) {\tiny $a_3$};
	\node at ($(a3)+(-.2,-.15)$) {\tiny $a_4$};
	\node at ($(a4)+(.1,-.3)$) {\tiny $a_0$};
	\node [black!60] at ($(a5)+(.25,-.1)$) {\tiny $a_1$};
	\node [black!20] at ($(b1)+(.03,.2)$) {\tiny $b_0$};
	\node at ($(b2)+(-.23,.1)$) {\tiny $b_1$};
	\node at ($(b3)+(-.15,-.2)$) {\tiny $b_2$};
	\node at ($(b4)+(.2,-.2)$) {\tiny $b_3$};
	\node [black!60] at ($(b5)+(.2,.2)$) {\tiny $b_4$};
	\node at ($(t2)+(-.05,-.15)$) {\tiny $t_1$};
	\node at ($(t3)+(.05,-.22)$) {\tiny $t_2$};
	\node at ($(t4)+(-.05,-.15)$) {\tiny $t_3$};
	\fill (c) circle (2pt);
	\foreach \i in {1,...,5}{
		\fill (a\i) circle (2pt);		
		\fill (b\i) circle (2pt);
	}
	\foreach \i in {2,3,4} {
		\fill (t\i) circle (2pt);
		\fill [cyan] (t\i) circle (1.7pt);
	}
	\foreach \i in {a5,b5,c} \fill [black!30] (\i) circle (1.7pt);
	\foreach \i in {b1,a2} \fill [black!20] (\i) circle (2pt);
\end{tikzpicture}
\vskip -.2cm
\caption{}	
\label{fig:lem314A}
\end{subfigure}
\hfill
\begin{subfigure}[b]{.32\textwidth}
\centering
\begin{tikzpicture}[scale=.9]			
	\coordinate (c2) at (0,-.1);
	\coordinate (b2) at (-1.5,-.1);
	\coordinate (a2) at (1.5,-.1);
	\coordinate (b0) at (-.45,3);
	\coordinate(a1) at (.45,3);
	\coordinate (a0) at (-1.9,.6);
	\coordinate (b1) at (1.9,.6);
	\coordinate (c0) at ($(b0)!.5!(a0)$);
	\coordinate (c1) at ($(b1)!.5!(a1)$);
	\coordinate (q) at (0, 1.1);
	\draw [thick, blue] (a0)--(b2)--(a1)--(b0)--(a2)--(b1)--(a0);
	\draw [thick, red] (b0) --(a0);
	\draw [thick, red] (b1) --(a1);
	\draw [thick, red] (b2) --(a2);
	\draw [thick, dashed, orange] (c0)--(c2);
	\draw [thick,green!80!black] (c0)--(q)--(c2)--(q)--(b1)--(q)--(a1);
	\foreach \i in {0,1,2}{
		\fill (a\i) circle (2pt);
		\fill (b\i) circle (2pt);
		\fill (c\i) circle (2pt);
	}
	\foreach \i in {q, a0, a1, c2} \draw [fill = cyan](\i) circle (2pt); 
	\node at ($(b0)+(-.23,.1)$) {\tiny $a_4$};
	\node at ($(c0)+(-.23,.1)$) {\tiny $b_1$};
	\node at ($(a0)+(-.23,.1)$) {\tiny $t_1$};
	\node at ($(b1)+(.23,.1)$) {\tiny $a_0$};
	\node at ($(c1)+(.23,.1)$) {\tiny $b_3$};
	\node at ($(a1)+(.23,.1)$) {\tiny $t_3$};
	\node at ($(b2)+(-.1,-.3)$) {\tiny $a_2$};
	\node at ($(c2)+(0,-.3)$) {\tiny $t_2$};
	\node at ($(a2)+(.1,-.3)$) {\tiny $b_2$};
	\node at ($(q)+(-.2,-.1)$) {\tiny $q$};
\end{tikzpicture}
\vskip -.2cm
\caption{}	
\label{fig:lem314B}
\end{subfigure}
\hfill
\begin{subfigure}[b]{.32\textwidth}
\centering
\begin{tikzpicture}[scale=1]			
	\foreach \i in {1,...,8} {
		\coordinate (a\i) at (22.5+\i*45:1);
	}
	\coordinate (b1) at (-2,.94);
	\coordinate (b2) at (2,.94);
	\coordinate (b3) at (2,-.94);
	\coordinate(b4) at (-2,-.94);
	\draw [green!80!black, thick] (b1) edge [bend left = 30] (b2);
	\draw [green!80!black, thick] (a2)--(b1)--(b4);
	\draw [black!50](b4)--(a5);
	\draw [black!50](a1) -- (a4) -- (a7) -- (a2) -- (a5) ;
	\draw [black!50](a5)-- (a8) -- (a3) ;
	\draw [black!50](a3)-- (a6) -- (a1);
	\draw [black!50](b4) -- (a3) -- (a7) -- (b2)--(a1);
	\draw [green!80!black, thick] (b1) -- (a4);
	\draw [black!50](a4)--(a8) -- (b3);
	\draw [black!50](b2)--(b3);
	\draw [black!50](a1) -- (a5);
	\draw [black!50](a2) -- (a6)--(b3);
	\draw [thick, blue](b3)--(b2)--(a1)--(a5)--(a2);
	\draw [thick, red](b2)--(a7)--(a2);
	\draw [thick, red](b4)--(a5);
	\draw [thick, red](b4)edge [bend right=30] (b3);
	\draw [thick, red](a1)--(a4);
	\foreach \i in {1,...,8}{
		\fill (a\i) circle (2pt);		
	}
	\foreach \i in {1,...,4} {
		\fill (b\i) circle (2pt);		
	}
	\foreach \i in {b1, b4, b2} \draw [fill=cyan] (\i) circle (2pt);
	\fill [black!30](a3) circle (1.7pt); 
	\fill [black!30] (a6) circle (1.7pt); 
	\fill [black!30] (a8) circle (1.7pt);
	\node at ($(a1)+(0,.2)$) {\tiny $a_4$};
	\node at ($(a2)+(0,.2)$) {\tiny $a_0$};
	\node at ($(a3)+(-.2,.1)$) {\tiny $a_1$};
	\node at ($(a4)+(-.2,-.1)$) {\tiny $b_1$};
	\node at ($(a5)+(0,-.2)$) {\tiny $b_2$};
	\node at ($(a6)+(0,-.2)$) {\tiny $c$};
	\node at ($(a7)+(.2,-.1)$) {\tiny $b_3$};
	\node at ($(a8)+(.2,.1)$) {\tiny $b_4$};
	\node at ($(b1)+(-.2,.1)$) {\tiny $q$};
	\node at ($(b3)+(.22,-.1)$) {\tiny $a_2$};
	\node at ($(b4)+(-.2,-.1)$) {\tiny $t_2$};
	\node at ($(b2)+(.2,.1)$) {\tiny $t_3$};
\end{tikzpicture}
\vskip -.2cm
\caption{}	
\label{fig:lem314C}
\end{subfigure}
\vskip -.2cm
\caption{The proof of Lemma~\ref{l:315}.}
\label{fig:lem314}
\vskip -.3cm
\end{figure}

\begin{proof}
	Assume contrariwise that there exist vertices $t_i\in \Ext(\grot, a_i)$ 
	for $i\in\{1, 2, 3\}$. Let us recall that this means 
	$a_{i-1}, a_{i+1}, b_i\in \Nn(t_i)$. The beautiful lemma yields $t_2a_2\in E(G)$ 
	(see Figure~\ref{fig:lem314A}). Since $G$ is triangle-free, $b_1, t_2\in\Nn(a_3)$
	implies $b_1t_2\not\in E(G)$. Together with Lemma~\ref{l:35} this
	ensures the existence of a common neighbour $q$ of $\{a_0, b_1, t_2, t_3\}$ 
	(see Figure~\ref{fig:lem314B}).
	But now the graph $\grot -a_3 -b_0 + q+t_2+t_3$ drawn in Figure~\ref{fig:lem314C} 
	has no independent set of size five, which contradicts $G\in \fD_4$.		
\end{proof}
 				
\begin{lemma}\label{l:316}
	Let $\grot\subseteq G\in \fD_4$ and let $a_1'$ be an $\grot$-twin of~$a_1$. 
	If $t_2'\in \Ext(\grot(a_1'), a_2)$ and there exists a common neighbour 
	of $a_1$, $b_1$, $t_2'$, then $a_0$ is reliable (with respect to $\grot$).
\end{lemma}

\vskip -.3cm 
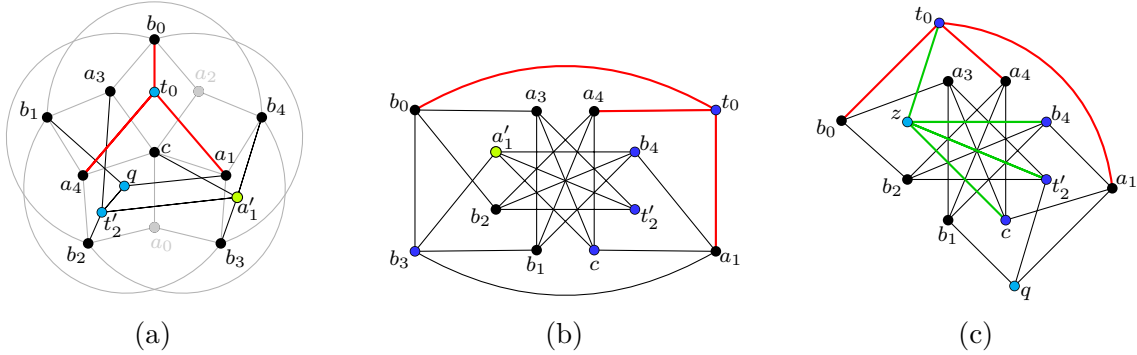
\begin{figure}[ht]
\centering
\begin{subfigure}[b]{.32\textwidth}
\centering
\begin{tikzpicture}[scale=1]			
	\coordinate (c) at (0,0);
	\foreach \i in {1,...,5} {
		\coordinate (a\i) at (-18+\i*72:1);
		\coordinate (b\i) at (18+\i*72:1.5);	
		\coordinate (d\i) at (18+\i*72:2);
		\draw [black!30](a\i)--(c);
	}
	\coordinate (t) at (-.7,-.8);
	\coordinate (q) at (-.4, -.45);
	\coordinate (a) at (1.1, -.6);
	\coordinate (t0) at (0,.8);
	\draw [black!30] (a1)--(b1)--(a2)--(b2)--(a3)--(b3)--(a4)--(b4)--(a5)--(b5)--(a1);
	\draw (a2)--(t)--(a)--(t)--(b3);
	\draw (c) -- (a) --(b5) -- (a) --(b4);
	\draw (b2)--(q)--(t)--(q)--(a5);
	\draw [thick, red] (b1)--(t0)--(a3)--(t0)--(a5);
	\draw (b2) [black!30, out=234, in=144] to (d3) [out=-36, in=234] to (b4);		
	\draw (b1) [black!30, out=162, in=72] to (d2) [out=-108, in=162] to (b3);
	\draw (b4) [black!30, out=18, in=-72] to (d5) [out=108, in=18] to (b1);
	\draw (b5) [black!30, out=90, in=0] to (d1) [out=180, in=90] to (b2);		
	\draw (b3) [black!30, out=-54, in=216] to (d4) [out=36, in=-54] to (b5);
	\node at ($(c)+(.15,.05)$) {\tiny $c$};
	\node [black!20]  at ($(a1)+(.1,.2)$) {\tiny $a_2$};
	\node at ($(a2)+(-.2,.2)$) {\tiny $a_3$};
	\node at ($(a3)+(-.15,-.15)$) {\tiny $a_4$};
	\node [black!20] at ($(a4)+(.1,-.23)$) {\tiny $a_0$};
	\node at ($(a5)+(0,.25)$) {\tiny $a_1$};
	\node at ($(b1)+(.03,.2)$) {\tiny $b_0$};
	\node at ($(b2)+(-.23,.1)$) {\tiny $b_1$};
	\node at ($(b3)+(-.15,-.15)$) {\tiny $b_2$};
	\node at ($(b4)+(.2,-.2)$) {\tiny $b_3$};
	\node at ($(b5)+(.2,.15)$) {\tiny $b_4$};
	\node at ($(t)+(.15,-.17)$) {\tiny $t_2'$};
	\node at ($(a) +(.15,-.15)$) {\tiny $a_1'$};
	\node at ($(q)+(.1,.13)$) {\tiny $q$};
	\node at ($(t0)+(.2,.05)$) {\tiny $t_0$};
	\fill (c) circle (2pt);
	\foreach \i in {1,...,5} {
		\fill (a\i) circle (2pt);		
		\fill (b\i) circle (2pt);
	}
	\foreach \i in {t,q,t0} {
		\fill (\i) circle (2pt);
		\fill [cyan] (\i) circle (1.7pt);
	}
	\draw [fill = lime] (a) circle (2pt);
	\foreach \i in { a1, a4} \fill [black!20] (\i) circle (2pt);
\end{tikzpicture}
\vskip -.2cm
\caption{}	
\label{fig:lem315A}
\end{subfigure}
\hfill
\begin{subfigure}[b]{.32\textwidth}
\centering
\begin{tikzpicture}[scale=1]			
	\foreach \i in {1,...,8} {
		\coordinate (a\i) at (22.5+\i*45:1);
	}
	\coordinate (b1) at (-2,.94);
	\coordinate (b2) at (2,.94);
	\coordinate (b3) at (2,-.94);
	\coordinate(b4) at (-2,-.94);
	\draw (b1) edge [thick, red, bend left = 30] (b2);
	\draw (b4) edge [bend right = 30 ] (b3);
	\draw (a2)--(b1)--(b4)--(a5);
	\draw (a1) -- (a4) -- (a7) -- (a2) -- (a5) -- (a8) -- (a3) -- (a6) -- (a1);
	\draw (b4) -- (a3)--(a7);
	\draw (b1) -- (a4) -- (a8) -- (b3);
	\draw (a1)--(b2)--(b3)--(a6);
	\draw (a1) -- (a5);
	\draw (a2) -- (a6);
	\draw [thick,red] (b3)--(b2)--(a1);
	\foreach \i in {1,...,8} {
		\fill (a\i) circle (2pt);		
	}
	\foreach \i in {1,...,4} {
		\fill (b\i) circle (2pt);		
	}
	\foreach \i in {b4,a8,a6,b2,a7} \fill [blue!80] (\i) circle (1.7pt);
	\draw [fill=lime] (a3) circle (2pt);
	\node at ($(a1)+(0,.2)$) {\tiny $a_4$};
	\node at ($(a2)+(0,.2)$) {\tiny $a_3$};
	\node at ($(a3)+(.1,.2)$) {\tiny $a_1'$};
	\node at ($(a4)+(-.2,-.1)$) {\tiny $b_2$};
	\node at ($(a5)+(0,-.2)$) {\tiny $b_1$};
	\node at ($(a6)+(0,-.2)$) {\tiny $c$};
	\node at ($(a7)+(.2,-.1)$) {\tiny $t_2'$};
	\node at ($(a8)+(.2,.1)$) {\tiny $b_4$};
	\node at ($(b1)+(-.2,.1)$) {\tiny $b_0$};
	\node at ($(b3)+(.2,-.1)$) {\tiny $a_1$};
	\node at ($(b4)+(-.2,-.1)$) {\tiny $b_3$};
	\node at ($(b2)+(.2,.1)$) {\tiny $t_0$};
\end{tikzpicture}
\vskip -.2cm
\caption{}	
\label{fig:lem315B}
\end{subfigure}
\hfill
\begin{subfigure}[b]{.32\textwidth}
\centering
\begin{tikzpicture}[scale=1]
	\foreach \i in {1,...,8} {
		\coordinate (a\i) at (22.5+\i*45:1);
	}
	\coordinate (b1) at (-.5,1.7);
	\coordinate (b2) at (1.8,-.5);
	\coordinate (b3) at (.5,-1.8);
	\coordinate (b4) at (-1.8,.4);
	\draw (b4)--(a4);
	\draw (a1) -- (a4) -- (a7) -- (a2) -- (a5) -- (a8) -- (a3) -- (a6) -- (a1);
	\draw (b4) --(a2);
	\draw (a7) -- (b3);
	\draw (a6)--(b2)--(b3)--(a5);
	\draw (a4)--(a8)--(b2);
	\draw (a1) -- (a5);
	\draw (a2) -- (a6);
	\draw (a3)--(a7);
	\draw [green!80!black, thick] (a6)--(a3)--(a7)--(a3)--(a8)--(a3)--(b1);
	\draw [red, thick] (b4)--(b1) --(a1);
	\draw (b1) edge [red, thick, bend left=40] (b2);
	\foreach \i in {1,...,8} {
		\fill (a\i) circle (2pt);		
	}
	\foreach \i in {1,...,4} {
		\fill (b\i) circle (2pt);		
	}
	\foreach \i in {a6,a7,a8,b1} \fill [blue!80] (\i) circle (1.7pt);
	\fill [cyan] (a3) circle (1.7pt);
	\fill [cyan] (b3) circle (1.7pt);
	\node at ($(a1)+(.2,.1)$) {\tiny $a_4$};
	\node at ($(a2)+(.2,.1)$) {\tiny $a_3$};
	\node at ($(a3)+(-.15,.1)$) {\tiny $z$};
	\node at ($(a4)+(-.2,-.1)$) {\tiny $b_2$};
	\node at ($(a5)+(0,-.2)$) {\tiny $b_1$};
	\node at ($(a6)+(0,-.18)$) {\tiny $c$};
	\node at ($(a7)+(.2,-.1)$) {\tiny $t_2'$};
	\node at ($(a8)+(.2,.1)$) {\tiny $b_4$};
	\node at ($(b1)+(-.2,.1)$) {\tiny $t_0$};
	\node at ($(b3)+(.15,-.1)$) {\tiny $q$};
	\node at ($(b4)+(-.2,-.1)$) {\tiny $b_0$};
	\node at ($(b2)+(.2,.1)$) {\tiny $a_1$};
\end{tikzpicture}
\vskip -.2cm
\caption{}	
\label{fig:lem315C}
\end{subfigure}
\vskip -.2cm
\caption{The proof of Lemma~\ref{l:316}.}
\label{fig:lem315}
\end{figure}

\begin{proof}
	Let $q$ be a common neighbour of $\{a_1,b_1,t_2'\}$ and assume for the sake of 
	contradiction that there exists some $t_0\in \Ext(\grot, a_0)$, so that 
	$a_1, a_4, b_0 \in \Nn(t_0)$ (see Figure~\ref{fig:lem315A}). 
	Since the only independent set of size five in the graph 
	$\grot - a_0-a_2 + a_1'+t_0+t_2'$ drawn in Figure~\ref{fig:lem315B} 
	is  $\{b_3, b_4, c, t_0, t_2'\}$, property~\Dp{4} guarantees the existence 
	of a common neighbour $z$ of this set.  
	But now the graph $\grot -a_0 -a_2-b_3+t_0+t_2'+q+z$ drawn in 
	Figure~\ref{fig:lem315C} has no independent set of size five, contrary 
	to $G\in \fD_4$. 
\end{proof}

\subsection{Independent sets}
Let us now return to Vega graphs and study their independent subsets. 
An independent set 
$T\subseteq V(\grot^{\mu\nu}_i)$ is said to be {\it small} if it 
intersects~$\{x,y\}$ and two of the sets $\Gr$, $\Gg$, $\Gb$. So if $\mu=1$, then 
all small sets contain $x$. 
We proceed with a classification of independent sets.

\begin{lemma}\label{l:317}
	If $T\subseteq V(\grot_i^{\mu\nu})$ is independent, then one of the following 
	six cases occurs. 
	\begin{enumerate}[label=\alabel]
		\item \label{it:a} There is a vertex of $V(\grot_i^{\mu\nu})$ whose 
		neighbourhood contains $T$;
		\item\label{it:b} $\mu=1$ and 
		$\{\re{u},\gr{v},\bl{w}\}\subseteq T\subseteq \{\re{u},\gr{v},\bl{w},x\}$;
		\item\label{it:c} $\nu=1$ and 
		$\{\gr{b,v},\re{i-1}\}\subseteq T\subseteq \gr{\{b,v\}}\cup \Gr$;
		\item\label{it:d} $\{\bl{c,w},\re{0}\}\subseteq T\subseteq \bl{\{c,w\}}\cup \Gr$; 
		\item\label{it:e} $\nu=0$ and 
		$\{\bl{c,w},\gr{2i-1}\}\subseteq T\subseteq \bl{\{c,w\}}\cup \Gg$;
		\item\label{it:f} $T$ is small.
	\end{enumerate}
\end{lemma}
\begin{proof}
	A vertex $q\in V(\grot_i^{\mu\nu})$ is said to {\it govern} $T$ 
	if $T\subseteq \Nn(q)$. Let~$\Phi$ be the set of colours~$\varphi$ 
	satisfying $T\cap \G_{\varphi}\ne\vn$. It is easily seen that $|\Phi|=3$ 
	is impossible. 
	If $|\Phi|=2$ there is a vertex $j\in \G_i$ 
	such that $T\cap \G_i\subseteq \Nn(j)$ and either~$j$ governs~$T$ or~$T$ is small. 
	We may henceforth suppose that $|\Phi|\le 1$. 
	
	Next, let~$\Psi$ be the set of colours 
	of the vertices in $T\cap \cC_6$. If $|\Psi|=3$ and neither~$x$ nor~$y$ governs~$T$, 
	then~\ref{it:b} holds. If $|\Psi|=2$, then $T\cap \cC_6$ consists of two vertices 
	and the vertex between them governs~$T$. If $|\Psi|\le 1$ and, 
	moreover, $|T\cap \cC_6|\le 1$, then
	there is a hexagonal vertex governing~$T$. In all remaining cases $T\cap \cC_6$ 
	is a pair of vertices of the same colour.
	
	But if $\re{\{a,u\}}\subseteq T$, then one of $\re{0}$, $\re{i-1}$ governs~$T$ 
	(depending on whether the vertices in~$T\cap \G_i$ are \gr{green} 
	or \bl{blue}). Similarly, if $\gr{\{b,v\}}\subseteq T$ and none 
	of $\gr{i}$, $\gr{2i-2}$, $\gr{2i-1}$ governs~$T$, then~\ref{it:c} holds. 
	Finally, if $\bl{\{c,w\}}\subseteq T$ and none of $\bl{2i}$, $\bl{3i-2}$ 
	governs~$T$, then~\ref{it:d} or~\ref{it:e} holds. 
\end{proof}

In the remainder of this subsection we study situations where 
$\grot^{\mu\nu}_i\subseteq G\in \fD_4$ and for some $q\in V(G)$ the set 
$\Nn(q)\cap V(\grot_i^{\mu\nu})$ is in one of the cases \ref{it:b}\,--\,\ref{it:f}.
We begin with a couple of simple applications of the cube lemma. 

\begin{lemma}\label{l:318}
	Let $\grot_i^{\mu\nu}\subseteq G\subseteq \fD_4$ and $q\in V(G)$. 
	\begin{enumerate}[label=\alabel]
		\item \label{it:318a} If $\re{u}, \gr{v}, \bl{w}\in \Nn(q)$, 
			then $x\in \Nn(q)$.
		\item\label{it:318b} If $\re{a}, \gr{b}, \bl{c}\in \Nn(q)$ and $\mu=0$, 
			then $y\in \Nn(q)$.
	\end{enumerate}
\end{lemma}

\begin{proof}
	Part~\ref{it:318a} follows from the fact that the cube
	\begin{center}
		\begin{tabular}{cc}
			$(Q)$&
		\begin{tabular}{c|c|c|c}
			$q$&  $\re{a}$&  $\gr{b}$ & $\bl{c}$\\ \hline
			$x$&  $\re{u}$&  $\gr{v}$ & $\bl{w}$
		\end{tabular}
			\end{tabular}
	\end{center}	
	cannot be induced. By $\sigma$-symmetry part~\ref{it:318b} holds as well. 
\end{proof}

\begin{lemma}\label{l:319}
	Suppose $\grot_i^{\mu\nu}\subseteq G\in \fD_4$ and that $q\in V(G)$ is 
	adjacent to $\gr{b}$, $\gr{v}$. 
	\begin{enumerate}[label=\alabel]
		\item\label{it:319a} If $\re{j}\in \Nn(q)\cap \Gr$, then 
			$\re{[0,j]}\subseteq \Nn(q)$.
		\item\label{it:319b} If $\bl{j}\in \Nn(q)\cap\Gb$, then 
			$\bl{[j,3i-2]}\subseteq \Nn(q)$.
	\end{enumerate}
\end{lemma}

\begin{proof}
	For the proof of part~\ref{it:319a} we consider any $\re{t}\in \re{[0,j)}$ 
	and look at the cube 
		\begin{center}
		\begin{tabular}{cc}
				$(Q)$&
		\begin{tabular}{c|c|c|c}
			$q$&  $\re{a}$&  $\re{u}$ & $\gr{t+i}$
			\\ \hline 
			$\re{t}$&  $\gr{b}$&  $\gr{v}$ & $\re{j}$
		\end{tabular} 
		\end{tabular}.
	\end{center}	
	Since $(Q)$ is not induced and $\re{j}\gr{(t+i)}\not\in E(G)$, 
	we have indeed $q\re{t}\in E(G)$.
	Similarly, to prove part~\ref{it:319b}, we observe that for given 
	$\bl{t}\in \bl{(j,3i-2]}$ the facts that the cube 
	\begin{center}
		\begin{tabular}{cc}
			$(Q)$&
	\begin{tabular}{c|c|c|c}
		$q$&  $\bl{c}$&  $\bl{w}$ & $\gr{t-i}$
		\\ \hline
		$\bl{t}$&  $\gr{b}$&  $\gr{v}$ & $\bl{j}$
	\end{tabular}
	\end{tabular}
\end{center}	
	is not induced and $\gr{(t-i)}\bl{j}\not\in E(G)$
	yield the edge $q\bl{t}\in E(G)$.
\end{proof}

\begin{cor}\label{cor:320}
	If $\grot_i^{\mu 1}\subseteq G \in \fD_4$ and some $q\in V(G)$ is adjacent 
	either to $\gr{b}$, $\gr{v}$, $\re{i-1}$ or to $\bl{c}$, $\bl{w}$, $\re{0}$, 
	then~$G$ has a subgraph isomorphic to~$\grot_i^{\mu 0}$.
\end{cor}

\begin{proof}
	Because of the automorphism~$\tau_1$ it suffices to treat the case 
	$\gr{b}, \gr{v}, \re{i-1}\in \Nn(q)$ (recall that $\tau_1(\bl{c})=\gr{b}$, 
	$\tau_1(\bl{w})=\gr{v}$, $\tau_1(\re{0})=\re{i-1}$). Lemma~\ref{l:319}\ref{it:319a} 
	yields $\Gr\subseteq \Nn(q)$ and, therefore,~$q$ can play the r\^{o}le 
	of~$\gr{2i-1}$.
\end{proof}

\begin{lemma}\label{l:321}
	Let $\grot_i^{\mu\nu}\subseteq G\in \fD_4$ and let $q\in V(G)$ be adjacent 
	to $\bl{c}$, $\bl{w}$. 
	\begin{enumerate}[label=\alabel]
		\item\label{it:321a} If $\re{j}\in \Nn(q)\cap\Gr$, then 
			$\re{[j,i-1]}\subseteq \Nn(q)$.
		\item\label{it:321b} If $\gr{j}\in \Nn(q)\cap \Gg$, then 
			$\gr{[i,j]}\subseteq \Nn(q)$.
	\end{enumerate}
\end{lemma}

\begin{proof}
	As in the proof of Lemma~\ref{l:319} we consider any vertices 
	$\re{t}\in \re{(j,i-1]}$, $\gr{s}\in \gr{[i,j)}$ and look at the cubes	
	\begin{center}
	\begin{tabular}{ccc}
	\begin{tabular}{cc}
		$(Q)$&
		\begin{tabular}{c|c|c|c}
			$q$&  $\re{a}$&  $\re{u}$ & $\bl{t-i}$
			\\ \hline
			$\re{t}$&  $\bl{c}$&  $\bl{w}$ & $\re{j}$
		\end{tabular}
	\end{tabular} &
	and &
	\begin{tabular}{cc}
		$(Q)$&
		\begin{tabular}{c|c|c|c}
			$q$&  $\gr{b}$&  $\gr{v}$ & $\bl{s+i}$
			\\ \hline
			$\gr{s}$&  $\bl{c}$&  $\bl{w}$ & $\gr{j}$
		\end{tabular}
	\end{tabular}
		\end{tabular},			
	\end{center}
	respectively.
\end{proof}

\begin{lemma}\label{l:321r}
	Let $\grot_i^{\mu\nu}\subseteq G\in \fD_4$ and let $q\in V(G)$ be adjacent 
	to $\re{a}$, $\re{u}$. 
	\begin{enumerate}[label=\alabel]
		\item\label{it:321ra} If $\gr{j}\in \Nn(q)\cap\Gg$, then 
			$\gr{[j,2i-1-\nu]}\subseteq \Nn(q)$.
		\item\label{it:321rb} If $\bl{j}\in \Nn(q)\cap \Gb$, then 
			$\bl{[2i,j]}\subseteq \Nn(q)$.
	\end{enumerate}
\end{lemma}

\begin{proof}
	Arguing similarly again, we consider any vertices $\gr{t}\in \gr{(j,2i-1-\nu]}$, 
	$\bl{s}\in \bl{[2i,j)}$ and look at the cubes	
	\begin{center}
		\begin{tabular}{ccc}
							\begin{tabular}{cc}
					$(Q)$&
					\begin{tabular}{c|c|c|c}
						$q$&  $\gr{b}$&  $\gr{v}$ & $\re{t-i}$
						\\ \hline
						$\gr{t}$&  $\re{a}$&  $\re{u}$ & $\gr{j}$
					\end{tabular}
				\end{tabular} &
								and &
								\begin{tabular}{cc}
					$(Q)$&
					\begin{tabular}{c|c|c|c}
						$q$&  $\bl{c}$&  $\bl{w}$ & $\re{s+i}$
						\\ \hline
						$\bl{s}$&  $\re{a}$&  $\re{u}$ & $\bl{j}$
					\end{tabular}
				\end{tabular}
							\end{tabular},
		\end{center}
	respectively.	
\end{proof}

\begin{lemma}\label{l:322}
	If $\grot_i^{\mu 0}\subseteq G\in \fD_4$ and some vertex $q\in V(G)$ is adjacent to 
	either $\bl{c}$, $\bl{w}$, $\re{0}$ or to $\bl{c}$, $\bl{w}$, $\gr{2i-1}$, 
	then~$G$ contains a subgraph isomorphic to~$\grot_{i+1}^{\mu 1}$.
\end{lemma}

	\begin{figure}[ht]
	\centering
	\def\si{3cm}
	\def\sj{1.6}
	\def\sk{1.8}
	\begin{subfigure}[b]{.32\textwidth}
	\centering
	\begin{tikzpicture}[scale=1]
		\phantom{\node at (0,2.3) {s};}
		\foreach \i in {1,...,7} {
			\coordinate (a\i) at (\i*60:1.5);
		}
		\draw [cyan, thick] (a1)--(a2);
		\draw [green!70!black, thick] (a2)--(a3)--(a4);
		\draw [cyan, thick] (a4)--(a5);
		\draw (a5)--(a6)--(a1);
		\foreach \i in {1,...,6} {
			\pgfmathsetmacro\j{\i+1};
			\fill (a\i) circle (1.5pt);		
		}
		\fill [blue] (a1) circle (1.5pt);
		\fill [blue] (a5) circle (1.5pt);
		\foreach \i in {2,3,4} \fill [green!70!black] (a\i) circle (1.5pt);
		\node [left] at (a2) {$\gr{b}$};
		\node [left] at (a3) {\small $\gr{2i-1}$};
		\node [left] at (a4) {$\gr{v}$};
		\node [right] at (a5) {$\bl{c}$};
		\node [right] at (a6) {$q$};
		\node [right] at (a1) {$\bl{w}$};
	\end{tikzpicture}
	\vskip -.1cm
	\caption{}	
	\label{fig:lem321A}
	\end{subfigure}
	\hfill
	\begin{subfigure}[b]{.32\textwidth}
	\centering
	\begin{tikzpicture}[scale=.6]			
		\coordinate (v) at (30:\si);
		\coordinate (a) at (90:\si);
		\coordinate (w) at (150:\si);
		\coordinate (b) at (210:\si);
		\coordinate (u) at (270:\si);
		\coordinate (c) at (330:\si);
		\coordinate (x) at (3.8,0);
		\foreach \i in {1,...,36}{
			\coordinate (v\i) at (\i*10:\sj cm);
		}
		\draw [thick](c)--(x);
		\draw [rounded corners=20] (x) -- (2.5, -2.8)--(0,-3.4)--(-2.2,-2.5)--(b);
		\draw [rounded corners=30] (x) -- (2.8,2.6)--(a);
		\draw  (a)--(v)--(c)--(u)--(b)--(w)--(a);
		\draw [red!70!black, thick,domain=-20:90] plot ({\sj*cos(\x)},{\sj*sin(\x)}); 
		\draw [blue, thick,domain=100:210] plot ({\sj*cos(\x)},{\sj*sin(\x)}); 
		\draw [green!60!black, thick,domain=230:330] plot ({\sj*cos(\x)},{\sj*sin(\x)}); 
		\draw [red!70!black ,domain=5:65, dashed] plot ({\sk*cos(\x)},{\sk*sin(\x)}); 
		\draw [blue, dashed, domain=115:202] plot ({\sk*cos(\x)},{\sk*sin(\x)}); 
		\draw [green!60!black, dashed,domain=240:327] plot ({\sk*cos(\x)},{\sk*sin(\x)}); 
		\foreach \i in {c, w}{
			\draw [fill=blue!75!white]  (\i) circle (3pt);
		}	
		\foreach \i in {a,u}{
			\draw [fill=red!75!white]  (\i) circle (3pt);
		}	
		\foreach \i in {v,b}{
			\draw [fill=green!75!white]  (\i) circle (3pt);
		}	
		\foreach \i in {9,34,7, 36}{
			\draw[red!75!black, very thick]  (v\i) circle (2pt);
			\fill[red!75!white]  (v\i) circle (2pt);
		}
		\foreach \i in {23,33}{
			\draw[green!75!black, very thick]  (v\i) circle (2pt);
			\fill[green!75!white]  (v\i) circle (2pt);
		}
		\foreach \i in {10,21}{
			\draw[blue!75!black, very thick]  (v\i) circle (2pt);
			\fill[blue!75!white]  (v\i) circle (2pt);
		}
		\fill (x) circle (3pt);
		\draw (v22) circle (3pt);
		\node at ($(a)+(0,-.3)$) {\tiny $\re{c}$};
		\node at ($(b)+(-.3,-.1)$) {\tiny $\gr{a}$};
		\node at ($(c)+(-.25,.25)$) {\tiny $\bl{b}$};
		\node at ($(v)+(-.25,-.25)$) {\tiny $\gr{u}$};
		\node at ($(w)+(-.3,.1)$) {\tiny $\bl{v}$};
		\node at ($(u)+(0,.3)$) {\tiny $\re{w}$};
		\node at ($(x)+(.2,.3)$) {\tiny $x$};
		\node at (-.1,.1) {\small missing};
		\node at (-.1,-.4) {\small vertex};
		\draw [->, thick] (-.1,-.7)[out=-90, in = 0] to (-1,-1);
		\node [scale = .8,blue] at ($(v10)+(-.25,.3)$) {\tiny $2i-1$};
		\node [scale = .8,blue] at ($(v21)+(-.3,0)$) {\tiny $i$};
		\node [scale = .8,green!60!black] at ($(v23)+(-.3,-.3)$) {\tiny $i-1$};
		\node [scale = .8,green!60!black] at ($(v33)+(.25,-.1)$) {\tiny $0$};
		\node [scale = .9,red!70!black] at ($(v9)+(.1,.3)$) {\tiny $q$};
		\node [scale = .8,red!70!black] at ($(v7)+(.1,.3)$) {\tiny $2i$};
		\node [scale = .7,red!70!black] at ($(v35)+(.57,.25)$) {\tiny $3i-2$};
		\node [scale = .9,red!70!black] at ($(v34)+(.25,0)$) {\tiny $z$};
	\end{tikzpicture}
	\vskip -.1cm
	\caption{}	
	\label{fig:lem321B}
	\end{subfigure}
	\hfill
	\begin{subfigure}[b]{.32\textwidth}
	\centering
	\begin{tikzpicture}[scale=.6]
		\coordinate (v) at (30:\si);
		\coordinate (a) at (90:\si);
		\coordinate (w) at (150:\si);
		\coordinate (b) at (210:\si);
		\coordinate (u) at (270:\si);
		\coordinate (c) at (330:\si);
		\foreach \i in {1,...,36}{
			\coordinate (v\i) at (\i*10:\sj cm);
		}
		\draw [thick](c)--(x);
		\draw [rounded corners=20] (x) -- (2.5, -2.8)--(0,-3.4)--(-2.2,-2.5)--(b);
		\draw [rounded corners=30] (x) -- (2.8,2.6)--(a);
		\draw  (a)--(v)--(c)--(u)--(b)--(w)--(a);
		\draw [red!60!black, thick,domain=-20:90] plot ({\sj*cos(\x)},{\sj*sin(\x)}); 
		\draw [blue, thick,domain=100:210] plot ({\sj*cos(\x)},{\sj*sin(\x)}); 
		\draw [green!60!black, thick,domain=230:330] plot ({\sj*cos(\x)},{\sj*sin(\x)}); 
		\draw [red!60!black, domain=-8:66, dashed] plot ({\sk*cos(\x)},{\sk*sin(\x)}); 
		\draw [blue, dashed, domain=115:202] plot ({\sk*cos(\x)},{\sk*sin(\x)}); 
		\draw [green!60!black, dashed,domain=240:300] plot ({\sk*cos(\x)},{\sk*sin(\x)}); 
		\foreach \i in {c, w}{
			\draw [fill=blue!75!white]  (\i) circle (3pt);
		}	
		\foreach \i in {a,u}{
			\draw [fill=red!75!white]  (\i) circle (3pt);
		}	
		\foreach \i in {v,b}{
			\draw [fill=green!75!white]  (\i) circle (3pt);
		}	
		\foreach \i in {9,34,7}{
			\draw[red!75!black, very thick]  (v\i) circle (2pt);
			\fill[red!75!white]  (v\i) circle (2pt);
		}
		\foreach \i in {23,33,31}{
			\draw[green!75!black, very thick]  (v\i) circle (2pt);
			\fill[green!75!white]  (v\i) circle (2pt);
		}
		\foreach \i in {10,21}{
			\draw[blue!75!black, very thick]  (v\i) circle (2pt);
			\fill[blue!75!white]  (v\i) circle (2pt);
		}
		\fill (x) circle (3pt);
		\draw (v22) circle (3pt);
		\node at ($(a)+(0,-.35)$) {\tiny $\re{b}$};
		\node at ($(b)+(-.3,-.1)$) {\tiny $\gr{c}$};
		\node at ($(c)+(-.3,.2)$) {\tiny $\bl{a}$};
		\node at ($(v)+(-.3,-.2)$) {\tiny $\gr{w}$};
		\node at ($(w)+(-.3,.1)$) {\tiny $\bl{u}$};
		\node at ($(u)+(0,.3)$) {\tiny $\re{v}$};
		\node at ($(x)+(.2,.3)$) {\tiny $x$};
		\node at (-.1,.1) {\small missing};
		\node at (-.1,-.4) {\small vertex};
		\draw [->, thick] (-.1,-.7)[out=-90, in = 0] to (-1,-1);
		\node [scale = .8,blue] at ($(v10)+(-.3,.25)$) {\tiny $i-1$};
		\node [scale = .8,blue] at ($(v21)+(-.3,0)$) {\tiny $0$};
		\node [scale = .8, green!60!black] at ($(v23)+(-.4,-.3)$) {\tiny $3i-2$};
		\node [scale = .8,green!60!black] at ($(v31)+(.1,-.3)$) {\tiny $2i$};
		\node [scale = .9,green!60!black] at ($(v33)+(.25,-.15)$) {\tiny $q$};
		\node [scale = .9,red!60!black] at ($(v9)+(.1,.25)$) {\tiny $z$};
		\node [scale = .8,red!70!black] at ($(v7)+(.1,.3)$) {\tiny $i$};
		\node [scale = .7,red!70!black] at ($(v34)+(.58,0)$) {\tiny $2i-1$};
	\end{tikzpicture}
	\vskip -.1cm
	\caption{}	
	\label{fig:lem321C}
	\end{subfigure}
	\vskip -.2cm
	\caption{The proof of Lemma~\ref{l:322}.}
	\label{fig:lem321}
	\vskip -.2cm
	\end{figure}

\begin{proof}
	Since $\tau_0$ fixes $\bl{c}$, $\bl{w}$ and exchanges $\re{0}$, $\gr{2i-1}$,
	it suffices to treat the case $\bl{c},\bl{w},\re{0}\in \Nn(q)$. 
	Lemma \ref{l:321}\ref{it:321a} yields $\Gr\subseteq \Nn(q)$. By~\Dp{2} applied to 
	the hexagon $\gr{v(2i-1)b}\bl{w}q\bl{c}$ (see Figure~\ref{fig:lem321A})
	there exists a vertex $z$ adjacent to either $\bl{c}$, $\bl{w}$, $\gr{2i-1}$ 
	or to $\gr{b}$, $\gr{v}$, $q$. 
	
	In the former case Lemma~\ref{l:321}\ref{it:321b} yields $\Gg\subseteq \Nn(z)$ 
	and the desired copy of $\grot^{\mu 1}_{i+1}$ is shown, with the possible 
	exception of $y$, in Figure~\ref{fig:lem321B}.
	
	So we can henceforth assume $\gr{b}, \gr{v}, q\in \Nn(z)$. Notice that $q\gr{i}$
	cannot be an edge of $G$, since otherwise $G$ contained the triangle $q\re{0}\gr{i}$. 
	Next, the cube
	\begin{center}
		\begin{tabular}{cc}
			$(Q)$&
			\begin{tabular}{c|c|c|c}
					$z$&$\bl{c}$&  $\bl{w}$&  $\gr{i}$
									\\ \hline
			$\bl{2i}$&	$\gr{b}$&  $\gr{v}$&   $q$
			\end{tabular}
		\end{tabular}
	\end{center}	
	cannot be induced, whence $\bl{2i}\in \Nn(z)$. Now Lemma~\ref{l:319}\ref{it:319b}
	tells us $\Gb\subseteq \Nn(z)$ and it remains to look at 
	Figure~\ref{fig:lem321C}.
\end{proof}
 		 
Recall that an independent 
set $T\subseteq V(\grot_i^{\mu\nu})$ 
is said to be small and if it intersects $\{x,y\}$ and two of the three 
sets $\Gr$, $\Gg$, $\Gb$. Since $T$ cannot intersect all three of them, 
there is a unique colour 
$\varphi \in \{\re{\mathrm{red}}, \gr{\mathrm{green}}, \bl{\mathrm{blue}}\}$
such that $T\cap \G_\varphi = \vn$; we call $\phi$ the {\it colour of $T$}. 
Due to $\sigma$-symmetry it usually
suffices to consider small sets containing $x$. 

Given a Vega graph~$\grot_i^{\mu\nu}$ we write $\fD_4(\grot_i^{\mu\nu})$ for the 
class of graphs in~$\fD_4$ that contain~$\grot_i^{\mu\nu}$ but no larger Vega graph,
i.e., no Vega graph~$\grot^{\mu'\nu'}_{i'}$ with more vertices than~$\grot_i^{\mu\nu}$. 
Eventually we shall show shat 
if $\grot_i^{\mu\nu}\subseteq G\in \fD_4(\grot_i^{\mu\nu})$, then there is 
no $q\in V(G)$ such that $\Nn(q)\cap V(\grot_i^{\mu\nu})$ is small.
In the special case $i=2$ this can often be inferred from earlier results using 
exceptional isomorphisms. 

\begin{lemma}\label{lem:2119}
	If $\grot_2^{\mu\nu}\subseteq G\in \fD_4(\grot_2^{\mu\nu})$, $q\in V(G)$, 
	and $T=\Nn(q)\cap V(\grot_i^{\mu\nu})$ is small, then $\mu=\nu=0$ and $T$
	is \re{red} or \gr{green}.
\end{lemma} 

\begin{proof}
	By $\sigma$-symmetry it suffices to consider the case $x\in T$. 
	Suppose first that $T$ is \bl{blue}, whence $T\supseteq \{x, \re{1},\gr{2}\}$. 
	
	\begin{figure}[h!]
	\centering
	\begin{tikzpicture}[scale=1]
		\coordinate (r) at (2,0);
			\foreach \i in {1,...,8} {
			\coordinate (a\i) at (\i*45:.9);
			\coordinate (c\i) at (180+\i*45:1.8);
		}
		\draw (c3) -- (r) -- (c5);		
		\draw [rounded corners=35] (r) -- (2,2)--(c7);
		\draw [rounded corners=35] (r) -- (2,-2)--(c1);
		\draw (a3) -- (a7);		
		\draw (r)--(c3);
		\draw (a6)--(c3)--(a8)--(a4)--(c7)--(a2)--(a6);		
		\draw (a2)--(c5)--(a8)--(a5)--(a2);
		\draw (a1)--(a4)--(c1)--(a6)--(a1);		
		\draw (a6)--(a3)--(a8)--(a3)--(c7);
		\draw (a2)--(a7)--(a4)--(a7)--(c3);
		\draw (c1)--(a5)--(a1)--(c5)--(r);
		\foreach \i in {1,3,5,7}{
			\fill (c\i) circle (2pt);
		}			
		\foreach \i in {1,...,8} {
			\fill (a\i) circle (2pt);		
		}
		\fill (r) circle (2pt);
		\foreach \i in {r,c3,c7} \fill [blue] (\i) circle (1.7pt);
		\foreach \i in {c5,a5,a2,a8} \fill [green] (\i) circle (1.7pt);
		\foreach \i in {a1,c1,a4,a6} \fill [red] (\i) circle (1.7pt);
		\node at ($(r)+(.25,0)$) {\small $\bl{4}$};
		\node at ($(a1)+(.2,-.1)$) {\small $\re{0}$};
		\node at ($(a2)+(0,.25)$) {$\gr{v}$};
		\node at ($(a3)+(-.2,-.05)$) {$x$};
		\node at ($(a4)+(-.3,0)$) {$\re{u}$};
		\node at ($(a5)+(-.2,.1)$) {\small $\gr{3}$};
		\node at ($(a6)+(0,-.25)$) {$\re{a}$};
		\node at ($(a7)+(.2,.1)$) {$y$};
		\node at ($(a8)+(.3,0)$) {$\gr{b}$};
		\node at ($(c1)+(-.2,-.1)$) {\small $\re{1}$};
		\node at ($(c3)+(.2,-.2)$) {$\bl{w}$};
		\node at ($(c5)+(.2,.1)$) {\small $\gr{2}$};
		\node at ($(c7)+(-.2,.1)$) {$\bl{c}$};
		\draw[ultra thick, -stealth] (2.7,-.3)--(4,-.3);
		\node at (3.35, 0) {\LARGE $\varrho$};
		\coordinate (r) at (7.6,0);
		\foreach \i in {1,...,8} {
			\coordinate (a\i) at ($(\i*45:.9)+(5.6,0)$);
			\coordinate (c\i) at ($(180+\i*45:1.8)+(5.6,0)$);
		}
		\draw (c3) -- (r) -- (c5);		
		\draw [rounded corners=35] (r) -- ($(2,2)+(5.6,0)$)--(c7);
		\draw [rounded corners=35] (r) -- ($(2,-2)+(5.6,0)$)--(c1);
		\draw (a3) -- (a7);		
		\draw (r)--(c3);
		\draw (a6)--(c3)--(a8)--(a4)--(c7)--(a2)--(a6);		
		\draw (a2)--(c5)--(a8)--(a5)--(a2);
		\draw (a1)--(a4)--(c1)--(a6)--(a1);		
		\draw (a6)--(a3)--(a8)--(a3)--(c7);
		\draw (a2)--(a7)--(a4)--(a7)--(c3);
		\draw (c1)--(a5)--(a1)--(c5)--(r);
		\foreach \i in {1,3,5,7}{
			\fill (c\i) circle (2pt);
		}			
		\foreach \i in {1,...,8} {
			\fill (a\i) circle (2pt);		
		}
		\fill (r) circle (2pt);
		\foreach \i in {r,c1,c5} \fill [blue] (\i) circle (1.7pt);
		\foreach \i in {c7,a4,a2,a7} \fill [green] (\i) circle (1.7pt);
		\foreach \i in {a3,c3,a8,a6} \fill [red] (\i) circle (1.7pt);
		\node  at ($(r)+(.25,0)$) {\small $\bl{4}$};
		\node at ($(a1)+(.2,-.1)$) {$x$};
		\node at ($(a2)+(0,.25)$) {$\gr{v}$};
		\node at ($(a3)+(-.2,-.05)$) {\small \re{$0$}};
		\node at ($(a4)+(-.3,0)$) {$\gr{b}$};
		\node at ($(a5)+(-.2,.1)$) {$y$};
		\node at ($(a6)+(0,-.25)$) {$\re{a}$};
		\node at ($(a7)+(.2,.1)$) {\small $\gr{3}$};
		\node at ($(a8)+(.3,0)$) {$\re{u}$};
		\node at ($(c1)+(-.2,-.1)$) {$\bl{w}$};
		\node at ($(c3)+(.2,-.2)$) {\small $\re{1}$};
		\node at ($(c5)+(.2,.1)$) {$\bl{c}$};
		\node at ($(c7)+(-.2,.1)$) {\small $\gr{2}$};
	\end{tikzpicture}
	\caption{The isomorphism $\rho$.}
	\label{fig:1419}
	\end{figure}	 
	
	Recall, that the exceptional isomorphism~$\rho$ introduced in~\S\ref{subsec:auto}
	maps~$\grot_2^{\mu\nu}$ onto $\grot_2^{\nu\mu}$ and the three 
	vertices $x$, $\re{1}$, $\gr{2}$ to $\re{0}$, $\bl{w}$, $\bl{c}$ 
	(see Figure~\ref{fig:1419}). 
	In the sequel such situations will be written as		
	\[
		(\grot_2^{\mu\nu}, \{x,\re{1},\gr{2}\}) 
		\cong 
		(\grot_2^{\nu\mu}, \{\re{0},\bl{w},\bl{c}\}) 
			\qquad (\textrm{via } \rho)\,.
	\]
	Depending on whether $\mu=1$ or $\mu=0$ we now get a contradiction 
	to $G\in \fD_4(\grot_i^{\mu\nu})$ from Corollary~\ref{cor:320} or 
	Lemma~\ref{l:322}.
		 				
	Suppose next that~$T$ is \re{red}, which implies $\{x,\gr{3},\bl{4}\}\subseteq T$ 
	and, therefore, $\nu=0$. If, in addition, $\mu=1$, then
	\begin{align*}
			(\grot_2^{10},\{x,\gr{3},\bl{4}\})
			&\cong (\grot_2^{01}, \{\re{0},y,\bl{4}\}) \qquad (\textrm{via } \rho)\\
			&\cong (\grot_2^{01}, \{\re{1},x,\gr{2}\}) \qquad 
			(\textrm{via } \sigma \circ \tau_1)
	\end{align*}
	reduces the current situation to the \bl{blue} case, which has already been 
	dealt with. Thus we have indeed $\mu=\nu=0$. 
	Finally, the case that $T$ is \gr{green} reduces to the earlier ones 
	by $\tau_\nu$-symmetry.
\end{proof}

By an {\it auxiliary path} 
in $\grot_i^{\mu\nu}$ we mean a path of length three in~$\Gamma_i$ 
whose first and last vertex have the same colour, also called the 
{\it colour of the path}. 
For instance, 
if $j$, $j+1$ are consecutive vertices of the same colour~$\varphi$, then 
\[
	\pi_j=j-(j+i)-(j+2i)-(j+1)
\]
is an auxiliary path whose colour is $\varphi$. 
So every Vega graph $\grot_i^{\mu\nu}$ possesses a \re{red} auxiliary 
path $\pi_{\re{0}}=\re{0}-\gr{i}-\bl{2i}-\re{1}$. If $i\ge 3$ there is for every 
vertex of $\G_i$ an auxiliary path starting in that vertex and, in particular, 
there are auxiliary paths of all colours. Every auxiliary path~$\pi$ 
in~$\grot_i^{\mu\nu}$ gives rise 
to a copy~$\grot(\pi)$ of the Mycielski-Gr\"otzsch graph in~$\grot_i^{\mu\nu}$ 
(see Figure~\ref{fig:331}).

\begin{figure}[ht]
\centering
\begin{subfigure}[b]{0.32\textwidth}
\centering
\begin{tikzpicture}[scale=1]
	\coordinate (c) at (0,0);
	\foreach \i in {1,...,5} {
		\coordinate (a\i) at (-18+\i*72:1);
		\coordinate (b\i) at (18+\i*72:1.5);	
		\coordinate (d\i) at (18+\i*72:2);
	}
	\fill (c) circle (2pt);
	\foreach \i in {1,...,5} {
		\draw (c)--(a\i);
		\fill (a\i) circle (2pt);		
		\fill (b\i) circle (2pt);
	}
	\draw (b1) [out=162, in=72] to (d2) [out=-108, in=162] to (b3);
	\draw (b2) [out=234, in=144] to (d3) [out=-36, in=234] to (b4);
	\draw (b3) [out=-54, in=216] to (d4) [out=36, in=-54] to (b5);
	\draw (b4) [out=18, in=-72] to (d5) [out=108, in=18] to (b1);
	\draw (b5) [thick, out=90, in=0] to (d1) [out=180, in=90] to (b2);
	\draw (b5) [ultra thick, red, dashed, out=90, in=0] to (d1) [out=180, in=90] to (b2);
	\draw [thick] (a2)--(b2);
	\draw [ultra thick, red, dashed] (a2)--(b2);
	\draw [thick] (a1)--(b5);
	\draw [ultra thick,red, dashed] (a1)--(b5);
	\draw (b3)--(a3)--(c)--(a5)--(b4);
	\draw (b3)--(a4)--(b4)--(a4)--(c);
	\draw (b1)--(a1)--(c)--(a2)--(b1);
	\draw (b2)--(a3);
	\draw (b5)--(a5);
	\foreach \i in {a1, a2, c, b1} \fill [red] (\i) circle (1.7pt);
	\foreach \i in {b5, a5, b3} \fill [green] (\i) circle (1.7pt);
	\foreach \i in {b2, a3, b4} \fill [blue] (\i) circle (1.7pt); 
	\node at ($(c)+(-.25,.08)$) {$\re{a}$};
	\node at ($(a3)+(-.2,-.15)$) {$\bl{w}$};
	\node at ($(a4)+(0,-.2)$) {$x$};
	\node at ($(a5)+(.2,-.1)$) {$\gr{v}$};
	\node at ($(b3)+(-.1,-.2)$) {$\gr{b}$};
 	\node at ($(b4)+(.1,-.2)$) {$\bl{c}$};
	\node at ($(b1)+(0,.2)$) {$\re{u}$};
\end{tikzpicture}
\vskip -.2cm
\caption{{a \re{red} path}}
\label{fig:331A} 		
\end{subfigure}
\hfill  
\begin{subfigure}[b]{0.32\textwidth}
\centering
\begin{tikzpicture}[scale=1]
	\coordinate (c) at (0,0);
	\foreach \i in {1,...,5}{
		\coordinate (a\i) at (-18+\i*72:1);
		\coordinate (b\i) at (18+\i*72:1.5);	
		\coordinate (d\i) at (18+\i*72:2);
	}
	\fill (c) circle (2pt);
	\foreach \i in {1,...,5} {
		\draw (c)--(a\i);
		\fill (a\i) circle (2pt);		
		\fill (b\i) circle (2pt);
	}
	\draw (b1) [out=162, in=72] to (d2) [out=-108, in=162] to (b3);
	\draw (b2) [  out=234, in=144] to (d3) [out=-36, in=234] to (b4);
	\draw (b3) [ out=-54, in=216] to (d4) [out=36, in=-54] to (b5);
	\draw (b4) [out=18, in=-72] to (d5) [out=108, in=18] to (b1);
	\draw (b5) [thick, out=90, in=0] to (d1) [out=180, in=90] to (b2);
	\draw (b5) [ultra thick,green,dashed,out=90, in=0] to (d1) [out=180, in=90] to (b2);
	\draw (b3)--(a3)--(c)--(a5)--(b4);
	\draw (b3)--(a4)--(b4)--(a4)--(c);
	\draw (b1)--(a1)--(c)--(a2)--(b1);
	\draw (b2)--(a3);
	\draw (b5)--(a5);
	\draw [thick] (a2)--(b2);
	\draw [ultra thick, green, dashed] (a2)--(b2);
	\draw [thick] (a1)--(b5);
	\draw [ultra thick,green, dashed] (a1)--(b5); 
	\foreach \i in {b5, a5, b3} \fill [red] (\i) circle (1.7pt);
	\foreach \i in {a1, a2, b1, c} \fill [green] (\i) circle (1.7pt);
	\foreach \i in {b2, a3, b4} \fill [blue] (\i) circle (1.7pt); 
	\node at ($(c)+(-.25,.12)$) {$\gr{b}$};
	\node at ($(a3)+(-.2,-.15)$) {$\bl{w}$};
	\node at ($(a4)+(0,-.2)$) {$x$};
	\node at ($(a5)+(.2,-.1)$) {$\re{u}$};
	\node at ($(b3)+(-.1,-.2)$) {$\re{a}$};
	\node at ($(b4)+(.1,-.2)$) {$\bl{c}$};
	\node at ($(b1)+(0,.2)$) {$\gr{v}$};
\end{tikzpicture}
\vskip -.2cm
\caption{{a \gr{green} path}}
\label{fig:331B} 		
\end{subfigure}
\hfill  
\begin{subfigure}[b]{0.32\textwidth}
\centering
\begin{tikzpicture}[scale=1]
	\coordinate (c) at (0,0);
	\foreach \i in {1,...,5} {
		\coordinate (a\i) at (-18+\i*72:1);
		\coordinate (b\i) at (18+\i*72:1.5);	
		\coordinate (d\i) at (18+\i*72:2);
	}
	\fill (c) circle (2pt);
	\foreach \i in {1,...,5} {
		\draw (c)--(a\i);
		\fill (a\i) circle (2pt);		
		\fill (b\i) circle (2pt);
	}
	\draw (b1) [out=162, in=72] to (d2) [out=-108, in=162] to (b3);
	\draw (b2) [out=234, in=144] to (d3) [out=-36, in=234] to (b4);
	\draw (b3) [out=-54, in=216] to (d4) [out=36, in=-54] to (b5);
	\draw (b4) [out=18, in=-72] to (d5) [out=108, in=18] to (b1);
	\draw (b5) [thick, out=90, in=0] to (d1) [out=180, in=90] to (b2);
	\draw (b5) [ultra thick,blue,dashed,out=90, in=0] to (d1) [out=180, in=90] to (b2);
	\draw [thick] (a2)--(b2);
	\draw [ultra thick, blue, dashed] (a2)--(b2);
	\draw [thick] (a1)--(b5);
	\draw [ultra thick,blue, dashed] (a1)--(b5);
	\draw (b3)--(a3)--(c)--(a5)--(b4);
	\draw (b3)--(a4)--(b4)--(a4)--(c);
	\draw (b1)--(a1)--(c)--(a2)--(b1);
	\draw (b2)--(a3);
	\draw (b5)--(a5);
	\foreach \i in {b5,b3,a5} \fill [red] (\i) circle (1.7pt);
	\foreach \i in {a2, a1, c, b1} \fill [blue] (\i) circle (1.7pt);
	\foreach \i in {b2, a3,b4} \fill [green] (\i) circle (1.7pt); 
	\node at ($(c)+(-.25,.08)$) {$\bl{c}$};
	\node at ($(a3)+(-.15,-.15)$) {$\gr{v}$};
	\node at ($(a4)+(0,-.2)$) {$x$};
	\node at ($(a5)+(.2,-.1)$) {$\re{u}$};
	\node at ($(b1)+(-.06,.2)$) {$\bl{w}$};
	\node at ($(b3)+(-.1,-.2)$) {$\re{a}$};
	\node at ($(b4)+(.15,-.2)$) {$\gr{b}$};
\end{tikzpicture}
\vskip -.2cm
\caption{{a \bl{blue} path}}
\label{fig:331C} 		
\end{subfigure}
\vskip -.2cm
\caption{A copy $\grot(\pi)$ of the Mycielski-Gr\"otzsch graph in $\grot_i^{\mu\nu}$.}
\label{fig:331}
\end{figure}
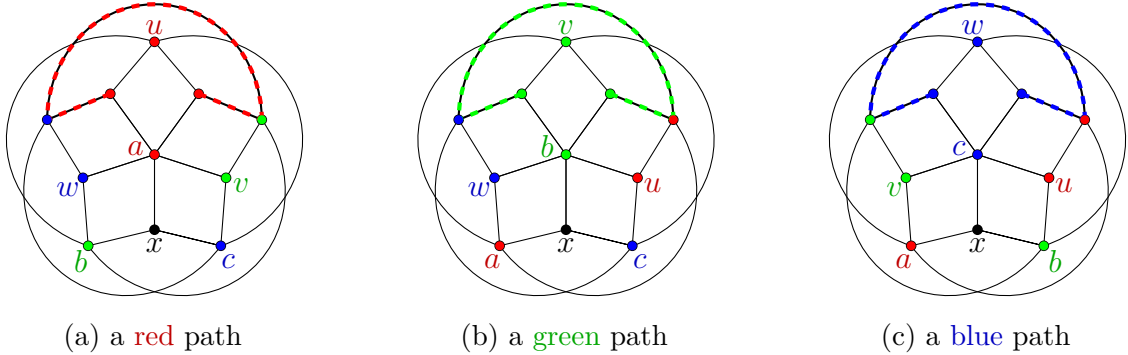
 
One can learn a lot by applying the results from the previous subsection to graphs
of the form~$\grot(\pi)$. 

\begin{lemma}\label{lem:2157}
	If $\grot_i^{\mu\nu}\subseteq G\in \fD_4(\grot_i^{\mu\nu})$
	and $q\in V(G)$, then $\Nn(q)\cap V(\grot_i^{\mu\nu})$ is not small.
\end{lemma}

\begin{proof}
	By $\sigma$- and $\tau_\nu$-symmetry it suffices to show that if $\Nn(q)$
	contains $x$ and a vertex from~$\Gg$, then it is disjoint to $\Gr\cup\Gb$.
	\goodbreak
	
	\smallskip
	
	{\it \hskip 2em First case: $\Nn(q)\cap \gr{[i, 2i-2]}\ne\vn$.}
	
	\smallskip
	
	In the special case $i=2$ this means $\gr{2}\in \Nn(q)$. 
	Since $G$ is triangle-free, $q$ cannot have a neighbour in $\Gb$ and by 
	Lemma~\ref{lem:2119} there are no neighbours of $q$ in $\Gr$ either. 
	Suppose next that $i\ge 3$, so that the interval $\gr{[i, 2i-2]}$ consists of more 
	than one vertex. Thus there exists a pair of consecutive vertices 
	$\gr{\{m, m+1\}}\subseteq \gr{[i, 2i-2]}$ such 
	that $\Nn(q)\cap \gr{\{m, m+1\}}\ne\vn$. Working with the auxiliary path
	\[
		\pi_{\gr{m}} 
		= 
		\gr{m}-\bl{(m+i)}-\re{(m-i+1)}-\gr{(m+1)}
	\]
	we construct the graph~$\grot(\pi_{\gr{m}})$ (see Figure~\ref{il:323A}).
	The \bl{blue} vertex~$\bl{m+1+i}$ is adjacent to $\gr{m+1}$, $\bl{c}$, $\bl{w}$ 
	and, therefore,~$\gr{m}$ is unreliable in~$\grot(\pi_{\gr{m}})$. 
	Similarly, the \re{red} vertex~$\re{m-i}$ exemplifies the unreliability 
	of~$\gr{m+1}$. Now Lemma~\ref{l:315} tells us that~$\re{u}$ and~$\bl{w}$ are 
	reliable. In particular, $q$ is neither in $\Ext(\grot(\pi_{\gr{m}}), \re{u})$
	nor in $\Ext(\grot(\pi_{\gr{m}}), \bl{w})$ and Lemma~\ref{l:312} implies 
	that~$q$ is an $\grot(\pi_{\gr{m}})$-twin of $\gr{b}$. Thus $q$ is adjacent 
	to~$\re{u}$,~$\bl{w}$ and, consequently, to no vertex in $\Gr\cup\Gb$.
	
\begin{figure}[h!]
\begin{subfigure}[b]{.32\textwidth}
\centering
\begin{tikzpicture}[scale=.9]
	\coordinate (c) at (0,0);
	\foreach \i in {1,...,5} {
		\coordinate (a\i) at (-18+\i*72:1);
		\coordinate (b\i) at (18+\i*72:1.5);	
		\coordinate (d\i) at (18+\i*72:2);
	}
	\fill (c) circle (2pt);
	\foreach \i in {1,...,5} {
		\draw (c)--(a\i);
		\fill (a\i) circle (2pt);		
		\fill (b\i) circle (2pt);
	}
	\draw (b1) [out=162, in=72] to (d2) [out=-108, in=162] to (b3);
	\draw (b2) [out=234, in=144] to (d3) [out=-36, in=234] to (b4);
	\draw (b3) [out=-54, in=216] to (d4) [out=36, in=-54] to (b5);
	\draw (b4) [out=18, in=-72] to (d5) [out=108, in=18] to (b1);
	\draw (b5) [thick, out=90, in=0] to (d1) [out=180, in=90] to (b2);
	\draw (b5) [ultra thick,green,dashed,out=90,in=0] to (d1) [out=180, in=90] to 	(b2);
	\draw (b3)--(a3)--(c)--(a5)--(b4);
	\draw (b3)--(a4)--(b4)--(a4)--(c);
	\draw (b1)--(a1)--(c)--(a2)--(b1);
	\draw (b2)--(a3);
	\draw (b5)--(a5);
	\draw [thick] (a2)--(b2);
	\draw [ultra thick, green, dashed] (a2)--(b2);
	\draw [thick] (a1)--(b5);
	\draw [ultra thick,green, dashed] (a1)--(b5); 
	\foreach \i in {b5, a5, b3} \fill [red] (\i) circle (1.7pt);
	\foreach \i in {a1, a2, b1, c} \fill [green] (\i) circle (1.7pt);
	\foreach \i in {b2, a3, b4} \fill [blue] (\i) circle (1.7pt); 
	\node at ($(c)+(-.25,.12)$) {\tiny $\gr{b}$};
	\node at ($(a3)+(-.2,-.15)$) {\tiny $\bl{w}$};
	\node at ($(a4)+(0,-.2)$) {\tiny $x$};
	\node at ($(a5)+(.2,-.1)$) {\tiny $\re{u}$};
	\node at ($(b3)+(-.1,-.15)$) {\tiny $\re{a}$};
	\node at ($(b4)+(.1,-.15)$) {\tiny $\bl{c}$};
	\node at ($(b1)+(0,.2)$) {\tiny $\gr{v}$};
	\node at ($(b2)+(-.48,0)$) {\tiny $\bl{m+i}$};
	\node at ($(b5)+(.85,-.0)$) {\tiny $\re{m-i+1}$};
	\node [scale=.9] at ($(a1) +(.32,.16)$) {\tiny $\gr{m+1}$};
	\node at ($(a2)+(-.15,.15)$){\tiny $\gr{m}$};
\end{tikzpicture} 	
\caption{$\grot(\pi_{\gr{m}})$}
\label{il:323A}
\end{subfigure}
\hfill
\begin{subfigure}[b]{.32\textwidth}
\centering
\begin{tikzpicture}[scale=.9]
	\coordinate (c) at (0,0);
	\foreach \i in {1,...,5} {
		\coordinate (a\i) at (-18+\i*72:1);
		\coordinate (b\i) at (18+\i*72:1.5);	
		\coordinate (d\i) at (18+\i*72:2);
	}
	\fill (c) circle (2pt);
	\foreach \i in {1,...,5} {
		\draw (c)--(a\i);
		\fill (a\i) circle (2pt);		
		\fill (b\i) circle (2pt);
	}
	\draw (b1) [out=162, in=72] to (d2) [out=-108, in=162] to (b3);
	\draw (b2) [out=234, in=144] to (d3) [out=-36, in=234] to (b4);
	\draw (b3) [out=-54, in=216] to (d4) [out=36, in=-54] to (b5);
	\draw (b4) [out=18, in=-72] to (d5) [out=108, in=18] to (b1);
	\draw (b5) [thick, out=90, in=0] to (d1) [out=180, in=90] to (b2);
	\draw (b5) [ultra thick,green,dashed,out=90,in=0] to (d1) [out=180, in=90] to 	(b2);	
	\draw (b3)--(a3)--(c)--(a5)--(b4);
	\draw (b3)--(a4)--(b4)--(a4)--(c);
	\draw (b1)--(a1)--(c)--(a2)--(b1);
	\draw (b2)--(a3);
	\draw (b5)--(a5);
	\draw [thick] (a2)--(b2);
	\draw [ultra thick, green, dashed] (a2)--(b2);
	\draw [thick] (a1)--(b5);
	\draw [ultra thick,green, dashed] (a1)--(b5); 
	\fill (a4) circle (1.7pt);
	\foreach \i in {b5, a5, b3} \fill [red] (\i) circle (1.7pt);
	\foreach \i in {a1, a2, b1, c} \fill [green] (\i) circle (1.7pt);
	\foreach \i in {b2, a3, b4} \fill [blue] (\i) circle (1.7pt); 
	\node at ($(c)+(-.25,.12)$) {\tiny $\gr{b}$};
	\node at ($(a3)+(-.2,-.15)$) {\tiny $\bl{w}$};
	\node at ($(a4)+(0,-.2)$) {\tiny $x$};
	\node at ($(a5)+(.2,-.1)$) {\tiny $\re{u}$};
	\node at ($(b3)+(-.1,-.15)$) {\tiny $\re{a}$};
	\node at ($(b4)+(.15,-.15)$) {\tiny $\bl{c}$};
	\node at ($(b1)+(0,.2)$) {\tiny $\gr{v}$};
	\node at ($(b2)+(-.24,0)$) {\tiny $\bl{m}$};
	\node at ($(b5)+(.42,0)$) {\tiny $\re{i-1}$};
	\node [scale = .9] at ($(a1) +(.31,.18)$) {\tiny $\gr{2i-1}$};
	\node at ($(a2)+(-.12,.2)$){\tiny $\gr{i}$};
\end{tikzpicture} 
\caption{$\grot(\omega)$}
\label{il:324C}
\end{subfigure}
\hfill
\begin{subfigure}[b]{.32\textwidth}
\centering
\begin{tikzpicture}[scale=.9]
	\coordinate (c) at (0,0);
	\foreach \i in {1,...,5} {
		\coordinate (a\i) at (-18+\i*72:1);
		\coordinate (b\i) at (18+\i*72:1.5);	
		\coordinate (d\i) at (18+\i*72:2);
	}
	\fill (c) circle (2pt);
	\foreach \i in {1,...,5} {
		\draw (c)--(a\i);
		\fill (a\i) circle (2pt);		
		\fill (b\i) circle (2pt);
	}
	\draw (b1) [out=162, in=72] to (d2) [out=-108, in=162] to (b3);
	\draw (b2) [out=234, in=144] to (d3) [out=-36, in=234] to (b4);
	\draw (b3) [out=-54, in=216] to (d4) [out=36, in=-54] to (b5);
	\draw (b4) [out=18, in=-72] to (d5) [out=108, in=18] to (b1);
	\draw (b5) [thick, out=90, in=0] to (d1) [out=180, in=90] to (b2);
	\draw (b5) [ultra thick,green,dashed,out=90,in=0] to (d1) [out=180, in=90] to 	(b2);	
	\draw (b3)--(a3)--(c)--(a5)--(b4);
	\draw (b3)--(a4)--(b4)--(a4)--(c);
	\draw (b1)--(a1)--(c)--(a2)--(b1);
	\draw (b2)--(a3);
	\draw (b5)--(a5);
	\draw [thick] (a2)--(b2);
	\draw [ultra thick, green, dashed] (a2)--(b2);
	\draw [thick] (a1)--(b5);
	\draw [ultra thick,green, dashed] (a1)--(b5); 
	\fill (a4) circle (1.7pt);
	\foreach \i in {b5, a5, b3} \fill [red] (\i) circle (1.7pt);
	\foreach \i in {a1, a2, b1, c} \fill [green] (\i) circle (1.7pt);
	\foreach \i in {b2, a3, b4} \fill [blue] (\i) circle (1.7pt); 
	\node at ($(c)+(-.25,.12)$) {\tiny $\gr{b}$};
	\node at ($(a3)+(-.2,-.15)$) {\tiny $\bl{w}$};
	\node at ($(a4)+(0,-.2)$) {\tiny $x$};
	\node at ($(a5)+(.2,-.1)$) {\tiny $\re{u}$};
	\node at ($(b3)+(-.1,-.15)$) {\tiny $\re{a}$};
	\node at ($(b4)+(.15,-.15)$) {\tiny $\bl{c}$};
	\node at ($(b1)+(0,.2)$) {\tiny $\gr{v}$};
	\node at ($(b2)+(-.24,0)$) {\tiny $\bl{2i}$};
	\node at ($(b5)+(.17,0)$) {\tiny $\re{1}$};
		\node [scale = .9] at ($(a1) +(.31,.18)$) {\tiny $\gr{2i-1}$};
	\node at ($(a2)+(-.12,.13)$){\tiny $\gr{i}$};
\end{tikzpicture} 
\caption{$\grot(\varpi)$}
\label{il:324CC}
\end{subfigure}
\caption{The proof of Lemma~\ref{lem:2157}.}
\label{il:323}
\end{figure} 

{\it \hskip 2em Second case: We have $\nu=0$ and $\gr{2i-1}\in \Nn(q)$.}
	
	\smallskip
	
	Due to $\Gr\subseteq \Nn(\gr{2i-1})$ we only need to derive a contradiction 
	from the assumption that there exists some blue $\bl{m}\in\Nn(q)\cap \Gb$.
	To this end we consider the \gr{green} auxiliary path 
	$\omega=\gr{i}-\bl{m}-\re{(i-1)}-\gr{(2i-1)}$
	(see Figure~\ref{il:324C}) and the associated graph $\grot(\omega)$.
	The vertices~$\re{0}$ and~$q$ witness that~$\gr{2i-1}$ and~$\re{u}$
	are unreliable; so $x$ is reliable by Lemma~\ref{l:315} or, in other 
	words, $\mu=1$. Now Lemma~\ref{lem:2119} tells us that $i\ge 3$.
	
	Next we apply Lemma~\ref{l:316} for the green auxiliary path 
	$\varpi=\gr{i}-\bl{2i}-\re{1}-\gr{(2i-1)}$ to $\grot(\varpi)$ with 
	the labelling $a_0=\re{u}$, $a_1=\gr{2i-1}$, $a_2=\gr{i}$, $b_1=\re{a}$ 
	and the additional vertices $a_1'=\gr{i+1}$, $t_2'=\bl{2i+1}$. 
	As~$\re{2}$ is a common neighbour of $\gr{2i-1}$, $\bl{2i+1}$, $\re{a}$, 
	we infer that~$\re{u}$ is reliable. Together with $x, \gr{2i-1}\in\Nn(q)$
	and Lemma~\ref{l:312} this entails that $q$ is an $\grot(\varpi)$-twin
	of~$\gr{b}$. But now $q\bl{wm}$ is a triangle in $G$, which is absurd.
\end{proof}

\begin{cor}\label{cor:326}
	If $\grot_i^{\mu\nu}\subseteq G \in \fD_4(\grot_i^{\mu\nu})$ and $\pi$ denotes 
	an auxiliary path in~$\grot_i^{\mu\nu}$ with colour~$\varphi$, then those 
	among $\re{u},\gr{v},\bl{w}$ whose colour is not $\varphi$ are reliable 
	in~$\grot(\pi)$.
\end{cor}

\begin{proof}
	Otherwise there existed a vertex $q\in V(G)$ for 
	which $\Nn(q)\cap V(\grot_i^{\mu\nu})$ is small 
	(see Figure~\ref{fig:331}), contrary to Lemma~\ref{lem:2157}.
\end{proof}

Summarising the work of this subsection, we can now establish a weak form of the 
attachment lemma with an inclusion as opposed to an equality.

\begin{lemma}\label{l:325}
	If $\grot_i^{\mu\nu}\subseteq G\in \fD_4(\grot_i^{\mu\nu})$, then for 
	every $q\in V(G)$ there is some $z\in V(\grot_i^{\mu\nu})$ such 
	that $\Nn(q)\cap V(\grot_i^{\mu\nu})\subseteq \Nn(z)\cap V(\grot_i^{\mu\nu})$.
\end{lemma}

\begin{proof}
	If no such vertex $z$ exists, then $T=\Nn(q)\cap V(\grot_i^{\mu\nu})$ satisfies 
	one of the five statements~\ref{it:b}\,--\,\ref{it:f} in Lemma~\ref{l:317}. 
	In case~\ref{it:b} Lemma~\ref{l:318}\ref{it:318a} 
	yields $T=\{\re{u},\gr{v},\bl{w}, x\}$, and $q$ can play the r\^{o}le of $y$ 
	in a copy of $\grot_i^{0\nu}$ in $G$, which is absurd. 
	Similarly, Corollary~\ref{cor:320}, Lemma~\ref{l:322}, and Lemma~\ref{lem:2157} 
	exclude the remaining cases.
\end{proof}

\subsection{The twin lemma}
The goal of this subsection is to prove the twin lemma for graphs 
in $\fD_4(\grot_i^{\mu\nu})$, which simply asserts that all these 
graphs have the $\grot_i^{\mu\nu}$-twin property. Since Vega graphs
have several different types of edges, the argument involves a case 
analysis. We begin with some edges, for which the twin property can 
be derived from the cube lemma alone. 
 
Edges of a Vega graph $\grot_i^{\mu\nu}$ that connect two vertices 
of $\Gamma_i$ are called {\it Andr\'asfai edges}. Such an edge $jm$ 
is said to be {\it long} if $j-m\neq \pm i$. Moreover, for $\nu=0$ 
the edge $\re{0}\gr{(2i-1)}$ is considered to be long as well. All 
other Andr\'asfai edges are {\it short}. As in the previous section, 
long edges are easier to handle than short ones.

\begin{lemma}\label{l:327}
	If $e$ denotes a long Andr\'asfai edge of a Vega graph $\grot_i^{\mu\nu}$, then 
	every graph $G\in \fD_4$ has the $(\grot_i^{\mu\nu}, e)$-twin property. 
	Moreover, if $\mu=0$ the same holds for $e=xy$.
\end{lemma}

\begin{proof}
	For the last statement we consider any $\grot_i^{0\nu}$-twins 
	$x'$ and $y'$ of $x$ and $y$, respectively. Since the cube
	\begin{center}
		\begin{tabular}{cc}
			$(Q)$&
		\begin{tabular}{c|c|c|c}
			$x'$&  $\re{u}$&  $\gr{v}$ & $\bl{w}$
			\\ \hline
			$y'$&  $\re{a}$&  $\gr{b}$ & $\bl{c}$
		\end{tabular}
			\end{tabular}
	\end{center}	
	cannot be induced, $x'y'$ is indeed an edge of $G$.

	Now let $e=jm$ be a long Andr\'asfai edge. Suppose first that $\re{j}$ is \re{red} 
	and $\gr{m}$ is \gr{green}. Due to $i<\gr{m}-\re{j}\le 2i-1$ 
	we have $\re{j}\neq\re{i-1}$ and, hence, there is a \gr{green} vertex~$\gr{j+i}$.
	Further\-more, $\re{m-i}$ is \re{red} and $\gr{(j+i)}-\re{(m-i)}=2i-(m-j)\in [i-1]$
	shows $\re{(m-i)}\gr{(j+i)}\not\in E(G)$. So if $j'$, $m'$ 
	are $\grot_i^{\mu\nu}$-twins of $\re{j}$, $\gr{m}$, then the cube
	
	\begin{center}
		\begin{tabular}{cc}
			$(Q)$&
		\begin{tabular}{c|c|c|c}
			$j'$&  $\gr{b}$&  $\gr{v}$ & $\re{m-i}$
			\\ \hline
			$m'$&  $\re{a}$&  $\re{u}$ & $\gr{j+i}$
		\end{tabular}
			\end{tabular}
	\end{center}	
	leads to the desired edge $j'm'\in E(G)$. 

 	Similarly, if $\gr{j}$ is \gr{green} and $\bl{m}$ is \bl{blue}, 
	then $\gr{j}\ne\gr{2i-1}$ and $\gr{(m-i)}\bl{(j+i)}\not\in E(G)$, 
	so that we can work with the cube
 	\begin{center}
 	\begin{tabular}{cc}
 		$(Q)$&
 	\begin{tabular}{c|c|c|c}
 		$j'$&  $\bl{c}$&  $\bl{w}$ & $\gr{m-i}$
 		\\ \hline
 		$m'$&  $\gr{b}$&  $\gr{v}$ & $\bl{j+i}$
 	\end{tabular}
 		\end{tabular}.
 \end{center}	

	The remaining case, where $e$ connects a \re{red} and a \bl{blue} vertex, 
	reduces to one of the previous two by $\tau_\nu$-symmetry.
\end{proof}

We proceed with short Andr\'asfai edges. 

\begin{lemma}\label{lem:short}
	If $e$ denotes a short Andr\'asfai edge of a Vega graph~$\grot_i^{\mu\nu}$,
	then every graph $G\in \fD_4(\grot_i^{\mu\nu})$ has the $(\grot_i^{\mu\nu}, e)$-twin 
	property.
\end{lemma}

\begin{proof}
	Suppose first that $e$ connects a \re{red} vertex with a \gr{green} vertex. 
	Recalling that the edge $\re{0}\gr{(2i-1)}$ is regarded as being long, we can 
	write~$e=\re{j}\gr{(j+i)}$ for some $\re{j}\in \re{[0,i-1]}$. We may further 
	assume $\re{j}\ne \re{i-1}$, because if the edge $e=\re{(i-1)}\gr{(2i-1)}$ exists,
	then~$\tau_0$ reflects it to~$\re{0}\gr{i}$, which corresponds to~$\re{j}=\re{0}$. 
	
	Working with the \re{red} auxiliary 
	path $\pi_{\re{j}}=\re{j}-\gr{(j+i)}-\bl{(j+2i)}-\re{(j+1)}$ we form the 
	graph $\grot(\pi_{\re{j}})$ (see Figure~\ref{il:332A}). If there are 
	non-adjacent $\grot_i^{\mu\nu}$-twins $j'$, $(j+i)'$ 
	of $\re{j}$, $\gr{j+i}$, then Lemma~\ref{l:38} applied to $\grot(\pi_{\re{j}})$ 
	yields a common neighbour $q$ of $j'$, $(j+i)'$, $\bl{c}$, $\bl{w}$ 
	(see Figure~\ref{fig:1244b}). 
	We colour \re{$j'$ red}, \gr{$(j+i)'$ green}, \bl{$q$ blue}, and construct a 
	copy of $\grot_{i+1}^{\mu\nu}$ whose Andr\'asfai part is shown in 
	Figure~\ref{il:332C}, thereby obtaining a contradiction 
	to $G\in \fD_4(\grot_i^{\mu\nu})$. The edges from~$\bl{q}$ 
	to $[\re{j+1}, \gr{j+i-1}]$ required here exist by Lemma \ref{l:321} applied 
	to~$\grot_i^{\mu\nu}(\re{j'})$ and~$\grot_i^{\mu\nu}(\gr{(j+i)'})$.
					
	\begin{figure}[h!]
	\begin{subfigure}[b]{.26\textwidth}
	\centering
	\begin{tikzpicture}[scale=.9]
		\coordinate (c) at (0,0);
		\foreach \i in {1,...,5} {
			\coordinate (a\i) at (-18+\i*72:1);
			\coordinate (b\i) at (18+\i*72:1.5);	
			\coordinate (d\i) at (18+\i*72:2);
		}
		\fill (c) circle (2pt);
		\foreach \i in {1,...,5} {
			\draw (c)--(a\i);
			\fill (a\i) circle (2pt);		
			\fill (b\i) circle (2pt);
		}
		\draw (b1) [out=162, in=72] to (d2) [out=-108, in=162] to (b3);
		\draw (b2) [out=234, in=144] to (d3) [out=-36, in=234] to (b4);
		\draw (b3) [out=-54, in=216] to (d4) [out=36, in=-54] to (b5);
		\draw (b4) [out=18, in=-72] to (d5) [out=108, in=18] to (b1);
	 	\draw (b5) [thick, out=90, in=0] to (d1) [out=180, in=90] to (b2);
	 	\draw (b5) [ultra thick,red,dashed,out=90,in=0] to (d1) [out=180, in=90] to (b2);
	 	\draw [thick] (a2)--(b2);
	 	\draw [ultra thick, red, dashed] (a2)--(b2);
	 	\draw [thick] (a1)--(b5);
	 	\draw [ultra thick,red, dashed] (a1)--(b5);
	 	\draw (b3)--(a3)--(c)--(a5)--(b4);
	 	\draw (b3)--(a4)--(b4)--(a4)--(c);
	 	\draw (b1)--(a1)--(c)--(a2)--(b1);
	 	\draw (b2)--(a3);
	 	\draw (b5)--(a5);
	 	\fill (a4) circle (1.7pt);
	 	\foreach \i in {a1, a2, c, b1} \fill [red] (\i) circle (1.7pt);
	 	\foreach \i in {b5, a5, b3} \fill [green] (\i) circle (1.7pt);
	 	\foreach \i in {b2, a3, b4} \fill [blue] (\i) circle (1.7pt); 
	 	\node at ($(c)+(-.2,.07)$) {\tiny{$\re{a}$}};
	 	\node at ($(a3)+(-.2,-.15)$) {\tiny{$\bl{w}$}};
	 	\node at ($(a4)+(0,-.2)$) {\tiny{$x$}};
	 	\node at ($(a5)+(.2,-.1)$) {\tiny{$\gr{v}$}};
	 	\node at ($(b3)+(-.1,-.2)$) {\tiny{$\gr{b}$}};
	 	\node at ($(b4)+(.15,-.15)$) {\tiny{$\bl{c}$}};
	 	\node at ($(b1)+(0,.18)$) {\tiny{$\re{u}$}};
	 	\node [scale = .9] at ($(b2)+(.55,-.07)$) {\tiny{$\bl{j+2i}$}};
	 	\node [scale = .9]  at ($(b5)+(-.45,-.13)$) {\tiny{$\gr{j+i}$}};
	 	\node at ($(a1) +(.18,.15)$) {\tiny{$\re{j}$}};
	 	\node [scale = .9] at ($(a2)+(-.25,.15)$){\tiny{$\re{j+1}$}};
	\end{tikzpicture}
	\caption{$\grot(\pi_{\re{j}})$}
	\label{il:332A}
	\end{subfigure}
	\hfill
	\begin{subfigure}[b]{.36\textwidth}
	\centering
	\begin{tikzpicture}[scale=.8]
		\coordinate (v) at (30:4cm);
		\coordinate (a) at (90:4cm);
		\coordinate (w) at (150:4cm);
		\coordinate (b) at (210:4cm);
		\coordinate (u) at (270:4cm);
		\coordinate (c) at (330:4cm);
		\coordinate (x) at (7,3);
		\coordinate (y) at (7,-3);
		\foreach \i in {1,...,36}{
			\coordinate (v\i) at (\i*10:2cm);
		}
		\draw [red!75!black, thick,domain=-20:100] plot ({2*cos(\x)},{2*sin(\x)}); 
		\draw [blue, thick,domain=110:220] plot ({2*cos(\x)},{2*sin(\x)}); 
		\draw [green!75!black, thick,domain=230:330] plot ({2*cos(\x)},{2*sin(\x)}); 
		\draw [red!75!black, ,domain=65:88, dashed] plot ({2.3*cos(\x)},{2.3*sin(\x)}); 
		\draw [blue, dashed, domain=117:138] plot ({2.3*cos(\x)},{2.3*sin(\x)}); 
		\draw [green!75!black, dashed,domain=305:328] plot ({2.3*cos(\x)},{2.3*sin(\x)}); 
		\draw [red!75!black,domain=-12:10, dashed] plot ({2.3*cos(\x)},{2.3*sin(\x)}); 
		\draw [blue, dashed, domain=185:208] plot ({2.3*cos(\x)},{2.3*sin(\x)}); 
		\draw [green!75!black, dashed,domain=235:247] plot ({2.3*cos(\x)},{2.3*sin(\x)}); 
		\draw [thick] (v4)--(v16)--(v28);
		\foreach \i in {10, 2,4,6,34}{
			\draw[red!75!black, very thick]  (v\i) circle (2pt);
			\fill[red!75!white]  (v\i) circle (2pt);
		}
		\foreach \i in {23,33,26,28,30}{
			\draw[green!75!black, very thick]  (v\i) circle (2pt);
			\fill[green!75!white]  (v\i) circle (2pt);
		}
		\foreach \i in {11,22, 14,16,18}{
			\draw[blue!75!black, very thick]  (v\i) circle (2pt);
			\fill[blue!75!white]  (v\i) circle (2pt);
		}
		\node [scale=.9,blue] at ($(v11)+(-.4,.25)$) {\tiny $3i-2$};
		\node [scale=.9,blue] at ($(v14)+(-.7,0)$) {\tiny $j+2i$};
		\node [scale=.9,blue] at ($(v16)+(-.3,0)$) {\tiny $q$};
		\node [scale=.9,blue] at ($(v18)+(-.9,0)$) {\tiny $j+2i-1$};
		\node [scale=.9,blue] at ($(v22)+(-.3,0)$) {\tiny $2i$};
		\node [scale=.9, green!60!black] at ($(v23)+(-.6,-.2)$) {\tiny $2i-1-\nu$};
		\node [scale=.9,green!60!black] at ($(v30)+(.6,-.3)$) {\tiny $i+j-1$};
		\node [scale=.9,green!60!black] at ($(v28)+(0,-.3)$) {\tiny $(i+j)'$};
		\node [scale=.9,green!60!black] at ($(v26)+(-.4,-.3)$) {\tiny $i+j$};
		\node [scale=.9,green!60!black] at ($(v33)+(.3,-.1)$) {\tiny $i$};
		\node [scale=.9,red!70!black] at ($(v10)+(.1,.3)$) {\tiny $0$};
		\node [scale=.9,red!70!black] at ($(v6)+(.3,.1)$) {\tiny $j$};
		\node [scale=.9,red!70!black] at ($(v2)+(.6,0)$) {\tiny $j+1$};
		\node [scale=.9,red!70!black] at ($(v4)+(.3,.05)$) {\tiny $j'$};
		\node [scale=.9,red!70!black] at ($(v34)+(.6,0)$) {\tiny $i-1$};
	\end{tikzpicture} 
	\caption{}
	\label{il:332C}
	\end{subfigure}
	\hfill
	\begin{subfigure}[b]{.36\textwidth}
	\centering
	\begin{tikzpicture}[scale=.8]
		\foreach \i in {1,...,36}{
			\coordinate (v\i) at (\i*10:2cm);
		}
		\draw [red!75!black, thick,domain=-20:100] plot ({2*cos(\x)},{2*sin(\x)}); 
		\draw [blue, thick,domain=110:220] plot ({2*cos(\x)},{2*sin(\x)}); 
		\draw [green!75!black, thick,domain=230:330] plot ({2*cos(\x)},{2*sin(\x)}); 
		\draw [red!75!black, ,domain=65:90, dashed] plot ({2.3*cos(\x)},{2.3*sin(\x)}); 
		\draw [blue, dashed, domain=117:138] plot ({2.3*cos(\x)},{2.3*sin(\x)}); 
		\draw [green!75!black, dashed,domain=305:328] plot ({2.3*cos(\x)},{2.3*sin(\x)}); 
		\draw [red!75!black ,domain=-12:13, dashed] plot ({2.3*cos(\x)},{2.3*sin(\x)}); 
		\draw [blue, dashed, domain=185:208] plot ({2.3*cos(\x)},{2.3*sin(\x)}); 
		\draw [green!75!black, dashed,domain=235:253] plot ({2.3*cos(\x)},{2.3*sin(\x)}); 
		\draw [thick] (v16)--(v4)--(v28);
		\foreach \i in {10, 2,4,6,34}{
			\draw[red!75!black, very thick]  (v\i) circle (2pt);
			\fill[red!75!white]  (v\i) circle (2pt);
		}
		\foreach \i in {23,33,26,28,30}{
			\draw[green!75!black, very thick]  (v\i) circle (2pt);
			\fill[green!75!white]  (v\i) circle (2pt);
		}
		\foreach \i in {11,22, 14,16,18}{
			\draw[blue!75!black, very thick]  (v\i) circle (2pt);
			\fill[blue!75!white]  (v\i) circle (2pt);
		}
		\node [scale=.9,blue] at ($(v11)+(-.4,.25)$) {\tiny{$3i-2$}};
		\node [scale=.9,blue] at ($(v14)+(-.7,0)$) {\tiny{$j+2i$}};
		\node [scale=.9,blue] at ($(v16)+(-.8,0)$) {\tiny{$(j+2i)'$}};
		\node [scale=.9,blue] at ($(v18)+(-.95,0)$) {\tiny{$j+2i-1$}};
		\node [scale=.9,blue] at ($(v22)+(-.3,0)$) {\tiny{$2i$}};
		\node [scale=.9,green!60!black] at ($(v23)+(-.6,-.2)$) {\tiny{$2i-1-\nu$}};
		\node [scale=.9,green!60!black] at ($(v30)+(.16,-.25)$) {\tiny{$j+i$}};
		\node [scale=.9,green!60!black] at ($(v28)+(.15,-.4)$) {\tiny{$(j+i)'$}};
		\node [scale=.9,green!60!black] at ($(v26)+(-0.6,-.35)$) {\tiny{$j+i+1$}};
		\node [scale=.9,green!60!black] at ($(v33)+(.3,-.1)$) {\tiny{$i$}};
		\node [scale=.9,red!70!black] at ($(v10)+(.1,.3)$) {\tiny{$0$}};
		\node [scale=.9,red!70!black] at ($(v6)+(.25,.25)$) {\tiny{$j$}};
		\node [scale=.9,red!70!black] at ($(v2)+(.6,0)$) {\tiny{$j+1$}};
		\node [scale=.9,red!70!black] at ($(v4)+(.2,.1)$) {\tiny{$q$}};
		\node [scale=.9,red!70!black] at ($(v34)+(.5,0)$) {\tiny{$i-1$}};
	\end{tikzpicture} 
	\caption{}
	\label{fig:332C}
	\end{subfigure}
	\caption{The proof of Lemma~\ref{lem:short}.}
	\label{fig:332}
	\end{figure} 						
	
	Next we consider the case that $e$ connects a \gr{green} vertex with a \bl{blue} 
	vertex and write $e=\gr{(j+i)}\bl{(j+2i)}$, where $\re{j}\in \re{[0,i-2]}$. 
	By Lemma~\ref{l:39} applied to $\grot(\pi_{\re{j}})$ there is a common 
	neighbour $q$ of $(j+i)'$, $(j+2i)'$, $\re{a}$, $\re{u}$ 
	(see Figures~\ref{fig:1244c} and~\ref{il:332A}). 
	We colour \gr{$(j+i)'$ green}, \bl{$(j+2i)'$ blue}, \re{$q$ red} and construct 
	a copy of $\grot_{i+1}^{\mu\nu}$ whose Andr\'asfai part is shown in 
	Figure~\ref{fig:332C}. The required edges from $\re{q}$ 
	to $[\gr{j+i+1}, \bl{j+2i-1}]$ are obtained from Lemma~\ref{l:321r}. 
	So as in the previous case we reach a
	contradiction to $G\in \fD_4(\grot_i^{\mu\nu})$. 
			
	Finally, the case that $e$ connects a \re{red} vertex to a \bl{blue} vertex 
	reduces to the previous ones by $\tau_\nu$-symmetry.		
\end{proof}

\begin{lemma}[Twin Lemma]\label{lem:vtwin}
	Every $G\in \fD_4(\grot_i^{\mu\nu})$ has the $\grot_i^{\mu\nu}$-twin property. 
\end{lemma}

\begin{proof}
	Let $e$ be an edge of $\grot_i^{\mu\nu}\subseteq G\in \fD_4(\grot_i^{\mu\nu})$, 
	for which we want to confirm the twin property. For the sake of contradiction we 
	assume that there are non-adjacent twins of the end vertices of $e$, indicated 
	in the usual way by primes. Owing to the two foregoing 
	lemmata we have $e\cap \cC_6\ne\vn$. 
	Moreover, by $\sigma$-symmetry we can suppose~$y\not\in e$, which means that one 
	of the following two main cases occurs.  
	
	\smallskip 
	
	{\it \hskip 2em First case: $e$ connects one of $\re{u}$, $\gr{v}$, $\bl{w}$ 
	to $\Gamma$.}
	
	\smallskip
	
	By $\tau_\nu$-symmetry we may assume $\re{u}\in e$ or $\bl{w}\in e$.
	If $e=\re{uj}$, there exists a \re{red} auxiliary path~$\pi$ one of whose 
	end vertices is $\re{j}$. By Lemma~\ref{l:38} applied to~$\grot(\pi)$, 
	there exists a vertex~$q$ adjacent to $u'$, $j'$, $x$, and an inner vertex 
	of $\pi$ 
	(see Figures~\ref{fig:1244b} and~\ref{fig:331A}). 
	But this means that $q$ has a small neighbourhood 
	in $\grot_i^{\mu\nu}(j')$, contrary to Lemma~\ref{lem:2157}.
										
	If $i\ge 3$ the case $e=\bl{wj}$, where $\bl{j}\in \Gb$, is similar, because there 
	exists a \bl{blue} auxiliary path starting with $\bl{j}$. 
	Finally, $i=2$ implies $\bl{j}=\bl{4}$ and due to 
	\[
		(\grot_2^{\mu\nu},\bl{4}\bl{w})\cong (\grot_2^{\nu\mu},\bl{4}\re{1})
		\qquad 
		(\textrm{via } \rho)
	\]
	(see Figure~\ref{fig:1419}) this case reduces to Lemma~\ref{lem:short}. 
		
	\smallskip 
	
	{\it \hskip 2em Second case: $e\cap \{\re{a},\gr{b},\bl{c}\}\ne\vn$.}
	
	\smallskip
	
	If $\re{a}\in e$ we just need to apply Lemma~\ref{l:37} to $\grot(\pi)$ 
	for an appropriate \re{red} auxiliary path~$\pi$. Provided that $i\ge 3$ 
	there exist auxiliary paths ending in all vertices and this argument generalises. 
	Thus it remains to consider the case that $i=2$ and $\{\gr{b}, \bl{c}\}\cap e\ne\vn$. 
	Using our isomorphisms and automorphisms, every possibility can be shown to 
	be equivalent to a case that has already been covered. 
	Indeed, we have 
	\[						
		(\grot_2^{\mu\nu},\bl{c}x) \cong (\grot_2^{\nu\mu},\gr{2}\re{0}),
		\quad
		(\grot_2^{\mu\nu},\bl{c}\bl{4}) \cong (\grot_2^{\nu\mu},\gr{2}\bl{4}),
		\quad
		(\grot_2^{\mu\nu},\bl{c}\gr{v}) \cong (\grot_2^{\nu\mu},\gr{2}\gr{v})
		\qquad (\textrm{via } \rho)
	\]
	(see Figure~\ref{fig:1419}), so that only the cases $e=\bl{c}\re{u}$ 
	and $\gr{b}\in e$ remain. The edge $e=\gr{2b}$ can be taken care of 
	by $\tau_\nu$-symmetry, because 
	\[
		(\grot_2^{\mu 0},\gr{2}\gr{b})\cong(\grot_2^{\mu 0},\re{1a})
		\quad \text{ and } \quad 
		(\grot_2^{\mu 1},\gr{2}\gr{b})\cong(\grot_2^{\mu 1},\bl{4c})\,.
	\]
	Next, we have 
	$(\grot_2^{\mu\nu},\bl{c}\re{u}) \cong (\grot_2^{\nu\mu},\gr{2}\gr{b})$ (via $\rho$)
	and, finally, the other edges containing $\gr{b}$ 
	are $\tau_\nu$-equivalent to certain edges containing $\re{a}$ or $\bl{c}$.
\end{proof}

\subsection{The attachment lemma}
In this subsection we establish the attachment lemma for graphs 
in $\fD_4(\grot_i^{\mu\nu})$, which allows us to complete the proof 
of Theorem~\ref{thm:1.3}\ref{it:1.3b}.

\begin{lemma}[Attachment Lemma]\label{l:329}
	If $\grot_i^{\mu\nu}\subseteq G\in \fD_4(\grot_i^{\mu\nu})$, 
	then every vertex $q\in V(G)$ is an $\grot_i^{\mu\nu}$-twin of some 
	vertex belonging to $\grot_i^{\mu\nu}$.
\end{lemma}

\begin{proof}
	Instead of ``$\grot_i^{\mu\nu}$-twin'' we will just write ``twin'' throughout 
	the argument. We proceed with fourteen claims of the form that if a given 
	vertex $q\in V(G)$ is adjacent to certain members of $V(\grot_i^{\mu\nu})$,
	then $q$ is indeed a twin of some vertex of $\grot_i^{\mu\nu}$.
		
	\begin{claim}\label{c:331}
		If $q$ is adjacent to at least two vertices among $\re{u}$, $\gr{v}$, $\bl{w}$, 
		then either 
		it is a twin of one of $\re{a}$, $\gr{b}$, $\bl{c}$ or~$\mu=0$ and $q$ is 
		a twin of $y$. 
	\end{claim}
		
	\begin{proof}
		Suppose first that $\gr{v}, \bl{w}\in \Nn(q)$. With an arbitrary red
		auxiliary path $\pi$ we form the graph $\grot(\pi)$ (see Figure~\ref{fig:331A}).
		The beautiful lemma yields $x\in \Nn(q)$; so if $\re{u}\in \Nn(q)$ holds as 
		well, then $q$ is a twin of $y$ (and $\mu=0$). Otherwise Lemma~\ref{l:312} 
		shows that $q$ is a $\grot(\pi)$-twin of $\re{a}$, which means that $q$ is 
		adjacent to the end vertices of $\pi$. As for every \re{red} 
		vertex $\re{j} \in \Gr$ there is an auxiliary \re{red} path starting in $\re{j}$,
		we conclude that if $\re{u}q\not\in E(G)$, then $\Gr\subseteq \Nn(q)$, 
		whence $q$ is a twin of $\re{a}$. 
		By $\tau_\nu$-symmetry the only other case we need to consider is 
		that $\re{u},\gr{v}\in \Nn(q)$ but $\bl{w}\not\in \Nn(q)$. If $i\ge 3$ 
		we can simply repeat the above argument with \bl{blue} auxiliary paths, 
		thereby learning that $q$ is a twin of $\bl{c}$. 
		
		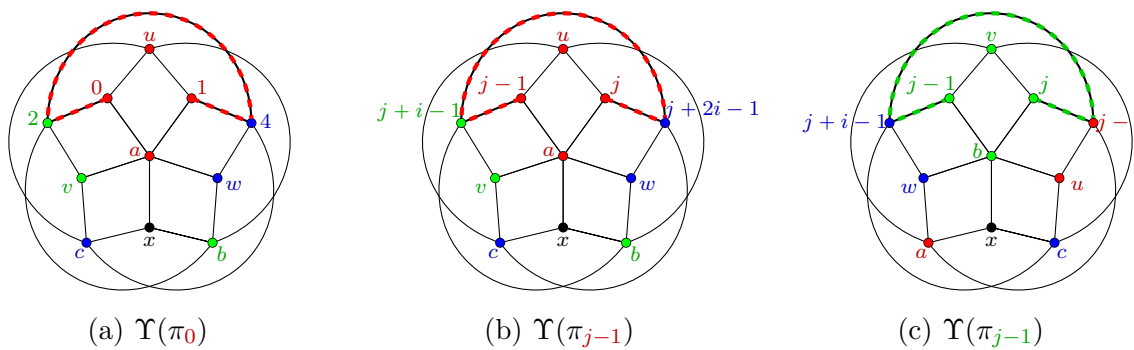
\begin{figure}[h!]
		\begin{subfigure}[b]{.32\textwidth}
		\centering
		\begin{tikzpicture}[scale=.95]
			\coordinate (c) at (0,0);
			\foreach \i in {1,...,5} {
				\coordinate (a\i) at (-18+\i*72:1);
				\coordinate (b\i) at (18+\i*72:1.5);	
				\coordinate (d\i) at (18+\i*72:2);
			}
			\fill (c) circle (2pt);
			\foreach \i in {1,...,5} {
				\draw (c)--(a\i);
				\fill (a\i) circle (2pt);		
				\fill (b\i) circle (2pt);
			}
			\draw (b1) [out=162, in=72] to (d2) [out=-108, in=162] to (b3);
			\draw (b2) [out=234, in=144] to (d3) [out=-36, in=234] to (b4);
			\draw (b3) [out=-54, in=216] to (d4) [out=36, in=-54] to (b5);
			\draw (b4) [out=18, in=-72] to (d5) [out=108, in=18] to (b1);
			\draw (b5) [thick, out=90, in=0] to (d1) [out=180, in=90] to (b2);
			\draw (b5) [ultra thick,red,dashed,out=90,in=0] to (d1) [out=180,in=90] to (b2);
			\draw [thick] (a2)--(b2);
			\draw [ultra thick, red, dashed] (a2)--(b2);
			\draw [thick] (a1)--(b5);
			\draw [ultra thick,red, dashed] (a1)--(b5);
			\draw (b3)--(a3)--(c)--(a5)--(b4);
			\draw (b3)--(a4)--(b4)--(a4)--(c);
			\draw (b1)--(a1)--(c)--(a2)--(b1);
			\draw (b2)--(a3);
			\draw (b5)--(a5);
			\foreach \i in {a1, a2, c, b1} \fill [red] (\i) circle (1.7pt);
			\foreach \i in {b5, a5, b3} \fill [blue] (\i) circle (1.7pt);
			\foreach \i in {b2, a3, b4} \fill [green] (\i) circle (1.7pt); 
			\node at ($(c)+(-.2,.07)$) {\tiny $\re{a}$};
			\node at ($(a3)+(-.2,-.15)$) {\tiny $\gr{v}$};
			\node at ($(a4)+(0,-.2)$) {\tiny $x$};
			\node at ($(a5)+(.23,-.1)$) {\tiny $\bl{w}$};
			\node at ($(b3)+(-.1,-.15)$) {\tiny $\bl{c}$};
			\node at ($(b4)+(.13,-.15)$) {\tiny $\gr{b}$};
			\node at ($(b1)+(0,.17)$) {\tiny $\re{u}$};
			\node at ($(b2)+(-.2,.07)$) {\tiny $\gr{2}$};
			\node at ($(b5)+(.2,.02)$) {\tiny $\bl{4}$};
			\node at ($(a1) +(.15,.15)$) {\tiny $\re{1}$};
			\node at ($(a2)+(-.15,.15)$){\tiny $\re{0}$};
		\end{tikzpicture}
		\vskip -.2cm
		\caption{$\grot(\pi_{\re{0}})$}
		\label{fig:1253a}
		\end{subfigure}
		\hfill
		\begin{subfigure}[b]{.32\textwidth}
		\centering
		\begin{tikzpicture}[scale=.95]
			\coordinate (c) at (0,0);
			\foreach \i in {1,...,5} {
				\coordinate (a\i) at (-18+\i*72:1);
				\coordinate (b\i) at (18+\i*72:1.5);	
				\coordinate (d\i) at (18+\i*72:2);
			}
			\fill (c) circle (2pt);
			\foreach \i in {1,...,5} {
				\draw (c)--(a\i);
				\fill (a\i) circle (2pt);		
				\fill (b\i) circle (2pt);
			}
			\draw (b1) [out=162, in=72] to (d2) [out=-108, in=162] to (b3);
			\draw (b2) [out=234, in=144] to (d3) [out=-36, in=234] to (b4);
			\draw (b3) [out=-54, in=216] to (d4) [out=36, in=-54] to (b5);
			\draw (b4) [out=18, in=-72] to (d5) [out=108, in=18] to (b1);
			\draw (b5) [thick, out=90, in=0] to (d1) [out=180, in=90] to (b2);
			\draw (b5) [ultra thick, red, dashed, out=90, in=0] to (d1) [out=180, in=90] to (b2);
			\draw [thick] (a2)--(b2);
			\draw [ultra thick, red, dashed] (a2)--(b2);
			\draw [thick] (a1)--(b5);
			\draw [ultra thick,red, dashed] (a1)--(b5);
			\draw (b3)--(a3)--(c)--(a5)--(b4);
			\draw (b3)--(a4)--(b4)--(a4)--(c);
			\draw (b1)--(a1)--(c)--(a2)--(b1);
			\draw (b2)--(a3);
			\draw (b5)--(a5);
			\foreach \i in {a1, a2, c, b1} \fill [red] (\i) circle (1.7pt);
			\foreach \i in {b5, a5, b3} \fill [blue] (\i) circle (1.7pt);
			\foreach \i in {b2, a3, b4} \fill [green] (\i) circle (1.7pt); 
			\node at ($(c)+(-.2,.07)$) {\tiny $\re{a}$};
			\node at ($(a3)+(-.2,-.15)$) {\tiny $\gr{v}$};
			\node at ($(a4)+(0,-.2)$) {\tiny $x$};
			\node at ($(a5)+(.23,-.1)$) {\tiny $\bl{w}$};
			\node at ($(b3)+(-.1,-.15)$) {\tiny $\bl{c}$};
			\node at ($(b4)+(.13,-.15)$) {\tiny $\gr{b}$};
			\node at ($(b1)+(0,.17)$) {\tiny $\re{u}$};
			\node at ($(b2)+(-.58,.15)$) {\tiny $\gr{j+i-1}$};
			\node at ($(b5)+(.66,.18)$) {\tiny $\bl{j+2i-1}$};
			\node at ($(a1) +(.15,.18)$) {\tiny $\re{j}$};
			\node at ($(a2)+(-.25,.18)$){\tiny $\re{j-1}$};
		\end{tikzpicture}
		\vskip -.2cm
		\caption{$\grot(\pi_{\re{j-1}})$}
		\label{fig:1253b}
		\end{subfigure}
		\hfill
		\begin{subfigure}[b]{.32\textwidth}
		\centering
		\begin{tikzpicture}[scale=.95]
			\coordinate (c) at (0,0);
			\foreach \i in {1,...,5} {
				\coordinate (a\i) at (-18+\i*72:1);
				\coordinate (b\i) at (18+\i*72:1.5);	
				\coordinate (d\i) at (18+\i*72:2);
			}
			\draw (b1) [out=162, in=72] to (d2) [out=-108, in=162] to (b3);
			\draw (b2) [out=234, in=144] to (d3) [out=-36, in=234] to (b4);
			\draw (b3) [out=-54, in=216] to (d4) [out=36, in=-54] to (b5);
			\draw (b4) [out=18, in=-72] to (d5) [out=108, in=18] to (b1);
			\draw (b5) [thick, out=90, in=0] to (d1) [out=180, in=90] to (b2);
			\draw (b5) [ultra thick, green!80!black, dashed, out=90, in=0] to (d1) [out=180, in=90] to (b2);
			\draw (b3)--(a3)--(c)--(a5)--(b4);
			\draw (b3)--(a4)--(b4)--(a4)--(c);
			\draw (b1)--(a1)--(c)--(a2)--(b1);
			\draw (b2)--(a3);
			\draw (b5)--(a5);
			\draw [thick] (a2)--(b2);
			\draw [ultra thick, green!80!black, dashed] (a2)--(b2);
			\draw [thick] (a1)--(b5);
			\draw [ultra thick,green!80!black, dashed] (a1)--(b5); 
			\fill (c) circle (2pt);
			\foreach \i in {1,...,5} {
				\draw (c)--(a\i);
				\fill (a\i) circle (2pt);		
				\fill (b\i) circle (2pt);
			}
			\foreach \i in {b5, a5, b3} \fill [red] (\i) circle (1.7pt);
			\foreach \i in {a1, a2, b1, c} \fill [green] (\i) circle (1.7pt);
			\foreach \i in {b2, a3, b4} \fill [blue] (\i) circle (1.7pt); 
			\node at ($(c)+(-.2,.08)$) {\tiny $\gr{b}$};
			\node at ($(a3)+(-.2,-.15)$) {\tiny $\bl{w}$};
			\node at ($(a4)+(0,-.2)$) {\tiny $x$};
			\node at ($(a5)+(.25,-.1)$) {\tiny $\re{u}$};
			\node at ($(b3)+(-.1,-.15)$) {\tiny $\re{a}$};
			\node at ($(b4)+(.1,-.15)$) {\tiny $\bl{c}$};
			\node at ($(b1)+(0,.17)$) {\tiny $\gr{v}$};
			\node at ($(b2)+(-.62,0)$) {\tiny $\bl{j+i-1}$};
			\node at ($(b5)+(.36,0)$) {\tiny $\re{j-i}$};
			\node at ($(a1) +(.15,.18)$) {\tiny $\gr{j}$};
			\node at ($(a2)+(-.25,.18)$){\tiny $\gr{j-1}$};
		\end{tikzpicture}
		\vskip -.2cm
		\caption{$\grot(\pi_{\gr{j-1}})$}\label{fig:1253c}
		\end{subfigure} 
		\vskip -.2cm
		\caption{The proof of Claim~\ref{c:331} and Claim~\ref{c:332}.}
		\label{il:331a}
		\end{figure} 		
		
		If $i=2$, however, only the \re{red} auxiliary path 
		$\pi_{\re{0}}=\re{0}-\gr{2}-\bl{4}-\re{1}$ might be available. 
		Corollary~\ref{cor:326} guarantees that~$\bl{w}$ is reliable 
		with respect to~$\grot(\pi_{\re{0}})$ (see Figure~\ref{fig:1253a}). 
		Thus Lemma~\ref{l:313} applied to this graph
		yields $x\in \Nn(q)$ (see Figure~\ref{fig:1138c}).  
		Next Lemma~\ref{l:314} and $q\not\in\Ext(\grot(\pi_0), x)$ imply 
		that some $\grot(\pi_{\re{0}})$-twin~$4'$ of~$\bl{4}$ is in $\Nn(q)$
		(see Figure~\ref{fig:1138d}). 
		Due to $i=2$ this vertex is actually a real twin of $\bl{4}$. 
		Now~$q$ is a $\grot_2^{\mu\nu}(4')$-twin of $\bl{c}$. By Lemma~\ref{f:24} 
		it follows that $q$ is a twin of $\bl{c}$.	 
	\end{proof}
	
	In view of $\sigma$-symmetry, the previous claim implies the following. 
	
	\begin{claim}\label{cor:430}
		If $\mu=0$ and~$q$ is adjacent to at least two among $\re{a}$, $\gr{b}$, $\bl{c}$, 
		then it is a twin of $\re{u}$, $\gr{v}$, $\bl{w}$, or $x$. \qed	
	\end{claim}
	
	Without the assumption on $\mu$ we still have the following weaker assertion. 
	
	\begin{claim}\label{c:330}
		If $q$ is adjacent to two vertices among $\re{a}$, $\gr{b}$, $\bl{c}$ and 
		to some vertex from $\G_i$, then $q$ is a twin of $\re{u}$, $\gr{v}$, or $\bl{w}$.
	\end{claim}
	
	\begin{proof}
		Arguing indirectly we assume that $q$ is a counterexample. 
		Let $\phi\in\{\re{\mathrm{red}}, \gr{\mathrm{green}}, \bl{\mathrm{blue}}\}$
		be a colour satisfying $\Nn(q)\cap\Gamma_\phi\ne\vn$.
		Claim~\ref{cor:430} tells us that $\mu=1$, 
		whence $\Gamma_\phi\not\subseteq \Nn(q)$. Consequently, there is a 
		pair of consecutive vertices $\{j, j+1\}\subseteq \Gamma_\phi$ such 
		that exactly one of $j$, $j+1$ is in $\Nn(q)$. 
		
		Recall that 
		\[
			\pi_j=j-(j+i)-(j+2i)-(j+1)
		\]
		is an auxiliary path of colour $\phi$. 
		By Lemma~\ref{l:313} applied to $\grot(\pi_j)$ there is a neighbour of $q$
		in $\Ext(\grot(\pi_j), s)$ for some $s\in\{\re{u}, \gr{v}, \bl{w}\}$ whose 
		colour is not $\phi$ (see Figure~\ref{fig:331}). But this contradicts the
		fact that $s$ is, by Corollary~\ref{cor:326}, reliable in $\grot(\pi_j)$.
	\end{proof}
	
	Here is a statement that can be viewed as an adaptation of Claim~\ref{l:28}
	to Vega graphs.
	
	\begin{claim}\label{c:332}
		If $j\ne \re{0}, \gr{i}$ and $j, j+i-1\in \Nn(q)$, then $q$ is a twin of $j-i$.
	\end{claim}
	
	\begin{proof}
		The condition $j\neq \re{0},\gr{i}$ just means that $j$, $j+i-1$ have distinct 
		colours. So by $\tau_\nu$-symmetry we may assume that $j$ is \re{red} 
		or \gr{green}. 
		If $\re{j}$ is \re{red}, we consider $\grot(\pi_{\re{j-1}})$, where
		\[
			\pi_{\re{j-1}} = \re{(j-1)}-\gr{(j+i-1)} -\bl{(j+2i-1)}-\re{j}
		\]
		is a \re{red} auxiliary path (see Figure~\ref{fig:1253b}). 
		Assume for the sake of contradiction 
		that $\bl{w}\not\in \Nn(q)$. In view of Lemma~\ref{l:313} there 
		exists some $y'\in \Nn(q)\cap \Ext(\grot(\pi_{\re{j-1}}), x)$.
		As~$y'$ is adjacent to $\re{u}$, $\gr{v}$, $\bl{w}$, and~$x$, we 
		have $\mu=0$ and $y'$ is a twin of $y$. But now $q$ has a small neighbourhood 
		in $\grot_i^{\mu\nu}(y')$, which is absurd. This proves $\bl{w}\in \Nn(q)$.
	
		Using the impossibility of small neighbourhoods again we obtain $x\not\in \Nn(q)$. 
		Therefore Lemma \ref{l:314} shows that some $\grot(\pi_{\re{j-1}})$-twin $c'$ 
		of $\bl{c}$ is in $\Nn(q)$. Claim~\ref{c:331} informs us that~$c'$ is, in fact, 
		a real twin of $\bl{c}$. So Lemma~\ref{l:321} applied to $\grot_i^{\mu\nu}(c')$ 
		yields $\re{[j,i-1]}\subseteq \Nn(q)$ and $\gr{[i,j+i-1]}\subseteq \Nn(q)$. 
		Altogether~$q$ is a $\grot_i^{\mu\nu}(c')$-twin of~$\bl{j-i}$ and 
		by Lemma~\ref{f:24} we are done. The case that~$\gr{j}$ is \gr{green} is 
		similar, the only difference being that now the path 
		$\pi_{\gr{j-1}}=\gr{(j-1)}-\bl{(j+i-1)} -\re{(j-i)}-\gr{j}$ 
		is \gr{green} (see Figure~\ref{fig:1253c}). 
	\end{proof}

	The next three claims analyse vertices adjacent to two opposite vertices 
	of the external hexagon.

	\begin{claim}\label{c:334}
		If $\re{a},\re{u}\in \Nn(q)$, then $q$ is a twin of some vertex belonging 
		to~$\Gr$. 
	\end{claim}
	
	\begin{proof}
		Due to Lemma \ref{l:321r} it suffices to show the following statements. 
		\begin{enumerate}[label=\nlabel]
			\item\label{it:4361} For every $\re{j}\in \re{[0,i-2]}$, either $\gr{j+i}$ 
				or $\bl{j+2i}$ is in $\Nn(q)$.
			\item\label{it:4362} If $\nu=0$, then $\gr{2i-1}\in \Nn(q)$.
		\end{enumerate}
		
		For the proof of~\ref{it:4361} we apply Corollary~\ref{cor:311} 
		to $\grot(\pi_{\re{j}})$, where, as usual, 
		\[
			\pi_{\re{j}}=\re{j}-\gr{(j+i)}-\bl{(j+2i)}-\re{(j+1)}
		\]
		(see Figures~\ref{fig:1138a} and~\ref{fig:2409a}).
		This shows that if $\gr{j+i}, \bl{j+2i} \not\in \Nn(q)$, then there 
		are $\grot(\pi_{\re{j}})$-twins of $\gr{j+i}$, $\bl{j+2i}$ adjacent to~$q$. 
		But by Claim~\ref{c:332} these twins had to be real twins and the twin lemma 
		would show that they are adjacent, thus creating a triangle with~$q$. 
	
		To verify~\ref{it:4362} we work with the \gr{green} auxiliary path 
		$\omega=\gr{i}-\bl{2i}-\re{(i-1)}-\gr{(2i-1)}$. In view of Lemma~\ref{l:325} 
		no vertex is adjacent to $\bl{c}$, $\bl{w}$, $\gr{2i-1}$ and thus~$\gr{i}$ 
		is reliable in $\grot(\omega)$ (see Figure~\ref{fig:2409b}). 
		So Lemma~\ref{l:313} discloses 
		that $\gr{(2i-1)}\not\in \Nn(q)$ is impossible. 
	\end{proof}

	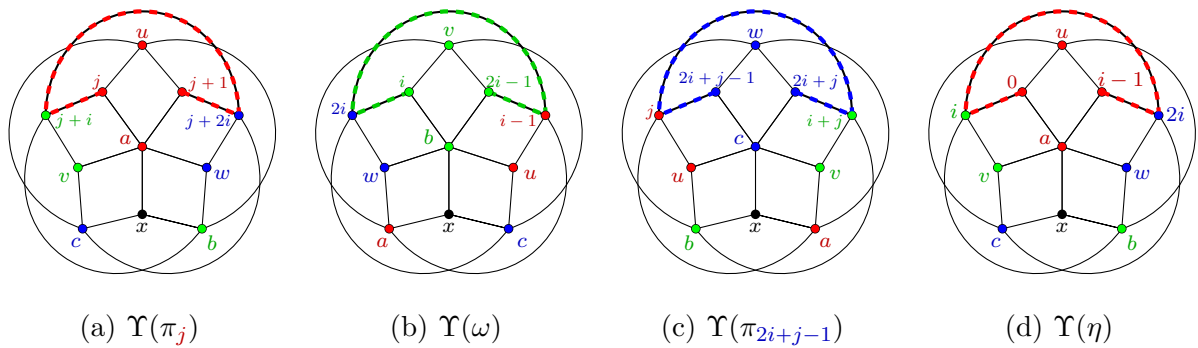
\begin{figure}[h!]
	\begin{subfigure}[b]{.24\textwidth}
	\centering
	\begin{tikzpicture}[scale=.9]
		\coordinate (c) at (0,0);
		\foreach \i in {1,...,5} {
			\coordinate (a\i) at (-18+\i*72:1);
			\coordinate (b\i) at (18+\i*72:1.5);	
			\coordinate (d\i) at (18+\i*72:2);
		}
		\fill (c) circle (2pt);
		\foreach \i in {1,...,5} {
			\draw (c)--(a\i);
			\fill (a\i) circle (2pt);		
			\fill (b\i) circle (2pt);
		}
		\draw (b1) [out=162, in=72] to (d2) [out=-108, in=162] to (b3);
		\draw (b2) [ out=234, in=144] to (d3) [out=-36, in=234] to (b4);
		\draw (b3) [ out=-54, in=216] to (d4) [out=36, in=-54] to (b5);
		\draw (b4) [out=18, in=-72] to (d5) [out=108, in=18] to (b1);
		\draw (b5) [thick, out=90, in=0] to (d1) [out=180, in=90] to (b2);
		\draw (b5) [ultra thick, red, dashed, out=90, in=0] to (d1) [out=180, in=90] to (b2);
		\draw [thick] (a2)--(b2);
		\draw [ultra thick, red, dashed] (a2)--(b2);
		\draw [thick] (a1)--(b5);
		\draw [ultra thick,red, dashed] (a1)--(b5);
		\draw (b3)--(a3)--(c)--(a5)--(b4);
		\draw (b3)--(a4)--(b4)--(a4)--(c);
		\draw  (b1)--(a1)--(c)--(a2)--(b1);
		\draw  (b2)--(a3);
		\draw  (b5)--(a5);
		\foreach \i in {a1, a2, c, b1} \fill [red] (\i) circle (1.7pt);
		\foreach \i in {b5, a5, b3} \fill [blue] (\i) circle (1.7pt);
		\foreach \i in {b2, a3, b4} \fill [green] (\i) circle (1.7pt); 
		\node at ($(c)+(-.25,.1)$) {\tiny{$\re{a}$}};
		\node at ($(a3)+(-.2,-.15)$) {\tiny{$\gr{v}$}};
		\node at ($(a4)+(0,-.2)$) {\tiny{$x$}};
		\node at ($(a5)+(.23,-.1)$) {\tiny{$\bl{w}$}};
		\node at ($(b3)+(-.1,-.2)$) {\tiny{$\bl{c}$}};
		\node at ($(b4)+(.15,-.2)$) {\tiny{$\gr{b}$}};
		\node at ($(b1)+(0,.18)$) {\tiny{$\re{u}$}};
		\node [scale=.8] at ($(b2)+(.42,-.08)$) {\tiny{$\gr{j+i}$}};
		\node[scale=.8] at ($(b5)+(-.45,-.1)$) {\tiny{$\bl{j+2i}$}};
		\node[scale=.8] at ($(a1) +(.39,.12)$) {\tiny{$\re{j+1}$}};
		\node[scale=.8] at ($(a2)+(-.08,.15)$){\tiny{$\re{j}$}};
	\end{tikzpicture}
	\caption{$\grot(\pi_{\re{j}})$}
	\label{fig:2409a}
	\end{subfigure}
	\hfill
	\begin{subfigure}[b]{.24\textwidth}
	\centering
	\begin{tikzpicture}[scale=.9]
		\coordinate (c) at (0,0);
		\foreach \i in {1,...,5} {
			\coordinate (a\i) at (-18+\i*72:1);
			\coordinate (b\i) at (18+\i*72:1.5);	
			\coordinate (d\i) at (18+\i*72:2);
		}
		\draw (b1) [out=162, in=72] to (d2) [out=-108, in=162] to (b3);
		\draw (b2) [out=234, in=144] to (d3) [out=-36, in=234] to (b4);
		\draw (b3) [out=-54, in=216] to (d4) [out=36, in=-54] to (b5);
		\draw (b4) [out=18, in=-72] to (d5) [out=108, in=18] to (b1);
		\draw (b5) [thick, out=90, in=0] to (d1) [out=180, in=90] to (b2);
		\draw (b5) [ultra thick, green!80!black, dashed, out=90, in=0] to (d1) [out=180, 	in=90] to (b2);
		\draw (b3)--(a3)--(c)--(a5)--(b4);
		\draw (b3)--(a4)--(b4)--(a4)--(c);
		\draw (b1)--(a1)--(c)--(a2)--(b1);
		\draw (b2)--(a3);
		\draw (b5)--(a5);
		\draw [thick] (a2)--(b2);
		\draw [ultra thick, green!80!black, dashed] (a2)--(b2);
		\draw [thick] (a1)--(b5);
		\draw [ultra thick,green!80!black, dashed] (a1)--(b5); 
		\fill (c) circle (2pt);
		\foreach \i in {1,...,5} {
			\draw (c)--(a\i);
			\fill (a\i) circle (2pt);		
			\fill (b\i) circle (2pt);
		}
		\foreach \i in {b5, a5, b3} \fill [red] (\i) circle (1.7pt);
		\foreach \i in {a1, a2, b1, c} \fill [green] (\i) circle (1.7pt);
		\foreach \i in {b2, a3, b4} \fill [blue] (\i) circle (1.7pt); 
		\node at ($(c)+(-.3,.15)$) {\tiny{$\gr{b}$}};
		\node at ($(a3)+(-.2,-.15)$) {\tiny{$\bl{w}$}};
		\node at ($(a4)+(0,-.2)$) {\tiny{$x$}};
		\node at ($(a5)+(.25,-.1)$) {\tiny{$\re{u}$}};
		\node at ($(b3)+(-.1,-.2)$) {\tiny{$\re{a}$}};
		\node at ($(b4)+(.2,-.2)$) {\tiny{$\bl{c}$}};
		\node at ($(b1)+(0,.2)$) {\tiny{$\gr{v}$}};
		\node [scale=.8] at ($(b2)+(-.2,.1)$) {\tiny{$\bl{2i}$}};
		\node [scale=.8] at ($(b5)+(-.4,-.09)$) {\tiny{$\re{i-1}$}};
		\node[scale=.8] at ($(a1) +(.3,.15)$) {\tiny{$\gr{2i-1}$}};
		\node[scale=.8] at ($(a2)+(-.1,.15)$){\tiny{$\gr{i}$}};
	\end{tikzpicture}
	\caption{$\grot(\omega)$}
	\label{fig:2409b}
	\end{subfigure}
	\hfill
	\begin{subfigure}[b]{.24\textwidth}
	\centering
	\begin{tikzpicture}[scale=.9]
		\coordinate (c) at (0,0);
		\foreach \i in {1,...,5} {
			\coordinate (a\i) at (-18+\i*72:1);
			\coordinate (b\i) at (18+\i*72:1.5);	
			\coordinate (d\i) at (18+\i*72:2);
		}
		\fill (c) circle (2pt);
		\foreach \i in {1,...,5} {
			\draw (c)--(a\i);
			\fill (a\i) circle (2pt);		
			\fill (b\i) circle (2pt);
		}
		\draw (b1) [out=162, in=72] to (d2) [out=-108, in=162] to (b3);
		\draw (b2) [out=234, in=144] to (d3) [out=-36, in=234] to (b4);
		\draw (b3) [out=-54, in=216] to (d4) [out=36, in=-54] to (b5);
		\draw (b4) [out=18, in=-72] to (d5) [out=108, in=18] to (b1);
		\draw (b5) [thick, out=90, in=0] to (d1) [out=180, in=90] to (b2);
		\draw (b5) [ultra thick, blue,dashed,out=90,in=0] to (d1) [out=180, in=90] to (b2);
		\draw [thick] (a2)--(b2);
		\draw [ultra thick, blue, dashed] (a2)--(b2);
		\draw [thick] (a1)--(b5);
		\draw [ultra thick, blue, dashed] (a1)--(b5);
		\draw (b3)--(a3)--(c)--(a5)--(b4);
		\draw (b3)--(a4)--(b4)--(a4)--(c);
		\draw (b1)--(a1)--(c)--(a2)--(b1);
		\draw (b2)--(a3);
		\draw (b5)--(a5);
		\foreach \i in {a1, a2, c, b1} \fill [blue] (\i) circle (1.7pt);
		\foreach \i in {b5, a5, b3} \fill [green] (\i) circle (1.7pt);
		\foreach \i in {b2, a3, b4} \fill [red] (\i) circle (1.7pt); 
		\node at ($(c)+(-.25,.1)$) {\tiny{$\bl{c}$}};
		\node at ($(a3)+(-.2,-.15)$) {\tiny{$\re{u}$}};
		\node at ($(a4)+(0,-.2)$) {\tiny{$x$}};
		\node at ($(a5)+(.23,-.1)$) {\tiny{$\gr{v}$}};
		\node at ($(b3)+(-.1,-.2)$) {\tiny{$\gr{b}$}};
		\node at ($(b4)+(.15,-.2)$) {\tiny{$\re{a}$}};
		\node at ($(b1)+(0,.18)$) {\tiny{$\bl{w}$}};
		\node [scale=.8]at ($(b2)+(-.13,.05)$) {\tiny{$\re{j}$}};
		\node[scale =.8] at ($(b5)+(-.4,-.1)$) {\tiny{$\gr{i+j}$}};
		\node[scale=.8] at ($(a1) +(.3,.15)$) {\tiny{$\bl{2i+j}$}};
		\node [scale = .8] at ($(a2)+(.02,.19)$){\tiny{$\bl{2i+j-1}$}};
	\end{tikzpicture}
	\caption{$\grot(\pi_{\bl{2i+j-1}})$}
	\label{fig:2409c}
	\end{subfigure}
	\hfill
	\begin{subfigure}[b]{.24\textwidth}
	\centering
	\begin{tikzpicture}[scale=.9]
	\coordinate (c) at (0,0);
	\foreach \i in {1,...,5} {
		\coordinate (a\i) at (-18+\i*72:1);
		\coordinate (b\i) at (18+\i*72:1.5);	
		\coordinate (d\i) at (18+\i*72:2);
	}
	\fill (c) circle (2pt);
	\foreach \i in {1,...,5} {
		\draw (c)--(a\i);
		\fill (a\i) circle (2pt);		
		\fill (b\i) circle (2pt);
	}
	\draw (b1) [out=162, in=72] to (d2) [out=-108, in=162] to (b3);
	\draw (b2) [out=234, in=144] to (d3) [out=-36, in=234] to (b4);
	\draw (b3) [out=-54, in=216] to (d4) [out=36, in=-54] to (b5);
	\draw (b4) [out=18, in=-72] to (d5) [out=108, in=18] to (b1);
	\draw (b5) [thick, out=90, in=0] to (d1) [out=180, in=90] to (b2);
	\draw (b5) [ultra thick, red, dashed, out=90, in=0] to (d1) [out=180, in=90] to (b2);
	\draw [thick] (a2)--(b2);
	\draw [ultra thick, red, dashed] (a2)--(b2);
	\draw [thick] (a1)--(b5);
	\draw [ultra thick,red, dashed] (a1)--(b5);
	\draw (b3)--(a3)--(c)--(a5)--(b4);
	\draw (b3)--(a4)--(b4)--(a4)--(c);
	\draw (b1)--(a1)--(c)--(a2)--(b1);
	\draw (b2)--(a3);
	\draw (b5)--(a5);
	\foreach \i in {a1, a2, c, b1} \fill [red] (\i) circle (1.7pt);
	\foreach \i in {b5, a5, b3} \fill [blue] (\i) circle (1.7pt);
	\foreach \i in {b2, a3, b4} \fill [green] (\i) circle (1.7pt); 
	\node at ($(c)+(-.25,.1)$) {\tiny{$\re{a}$}};
	\node at ($(a3)+(-.2,-.15)$) {\tiny{$\gr{v}$}};
	\node at ($(a4)+(0,-.2)$) {\tiny{$x$}};
	\node at ($(a5)+(.23,-.1)$) {\tiny{$\bl{w}$}};
	\node at ($(b3)+(-.1,-.2)$) {\tiny{$\bl{c}$}};
	\node at ($(b4)+(.15,-.2)$) {\tiny{$\gr{b}$}};
	\node at ($(b1)+(0,.18)$) {\tiny{$\re{u}$}};
	\node at ($(b2)+(-.17,.07)$) {\tiny{$\gr{i}$}};
	\node at ($(b5)+(.25,.02)$) {\tiny{$\bl{2i}$}};
	\node at ($(a1) +(.28,.17)$) {\tiny{$\re{i-1}$}};
	\node [scale=.9] at ($(a2)+(-.15,.15)$){\tiny{$\re{0}$}};
	\end{tikzpicture}
	\caption{$\grot(\eta)$}
	\label{fig:2409d}
	\end{subfigure}
	\caption{The proofs of Claim~\ref{c:334} and Claim~\ref{l:335}.}
	\label{il:334}
	\end{figure} 	
	
	\begin{claim}\label{l:335}
		If $\bl{c},\bl{w}\in \Nn(q)$, then $q$ is a twin of some vertex in $\Gb$.
	\end{claim}
	
	\begin{proof}
		Because of Lemma~\ref{l:321} it suffices to show that
		\begin{enumerate}[label=\nlabel]
			\item\label{it:335a} for every $\re{j}\in \re{[1,i-2]}$, either $\re{j}$ 
				or $\gr{j+i}$ is in $\Nn(q)$;
				\item\label{it:335b} $\re{i-1}\in \Nn(q)$;
				\item\label{it:335c} and $\gr{i}\in \Nn(q)$. 
		\end{enumerate}
		
		The argument establishing~\ref{it:335a} is very similar to the previous proof 
		but uses the \bl{blue} auxiliary path
		\[
			\pi_{\bl{2i+j-1}}=\bl{(2i+j-1)}-\re{j}-\gr{(i+j)}-\bl{(2i+j)}
		\]
		instead (see Figure~\ref{fig:2409c}); we omit the details. 
					
		For the remaining two statements we work with the \re{red} auxiliary path 
		\[
			\eta = \re{0}-\gr{i}-\bl{2i}-\re{(i-1)}
		\]
		(see Figure~\ref{fig:2409d}). 				
		Assume first that contrary to~\ref{it:335b} we have $\re{i-1}\not\in \Nn(q)$. 
		By Lemma~\ref{l:313} there exists 
		some $z\in \Nn(q)\cap \Ext(\grot(\omega), \re{0})$. Due to 
		$\gr{b},\gr{v},\re{i-1}\in \Nn(z)$ and Lemma \ref{l:319}\ref{it:319a} we have 
		$\Gr\subseteq \Nn(z)$, wherefore $z$ is a twin of $\gr{2i-1}$ (and $\nu=0$). 
		But this means that $q$ violates Lemma~\ref{l:325} with respect 
		to $\grot_i^{\mu\nu}(z)$. Thereby~\ref{it:335b} is proved.
		
		Suppose next that $\gr{i}\not\in \Nn(q)$ and observe that due 
		to Lemma~\ref{l:325} the vertex~$\re{i-1}$ is reliable with respect 
		to $\grot(\eta)$. So Lemma \ref{l:314} tells us that $\Nn(q)$ contains 
		some $\grot(\eta)$-twins~$0'$,~$i'$ of $\re{0}$, $\gr{i}$, respectively. 
		By Lemma~\ref{l:321r}\ref{it:321ra} $0'$ is actually a real twin of $\re{0}$
		and by the case~$j=\bl{2i}$ of Claim~\ref{c:332} $i'$ is also real twin 
		of $\gr{i}$. Now the twin lemma discloses $0'i'\in E(G)$ and together with $q$
		this edge closes a triangle, which is absurd. This concludes the proof 
		of~\ref{it:335c}.
	\end{proof}

	\begin{claim}\label{c:336}
		If $\gr{b},\gr{v}\in \Nn(q)$, then $q$ is a twin of some vertex in $\Gg$.
	\end{claim}
	
	\begin{proof}
		By $\tau_\nu$-symmetry this follows from the two previous claims. 
	\end{proof}
	
	So far all our claims assume that at least two neighbours of $q$ 
	in $V(\grot^{\mu\nu}_i)$ are given. This is not surprising, because all 
	the results in~\S\ref{subsec:43} are of this form, and up to this point 
	no other arguments have been utilised. Eventually we need to cover less 
	restrictive cases of the attachment lemma as well. Accordingly, we shall use 
	a hexagon argument in the claim after the next one, where we show that 
	neighbours of $x$ are twins of $\re{a}$, $\gr{b}$, $\bl{c}$, or $y$. 
	Preparing ourselves for this task we 
	establish an important special case first. 
	  
	\begin{claim}\label{c:333}
		If $q$ is adjacent to $x$ and a vertex belonging to $\Gamma_i$, then it is a 
		twin of $\re{a}$, $\gr{b}$, or~$\bl{c}$. 
	\end{claim}
		
	\begin{proof}
		If, for instance, $x,\re{j}\in \Nn(q)$, where $\re{j}\in \Gr$, we take 
		a \re{red} auxiliary path $\pi$ one of whose end vertices is $j$ and apply 
		Lemma~\ref{l:312} to~$\grot(\pi)$ (see Figure~\ref{fig:331A}). 
		As~$q$ cannot have a small neighbourhood 
		in $\grot_i^{\mu\nu}$, this yields $\gr{v},\bl{w}\in \Nn(q)$ and due to 
		Claim~\ref{c:331} $q$ is a twin of~$\re{a}$. By $\tau_\nu$-symmetry the only 
		other possibility we need to consider is that 
		$\bl{j}, x\in \Nn(q)$ holds for some $\bl{j}\in \bl{\Gb}$. 
		The case $i\ge 3$ is similar, because then there is a \bl{blue} auxiliary 
		path starting in $\bl{j}$.
		Suppose, finally, that~$i=2$ and $\bl{4}, x\in \Nn(q)$. 
		Since $\rho(\bl{4})=\bl{4}$ and $\rho(x)=\re{0}$, the case $j=\bl{4}$ of 
		Claim~\ref{c:332} now shows that $q$ is a twin of $\rho^{-1}(\gr{2})$, i.e., 
		of~$\bl{c}$. 
	\end{proof}
	
	\begin{claim}\label{c:337}
		If $x\in \Nn(q)$, then $q$ is a twin of $\re{a}$, $\bl{b}$, $\gr{c}$, or $y$. 
	\end{claim}

	\begin{proof}
		We first show that if $\re{u}\not\in \Nn(q)$, then $q$ is a twin of $\re{a}$. 
		To this end we take an arbitrary common neighbour $z$ of $q$, $\re{u}$
		and form the hexagon shown in Figure~\ref{il:337A}.
		Since~$G$ satisfies~\Dp{2}, there is a common neighbour $t$ 
		of $\{q,\re{a},\re{u}\}$ or $\{x,\re{0},z\}$.
		
		If $q,\re{a},\re{u}\in\Nn(t)$, then 
		Claim~\ref{c:334} informs us that $t$ is a twin of some $\re{j}\in \Gr$.
		Among the neighbours of $q$ in $\grot_i^{\mu\nu}(t)$ there are $x$ 
		and the \re{red} vertex $\re{t}\in \Gr$. 
		Therefore, Claim~\ref{c:333} implies that $q$ is a $\grot_i^{\mu\nu}(t)$-twin 
		of~$\re{a}$ and Lemma~\ref{f:24} reveals that $q$ is a real twin of $\re{a}$ 
		as well.
		
		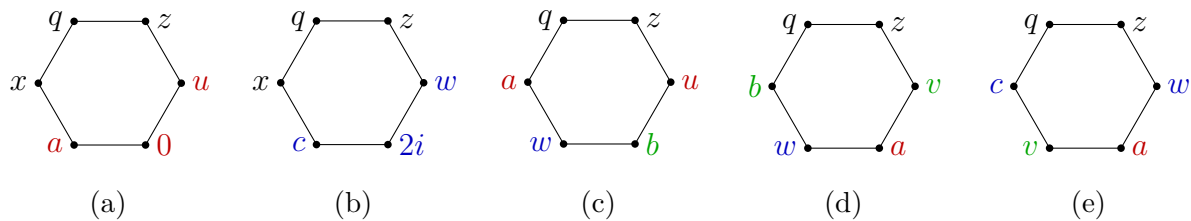
\begin{figure}[h!]
\begin{subfigure}[b]{.19\textwidth}
\centering
\begin{tikzpicture}[scale=.95]
\foreach \i in {1,...,7} {
	\coordinate (a\i) at (\i*60:1);
}
\foreach \i in {1,...,6} {
	\pgfmathsetmacro\j{\i+1};
	\draw (a\i)--(a\j);
	\fill (a\i) circle (1.5pt);		
}
\node [left] at (a2) {$q$};
\node [left] at (a3) {$x$};
\node [left] at (a4) {$\re{a}$};
\node [right] at (a5) {$\re{0}$};
\node [right] at (a6) {$\re{u}$};
\node [right] at (a1) {$z$};
\end{tikzpicture}
\caption{}
\label{il:337A}
\end{subfigure}
\hfill
\begin{subfigure}[b]{.19\textwidth}
\centering
\begin{tikzpicture}[scale=.95]
\foreach \i in {1,...,7} {
	\coordinate (a\i) at (\i*60:1);
}
\foreach \i in {1,...,6} {
	\pgfmathsetmacro\j{\i+1};
	\draw (a\i)--(a\j);
	\fill (a\i) circle (1.5pt);		
}
\node [left] at (a2) {$q$};
\node [left] at (a3) {$x$};
\node [left] at (a4) {$\bl{c}$};
\node [right] at (a5) {$\bl{2i}$};
\node [right] at (a6) {$\bl{w}$};
\node [right] at (a1) {$z$};
\end{tikzpicture}
\caption{}
\label{il:337B}
\end{subfigure}
\hfill
\begin{subfigure}[b]{.19\textwidth}
\centering
\begin{tikzpicture}[scale=.95]
\foreach \i in {1,...,7} {
	\coordinate (a\i) at (\i*60:1);
}
\foreach \i in {1,...,6} {
	\pgfmathsetmacro\j{\i+1};
	\draw (a\i)--(a\j);
	\fill (a\i) circle (1.5pt);		
}
\node [left] at (a2) {$q$};
\node [left] at (a3) {$\re{a}$};
\node [left] at (a4) {$\bl{w}$};
\node [right] at (a5) {$\gr{b}$};
\node [right] at (a6) {$\re{u}$};
\node [right] at (a1) {$z$};
\end{tikzpicture}
\caption{}
\label{il:340A}
\end{subfigure}
\hfill
\begin{subfigure}[b]{.19\textwidth}
\centering
\begin{tikzpicture}[scale=.95]
\foreach \i in {1,...,7} {
	\coordinate (a\i) at (\i*60:1);
}
\foreach \i in {1,...,6} {
	\pgfmathsetmacro\j{\i+1};
	\draw (a\i)--(a\j);
	\fill (a\i) circle (1.5pt);		
}
\node [left] at (a2) { $q$};
\node [left] at (a3) {$\gr{b}$};
\node [left] at (a4) {$\bl{w}$};
\node [right] at (a5) {$\re{a}$};
\node [right] at (a6) {$\gr{v}$};
\node [right] at (a1) {$z$};
\end{tikzpicture}
\caption{}
\label{il:340B}
\end{subfigure}
\hfill
\begin{subfigure}[b]{.19\textwidth}
\centering
\begin{tikzpicture}[scale=.95]
\foreach \i in {1,...,7} {
	\coordinate (a\i) at (\i*60:1);
}
\foreach \i in {1,...,6} {
	\pgfmathsetmacro\j{\i+1};
	\draw (a\i)--(a\j);
	\fill (a\i) circle (1.5pt);		
}
\node [left] at (a2) {$q$};
\node [left] at (a3) {$\bl{c}$};
\node [left] at (a4) {$\gr{v}$};
\node [right] at (a5) {$\re{a}$};
\node [right] at (a6) {$\bl{w}$};
\node [right] at (a1) {$z$};
\end{tikzpicture}
\caption{}\label{il:340C}
\end{subfigure}
\vskip -.2cm
\caption{Hexagon arguments.}
\label{fig:hex}
\end{figure} 
		
		Suppose next that $x,\re{0},z\in \Nn(t)$. By Claim~\ref{c:333}, $t$ is a twin 
		of $\re{a}$. For simplicity we may assume $t=\re{a}$. Now Claim~\ref{c:334} 
		shows that $z$ is a twin of some $\re{j}\in \Gr$ and a final application of 
		Claim~\ref{c:333} reveals that~$q$ is again a twin of $\re{a}$. 
		Summarising the discussion so far, we have thereby completed the 
		case $\re{u}\not\in \Nn(q)$. 
		
		A similar argument based on the hexagon in Figure~\ref{il:337B}
		shows that if $\bl{w}\not\in \Nn(q)$, then $q$ is a twin of $\bl{c}$. 
		In the only remaining case, $\re{u},\bl{w}\in \Nn(q)$, Claim~\ref{c:331}
		implies that~$q$ is a twin of $\gr{b}$ or~$y$.
	\end{proof}
	
	This has the following consequence. 
	
	\begin{claim}\label{c:y}
		If $\mu=0$ and some twin $y'$ of $y$ is adjacent to $q$, then~$q$ is 
		a twin of $\re{u}$, $\bl{v}$, $\gr{w}$, or $x$. 
	\end{claim}
		
	\begin{proof}
		For $y=y'$ this can be seen by applying the automorphism $\sigma$ to 
		the previous claim.
		By Lemma~\ref{f:24} the general case follows.	
	\end{proof}

	Our next major goal is an attachment lemma for neighbours of~$\re{a}$ 
	(cf. Claim~\ref{c:3400}) and again we commence with a special case. 
	
	\begin{claim}\label{c:abc}
		If $\re{a}\in \Nn(q)$ and, moreover, $\{\gr{b}, \bl{c}\}\cap \Nn(q)\ne\vn$,  
		then~$q$ is a twin of $\gr{v}$, $\bl{w}$, or $x$. 
	\end{claim}
		
	\begin{proof}
		We will only display the argument for the case $\re{a}, \gr{b}\in \Nn(q)$,
		the other case being similar. Using the \re{red} auxiliary path 
		$\pi_{\re{0}}=\re{0}-\gr{i}-\bl{2i}-\re{1}$ we form the graph 
		$\grot(\pi_{\re{0}})$. 
		By Corollary~\ref{cor:311} there is a $\grot(\pi_{\re{0}})$-twin of $c$
		or $\bl{2i}$ adjacent to $q$.
		
		Suppose first that $c'\in\Nn(q)$ holds for some $\grot(\pi_{\re{0}})$-twin~$c'$ 
		of~$c$. Claim~\ref{c:331} tells us that $c'$ is a real twin of $c$ and due to 
		Lemma~\ref{l:318}\ref{it:318b} $q$ is a $\grot^{\mu\nu}_i(c')$-twin of $x$.
		So by Lemma~\ref{f:24} $q$ is a twin of $x$.
		
		It remains to consider the case that some $\grot(\pi_{\re{0}})$-twin $(2i)'$
		of $2i$ is adjacent to $q$. For~${j=\re{1}}$ Claim~\ref{c:332}
		shows that $(2i)'$ is actually a real twin of $2i$ and, therefore, 
		Claim~\ref{c:330} is applicable to $\grot^{\mu\nu}_i((2i)')$.
		Together with Lemma~\ref{f:24} we conclude that $q$ is a twin of $\bl{w}$.
	\end{proof}
			
	\begin{claim}\label{c:3400}
		If $q$ is adjacent to $\re{a}$, then it is a twin of some vertex 
		of $\grot_i^{\mu\nu}$.
	\end{claim}	

	\begin{proof}
		If $\re{u}\in\Nn(q)$ the desired conclusion can be drawn from Claim~\ref{c:334}.
		Assuming $\re{u}\not\in\Nn(q)$ from now on we take a common neighbour $z$ 
		of $q$, $\re{u}$ and form the hexagon shown in Figure~\ref{il:340A}.
		
		If there is a common neighbour~$t$ of $q$, $\re{u}$, $\bl{w}$, then 
		Claim~\ref{c:331} shows that $t$ is a twin of $\gr{b}$, or $\mu=0$ 
		and $t$ is a twin of $y$. In the latter case we use Claim~\ref{c:y}
		and in the former case we appeal to Claim~\ref{c:abc} and Lemma~\ref{f:24}.
		
		Since $G$ satisfies~\Dp{2}, it only remains to consider the case 
		that some common neighbour~$t$ of $\re{a}$, $\gr{b}$, $z$ exists. 
		By Claim~\ref{c:abc}, $t$ is a twin of $\bl{w}$ or $x$. Due to 
		Lemma~\ref{f:24} we can assume, for simplicity, that $t\in\{\bl{w}, x\}$. 
		Now the Claims~\ref{c:331} and~\ref{c:337} tell us that $z$ is a twin 
		of~$\gr{b}$,~$\bl{c}$ or~$y$. Finally, Claim~\ref{c:y} or 
		Claim~\ref{c:abc} and Lemma~\ref{f:24} complete the proof. 
	\end{proof}
	
	We also need a version of this claim with~$\gr{b}$ or~$\bl{c}$ instead of~$\re{a}$,
	which we prepare as follows. 
	
	\begin{claim}\label{clm:bc}
		If $q$ is adjacent to $\gr{b}$ and $\bl{c}$, then it is a twin of $\re{u}$ or $x$.
	\end{claim}
	
	\begin{proof}
		If $\re{a}\in\Nn(q)$, then Claim~\ref{c:3400} shows that $q$ is a twin of $x$.
		Otherwise, we take a common neighbour $z$ of $\re{a}$, $q$ and deduce from 
		Claim~\ref{c:3400} that $z$ is a twin of some vertex $t\in V(\grot^{\mu\nu}_i)$
		adjacent to $\re{a}$. Since neither $q\gr{b}z$ nor $q\bl{c}z$ is a triangle 
		in $G$, we have $z\not\in (\Nn(b)\cup\Nn(c))$. Consequently, $t$ is non-adjacent 
		to $\gr{b}$ and $\bl{c}$ and, altogether, only the possibility $t\in \Gr$ remains. 
		Now Claim~\ref{c:330} shows that $q$ is a $\grot^{\mu\nu}_i(t)$-twin of~$\re{u}$ 
		and another application of Lemma~\ref{f:24} concludes the argument. 
	\end{proof}
	
	Next we can repeat the proof of Claim~\ref{c:3400} with the hexagon in 
	Figure~\ref{il:340B} or~\ref{il:340C}, thus obtaining the following statement. 
	
	\begin{claim}\label{c:ende}
		If $\gr{b}$ or $\bl{c}$ is in $\Nn(q)$, then $q$ is a twin of some vertex 
		of $\grot_i^{\mu\nu}$. \qed
	\end{claim}
		
	Now the attachment lemma is clear. Given any $q\in V(G)$ we can either apply 
	Claim~\ref{c:337} directly (if $qx\in E(G)$), or there is a common neighbour $z$ 
	of $q$ and $x$, which then has to be a twin of $\re{a}$, $\gr{b}$, $\bl{c}$, or $y$.
	As usual, Lemma~\ref{f:24} allows us to assume that, actually, $z$ is one of those 
	four vertices. Depending on $z$ we now use Claim~\ref{c:3400},~\ref{c:ende}, 
	or~\ref{c:y}. 
\end{proof}

Let us end this section by pointing out that owing to the twin lemma 
(cf.\ Lemma~\ref{lem:vtwin})	and the attachment lemma (cf.\ Lemma~\ref{l:329})
Lemma~\ref{lem:2407} implies Theorem~\ref{thm:1.3}\ref{it:1.3b}.

\section{Concluding remarks}\label{sec:final}

\subsection{Finite graphs}
Whenever we appealed to the property~\Dp{4} in the proof of Theorem~\ref{thm:D},
the list of $3m$ vertices $x_1, \dots, x_{3m}$ we specified had no independent 
subset $\{x_i\colon i\in I\}$ such that $|I|\ge m+2$. Originally we thought that our 
intended application to Ramsey-Tur\'an theory, i.e., the proof of~\eqref{eq:LPR},
required that we establish Theorem~\ref{thm:D} with such `restricted applications' 
of~\Dp{4} only. While it turned out later that this extra caution could be avoided,
we would still like to record the stronger statement for potential future references. 

Let us say for a positive integer~$k$ that a graph $G$ has the property~\Qp{k} 
if for every $m\in [k]$ and every sequence $x_1, \dots, x_{3m}$ of vertices of~$G$ 
there is an index set $I\subseteq [3m]$ such that $U=\{x_i\colon i\in I\}$ 
is independent and either $|I|\ge m+2$, or $|I|=m+1$ and $U$ has a common neighbour. 

\begin{thm}\label{thm:Q}
	A maximal triangle-free graph satisfies \Qp{4} if and only if 
	it is a blow-up of either an Andr\'asfai or a Vega graph. \qed
\end{thm}

Recall that by Lemma~\ref{l:31} the members of $\fD_4$ contain no induced cubes.
We also found several other forbidden induced subgraphs for the class~$\fD_4$,
such as the graph $N$ depicted in Figure~\ref{fig:NN},
but so far we did not complete our analysis of the situation. 

\begin{conj}
	There exists a finite family $\ccF$ of graphs such that $\fD_4$ is the class of 
	maximal triangle-free graphs with at least two vertices not possessing induced 
	subgraphs in~$\ccF$. 
\end{conj}

Let us point out that a somewhat similar result for the class $\fA$ of maximal 
triangle-free, $\grot$-free graphs and the forbidden family $\{C_6\}$ is established 
in Section~\ref{sec:A}. Thus a solution of the above problem might lead to a 
different (albeit longer) proof of Theorem~\ref{thm:D}. 

Another problem suggested by Theorem~\ref{thm:D} is whether the assumption \Dp{4} 
is really necessary or whether a more elaborate argument would show that \Dp{3}
suffices. It is probably not very difficult to rule this out, but we lacked the 
energy for doing so. 

\begin{conj}\label{c:34}
	There is a finite maximal triangle-free graph $G$ satisfying \Dp{3} but not \Dp{4}.
\end{conj}

The last finitary problem we would like to mention is the well-known question to 
characterise the maximal triangle-free graphs $G$ on $n$ vertices 
satisfying $\delta(G)\ge n/3$. 
This family does not consist exclusively of blow-ups of Andr\'asfai and Vega graphs, 
since it contains, for instance, the Cayley graphs associated to 
\[
	\ZZ/6k\ZZ \qand \{\pm k, \pm (k+1), \dots, \pm (2k-1)\}.
\]
This graph contains the induced hexagon $0-2k-3k-4k-5k-0$ but no vertex with three 
neighbours on this hexagon and, therefore, it violates~\Dp{2}. Another graph interesting 
in this context is shown in Figure~\ref{fig:41} (see also Figure~\ref{fig:lem314C} for 
a less symmetric drawing of the same graph). As all these additional examples
are $(n/3)$-regular, one may wonder whether the following is true. 

\begin{quest}
	Let $G$ be a maximal triangle-free graph on $n$ vertices such 
	that $\delta(G)\ge n/3$,
	but at least one vertex of $G$ has degree larger than $n/3$. 
	Does it follow that~$G$ is a blow-up of either an Andr\'asfai or a Vega graph? 
\end{quest}

\begin{figure}[ht]
\centering
\begin{tikzpicture}[scale=.7]
	\foreach \i in {1,...,8} {
		\coordinate (a\i) at (22.5+\i*45:1);
	}
	\coordinate (b1) at (-2.2,1.8);
	\coordinate (b2) at (2.2,1.8);
	\coordinate (b3) at (2.2,-1.8);
	\coordinate(b4) at (-2.2,-1.8);
	\foreach \i in {1,...,8} {
		\fill (a\i) circle (1.5pt);		
	}
	\foreach \i in {1,...,4} {
		\fill (b\i) circle (1.5pt);		
	}
	\draw (b1) -- (b2);
	\draw (b4) -- (b3);
	\draw (a2)--(b1)--(b4)--(a5);
	\draw (a1) -- (a4) -- (a7) -- (a2) -- (a5) -- (a8) -- (a3) -- (a6) -- (a1);
	\draw (b1) -- (a3) -- (a7) -- (b3);
	\draw (b4) -- (a4) -- (a8) -- (b2);
	\draw (a1)--(b2)--(b3)--(a6);
	\draw (a1) -- (a5);
	\draw (a2) -- (a6);
\end{tikzpicture}
\caption{A triangle-free graph $G$ with $v(G)= 12$ and $\alpha(G)=4$.}
\label{fig:41}
\end{figure}
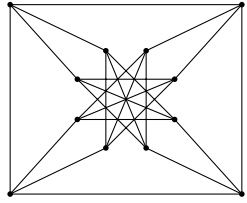

\subsection{Infinite graphs}
One may wonder whether results analogous to Theorem~\ref{thm:1.3} hold for infinite 
graphs as well. Before we discuss this matter we present three examples of infinite 
maximal triangle-free graphs, which have the property~\Dp{k} for every $k\ge 1$. 

\medskip

{\bf I. The circular Andr\'asfai graph $G_\omega$.}

\smallskip

Let $\xi\in S^1\subseteq \CC$ be a complex number of modulus $1$ which is not a root
of unity. Consider the graph $G_\omega$ with vertex set $V=\{\xi^n\colon n\in \NN\}$ in 
which two vertices are adjacent if their distance exceeds $\sqrt 3$. Since~$V$ 
is dense in $S^1$, the graph $G_\omega$ is maximal triangle-free. Moreover, for every 
sequence of $3m$ vertices from $V$ there is an open arc~$I\subseteq S^1$ of 
length~$2\pi/3$ containing~$m+1$ of them. Consequently,~$G_\omega$ satisfies~\Dp{k} 
for every $k\ge 1$.  

\medskip

{\bf II. Generalised Andr\'asfai graphs $\G_X$.}

\smallskip

Given a dense linear order $(X, \le)$ without minimal or maximal elements 
we define the generalised Andr\'asfai graph~$\G_X$ to be the graph with vertex set
\[
	V(\G_X)
	=
	\{\re{r_x}\colon x\in X\}
	\cup 
	\{\gr{g_x}\colon x\in X\}
	\cup 
	\{\bl{b_x}\colon x\in X\}
\] 
and all edges
\begin{enumerate}
	\item[$\bullet$] $\re{r_x}\gr{g_x}$, $\gr{g_x}\bl{b_x}$, where $x\in X$,
	\item[$\bullet$] as well as $\re{r_x}\gr{g_y}$, $\gr{g_x}\bl{b_y}$, 
		$\bl{b_x}\re{r_y}$, where $x<y$.
\end{enumerate}
It can easily be checked that $\G_X$ is maximal triangle-free and that 
if $x_1<x_2<\dots<x_k$ are in~$X$, then the vertices
\[
	\{\re{r_{x_1}},\re{r_{x_2}},\dots,\re{r_{x_k}}\}
	\cup  
	\{\gr{g_{x_1}},\gr{g_{x_2}},\dots,\gr{g_{x_k}}\}
	\cup  
	\{\bl{b_{x_1}},\dots,\bl{b_{x_{k-1}}}\}
\] 
span a copy of $\Gamma_k$ in $\G_X$. 
This fact immediately implies that $G_X$ has the property~\Dp{k} for 
every~$k\ge 1$.  

\medskip\goodbreak

{\bf III. Generalised Vega graphs $\grot^{\mu 0}_X$.}

\smallskip

Starting from the graph $\G_X$ defined in the previous example we can construct 
two generalised Vega graphs $\grot^{00}_X$ and $\grot^{10}_X$. To this end we 
take an external hexagon $\re{a}\gr{v}\bl{c}\re{u}\gr{b}\bl{w}$ and, as usual, 
we connect 
\begin{enumerate}
	\item[$\bullet$] $\re{a}$, $\re{u}$ to $\{\re{r_x}\colon x\in X\}$; 
	\item[$\bullet$] $\gr{b}$, $\gr{v}$ to $\{\gr{g_x}\colon x\in X\}$; 
	\item[$\bullet$] and $\bl{c}$, $\bl{w}$ to $\{\bl{b_x}\colon x\in X\}$. 
\end{enumerate}
Next, we join another vertex $x$ to $\re{a}$, $\gr{b}$, $\bl{c}$,
thereby obtaining~$\grot^{10}_X$. Finally~$\grot^{00}_X$ has a further vertex $y$ 
adjacent to $\re{u}$, $\gr{v}$, $\bl{w}$, and $x$.
   
As none of these graphs is a blow-up of a finite Andr\'asfai of Vega graph, the most 
na\"ive extension of Theorem~\ref{thm:1.3} to infinite graphs is false. 

\begin{thm}\label{thm:inf}
	There exists an infinite maximal triangle-free graph satisfying~\Dp{k} for 
	every~$k\ge 1$ that fails to be a blow-up of any finite graph. \qed
\end{thm}

We are optimistic, however, that the following `local version' of Theorem~\ref{thm:1.3} 
holds for infinite graphs.

\begin{conj}\label{c:strong}
	For every maximal triangle-free graph $G$ with property~\Dp{4} and every finite set 
	of vertices $W\subseteq V(G)$ there exists a finite set $U$ with 
	$W\subseteq U\subseteq V(G)$ which spans a blow-up of either an Andr\'asfai 
	or a Vega graph in $G$.
\end{conj}

Notice that this is true for finite graphs, where one just needs to take $U=V(G)$, 
and for the infinite graphs $G_\omega$, $\G_X$, $\grot_X^{\mu0}$ constructed above. 
Moreover, Conjecture~\ref{c:strong} implies that the following statement holds 
for $\ell=4$. 

\begin{conj}
	There exists a natural number $\ell$ such that each maximal triangle-free graph 
	satisfying~\Dp{\ell} has the property~\Dp{k} for every~$k\ge 1$.
\end{conj}

Next we would like to address the infinitary version of Conjecture~\ref{c:34}.
Let us recall that Henson~\cite{He} constructed a countable homogeneous triangle-free 
graph, which he denoted by $U_3$. Its main property is that for all disjoint finite 
sets $A, B\subseteq V(U_3)$ such that $A$ is independent in $U_3$ there is a vertex $v$
adjacent to the vertices in $A$ and non-adjacent to the vertices in $B$. 
In particular, $U_3$ is a maximal triangle-free graph containing all 
finite or countably infinite triangle-free graphs as induced subgraphs. 

\begin{thm}
	Henson's graph $U_3$ satisfies \Dp{3} but not \Dp{4}.
\end{thm}

\begin{proof}
	It is a well known elementary Ramsey theoretic fact that for all $m\in\{1, 2, 3\}$ 
	the partition relation $3m\longrightarrow (3, m+1)$ asserting that 
	triangle-free graphs on $3m$ vertices contain independent sets of size $m+1$ holds. 
	Consequently, for every sequence 
	$x_1, \dots, x_{3m}$ of vertices of $U_3$ there is a set $I\subseteq [3m]$ of 
	size $m+1$ such that $A=\{x_i\colon i\in I\}$ is independent. As all finite 
	independent subsets of $V(U_3)$ have common neighbours, this proves that $U_3$ 
	satisfies~\Dp{3}.
	
	On the other hand, as $U_3$ contains the graph depicted in Figure~\ref{fig:41}, 
	it has twelve vertices no five of which possess a common neighbour. Thus $U_3$
	does not have the property~\Dp{4}.
\end{proof} 
	 
Finally, we would like to point out that there is an infinite minimum-degree
version of the Brandt-Thomass\'e theorem. The result that follows concerns graphs
on $\NN$. Given such a graph $G$ and a vertex $v\in V(G)$ we write 
\[
	\underline{\deg}(v)=\liminf_{n\to \infty} \frac{|\{1,2,\dots,n\}\cap\Nn(v)|}n\,.
\]

\begin{thm}\label{thm:ramsey}
	Let $G$ be a maximal triangle-free graph on $\NN$. 
	If $\inf\{\underline{\deg}(v)\colon v\in \NN\}>1/3$,
	then $G$ is a blow-up of either an Andr\'asfai or a Vega graph.
\end{thm}

\begin{proof}
	Choose $\eps>0$ such that $\underline{\deg}(v)>1/3+\eps$ holds for every $v\in \NN$.
	Given any sequence $x_1, \dots, x_{3m}$ of vertices of $G$, there is some positive 
	integer $n$ such that for every $i\in [3m]$ 
	we have $|\Nn(x_i)\cap [n]|\ge (1/3+\eps)n$. Now a counting argument leads to 
	some $y\in [n]$ for which the
	set $I=\{i\in [3m]\colon x_iy\in E(G)\}$ has at least the size $m+3m\eps$. 
	Consequently, $G$ satisfies~\Dp{k} for every $k\ge 1$ and, in particular,~$G$ 
	satisfies~\Dp{4}.
	 	
	The only moment in the proof of Theorem~\ref{thm:1.3} employing the finiteness 
	of~$G$ is that at the very end we choose a maximal Andr\'asfai or Vega 
	graph contained in it. Thus it remains to show that there exists some constant $C$
	such that every Andr\'asfai or Vega subgraph of~$G$ has at most~$C$ vertices. 
	
	If for some $m\in\NN$ there is a copy of $\G_{m+1}$ in $G$ 
	and $U=\{x_1, \dots, x_{3m}\}$ contains $3m$ distinct vertices of this copy, 
	then every independent subset of $I\subseteq U$ 
	satisfies $|I|\le \alpha(\G_{m+1})=m+1$ and, therefore, the conclusion of our 
	first paragraph yields $m\le (3\eps)^{-1}$. This shows that the Andr\'asfai 
	subgraphs of $G$ are indeed bounded. For Vega graphs the argument is similar, 
	using in addition that each $\grot^{\mu\nu}_i$ has for $\kappa=9i-(6+\mu+\nu)$ 
	a blow-up on $3\kappa-1$ vertices with independence number $\kappa$ (see~\cite{BP}
	or~\cite{Vega}). 
\end{proof}

\end{document}